\documentclass[11pt,reqno,UTF8,twoside]{amsart}
\usepackage{amssymb}
\usepackage[dvips]{epsfig}
\usepackage{booktabs}
\usepackage{multirow}
\usepackage{bbm}
\usepackage{xltabular}
\usepackage{tabularx}
\usepackage{makecell} 
\usepackage{ragged2e} 
\usepackage{graphicx}
\usepackage{color}
\usepackage{amsmath}
\allowdisplaybreaks
\usepackage{amssymb}
\usepackage{amsfonts}
\usepackage{geometry}
\usepackage{xstring}
\usepackage{amsthm}
\usepackage[normalem]{ulem}    
\usepackage{bm}
\usepackage{cite}
\usepackage{tikz}
\makeatletter
\@namedef{subjclassname@2020}{\textup{2020} Mathematics Subject 
	Classification}
\makeatother

\newcommand{\Z}{\mathbb{Z}}

\newcommand{\R}{\mathbb{R}}

\newcommand{\E}{\mathbb{E}}

\newcommand{\spec}{{\rm Spec}}

\newcommand{\bfv}{\mathbf{v}}
\newcommand{\bfw}{\mathbf{w}}
\newcommand{\event}{{\rm Event}}

\newcommand{\tq}{\tilde{q}}
\newcommand{\ts}{\tilde{s}}
\newcommand{\tr}{\tilde{r}}
\renewcommand{\tt}{\tilde{t}}
\newcommand{\tm}{\widetilde{m}}

\newcommand{\barLambda}{\overline{\Lambda}}
\newcommand{\length}{{\rm length}}

\newcommand{\bfa}{\mathbf{a}}
\newcommand{\bfe}{\mathbf{e}}

\renewcommand{\P}{\mathbb{P}}
\newcommand{\spread}{{\rm Spread}}
\newcommand{\baromega}{\overline{\omega}}
\newcommand{\hatomega}{\widehat{\omega}}

\newcommand{\mcA}{\mathcal{A}}
\newcommand{\mcB}{\mathcal{B}}
\newcommand{\mcD}{\mathcal{D}}
\newcommand{\mcC}{\mathcal{C}}
\newcommand{\mcE}{\mathcal{E}}
\newcommand{\mcF}{\mathcal{F}}

\newcommand{\mcP}{\mathcal{P}}
\newcommand{\mcS}{\mathcal{S}}
\newcommand{\mcI}{\mathcal{I}}
\newcommand{\mcQ}{\mathcal{Q}}
\newcommand{\mcR}{\mathcal{R}}
\newcommand{\mcK}{\mathcal{K}}
\newcommand{\mcT}{\mathcal{T}}
\newcommand{\mcX}{\mathcal{X}}
\newcommand{\mcL}{\mathcal{L}}
\newcommand{\mcM}{\mathcal{M}}
\newcommand{\mcY}{\mathcal{Y}}
\newcommand{\mcZ}{\mathcal{Z}}
\newcommand{\mcO}{\mathcal{O}}
\newcommand{\mcG}{\mathcal{G}}
\newcommand{\mcW}{\mathcal{W}}

\newcommand{\whn}{\widehat{n}}

\newcommand{\wtlambda}{\widetilde{\lambda}}

\newcommand{\barB}{\overline{B}}

\newcommand{\wtSj}{\widetilde{S}^{(j)}}

\newcommand{\wtSjminusone}{\widetilde{S}^{(j-1)}}
\newcommand{\smallSj}{s^{(j)}}
\newcommand{\smallSjminusone}{s^{(j-1)}}
\newcommand{\diffSj}{\Delta {S}^{(j)}}
\newcommand{\partialSj}{\partial {S}^{(j)}}
\newcommand{\partialSjminusone}{\partial {S}^{(j-1)}}

\newcommand{\wtGjminusone}{\widetilde{G}^{(j-1)}}

\newcommand{\partialside}{\partial^{\text{side}}}

\newcommand{\dist}{{\rm dist}}
\newcommand{\diam}{{\rm diam}}

\newtheorem{thm}{Theorem}[section]
\newtheorem{lem}[thm]{Lemma}
\newtheorem{prop}[thm]{Proposition}
\newtheorem{cor}{\bf Corollary}[section]
\newtheorem{Def}{Definition}[section]

\newtheorem{rmk}{\bf Remark}[section]

\theoremstyle{definition}

\newtheorem{ex}[thm]{Example}

\newtheorem{claim}[thm]{\bf Claim}
\newtheorem{prob}[thm]{Problem}

\numberwithin{equation}{section}

\usepackage[colorlinks=true,pdfstartview=FitV,linkcolor=magenta,citecolor=cyan]{hyperref}


\usepackage{caption}

\begin{document}
\title[hierarchical Anderson Bernoulli Model]{Anderson Localization for the hierarchical Anderson-Bernoulli model on $\Z^d$}

\author[Liu]{Shihe Liu}
\address[S. Liu] {School of Mathematical Sciences,
Peking University,
Beijing 100871,
China}
\email{2301110021@stu.pku.edu.cn}

\author[Shi]{Yunfeng Shi}
\address[Y. Shi] {School of Mathematics,
Sichuan University,
Chengdu 610064,
China}
\email{yunfengshi@scu.edu.cn}

\author[Zhang]{Zhifei Zhang}
\address[Z. Zhang] {School of Mathematical Sciences,
Peking University,
Beijing 100871,
China}
\email{zfzhang@math.pku.edu.cn}

\begin{abstract}
In this paper, we prove Anderson localization for a hierarchical Anderson-Bernoulli model on $\mathbb{Z}^d$ with arbitrary $d\geq 1$, where the potential is characterized by a geometric hierarchical structure combined with fluctuations induced by independent and identically distributed (i.i.d.) Bernoulli random variables. To the best of our knowledge, this is the first localization result for a model with i.i.d. two‑point Bernoulli potential in dimensions $d\geq 4$.  Notably, our approach is independent of any unique continuation result, instead combining a weak (one-dimensional) transversality estimate with a martingale argument. As such, we hope that our method may shed new light on the eventual proof of Anderson localization for the standard Anderson-Bernoulli model  on $\Z^d$ with $d\geq 4.$ Our method is also applicable to proving a probabilistic unique continuation result on $\Z^d$.
\end{abstract}

\maketitle

\tableofcontents

\section{Introduction}

\subsection{Background and related results}

The study of Anderson localization (i.e., pure point spectrum with exponentially decaying eigenfunctions) for Schr\"odinger operator on $\Z^d$
with independent and identically distributed (i.i.d.) random potentials has attracted considerable attention over the decades. For the Anderson model—where the random potential of the Schr\"odinger operator  is   is uniformly distributed over an interval—Anderson conjectured in \cite{And58} that sufficiently strong disorder would induce eigenstates to be exponentially localized within a finite region. This conjecture was later independently proven in \cite{FMSS85, DLS85, SW86}, building on the pivotal Green’s function estimates developed by Fr\"ohlich-Spencer \cite{FS83}.  In that work, the authors devised a novel multiscale analysis framework to address resonant phenomena, with a core focus on the {\bf hierarchical structure of resonant blocks}—subsystems where Green’s functions cannot be tightly controlled (e.g., lacking off-diagonal exponential decay bounds). Specifically, they demonstrated that in the regime of strong disorder or extreme energies, any finite subset can be decomposed, with high probability, into a union of non-resonant blocks of varying sizes $d_k$ that satisfy the scaling relation $d_{k+1}=d_k^{\alpha},\ \alpha>1$, known as Newtonian scales. The probabilistic bounds underpinning this decomposition were derived using the a priori Wegner estimate \cite{Weg81}—which requires the random potential to have at least a continuous distribution—and a suite of combinatorial entropy estimates.

Because the Green’s function estimates in \cite{FS83} were established for a fixed energy, the result only confirmed the absence of diffusion (or equivalently, vanishing conductivity) rather than full Anderson localization. Later, Martinelli and Scoppola \cite{MS85} integrated the estimates from \cite{FS83} with a generalized eigenvalue argument (namely, Shnol’s theorem) to rigorously establish the absence of absolutely continuous spectrum. Notably, to elucidate the physical mechanism driving localization, Jona-Lasinio, Martinelli, and Scoppola \cite{JLMS84,JLMS85} introduced the concept of a {\bf deterministic hierarchical potential}, constructed by iterating a basic geometric building block across an ascending sequence of Newtonian-type length scales. This innovation enabled the first formal connection between the quantum tunneling analysis of semiclassical limit Schrödinger equations \cite{JLMS81} and the theory of Anderson localization.  In particular, they proved that for the Schr\"odinger operator on $\Z^d$ with a deterministic symmetric hierarchical potential, even weak perturbations by i.i.d. random potentials (with \textbf{absolutely continuous distributions}) trigger a delocalization-to-localization transition for the low-lying spectrum. This hierarchical Anderson model can be regarded as a simplified ``toy mode'' of the standard Anderson model—one that is, in principle, physically realizable \cite{JLMS84,WSPZ89}—as it presupposes that resonant sets inherently exhibit a geometrically hierarchical structure. From a quantum tunneling perspective, this hierarchical architecture dictates the spatial arrangement of potential barriers and wells in the system. Since localization fundamentally implies that a quantum particle cannot escape a fixed region over long timescales (i.e., it remains confined within a potential well without undergoing tunneling), localization becomes increasingly probable as barriers grow higher and wider. In fact, the authors of \cite{FMSS85} (see \cite{MS87} for a comprehensive exposition) successfully proved Anderson localization for the standard Anderson model by showing that typical configurations of its random potential share a similar hierarchical structure to those in the model proposed in \cite{JLMS85}. It is worth noting that the hierarchical structure of resonant sets also manifests in the context of deterministic quasi-periodic Schrödinger operators, as documented in \cite{FSW90,BGS02,Bou07,CSZ23}.

However, the results discussed above primarily concern Schrödinger operators on $\Z^d$ with random potentials possessing bounded densities—even the work in \cite{JLMS85} only established localization for random potentials with bounded densities—and we refer readers to \cite{Kir08, AW15, AM93} for additional findings pertaining to continuously distributed random potentials. In that framework, the proof of the Wegner estimate, which quantifies the probability of resonance events, is standard and straightforward. By contrast, our focus is directed toward the scenario where the i.i.d. random potential follows a two-point Bernoulli distribution, namely the well-known Anderson-Bernoulli model (ABM). In the ABM, the measure concentration effects induced by the potential invalidate the traditional approaches to deriving the Wegner estimate, thereby posing a major obstacle to proving localization. Indeed, in the case of Bernoulli potentials, probabilistic bounds on the Green’s function cannot be obtained by perturbing a single random variable; instead, bounds on ``rare events'' must be established by accounting for the dependence of eigenvalues on a large ensemble of random variables.

For the one-dimensional ABM, this challenge can be circumvented via methods specific to one-dimensional lattices, such as transfer matrix-based methods and the Furstenberg–LePage approach (cf. \cite{CKM87,SVW98,DSS02}). In higher dimensions, however, the transfer matrices essential to one-dimensional dynamical systems techniques are generally unavailable, rendering multiscale analysis the more viable methodological candidate. The foundational work addressing the localization problem for the alloy-type ABM in higher dimensions was pioneered by Bourgain in \cite{Bou04}, where he developed a novel multiscale iteration scheme for Green’s function estimates incorporating a key {\bf free sites argument} and proved localization at the spectral edges. The methodology introduced in \cite{Bou04} was further advanced by Bourgain and Kenig in \cite{BK05}, in which they achieved a breakthrough by establishing localization at the spectral edges for the ABM defined on $\mathbb{R}^d$.

In \cite{Bou04} and \cite{BK05} (see \cite{GK13} for refinements), transversality served as a critical ingredient for proving localization. In the former work, transversality is inherent to the structure of the alloy-type potential; in the latter, it relies on a unique continuation theorem for elliptic equations in $\R^d$ (cf. \cite[Lemma 3.15]{BK05}). However, as noted by Jitomirskaya \cite{Jit07}, an analogous unique continuation statement fails on the lattice  $\Z^d$—a limitation that posed a major obstacle to proving localization for the discrete ABM. The problem of establishing localization at the spectral edges for the ABM on  $\Z^d$ with $d\ge 2$
remained open until the influential work of Ding and Smart \cite{DS20}, in which they established a probabilistic variant of the unique continuation result on $\Z^2$, partially building on the techniques of \cite{BLMS}, and thereby proved Anderson localization for the ABM on $\Z^2$. Subsequently, Li and Zhang \cite{LZ22} extended this two-dimensional framework to prove a unique continuation result and resolve the localization problem at the spectral edges for the ABM on $\Z^3$, supplying the necessary transversality. Furthermore, the transversality result on $\Z^2$ has been leveraged to establish localization for the ABM on $\Z^2$
under strong disorder \cite{Li22}. For localization results concerning alloy-type ABM models with long-range hopping, we also refer readers to \cite{LSZ25}.

{\bf The problem of establishing localization at the spectral edges for the ABM on $\Z^d$ with $d\geq 4$, and even for the hierarchical ABM as investigated in \cite{JLMS85}, remains open}. (Some attempts have been made to prove localization for higher-dimensional Schrödinger operators with singular potentials in \cite{Imb21}, where the random potential is discrete and assumes $N\gg1$ distinct values.)   A fundamental obstacle lies in the absence of an appropriate characterization of transversality on higher-dimensional lattices. In this article, we establish localization in a neighborhood of the bottom of the spectrum for the higher-dimensional hierarchical ABM in arbitrary dimensions. We consider a significant contribution of this work to be our ability to prove localization using a weaker form of transversality than that required by unique continuation theorems, as well as our uncovering of a { special martingale structure} inherent to higher-dimensional lattices in the course of the proof. To the best of our knowledge, this constitutes the first localization result for a model with an i.i.d. two-point Bernoulli potential in dimensions $d\geq 4$.

Finally, we emphasize that \textbf{the method developed in this paper can also be applied to proving a probabilistic unique continuation result (where the probability depends on the initial data) on $\Z^d$ for arbitrary $d\geq 2$} (see Theorem \ref{UC for d arbitrary}). 
In its proof, the same martingale with a special structure reemerges—a fact we believe is of independent interest.

\subsection{The deterministic hierarchical potential}

In this section, we present the precise definitions of our core models. This class of models was originally introduced in \cite{JLMS84, JLMS85} for the purpose of investigating the quantum tunneling behavior of quantum particles.

Let $d_0>1,\alpha>1$ and consider the {\bf Newtonian scales}
\begin{equation}
    d_{k+1}=d_k^{\alpha} ,\  k\geq 0.
\end{equation}
Let $|\cdot|$ and $|\cdot|_1$ be the $\ell^\infty$-norm and $\ell^1$-norm on $\Z^d$, respectively. Denote for $j\in\Z^d$ and $L\geq 0, $
\[Q_L(j)=\{x\in \Z^d:\ |x-j|\leq L\}.\]
Denote by $\lfloor\cdot\rfloor$ by the integer part of a real number.  Throughout this paper, we let   
\begin{equation}\label{k-th block}
  \Lambda_k(j):=\left\{ x\in \Z^d:\  |x-j|\leq 2\lfloor 3d_k \rfloor \right\},\  k\geq 0.
\end{equation}
In particular,  write  $\Lambda_k=\Lambda_k(0)$.

Now, consider a deterministic  potential $V:\ \Z^d\rightarrow \{0,h\}$ with some $h>4d+1$.
The \textbf{potential well} of $V$ is defined by 
\[
\mcW=\left\{x\in \Z^d:\  V(x)=0   \right\},
\]
and  the \textbf{potential barrier} is defined by 
\[
\mcB=\left\{x\in \Z^d: \ V(x)=h   \right\}.
\]
Here,    $h$ is called  the \textbf{height}  of the potential barrier, and 
 $\alpha>1$ in the Newtonian scales describes \textbf{the width} of the potential barrier.

In this article,  we  mainly  focus on  potentials with a specific geometric hierarchical structure.  Let $N\geq 1$ be a fixed positive integer. 
\begin{Def}\label{1-hierarchical structure}
Let $V:\ \Z^d\to \{0,h\}$ and $\mathcal W$ be as above. We say that \textbf{$V$ is $(1,N)$-hierarchical in  $\Lambda_1(j)$} if 
\[
\mcW \cap \Lambda_1(j)= \bigcup_{1\leq s \leq N_1} C^s_0
\]
with   
\begin{itemize}
    \item $N_1\leq N$,
    \item $\diam(C_0^s)\leq 4 \lfloor 3d_0 \rfloor$ for all $1\leq s\leq N_1$,
    \item $C_0^1 \subset \Lambda_0(j)$ and $\dist(C_0^s,C_0^{s'})\geq 2 d_1$ for any $s\neq s'$,
\end{itemize}
where both $\diam(\cdot)$ and $\dist(\cdot, \cdot)$ are induced by the $\ell^\infty$-norm $|\cdot|$ on $\Z^d$. 
\end{Def}
Next, we iteratively define the hierarchical structure of higher levels.
\begin{Def}\label{k-hierarchical structure}
For $k\geq 2$, we say that \textbf{$V$ is $(k,N)$-hierarchical in $\Lambda_k(j)$} if 
\[
\mcW \cap \Lambda_k(j)= \bigcup_{1\leq s \leq N_k} C^s_{k-1}
\]
with  
\begin{itemize}
    \item $N_k\leq N$,
    \item $C_{k-1}^1 \subset \Lambda_{k-1}(j)$ and $\dist(C_{k-1}^s,C_{k-1}^{s'})\geq 2 d_{k}$ for any $s\neq s'$,
    \item $V$ is $(k-1,N)$-hierarchical in  $\Lambda_{k-1}(j)$,
    \item For each $2\leq s\leq N_{k}$, there exists  some  $\Lambda_{k-1}(j_s)\subset \Lambda_k(j)$ such that $C_{k-1}^s\subset \Lambda_{k-1}(j_s)$ and $V$ is $(k-1,N)$-hierarchical in  $\Lambda_{k-1}(j_s)$.
\end{itemize}
\end{Def}
\begin{Def}
We say  $V:\ \Z^d\to\{0,h\}$ is a \textbf{hierarchical potential} if it is $(k,N)$-hierarchical in $\Lambda_k$ for all $k\geq 1$.
\end{Def}
\begin{rmk}
The parameter $N$  characterizing  the number of potential wells at each hierarchical level, is referred to   the \textbf{well density} (cf. Figure \ref{hierarchical structure figure}).
\end{rmk}

\begin{figure}[htbp]
	\centering

\tikzset{every picture/.style={line width=0.75pt}} 

\begin{tikzpicture}[x=0.75pt,y=0.75pt,yscale=-0.75,xscale=0.75]
\draw   (63,32) -- (538.76,32) -- (538.76,519.77) -- (63,519.77) -- cycle ;
\draw  [fill={rgb, 255:red, 74; green, 144; blue, 226 }  ,fill opacity=1 ][dash pattern={on 0.84pt off 2.51pt}] (252.19,225.96) -- (349.58,225.96) -- (349.58,325.81) -- (252.19,325.81) -- cycle ;
\draw  [color={rgb, 255:red, 0; green, 0; blue, 0 }  ,draw opacity=1 ][fill={rgb, 255:red, 74; green, 144; blue, 226 }  ,fill opacity=1 ][dash pattern={on 0.84pt off 2.51pt}] (91.58,57.77) -- (188.97,57.77) -- (188.97,157.62) -- (91.58,157.62) -- cycle ;
\draw  [fill={rgb, 255:red, 74; green, 144; blue, 226 }  ,fill opacity=1 ][dash pattern={on 0.84pt off 2.51pt}] (212.03,403.56) -- (309.42,403.56) -- (309.42,503.41) -- (212.03,503.41) -- cycle ;
\draw  [fill={rgb, 255:red, 74; green, 144; blue, 226 }  ,fill opacity=1 ][dash pattern={on 0.84pt off 2.51pt}] (441.48,129.52) -- (538.86,129.52) -- (538.86,229.37) -- (441.48,229.37) -- cycle ;
\draw  [fill={rgb, 255:red, 208; green, 2; blue, 27 }  ,fill opacity=1 ] (130.11,97.28) -- (150.43,97.28) -- (150.43,118.11) -- (130.11,118.11) -- cycle ;
\draw  [fill={rgb, 255:red, 208; green, 2; blue, 27 }  ,fill opacity=1 ] (168.65,136.78) -- (188.97,136.78) -- (188.97,157.62) -- (168.65,157.62) -- cycle ;
\draw  [fill={rgb, 255:red, 208; green, 2; blue, 27 }  ,fill opacity=1 ] (95.7,69.05) -- (116.02,69.05) -- (116.02,89.88) -- (95.7,89.88) -- cycle ;
\draw  [fill={rgb, 255:red, 208; green, 2; blue, 27 }  ,fill opacity=1 ] (290.72,265.47) -- (311.04,265.47) -- (311.04,286.3) -- (290.72,286.3) -- cycle ;
\draw  [fill={rgb, 255:red, 208; green, 2; blue, 27 }  ,fill opacity=1 ] (250.57,443.07) -- (270.89,443.07) -- (270.89,463.91) -- (250.57,463.91) -- cycle ;
\draw  [color={rgb, 255:red, 0; green, 0; blue, 0 }  ,draw opacity=1 ][fill={rgb, 255:red, 208; green, 2; blue, 27 }  ,fill opacity=1 ] (480.01,169.02) -- (500.33,169.02) -- (500.33,189.86) -- (480.01,189.86) -- cycle ;
\draw  [fill={rgb, 255:red, 208; green, 2; blue, 27 }  ,fill opacity=1 ] (441.48,129.52) -- (461.8,129.52) -- (461.8,150.35) -- (441.48,150.35) -- cycle ;
\draw  [fill={rgb, 255:red, 208; green, 2; blue, 27 }  ,fill opacity=1 ] (465.1,207.84) -- (485.42,207.84) -- (485.42,228.67) -- (465.1,228.67) -- cycle ;
\draw  [fill={rgb, 255:red, 208; green, 2; blue, 27 }  ,fill opacity=1 ] (518.54,208.53) -- (538.86,208.53) -- (538.86,229.37) -- (518.54,229.37) -- cycle ;
\draw  [fill={rgb, 255:red, 208; green, 2; blue, 27 }  ,fill opacity=1 ] (518.54,129.52) -- (538.86,129.52) -- (538.86,150.35) -- (518.54,150.35) -- cycle ;
\draw  [fill={rgb, 255:red, 208; green, 2; blue, 27 }  ,fill opacity=1 ] (252.19,304.98) -- (272.51,304.98) -- (272.51,325.81) -- (252.19,325.81) -- cycle ;
\draw  [fill={rgb, 255:red, 208; green, 2; blue, 27 }  ,fill opacity=1 ] (279.25,412.49) -- (299.57,412.49) -- (299.57,433.33) -- (279.25,433.33) -- cycle ;
\draw  [fill={rgb, 255:red, 208; green, 2; blue, 27 }  ,fill opacity=1 ] (232.21,478.36) -- (252.53,478.36) -- (252.53,499.19) -- (232.21,499.19) -- cycle ;
\draw    (250.83,224.49) -- (190.33,159.09) ;
\draw [shift={(188.97,157.62)}, rotate = 47.23] [color={rgb, 255:red, 0; green, 0; blue, 0 }  ][line width=0.75]    (10.93,-3.29) .. controls (6.95,-1.4) and (3.31,-0.3) .. (0,0) .. controls (3.31,0.3) and (6.95,1.4) .. (10.93,3.29)   ;
\draw [shift={(252.19,225.96)}, rotate = 227.23] [color={rgb, 255:red, 0; green, 0; blue, 0 }  ][line width=0.75]    (10.93,-3.29) .. controls (6.95,-1.4) and (3.31,-0.3) .. (0,0) .. controls (3.31,0.3) and (6.95,1.4) .. (10.93,3.29)   ;
\draw    (251.96,333.43) -- (348.33,334.61) ;
\draw [shift={(348.33,334.61)}, rotate = 180.7] [color={rgb, 255:red, 0; green, 0; blue, 0 }  ][line width=0.75]    (0,5.59) -- (0,-5.59)(10.93,-3.29) .. controls (6.95,-1.4) and (3.31,-0.3) .. (0,0) .. controls (3.31,0.3) and (6.95,1.4) .. (10.93,3.29)   ;
\draw [shift={(251.96,333.43)}, rotate = 0.7] [color={rgb, 255:red, 0; green, 0; blue, 0 }  ][line width=0.75]    (0,5.59) -- (0,-5.59)(10.93,-3.29) .. controls (6.95,-1.4) and (3.31,-0.3) .. (0,0) .. controls (3.31,0.3) and (6.95,1.4) .. (10.93,3.29)   ;

\draw (562.73,501.52) node [anchor=north west][inner sep=0.75pt]   [align=left] {$\Lambda_k(j)$};
\draw (361.85,309.81) node [anchor=north west][inner sep=0.75pt]   [align=left] {$\Lambda_{k-1}(j)$};
\draw (221.42,172.2) node [anchor=north west][inner sep=0.75pt]   [align=left] {$\sim d_k$};
\draw (285.47,338.04) node [anchor=north west][inner sep=0.75pt]   [align=left] {$\sim d_{k-1}$};
\draw (298.75,248.65) node [anchor=north west][inner sep=0.75pt]   [align=left] {$\Lambda_{k-2}(j)$};

\end{tikzpicture}

\caption{Hierarchical structure with well density $N\leq 5$.}
\label{hierarchical structure figure}
	
\end{figure}
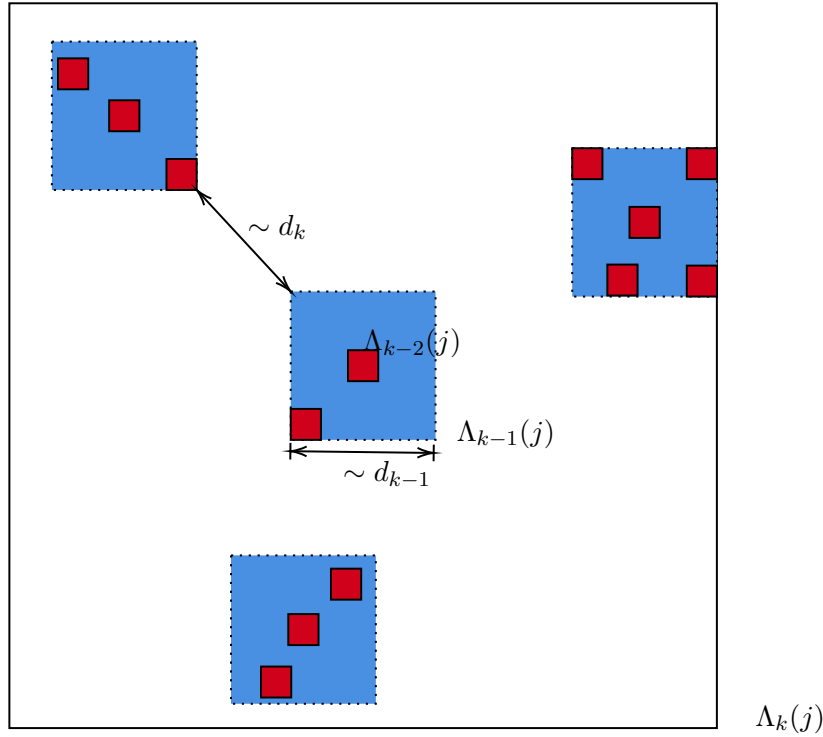

An important example of a hierarchical potential is the \textbf{symmetric hierarchical potential} introduced in \cite{JLMS84, JLMS85}: 
\begin{ex}[cf. \cite{JLMS84, JLMS85}]\label{symmetric hierarchical potential}
The  hierarchical potential $V_{\text{sy}}:\ \Z^d\to\{0,h\}$ is called a \textbf{symmetric hierarchical potential} if, within $\Lambda_1$,   $C_0^1 = \Lambda_0$, and for every $\Lambda_k$ with $k \geq 1$, 
\begin{itemize}
\item $N_k = 2d + 1$,
\item Each component $C_{k-1}^s$ for $2 \leq s \leq 2d + 1$ is obtained via  translating $C_{k-1}^1$ along each  coordinate axis  by the  distance of $\pm 2(\lfloor 3d_k \rfloor - \lfloor 3d_{k-1} \rfloor)$.
\end{itemize}
See \cite[Fig. 2]{JLMS85} for an intuitive illustration.
\end{ex}

\subsection{Main results}
We define
\begin{equation}\label{discrete Laplacian}
  \Delta f(x)=\sum_{y:\ y\sim x} f(y), 
\end{equation}
where  $y\sim x$ means $|x-y|_1=1$.  Let $V_{\text{hi}}:\ \Z^d\to\{0,h\}$ be a hierarchical potential  defined in  the the previous section  satisfying  the following properties:
\begin{itemize}
	\item The height of potential barrier $h>4d+1$, 
	\item The width of the potential barrier $\alpha>1$, 
	\item The well density $N\geq 1$.
\end{itemize}
Consider the deterministic Hamiltonian
\[H_0=2d-\Delta +V_{\text{hi}}.\]
Denote by $\spec(\cdot)$ the spectrum of an operator.  According to \cite[Proposition 3.1]{JLMS85}, the spectrum of $H_0$ satisfies 
\begin{align*}
\spec(H_0) \cap [h,\infty) =[h,h+4d],\ \spec(H_0) \cap [0,h) \subset [0,4d].
\end{align*}
It was established  in  \cite[Theorem 4.2]{JLMS85} that, at the bottom of the spectrum (i.e., $\spec(H_0) \cap [0,h)$), $H_0$  (with a symmetric hierarchical potential)  can  exhibit  the {\bf delocalization behavior}. More precisely,  if we take $V_{\text{hi}} = V_{\text{sy}}$ as in Example \ref{symmetric hierarchical potential}, then the mean square displacement (MSD) 
\[r^2(t)= \sum_{x\in \Z^d} |x|^2 \cdot \left| (e^{-itH_0}P_{[0,4d]}(H_0) \delta_0) (x) \right|^2\]
satisfies  a lower bound
  \begin{align}\label{MSD111}
  r^2(t_k) \gtrsim (2d+1)^{-k} d_k\to  +\infty \ {\rm as} \ k\to\infty
 \end{align}
for a subsequence ${t_k}$ with $t_k \to \infty$,  where $\delta_0(x)=\delta_{x,0}$ and $P_{[0,4d]}(H_0)$ is the spectral projection of $H_0$ on $[0,4d]$.

The celebrated work \cite{JLMS85} showed  that once the i.i.d.  random perturbations with {\bf the absolutely continuous distribution} were introduced into $H_0$, both Anderson localization and dynamical localization (i.e., $\E \sup\limits_{t\in\R} r^2(t)<+\infty$) hold true. This result thus illustrates the phenomenon of the instability of quantum tunneling.

Notably, the proof in \cite{JLMS85} relies on the a priori Wegner estimate and Fröhlich-Spencer estimates \cite{FS83}—the former of which is unavailable for singular random distributions, such as i.i.d. Bernoulli distributions. This limitation motivates our study of localization at the bottom of the spectrum for the randomly perturbed Hamiltonian $H_0$ with Bernoulli random variables. We therefore consider the \textbf{random hierarchical potential}\begin{equation}\label{random hierarchical potential}
V(\omega) = V_{\text{hi}} + \beta V_{\text{r}}(\omega),
\end{equation}
where  $\{V_{\text{r}, x}(\omega)\}_{x \in \mathbb{Z}^d}$ is a family of  i.i.d.  random variables on   the  probability space $\left(\Omega=\{0,1\}^{\Z^d}, \  \mathbb{P}\right)$ with 
\begin{equation}\label{Bernoulli distribution}
  \P(V_{\text{r},x}(\omega)=0)=\P(V_{\text{r},x}(\omega)=1)=\frac{1}{2}\ {\rm for\ all}\ \forall x\in \Z^d.
\end{equation}
For simplicity, we also write   $\{0,1\}^{\Z^d}\ni\omega$ with $\omega(x)=V_{\text{r}, x}(\omega)$  for $x\in \Z^d$.

In this paper, we study the  hierarchical ABM on $\Z^d$ 
\begin{equation}\label{our model}
	H(\omega)=H_0+V(\omega)=2d-\Delta +V_{\rm hi}+\beta V_{\text{r}}(\omega),
\end{equation}
where the  coupling strength $0<\beta\leq 1$ and $V_{\text{r}}(\omega)$ is the random Bernoulli potential. 
A  standard argument (see, e.g., \cite[Lemma 2.1.2]{Kri08}) shows for all $\omega$, 
\begin{equation}\label{spectral inclusion}
  \spec(H(\omega)) \subset [0,4d+\beta]\cup [h,h+4d+\beta]. 
\end{equation} 
\begin{rmk}
The condition $h>4d+1$ is imposed solely to ensure that the spectrum decomposes into two disjoint intervals, as illustrated in \eqref{spectral inclusion}.
\end{rmk}

Our main results are stated as follows.

\begin{thm}\label{low dimension localization}  Let $H(\omega)$ be defined by \eqref{our model} with  $1\leq d \leq 3$. Then for any $0 < \beta \leq 1$, the operator $H(\omega)$ exhibits the Anderson localization in the interval $[0, h) \cap \spec(H(\omega))$ almost surely.
\end{thm}

If additional conditions on $h$, $\alpha$ and $N$ are imposed,   we can establish the Anderson  localization  for the hierarchical ABM in  higher  dimensions. 

\begin{thm}\label{high dimension localization}
Let $H(\omega)$ be defined by \eqref{our model} with  $d \geq 4$. There is some  $h_0=h_0(d)>4d+1$  depending  only on $d$   such that,  if $h>h_0$ and $\alpha > N$, then for any $0<\beta\leq 1$,   the operator $H(\omega)$ exhibits  the Anderson localization in the interval $[0, h) \cap \spec(H(\omega))$ almost surely.
\end{thm}

\begin{rmk}
Indeed, the proof is also valid for dimensions $1\leq d\leq 3$. The constraints imposed in Theorem \ref{high dimension localization} are motivated by the underlying mechanism of quantum tunneling: localization arises when particle tunneling is suppressed, a condition that is favored by taller (larger 
$h$) and wider (larger $\alpha$) potential barriers, as well as a lower density of potential wells (smaller $N$). Given that the propensity toward delocalization intensifies with increasing spatial dimensions \cite{Bou03, LSZ25b}, the enhanced barrier configurations (large $h$, $\alpha > N$) are precisely the adjustments required to counteract this dimensional delocalization effect.
\end{rmk}

Finally, under the  assumptions of Theorems \ref{low dimension localization} and \ref{high dimension localization}, we  can prove the following   dynamical localization result.
\begin{thm}\label{dynamical localization} Under the assumptions of Theorem \ref{low dimension localization} ($d \leq 3$) and Theorem \ref{high dimension localization} ($d \geq 4$), the following holds  true for any $0 < \beta \leq 1$: for $\P$ a.e. $\omega$,  we have 
 \[\sup_{t\in\R} r^2(t,\omega) = \sup_{t\geq 0} \sum_{x\in \Z^d}  |x|^2\cdot \left| (e^{-itH(\omega)}P_{[0,h) }(H(\omega)) \delta_0) (x)  \right|^2\leq C(\omega)<\infty.\]   
\end{thm}
\begin{rmk}
Considering the delocalization estimate \eqref{MSD111} for the deterministic $H_0$ with a symmetric hierarchical potential,  this theorem  indicates  the instability of quantum tunneling  under   any non-zero  {\bf singular} random perturbations. 
\end{rmk}

\subsection{New ingredients of proof}

The proof follows the spirit of \cite{FS83, FMSS85}. However, in the present work, we must overcome the core challenge posed by the lack of a priori Wegner estimates when dealing with Bernoulli random potentials. The main novel contributions of this paper are as follows:

\begin{itemize}
  \item We establish Wegner estimates for the hierarchical ABM across different dimensions by leveraging distinct forms of transversality. For dimensions $d\le 3$, we employ existing unique continuation results in conjunction with the techniques developed in \cite{DS20}; for dimensions $d\ge 4$, we rely solely on the weaker transversality afforded by the so-called ``cone property'' (cf. Proposition \ref{cone property}). The proof for the case $d\ge 4$ necessitates increasing the height and width of the model’s potential barriers—a step that admits a natural interpretation from the perspective of quantum tunneling: taller and wider barriers compensate for the weaker transversality, thereby suppressing quantum tunneling effects.

  \item The proof for $d\ge 4$ draws substantial inspiration from the methodology in \cite{IM16,Imb21}. For our model, we refine and extend the ``iterative Schur complement argument'' introduced in those works. We devise a simpler and more effective approach to isolate the dominant transversal term in the decomposition (Lemma \ref{key decomposition lemma}), which enables us to eliminate the dependence of the barrier height on the coupling constant (see Section \ref{remove beta dependent}). Additionally, we invoke the approximate orthogonality of eigenvectors to achieve sharp control over the number of eigenvalues within a given interval.

  \item In standard multi-scale analysis, iterations are performed over Newtonian scales. A novel facet of our proof for the higher-dimensional Wegner estimate is the interpolation of a finer 
$a$-adic scale ($a\ge 5$) between consecutive Newtonian scales. On this $a$-adic scale, we prove that the eigenvalue count is strictly monotonic in the large deviation sense. We then revert to the Newtonian scale to derive the desired Wegner estimate. We anticipate that this scale-comparison strategy offers a fresh perspective on multi-scale analysis itself.

    \item For dimensions $d\le 3$, we apply the Sperner family lemma introduced by Bourgain (cf. Lemma 2.1 in \cite{Bou04}, later extended in \cite{DS20}) to bound the relevant probabilities. In contrast, for higher dimensions, we construct a so-called ``site-mixed'' martingale to control the tail probabilities in the Wegner estimate (see Section \ref{martingale section}). This martingale possesses a distinctive characteristic: its randomness stems not only from the potential values at individual sites but also from the selection of the sites themselves. Furthermore, the same martingale argument is utilized to prove a unique continuation result on lattices of arbitrary dimension (see Appendix \ref{Appendix UC}).

\end{itemize}

\subsection{Notations}

The  notations used in this paper can be collected as follows. 
\begin{itemize}
\item  For two nonnegative quantities $f$ and $g$, we write $f\lesssim g$, if there is an absolute constant
$C > 0$ such that $f\leq Cg$. If we want to emphasize that $C$ depends on some parameters
$a,b,\cdots$ independent of $f,g$, then we write $f\lesssim_{a,b,\cdots}g$. We denote $f\sim g$ if  both $f\lesssim g$ and $g\lesssim f$ (this should be distinguished with $x\sim y$ whenever  $x,y\in\Z^d$, which means $|x-y|_1=1$). In some cases, we denote  $f\ll g$  if there is some small  enough $c>0$ independent of $f,g$ so that $f\leq c g.$ 
\item For any $a,b\in\R$, denote $a\wedge b=\min\{a,b\}$ and $a\vee b=\max\{a,b\}$. 
  \item  Denote by $\#$ the cardinality of a set. In particular, when applied to a subset of the spectrum, for instance, $\#(\spec(H)\cap A)$, it means  the cardinality counting multiplicities. 
  \item Throughout this paper, the symbols $\dist(\cdot)$ and $\diam(\cdot, \cdot)$ denote the metric induced by $|\cdot|$ on $\Z^d$ or the standard Euclidean metric on $\R$, respectively, with the understanding that the context will prevent any ambiguity.
  \item Let $\langle \cdot, \cdot \rangle$ denote  the standard inner product on $\R^n$.
  \item For a subset $\Lambda \subset \mathbb{Z}^d$, let $R_{\Lambda}$ denote the restriction operator  on  $\ell^2(\Lambda)$. The restriction of $H$ to $\Lambda$ with Dirichlet boundary condition  is then given by  $H_{\Lambda} = R_{\Lambda} H R_{\Lambda}$. We also denote $V_{\Lambda}=R_{\Lambda}V R_{\Lambda}$; the notations $V_{\text{hi},\Lambda}$ and $V_{\text{r},\Lambda}$ are defined analogously.
  \item For a subset $\Lambda\subset \Z^d$, we define its boundaries to be  
\[\partial^{-}\Lambda =\{y\in \Lambda: \ \exists  x\notin \Lambda\ {\rm such\ that} \ x\sim y\},\]
\[\partial^{+}\Lambda =\{y \notin \Lambda: \ \exists  x\in \Lambda\ {\rm such\ that} \ x\sim y\}.\]
\item By $\bm 1_{A}(\cdot)$ we mean the characteristic  function on $A$. 
\end{itemize}

\subsection{Structure of the paper}

The paper  is organized  as follows. In Section \ref{section transversality}, we introduce  some basic facts on the  transversality analysis. The central part of the work, namely, a proof of  the Wegner estimate involving the Bernoulli random variables, is presented in Section \ref{section dim less 3} ($d \leq 3$) and Section \ref{section dim bigger 4} ($d \geq 4$). In Section \ref{section eliminate the energy}, we  will  eliminate  the energy parameters  from the (probabilistic) Wegner estimate,  leading to the proof  of our main localization results (Theorems \ref{low dimension localization}, \ref{high dimension localization}, and \ref{dynamical localization}).  Supplementary materials, including a probabilistic unique continuation result and some useful lemmas, are  provided in the Appendices.

\section{The transversality analysis}\label{section transversality}
This section presents some basic facts regarding transversality estimates, which play a crucial and fundamental role in the analysis of tight-binding models involving Bernoulli random potentials. In such problems, transversality typically manifests itself by ensuring that the corresponding eigenfunctions do not decay rapidly.

We first introduce the useful cone property with consideration to the propagation of solutions to the discrete Schrödinger equation induced by our hierarchical ABM. Next, we prove several unique continuation results across different spatial dimensions. Of particular importance is a probabilistic unique continuation result (cf. Theorem \ref{UC for d arbitrary}) on $\Z^d$ for $d\geq 4,$ whose proof is based on the cone property combined with the martingale argument.

Consider the  Schr\"odinger equation on $\mathbb{Z}^d$,
\begin{equation}\label{harmonic equation}
\Delta u = W u,
\end{equation}
where $\Delta$ is  defined in \eqref{discrete Laplacian} and $W \in \ell^\infty(\mathbb{Z}^d)$ is a bounded potential. Our aim is to achieve certain  lower bound on the decay rate of its solutions, at least on a portion of the lattice. 
Now,  recalling  \eqref{spectral inclusion} and $h>4d+1\geq 4d+\beta$, we  fix throughout this paper 
\begin{equation}\label{iota}
  \iota = \min\left\{\frac{1}{100} , \frac{h-4d-\beta}{2} \right\}.
\end{equation}
Then $h>4d+\beta+\iota$. We focus on solutions  of  the random equation
\begin{equation}\label{random harmonic equation}
H(\omega)u - \lambda u = \big(2d + V(\omega) - \lambda\big)u - \Delta u = 0,
\end{equation}
where $V(\omega)$ is given by \eqref{random hierarchical potential} and $\lambda \in [-\iota, 4d+\beta+\iota]$. Clearly, \eqref{random harmonic equation} is a special case of \eqref{harmonic equation} with
\[W=2d+V(\omega)-\lambda,\  \| W\|_{\infty}:=\sup_{x\in\Z^d}|W(x)|\leq 2d+h+\beta+ \iota .\]

\subsection{Cone property}
A simple yet important form of transversality is characterized by the "cone property." This property leverages a pigeonhole argument to derive a lower bound for solutions propagating within a specified cone.

In the following, we  let $\{\mathbf{e}_k\}_{k=1}^d$ denote the standard basis vectors of $\mathbb{Z}^d$, namely, for each $1\leq k\leq d$, $ ({\bm e}_k)(j)=\delta_{k,j}$.  For any $x_0 \in \mathbb{Z}^d$, any $1 \leq k \leq d$  and $\sigma \in \{1,-1\}$, define the cone with vertex $x_0$ and direction $\sigma \mathbf{e}_k$  to be   \[\mathfrak{C}(x_0,\sigma \bfe_k)=\{y\sim x_0+\sigma \bfe_k:\ y\neq x_0\}.\]
\begin{prop}[{\bf Cone property}]\label{cone property}
Let $u$ be a solution of \eqref{harmonic equation}. 
  Then we have 
  \begin{equation}\label{cone property inequality}
    \max_{y\in \mathfrak{C}(x_0,\sigma \bfe_k) } |u(y)|\geq \frac{1}{\|W\|_{\infty}+2d-1} |u(x_0)|.
  \end{equation}
\end{prop}
\begin{proof}[Proof of Proposition \ref{cone property}.]
Evaluating  at the site $x_0 + \sigma \mathbf{e}_k$ to the equation \eqref{harmonic equation}, we get 
  \[\sum_{ y\sim x_0+\sigma \bfe_k} u(y)=W(x_0+\sigma \bfe_k) \cdot u(x_0+\sigma \bfe_k).\]
  Therefore,
  \begin{align*}
    |u(x_0)| & \leq \sum_{y\sim x_0+\sigma \bfe_k \atop y\neq x_0}|u(y)|+\| W\|_{\infty} |u(x_0+\sigma \bfe_k)| \\
      &\leq \left(\|W\|_{\infty}+2d-1\right)\cdot \max_{y\in \mathfrak{C}(x_0,\sigma \bfe_k) } |u(y)|,
  \end{align*}
which implies \eqref{cone property inequality}.
\end{proof}

An immediate corollary of the cone property is the following lemma concerning solutions to the Dirichlet problem for \eqref{random harmonic equation} within a block. 

\begin{lem}\label{transversality lemma}
Let $B_1=Q_l(x)$ be a block  and let $B_2=Q_{L+\ell}(x)$ be a larger one   with   $L, \ell>0$. Consider a solution $u$ to the Dirichlet problem
    \begin{equation*}
        \begin{cases}
            H(\omega)u - \lambda u = 0 \ \text{in } B_2, \\
            u \equiv 0 \ \text{on } \partial^{+} B_2,
        \end{cases}  
    \end{equation*}
    where $\lambda \in [-\iota, 4d+\beta+\iota]$. Define the influence of $u$ at a boundary point $y \in \partial^+ B_2$ to be 
    \begin{equation}
        \mathcal{I}_u(y) := \left| \sum_{\substack{x \sim y \\ x \in B_2}} u(x) \right|.
    \end{equation}
    Then there exists a point $\bar{y} \in \partial^+ B_2$ such that
    \[
    \mathcal{I}_{u}(\bar{y}) \geq \gamma ^{L + \ell} \cdot \max_{n\in B_1} |u(n)|,
    \]
    where $\gamma = \dfrac{1}{4d + h+\beta }$.
\end{lem}

\begin{proof}[Proof of Lemma \ref{transversality lemma}.]
    Suppose that $|u|$ attains its maximum  (in $B_1$)  at some  $n_0 \in B_1$. We choose the coordinate direction $\mathbf{e}_d$ (the argument works for any direction). Applying Proposition \ref{cone property} with
\begin{equation}\label{cone property parameter}
      W=2d+V(\omega)-\lambda,\  \frac{1}{\|W\|_{\infty}+2d-1} \geq \frac{1}{4d+h+\beta+\iota-1}\geq \gamma,
\end{equation}
 we can construct a chain of sites $n_0, n_1, n_2, \cdots, n_s$ such that
 \begin{equation}\label{Chain 1}
  n_{j+1}\in \mathfrak{C}(n_j,\bfe_d)\cap B_2  \  {\rm for} \ 0\leq j\leq s-1 
\end{equation}
and 
\begin{equation}\label{Chain transversality}
  \ |u(n_j)|\geq \gamma^{j}|u(n_0)| \ {\rm for} \   0\leq j\leq s.
\end{equation}

  \begin{figure}[htbp]
    \centering

\tikzset{every picture/.style={line width=0.75pt}} 

\begin{tikzpicture}[x=0.75pt,y=0.75pt,yscale=-0.55,xscale=0.55]

\draw   (152,95) -- (522.71,95) -- (522.71,465.71) -- (152,465.71) -- cycle ;
\draw   (312.36,255.36) -- (362.36,255.36) -- (362.36,305.36) -- (312.36,305.36) -- cycle ;
\draw  [dash pattern={on 0.84pt off 2.51pt}]  (152,106) -- (337.36,291.36) ;
\draw  [dash pattern={on 0.84pt off 2.51pt}]  (522.71,106) -- (337.36,291.36) ;
\draw    (309.71,283.29) -- (155.71,283.29) ;
\draw [shift={(153.71,283.29)}, rotate = 360] [color={rgb, 255:red, 0; green, 0; blue, 0 }  ][line width=0.75]    (10.93,-3.29) .. controls (6.95,-1.4) and (3.31,-0.3) .. (0,0) .. controls (3.31,0.3) and (6.95,1.4) .. (10.93,3.29)   ;
\draw [shift={(311.71,283.29)}, rotate = 180] [color={rgb, 255:red, 0; green, 0; blue, 0 }  ][line width=0.75]    (10.93,-3.29) .. controls (6.95,-1.4) and (3.31,-0.3) .. (0,0) .. controls (3.31,0.3) and (6.95,1.4) .. (10.93,3.29)   ;
\draw    (312.71,314.29) -- (361.71,314.29) ;
\draw [shift={(361.71,314.29)}, rotate = 180] [color={rgb, 255:red, 0; green, 0; blue, 0 }  ][line width=0.75]    (0,5.59) -- (0,-5.59)(10.93,-3.29) .. controls (6.95,-1.4) and (3.31,-0.3) .. (0,0) .. controls (3.31,0.3) and (6.95,1.4) .. (10.93,3.29)   ;
\draw [shift={(312.71,314.29)}, rotate = 0] [color={rgb, 255:red, 0; green, 0; blue, 0 }  ][line width=0.75]    (0,5.59) -- (0,-5.59)(10.93,-3.29) .. controls (6.95,-1.4) and (3.31,-0.3) .. (0,0) .. controls (3.31,0.3) and (6.95,1.4) .. (10.93,3.29)   ;
\draw [color={rgb, 255:red, 208; green, 2; blue, 27 }  ,draw opacity=1 ][line width=1.5]    (337.36,291.36) -- (328.71,270.29) ;
\draw [shift={(328.71,270.29)}, rotate = 247.7] [color={rgb, 255:red, 208; green, 2; blue, 27 }  ,draw opacity=1 ][fill={rgb, 255:red, 208; green, 2; blue, 27 }  ,fill opacity=1 ][line width=1.5]      (0, 0) circle [x radius= 4.36, y radius= 4.36]   ;
\draw [shift={(337.36,291.36)}, rotate = 247.7] [color={rgb, 255:red, 208; green, 2; blue, 27 }  ,draw opacity=1 ][fill={rgb, 255:red, 208; green, 2; blue, 27 }  ,fill opacity=1 ][line width=1.5]      (0, 0) circle [x radius= 4.36, y radius= 4.36]   ;
\draw [color={rgb, 255:red, 208; green, 2; blue, 27 }  ,draw opacity=1 ][line width=1.5]    (336.71,250.29) -- (328.71,270.29) ;
\draw [shift={(328.71,270.29)}, rotate = 111.8] [color={rgb, 255:red, 208; green, 2; blue, 27 }  ,draw opacity=1 ][fill={rgb, 255:red, 208; green, 2; blue, 27 }  ,fill opacity=1 ][line width=1.5]      (0, 0) circle [x radius= 4.36, y radius= 4.36]   ;
\draw [shift={(336.71,250.29)}, rotate = 111.8] [color={rgb, 255:red, 208; green, 2; blue, 27 }  ,draw opacity=1 ][fill={rgb, 255:red, 208; green, 2; blue, 27 }  ,fill opacity=1 ][line width=1.5]      (0, 0) circle [x radius= 4.36, y radius= 4.36]   ;
\draw [color={rgb, 255:red, 208; green, 2; blue, 27 }  ,draw opacity=1 ][line width=1.5]    (336.71,250.29) -- (328.07,229.21) ;
\draw [shift={(328.07,229.21)}, rotate = 247.7] [color={rgb, 255:red, 208; green, 2; blue, 27 }  ,draw opacity=1 ][fill={rgb, 255:red, 208; green, 2; blue, 27 }  ,fill opacity=1 ][line width=1.5]      (0, 0) circle [x radius= 4.36, y radius= 4.36]   ;
\draw [shift={(336.71,250.29)}, rotate = 247.7] [color={rgb, 255:red, 208; green, 2; blue, 27 }  ,draw opacity=1 ][fill={rgb, 255:red, 208; green, 2; blue, 27 }  ,fill opacity=1 ][line width=1.5]      (0, 0) circle [x radius= 4.36, y radius= 4.36]   ;
\draw [color={rgb, 255:red, 208; green, 2; blue, 27 }  ,draw opacity=1 ][line width=1.5]    (327.71,209.29) -- (328.07,229.21) ;
\draw [shift={(328.07,229.21)}, rotate = 88.97] [color={rgb, 255:red, 208; green, 2; blue, 27 }  ,draw opacity=1 ][fill={rgb, 255:red, 208; green, 2; blue, 27 }  ,fill opacity=1 ][line width=1.5]      (0, 0) circle [x radius= 4.36, y radius= 4.36]   ;
\draw [shift={(327.71,209.29)}, rotate = 88.97] [color={rgb, 255:red, 208; green, 2; blue, 27 }  ,draw opacity=1 ][fill={rgb, 255:red, 208; green, 2; blue, 27 }  ,fill opacity=1 ][line width=1.5]      (0, 0) circle [x radius= 4.36, y radius= 4.36]   ;
\draw [color={rgb, 255:red, 208; green, 2; blue, 27 }  ,draw opacity=1 ][line width=1.5]    (335.71,189.29) -- (327.71,209.29) ;
\draw [shift={(335.71,189.29)}, rotate = 111.8] [color={rgb, 255:red, 208; green, 2; blue, 27 }  ,draw opacity=1 ][fill={rgb, 255:red, 208; green, 2; blue, 27 }  ,fill opacity=1 ][line width=1.5]      (0, 0) circle [x radius= 4.36, y radius= 4.36]   ;
\draw [color={rgb, 255:red, 208; green, 2; blue, 27 }  ,draw opacity=1 ][line width=1.5]    (335.36,169.36) -- (335.71,189.29) ;
\draw [shift={(335.36,169.36)}, rotate = 88.97] [color={rgb, 255:red, 208; green, 2; blue, 27 }  ,draw opacity=1 ][fill={rgb, 255:red, 208; green, 2; blue, 27 }  ,fill opacity=1 ][line width=1.5]      (0, 0) circle [x radius= 4.36, y radius= 4.36]   ;
\draw [color={rgb, 255:red, 208; green, 2; blue, 27 }  ,draw opacity=1 ][line width=1.5]    (343.36,149.36) -- (335.36,169.36) ;
\draw [shift={(343.36,149.36)}, rotate = 111.8] [color={rgb, 255:red, 208; green, 2; blue, 27 }  ,draw opacity=1 ][fill={rgb, 255:red, 208; green, 2; blue, 27 }  ,fill opacity=1 ][line width=1.5]      (0, 0) circle [x radius= 4.36, y radius= 4.36]   ;
\draw [color={rgb, 255:red, 208; green, 2; blue, 27 }  ,draw opacity=1 ][line width=1.5]    (351.36,129.36) -- (343.36,149.36) ;
\draw [shift={(351.36,129.36)}, rotate = 111.8] [color={rgb, 255:red, 208; green, 2; blue, 27 }  ,draw opacity=1 ][fill={rgb, 255:red, 208; green, 2; blue, 27 }  ,fill opacity=1 ][line width=1.5]      (0, 0) circle [x radius= 4.36, y radius= 4.36]   ;
\draw [color={rgb, 255:red, 208; green, 2; blue, 27 }  ,draw opacity=1 ][line width=1.5]    (351.36,129.36) -- (342.71,108.29) ;
\draw [shift={(342.71,108.29)}, rotate = 247.7] [color={rgb, 255:red, 208; green, 2; blue, 27 }  ,draw opacity=1 ][fill={rgb, 255:red, 208; green, 2; blue, 27 }  ,fill opacity=1 ][line width=1.5]      (0, 0) circle [x radius= 4.36, y radius= 4.36]   ;
\draw [shift={(351.36,129.36)}, rotate = 247.7] [color={rgb, 255:red, 208; green, 2; blue, 27 }  ,draw opacity=1 ][fill={rgb, 255:red, 208; green, 2; blue, 27 }  ,fill opacity=1 ][line width=1.5]      (0, 0) circle [x radius= 4.36, y radius= 4.36]   ;
\draw [color={rgb, 255:red, 0; green, 0; blue, 0 }  ,draw opacity=1 ][line width=0.75]  [dash pattern={on 0.84pt off 2.51pt}]  (342.71,82.29) -- (342.71,108.29) ;
\draw [shift={(342.71,82.29)}, rotate = 90] [color={rgb, 255:red, 0; green, 0; blue, 0 }  ,draw opacity=1 ][fill={rgb, 255:red, 0; green, 0; blue, 0 }  ,fill opacity=1 ][line width=0.75]      (0, 0) circle [x radius= 3.35, y radius= 3.35]   ;

\draw (373,282) node [anchor=north west][inner sep=0.75pt]   [align=left] {$B_1$};
\draw (540,463) node [anchor=north west][inner sep=0.75pt]   [align=left] {$B_2$};
\draw (226,261) node [anchor=north west][inner sep=0.75pt]   [align=left] {$L$};
\draw (335,320) node [anchor=north west][inner sep=0.75pt]   [align=left] {$\ell$};
\draw (354,73) node [anchor=north west][inner sep=0.75pt]   [align=left] {$\bar{y}$};
\draw (345.36,152.36) node [anchor=north west][inner sep=0.75pt]   [align=left] {\textcolor[rgb]{0.82,0.01,0.11}{$n_j$}};

\end{tikzpicture}
\caption{Visualization of the proof  of  Lemma \ref{transversality lemma}.}
  \end{figure}
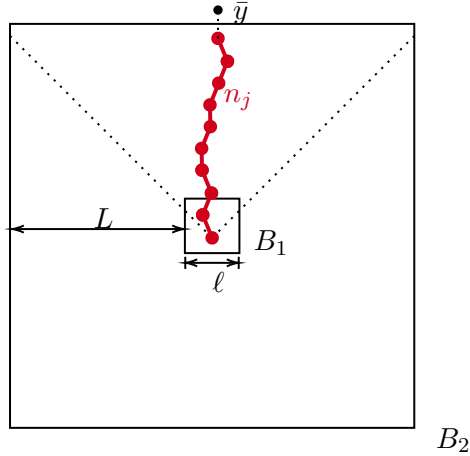

 Regarding \eqref{Chain 1},  if  $n_j \in \partial^- B_2$, then the  Dirichlet boundary condition ensures  that  $n_{j+1}$  remains inside $B_2$. Moreover, \eqref{Chain 1} implies
    \begin{equation}\label{Chain length}
        1 \leq \langle n_{j+1} - n_j, \mathbf{e}_d \rangle \leq 2,
    \end{equation}
    which gives $s \leq \langle n_s - n_0, \mathbf{e}_d \rangle \leq L + \ell$. Consequently,
    \[
    |u(n_s)| \geq \gamma^{L+\ell} |u(n_0)|.
    \]
    The inequality \eqref{Chain length} together with the Dirichlet boundary condition also guarantees that the terminal point $n_s$ lies on  $\partial^- B_2$ in the direction $\mathbf{e}_d$. Taking $\bar{y} = n_s + \mathbf{e}_d$, we obtain
    \[
    \mathcal{I}_u(\bar{y}) = |u(n_s)| \geq \gamma^{L+\ell} \cdot |u(n_0)|.
    \]

\end{proof}

The proof of Lemma \ref{transversality lemma} reveals that the cone property essentially furnishes a transversality estimate only along a one-dimensional subset of $\Z^d$
. This is somewhat elementary, as one typically expects transversality estimates to hold over {\bf a larger portion of the lattice}. We therefore proceed to present several unique continuation results, which yield stronger transversality estimates.

\subsection{The unique continuation results}
Unique continuation results were first employed by Bourgain \cite{BK05} to derive the transversality estimates required for proving Anderson localization in Anderson-Bernoulli models defined on $\mathbb{R}^d$. In the continuous setting, solutions to Equation \eqref{harmonic equation}—where the discrete operator is replaced by the standard Laplacian on $\R^d$—satisfy the following lower bound\[\sup_{|x-y|\leq 1}|u(y)|\geq c' \exp\{-c'|x|^{\frac{4}{3}}\log|x|\}\]
for all $x \in \mathbb{R}^d$ with $|x| > R$, where the constants $c', R > 0$ depend only on $\|W\|_{\infty}$ and   $d$ (cf. \cite[Lemma 3.10]{BK05}). In other words, transversality  estimate holds ``almost everywhere''  on $\mathbb{R}^d$.

However, in the discrete setting (on $\Z^d$), several examples (e.g., \cite[Theorem 2]{Jit07} and \cite[Remark 1.3]{BLMS}) demonstrate that deterministic transversality estimates cannot be expected to hold over a full-dimensional subset of $\Z^d$. The existing results in this framework can be summarized as follows:

For the one-dimensional case on $\Z$, the unique continuation property is deterministic and follows directly from the cone property.

\begin{thm}[{\bf A deterministic unique continuation result  on  $\Z$}]\label{UC for d=1}
Let $u$ satisfy \eqref{random harmonic equation} in an interval $Q_L(j) \subset \mathbb{Z}$. Then there exists a subset $T \subset Q_L(j)$ such that
\begin{equation}\label{transversality set, d=1}
  \# T\geq \frac{1}{2}\cdot \# Q_L(j) 
\end{equation} 
  with 
  \begin{equation}\label{transversality, d=1}
      |u(n)| \geq \gamma ^L \cdot |u(j)|\ {\rm for}\  \forall n\in T.
  \end{equation}
\end{thm}

For  the  case of $\Z^2$, a unique continuation result   was first established for the harmonic equation $\Delta u = 0$ (i.e., $W \equiv 0$) in \cite{BLMS}. Building on this work, Ding and Smart adapted the methodology of \cite{BLMS} to prove a probabilistic unique continuation result on $\mathbb{Z}^2$ (cf.  \eqref{transversality set, d=2}), which enabled them to establish localization for the ABM on $\Z^2$ \cite{DS20}. Later, Li \cite{Li22} refined the result of \cite{DS20} by leveraging the Bourgain-Tzafriri restricted invertibility theory. A minor adaptation of the proofs from \cite{Li22, DS20} to our hierarchical ABM framework yields the following result:

\begin{thm}[{\bf A probabilistic unique continuation result on $\Z^2$}]\label{UC for d=2}
For every small $\varepsilon>0$, there exists some  $C_1 \gg 1$ (depending  only on $h,\varepsilon,\beta$) such that if 
\begin{itemize}
  \item [(1)] $Q \subset\Z^2$ a block of  length $L> C_1$,
  \item [(2)] $Q' \subset Q$ a block  of length $\ell <\varepsilon^{2} L$,
  \item [(3)] $\lambda_0\in [-1,9+\beta]$,
  \item [(4)] $\mcE_{uc}^{\varepsilon,L_0}(Q,Q',\lambda_0)$ denotes  the event that
    \begin{equation*}
          \begin{cases}
            |\lambda-\lambda_0|\leq \exp\{-C_1 \cdot (L \log L)\} \\
            H(\omega)u-\lambda u=0 \ {\rm in} \ Q\\
          \end{cases}      
    \end{equation*}    
    implies there exists a subset $T\subset Q\setminus Q'$, such that 
  \begin{equation}\label{transversality set, d=2}
        \# T\geq \varepsilon^3  \cdot L^2
  \end{equation}
  and 
  \begin{equation}\label{transversality, d=2}
      |u(n)| \geq \exp\{-C_1 \cdot (L\log L) \} \cdot \| u\|_{\ell^{\infty}(Q')} \ {\rm for} \ \forall n\in T,
  \end{equation}
\end{itemize}
then we have $\P(\mcE_{uc}^{\varepsilon,L_0}(Q,Q',\lambda_0)|V_{\text{r},Q'})\geq 1 -\exp\{ -\varepsilon L^{\frac{2}{3}}\}$. Here $\P(\cdot|\cdot)$ means  the condition probability on $V_{\text{r},Q'}$ with $V_{\text{r},Q'}$  denoting the random variables restricted in $Q'$ and $\P$ being  only about random variables in $Q$.
\end{thm}
\begin{proof}[Proof of Theorem \ref{UC for d=2}]
The proof of Theorem \ref{UC for d=2}  follows from the argument in Chapter 5 of \cite{Li22}, with two key adaptations. First, the upper bound in \cite[(5.13)]{Li22} is replaced by $(30 + h + 2\beta)^{(t-2)\vee 0}$. Second, our potential $V(\omega)$ is defined by \eqref{random hierarchical potential}, and consequently the distance appearing in \cite[(5.76)]{Li22} becomes
\begin{align*}
&\ \ \ \dist(P_{\mcS_0}(\vec{V}_{\text{hi},t'}+\beta \vec{V}_{\text{r},t'}),(M_{S_0}A_{S_0})^{-1} (\Gamma-v^*|_{S_0}))\\
&=\dist(P_{\mcS_0}(\beta \vec{V}_{\text{r},t'}),(M_{S_0}A_{S_0})^{-1} (\Gamma-v^*|_{S_0})-P_{\mcS_0}(\vec{V}_{\text{hi},t'})). 
\end{align*}
Geometrically, this means that the affine space $(M_{S_0}A_{S_0})^{-1} (\Gamma-v^*|_{S_0})$ is translated by the deterministic vector $P_{\mathcal{S}0}(\vec{V}_{\text{hi},t'})$. Since the probabilistic estimate in \cite[Lemma 1.6]{Li22} depends only on the dimension of this affine space, it remains valid.

According to the above discussion, Theorem \ref{UC for d=2} follows immediately once we verify that $Q'$ is $\varepsilon$-regular in $Q$ (see Condition 2 of \cite[Lemma 3.5]{Li22}). For the definitions  of  ``$\varepsilon$-sparse" and ``$\varepsilon$-regular", we refer to Definitions 3.1–3.4 in \cite{DS20}. 

Suppose that $Q'$ is not $\varepsilon$-sparse in some tilted square $R$. By definition, this implies that the side length of $R$ must be less than $\varepsilon^{-1}\ell$ and that $R$ intersects with  $Q'$. Consequently, $R$ must be contained in the $10\varepsilon^{-1}\ell$-neighborhood of $Q'$, which is denoted  by $\widetilde{Q}'$.

Now, let $R_1, \cdots, R_k \subset Q$ be a collection of disjoint tilted squares,  in each of which $Q'$ is not $\varepsilon$-sparse. Since they are disjoint and are all  contained in $\widetilde{Q}'$, we have

\[\sum_{s}\# R_s \leq \# \widetilde{Q}' \leq (20 \varepsilon^{-1} +1)^2 \ell^2 .\]
Using the assumption  $\ell \leq \varepsilon^2 L$ and taking $\varepsilon > 0$ to be sufficiently small, we obtain
\[(20 \varepsilon^{-1} +1)^2 \ell^2 \leq \varepsilon L^2 =\varepsilon\cdot  \# Q.\]
Thus, $Q'$ is $\varepsilon$-sparse in $Q$ and we finish the proof.
\end{proof}

For the case  of $\Z^3$, Li and Zhang \cite{LZ22} used the cone property and the geometric structure of $\mathbb{Z}^3$ to prove the following unique continuation result.
\begin{thm}[{\bf A deterministic unique continuation result  on  $\Z^3$}, cf. Theorem 1.3 in \cite{LZ22}]\label{UC for d=3} There exists a constant $p>\frac{3}{2}$ such that the following holds. Let $u$ satisfy \eqref{random harmonic equation} in a  block   $Q_L(j) \subset \mathbb{Z}$. Then there is some  $C_1>0$ (depending only  on $h,\beta$) such that,  for all $L\geq C_1$, there exists a subset $T \subset Q_L(j)$ satisfying \begin{equation}\label{transversality set, d=3}
  \# T\geq L^p
\end{equation} 
  so that 
  \begin{equation}\label{transversality, d=3}
      |u(n)| \geq \exp\{-C_1 L\} \cdot |u(j)|\ {\rm for}\ \forall  n\in T.
  \end{equation}
\end{thm}

While the aforementioned unique continuation results have been established for $d=1,2,3$, they do not ensure that the transversal set $T$ in \eqref{transversality set, d=3} is of full dimension when $d=3$. In the present work, we prove that for $d \geq 3$, transversality holds on a full-dimensional subset with high probability, where the probability in question depends on the initial data.

\begin{thm}\label{UC for d arbitrary}
 Let $d \geq 2$. Consider a solution $u$ on $\mathbb{Z}^d$ to the equation
    \[
    H(\omega) u - \lambda u = 0,
    \]
    subject to the boundary condition $u \equiv u_0$ on the union of two hyperplanes $\mathcal{P}_0 \cup \mathcal{P}_1$, where $\mathcal{P}_k = \{ n \in \mathbb{Z}^d: \ n_d = k \}$.
    
Then there exist  some $\varepsilon=\varepsilon(d)>0$ and  some $C_1 = C_1(h, \beta, d) \gg 1$ with the following property: For any $L \geq C_1 $, we define  $\mathcal{E}_{\mathrm{uc}}(L , u_0, \lambda)$ to be the event  that there exists a subset $T \subset Q_L(0)$ satisfying
   \begin{equation}\label{transversality set, d arbitrary}
        \# T\geq \varepsilon  \cdot L^d
  \end{equation}
  so that 
  \begin{equation}\label{transversality, d arbitrary}
      |u(n)| \geq \exp\{-C_1 L \} \cdot |u(0)| \ {\rm for}\  \forall  n\in T.
  \end{equation}
  Then we have 
    \[
    \mathbb{P}\big( \mathcal{E}_{\mathrm{uc}}(L, u_0, \lambda) \big) \geq 1 - \exp\{-\varepsilon  L\}, 
    \]
    where $\P$ only depends  on random variables in $Q_L(0)$.
  \end{thm}

The proof of Theorem \ref{UC for d arbitrary} bears strong similarity to that of the high-dimensional Wegner estimate in Section \ref{section dim bigger 4}, as both rely on a martingale argument grounded in the cone property—an approach we deem to be of independent interest. For this reason, we defer the detailed proof of Theorem \ref{UC for d arbitrary} to Appendix \ref{Appendix UC}.

We remark  that, unlike Theorem \ref{UC for d=2}, the probabilistic estimate in Theorem \ref{UC for d arbitrary}  depends on   the initial data $u_0$. However,  for the special case   $d=2$, Theorem \ref{UC for d arbitrary} yields an exponential tail bound for the probabilistic estimate, in contrast to the sub-exponential bound furnished by Theorem \ref{UC for d=2}.
 
This observation naturally gives rise to the following open problem:

\begin{prob}\label{uniform initial data}
Can we obtain a probabilistic  estimate that is independent of  the initial data in Theorem \ref{UC for d arbitrary}, while preserving the essential bounds given in \eqref{transversality set, d arbitrary} and \eqref{transversality, d arbitrary}?
\end{prob}

A positive solution to Problem \ref{uniform initial data} would likely pave the way for a proof of Anderson localization at the bottom of the spectrum for the standard ABM on $\mathbb{Z}^d \ (d\geq 4)$.

\section{Wegner estimate on $\Z^d$ with  $d=1,2,3$}\label{section dim less 3}

In this section, we establish the key Wegner estimate for dimensions $1\leq d \leq 3$, —a result that plays a central role in the proof of Anderson localization. Our approach integrates the unique continuation results presented in Theorems \ref{UC for d=1}, \ref{UC for d=2}, and \ref{UC for d=3} with the free-sites argument originally developed in \cite{Bou04} (see \cite{BK05, DS20} for subsequent generalizations).

\subsection{Preliminary lemmas}
In this part,  we make no restriction on the dimension $d$.   From now on and for simplicity, we omit the explicit dependence of $H(\omega)$ on the randomness $\omega$ when no confusion arises. For a subset $\Lambda \subset \mathbb{Z}^d$, the corresponding Green's function on $\Lambda$ is defined by 
\[G_{\Lambda}(E)=(H_{\Lambda}-E)^{-1},\ E\in \R.\]
We say that  a finite set $\Lambda\subset\Z^d$ is  {non-resonant} if it is contained within the barrier region, i.e.,
\[\Lambda \subset \mcB =\{x\in \Z^d:\ V_{\rm hi}(x)=h \}.\]
Recall that we  have taken  $\iota$ (which essentially depends on $d,h,\beta$)  to be  \eqref{iota}. 

For each non-resonant set, we have 
\begin{lem}\label{non-resonant Green's function}
  If $\Lambda$ is non-resonant, then for any  $E\in [-\iota,4d+\beta+\iota]$, we have
  \begin{align}
   \label{non-resonant L2 norm} \|G_{\Lambda}(E)\|&\leq  \frac{1}{h-4d-\beta-\iota},\\
    \label{no-resonant off-diagonal decay} |G_{\Lambda}(x,y;E)| &\leq  \frac{1}{h-4d-\beta-\iota} \cdot \exp\{-\gamma_0 |x-y|_1\} \ {\rm for}\  \forall x,y \in \Lambda, 
  \end{align}
  where  $\gamma_0=\log(1+\frac{h-4d-\beta-\iota}{2d})>0$.
\end{lem}
\begin{proof}[Proof of Lemma \ref{non-resonant Green's function}]
  This proof is based on the Neumann series argument.  Indeed, 
  expanding the Green's function by the Neumann series, we get 
  \begin{align}\label{Neumann series expansion}
  \notag  G_{\Lambda}(E) & = (2d-E+V_{\Lambda}-\Delta_{\Lambda})^{-1} \\
           &=(V_{\Lambda}+2d-E)^{-1} \sum_{s\geq 0} \left(  \Delta_{\Lambda} (V_{\Lambda}+2d-E)^{-1}\right). 
  \end{align}
  Since $\Lambda$ is non-resonant and $E\in [-\iota,4d+\beta+\iota]$, we have
  \[\| (V_{\Lambda}+2d-E)^{-1} \| =\| (\beta V_{\text{r},\Lambda}+h+2d-E)^{-1}\| \leq \frac{1}{h+2d-E}\leq \frac{1}{h-2d-\beta-\iota}.\]
  Since  $\| \Delta_{\Lambda} \| \leq 2d$, we obtain  
  \[\| G_{\Lambda}(E)\| \leq \frac{\|(V_{\Lambda}+2d-E)^{-1}\| }{1-2d \|(V_{\Lambda}+2d-E)^{-1}\| }\leq \frac{1}{h-4d-\beta-\iota}.\]

  For the off-diagonal decay estimate,    we have
  \begin{align*}
     G_{\Lambda}(x,y;E)  & =\sum_{{\rm path}\ g=(x_0=x, x_1,\cdots,x_s=y)\subset \Lambda} \frac{1}{V_{x_0}+2d-E}\cdot \frac{1}{V_{x_1}+2d-E}\cdots \frac{1}{V_{x_s}+2d-E}.
  \end{align*}
 The sum runs over all paths $g=(x_0=x, x_1,\cdots,x_s=y)$ from $x$ to $y$ with $x_{k-1}\sim x_k$. The number of such paths of length $s$ is at most $(2d)^s$, since each site in $\Z^d$ has $2d$ neighbors. Moreover, any such path must have length at least $|x-y|_1$. Therefore,
   \begin{align*}
    |G_{\Lambda}(x,y;E)| & \leq \sum_{s \geq |x-y|_1} (2d)^s \left(\frac{1}{h-2d-\beta-\iota}\right)^{s+1} =\frac{1}{h-4d-\beta-\iota} \left(\frac{2d}{h-2d-\beta-\iota}\right)^{|x-y|_1},
  \end{align*}
  which implies  \eqref{no-resonant off-diagonal decay}.
\end{proof}

Now recall the notations $\Lambda_k, C_k^s, \Lambda_k(j_s)$ from \eqref{k-th block}, Definition \ref{1-hierarchical structure} and Definition \ref{k-hierarchical structure}. Denote by $\overline{\Lambda}_k$ and $\overline{\Lambda}_{k}^s$ the $\frac{1}{10}d_{k+1}$-neighborhood of $\Lambda_k$ and $\Lambda_k(j_s)$ (for $s\geq 2$), respectively.

The following lemma gives an upper bound  on  the number of eigenvalues of $H_{\overline{\Lambda}_k}$ in the low energy region.
\begin{lem}\label{eigenvalue number}
  There is a scale $k^{(1)}$ depending only on $d,\beta,h,\alpha,d_0$ such that for all $k\geq k^{(1)}$, we have 
  \[\#\left( \spec(H_{\barLambda_k})\cap [-\iota,4d+\beta+\iota] \right)\leq 2 N^{k-k^{(1)}} (16d_{k^{(1)}}+1)^{d}. \] 
\end{lem}
\begin{proof}[Proof of Lemma \ref{eigenvalue number}]
Let $k \geq k^{(1)}$  and let $E \in [-\iota, 4d+\beta+\iota]$ be an eigenvalue of $H_{\barLambda_k}$ with a corresponding normalized eigenfunction $\psi_E$. Denote by $U_{k,k^{(1)}}$ the union of all $\Lambda_{k^{(1)}}(j_s) \subset \Lambda_{k}$ at the $k^{(1)}$-th hierarchical level, and let $\widehat{U}_{k,k^{(1)}}$ be the $2d_{k^{(1)}}$-neighborhood of $U_{k,k^{(1)}}$. 
We have 
  \[(H_{\barLambda_k}-E)=(H_{U_{k,k^{(1)}}}-E)\oplus (H_{\barLambda_k\setminus U_{k,k^{(1)}}}-E)+\Gamma,\]
where $\Gamma$ is the connecting operator  $U_{k,k^{(1)}}$ and $\barLambda_k\setminus U_{k,k^{(1)}}$ via  $-\Delta$. Applying the Poisson's  formula yields 
\[\psi_E=-\left(G_{U_{k,k^{(1)}}}(E)\oplus G_{\barLambda_k\setminus U_{k,k^{(1)}}}(E)\right)  \Gamma \psi_E.\]
Clearly, $\barLambda_k\setminus U_{k,k^{(1)}}$ is non-resonant. Applying Lemma \ref{non-resonant Green's function}, we obtain that for every $x\in \barLambda_k\setminus \widehat{U}_{k,k^{(1)}},$
\begin{align}\label{k^(1) large 1}
\notag  |\psi_E(x)|& \leq \sum_{w\in \partial^+ U_{k,k^{(1)}} \atop w'\in \partial^-U_{k,k^{(1)}},\ w\sim w'} \left| G_{\barLambda_k\setminus U_{k,k^{(1)}}}(x,w;E) \right| \cdot |\psi_E(w')| \\
      \notag       &\leq  \frac{2d}{h-4d-\beta-\iota} \cdot \sum_{w\in \partial^+ U_{k,k^{(1)}} } \exp\{-\gamma_0 |x-w|_1\}  \lesssim_{d,h,\beta} \sum_{s\geq \dist(x,U_{k,k^{(1)}})}s^{d-1} \exp\{-\gamma_0 s\} \\
      &\leq \exp\{-\frac{1}{2}\gamma_0 \cdot \dist(x,U_{k,k^{(1)}})\}.
\end{align}
From \eqref{k^(1) large 1}, we have 
\begin{align}\label{k^(1) large 2}
 \notag \|\psi_E\|_{\ell^2(\barLambda_k\setminus \widehat{U}_{k,k^{(1)}})}^2 & \leq \sum_{x\in \barLambda_k\setminus \widehat{U}_{k,k^{(1)}}} \exp\{-\gamma_0\cdot \dist(x,U_{k,k^{(1)}})\} \\
   & \leq \exp\{-\gamma_0 d_{k^{(1)}}\}. 
\end{align}
Indeed, \eqref{k^(1) large 2} also means that 
\begin{equation}\label{k^(1) large 3}
  \|\psi_E\|_{\ell^2(\widehat{U}_{k,k^{(1)}})}^2 \geq 1-\exp\{-\gamma_0 d_{k^{(1)}}\}\geq \frac{1}{2}. 
\end{equation}
The estimates  \eqref{k^(1) large 1}, \eqref{k^(1) large 2} and \eqref{k^(1) large 3}  remain  valid provided $\dist(x,U_{k,k^{(1)}})\geq 2d_{k^{(1)}}$ is chosen sufficiently large, or   for $k^{(1)}$ sufficiently large (depending only on $d,h,\beta,\alpha,d_0$).  Combining  those estimates together, we can apply the following Hilbert-Schmidt argument: 
\begin{align*}
  \# \widehat{U}_{k,k^{(1)}} & = \sum_{x\in \widehat{U}_{k,k^{(1)}}} \|\delta_x \|^2 \geq \sum_{x\in \widehat{U}_{k,k^{(1)}}} \sum_{E\in \spec(H_{\barLambda_k})\cap [-\iota,4d+\beta+\iota]} |\psi_E(x)|^2  \\
      &=\sum_{E\in \spec(H_{\barLambda_k})\cap [-\iota,4d+\beta+\iota]} \|\psi_E\|_{\ell^2(\widehat{U}_{k,k^{(1)}})}^2 \geq \frac{1}{2} \cdot \# \left(  \spec(H_{\barLambda_k})\cap [-\iota,4d+\beta+\iota] \right),
\end{align*}
which yields 
\[\# \left(  \spec(H_{\barLambda_k})\cap [-\iota,4d+\beta+\iota] \right) \leq 2 \# \widehat{U}_{k,k^{(1)}} \leq 2 N^{k-k^{(1)}} (16d_{k^{(1)}}+1)^d. \]
Here, we used the fact that there are at most $N^{k-k^{(1)}}$-many $\Lambda_{k^{(1)}}(j_s)\subset \Lambda_k$ at the $k^{(1)}$-th hierarchical level, and the $2d_{k^{(1)}}$-extension of each $\Lambda_{k^{(1)}}(j_s)$ has its cardinality less than $(16d_{k^{(1)}}+1)^d$.

\end{proof}

\begin{rmk}\label{Lambda' has less eigenvalue number}
  Indeed, if  $\Lambda'_k$ denotes  the $2d_k$-neighborhood of $\Lambda_k$, then the similar argument as in  the proof of Lemma \ref{eigenvalue number} will  lead to 
   \[\# \left(  \spec(H_{\Lambda'_k})\cap [-\iota,4d+\beta+\iota] \right) \leq 2 \# \widehat{U}_{k,k^{(1)}} \leq 2 N^{k-k^{(1)}} (16d_{k^{(1)}}+1)^d. \]
\end{rmk}

\subsection{Proof of the Wegner estimate for $d=1,2,3$}
Now let  $1\leq d\leq 3$. The remaining part of this section is devoted to proving the following Wegner estimate. 
\begin{thm}[{\bf Wegner estimate for $d=1,2,3$}]\label{Wegner estimate for d=1,2,3}
Let $d=1,2, 3$. For any $0<\beta \leq 1$, there exist  some  $q>0$,   $\varepsilon>0$ and    $k_{\text{in}}>0$ (all of them depend only on $h,\beta,\alpha,d_0,N$) such that,  for all $k \geq k_{\text{in}}$ and $E\in [-\frac{\iota}{2},4d+\beta+\frac{\iota}{2}]$, we have 
\begin{equation}\label{Wegner poly probability}
\P\left( \dist\big(\spec(H_{\Lambda}), E\big) < \exp\{-d_{k+1}^{1-\varepsilon}\} \right) \leq d_{k+1}^{-q},
\end{equation}
where $\Lambda$ can be either $\overline{\Lambda}_k$ or $\overline{\Lambda}_k^s$ (with $s\geq 2$), and the probability $\P$ is taken only over the random variables  within $\Lambda$.
\end{thm}

\begin{proof}[Proof of Theorem \ref{Wegner estimate for d=1,2,3}]
Indeed, since  each  $\overline{\Lambda}_k^s$ shares  the same hierarchical structure as $\overline{\Lambda}_k$, it  suffices  to analyze the case of  $\Lambda= \overline{\Lambda}_k$. 

Let $\Lambda'_k$ denote the $2d_k$-neighborhood of $\Lambda_k$. We define the set of frozen sites  to be  $$\mcF = \Lambda'_k$$ and the complementary set of {\bf free sites}  as $$\mcS = \barLambda_k \setminus \Lambda'_k.$$ 
Choose $0<\varepsilon \ll1$   and $C_1 \gg1$  (both  depending only on $h, \alpha, \beta$) such that the unique continuation Theorems \ref{UC for d=1}, \ref{UC for d=2} and \ref{UC for d=3} are applicable, and  in addition $(1-\varepsilon)(1-2\varepsilon) > 1/\alpha$. Define the length scales
\[L_0=\diam(\barLambda_k),\ L_1=L_0^{1-2\varepsilon},\ L_2=L_1^{1-2\varepsilon}.\]
With the  choice of $\varepsilon$,  we then have
\begin{equation}\label{DS20 scale relationship}
  \diam(\Lambda_k')\sim d_k=d_{k+1}^{\frac{1}{\alpha}}\sim L_0^{\frac{1}{\alpha}}\ll L_2. 
\end{equation}

For  $k\gg1$, we have the following claim about the transversality estimate.
\begin{claim}\label{Claim UC}
Denote by $\mcE_{uc}=\mcE_{uc}(\barLambda_k, E)$ the event that the following holds: For any eigenvalue $\lambda$ with $|\lambda - E| \leq \exp\{-C_1 L_2 \log L_2\}$ and any corresponding eigenfunction $\psi_{\lambda}$ of $H_{\barLambda_k}$, one can find a subset $T(\omega, \psi_{\lambda}) \subset \mcS$ on which $\psi_{\lambda}$ is transversal, i.e.,
\[|\psi_{\lambda}(x)|\geq \exp\{-C_1L_2\log L_2\}\|\psi_{\lambda}\|_{\ell^{\infty}(\Lambda_k')} .\]
Moreover,
\begin{equation}\label{rho Sperner}
  \# T(\omega,\psi_\lambda)\geq N_{uc}:= 
  \begin{cases}
    \varepsilon L_2, \ d=1\\
    \varepsilon^3 L^2_2, \ d=2\\
     L_2^{p}, \ d=3,
  \end{cases}
\end{equation}
where $p>\frac{3}{2}$ is given in Theorem \ref{UC for d=3}. 
Then we have 
\[\P(\mcE_{uc}|V_{\text{r},\mcF })\geq 1-\exp\{-\varepsilon L^{\frac{2}{3}}_2\}.\]
\end{claim}
\begin{proof}[Proof of Claim \ref{Claim UC}]

Let $k\gg1$ be  such that $L_2>C_1$. 

For $d=1$, assume $\psi_{\lambda}$ attains its maximum within $\Lambda_k'$ at $x_0$. Observe that by \eqref{DS20 scale relationship},  we have 
\[\Lambda_k'\subset Q_{L_2 -8d_{k}}(x_0)\subset Q_{L_2}(0)\subset Q_{\frac{1}{2}L_0}(0)=\barLambda_k.\] 
Now we can  apply Theorem \ref{UC for d=1} to the cube $Q_{L_2 -8d_{k}}(x_0)$. Deterministically,  there exists a $\overline{T}(\omega,\psi_{\lambda})\subset Q_{L_2 -8d_{k}}(x_0)$ such that 
\[\# \overline{T}(\omega,\psi_\lambda) \geq \frac{1}{2}\cdot \# Q_{L_2 -8d_{k}}(x_0) =L_2-8d_k,\]
and for every $x\in \overline{T}(\omega,\psi_\lambda)$,
\[|\psi_\lambda(x)|\geq \gamma^{L_2-8d_k}|\psi_{\lambda}(x_0)|\geq  \exp\{-C_1 L_2\log L_2\}|\psi_{\lambda}(x_0)|.\]
Finally, we take $T(\omega,\psi_\lambda)=\overline{T}(\omega,\psi_\lambda)\setminus \Lambda_k'\subset \mcS$, and therefore 
\[\# T (\omega,\psi_\lambda) \geq L_2-8d_k-\# \Lambda_k'\geq \varepsilon L_2 .\]

For $d=3$, assume again that $\psi_{\lambda}$ attains its maximum within $\Lambda_k'$ at some  $x_0$. We first apply cone property (Proposition \ref{cone property}) to obtain a chain $(x_0,n_1,n_2, \cdots, x'_0)$ in the direction $\bfe_d$ starting at $x_0$ and terminating at $x'_0$. For this chain, inequality \eqref{Chain length} holds. Consequently, we may assume without loss of generality that $|x_0'|\in \{2 L_2+8d_k,2 L_2+8d_k-1\}$. The length of the chain is then at most $2L_2+16d_0$, and therefore
\[|\psi_{\lambda}(x_0')|\geq \gamma^{2 L_2+16d_k} |\psi_{\lambda}(x_0)|.\]
Notice that $Q_{L_2}(x_0')\subset \mcS$. Apply Theorem \ref{UC for d=3} to the cube $Q_{L_2}(x_0')$. Deterministically,  there exists a $T(\omega,\psi_{\lambda})\subset Q_{L_2}(x_0')\subset \mcS $ such that 
\[\# T(\omega,\psi_\lambda) \geq L_2^p,\]
and for every $x\in T(\omega,\psi_\lambda)$,
\begin{align*}
  |\psi_\lambda(x)| & \geq \exp\{-C_1  L_2\}  |\psi_{\lambda}(x_0')|\\
    &\geq  \exp\{-C_1 L_2\} \gamma^{2 L_2+16d_k} |\psi_{\lambda}(x_0)|\geq  \exp\{-C_1 L_2\log L_2\}|\psi_{\lambda}(x_0)|.
\end{align*}

The argument as above for $d=1$ and $d=3$ is deterministic, that is, it holds for all $\omega\in\Omega$ and all   $\lambda\in [-\iota,4d+\beta+\iota]$ (which contains the interval $[-\exp\{-C_1L_2\log L_2\}+E, \exp\{-C_1L_2\log L_2 \}+E]$, supposing $k\gg1$). Consequently, we have
\[\P(\mcE_{uc}|V_{\text{r},\mcF })=1 \geq 1-\exp\{-\varepsilon L^{\frac{2}{3}}_2\}.\]

For $d=2$, we have
\[\Lambda_k'\subset Q_{\frac12 L_2}(0) \subset \barLambda_k.\]
From  \eqref{DS20 scale relationship}, it follows that
\[\diam(\Lambda_k')\leq 16d_k \leq \varepsilon^2 L_2=\varepsilon^2 \cdot \diam(Q_{\frac{1}{2} L_2}(0)).\]
Therefore, we can  apply Theorem \ref{UC for d=2} to the block $Q_{\frac{1}{2} L_2}(0)$. This yields that for all eigenvalues $\lambda$ satisfying
\[|\lambda-E|\leq \exp\{-C_1 L_2\log L_2 \},\]
there is a set $T(\omega,\psi_{\lambda}) \subset Q_{\frac{1}{2}L_2}(0)\setminus \Lambda_k' \subset \mcS$ such that
\[\# T(\omega,\psi_\lambda) \geq \varepsilon^3 L_2^2,\]
and for every $x\in T(\omega,\psi_\lambda)$,
\begin{align*}
  |\psi_\lambda(x)| \geq  \exp\{-C_1 L_2\log L_2\} \|\psi_{\lambda}\|_{\ell^2(\Lambda_k')}.
\end{align*}
Finally, Theorem \ref{UC for d=2} also provides the probability estimate
\[\P(\mcE_{uc}|V_{\text{r},\mcF })\geq  1-\exp\{-\varepsilon L^{\frac{2}{3}}_2\}.\]
Combining the above discussions, we complete the proof of the Claim.
\end{proof}

With Claim \ref{Claim UC}  in hand, we can  proceed with the proof of the Wegner estimate. Let
\[\lambda_1(H_{\barLambda_k})\geq \lambda_2(H_{\barLambda_k})\geq \cdots \geq \lambda_{\# \barLambda_k} (H_{\barLambda_k})\]
denote the eigenvalues of $H_{\barLambda_k}$ arranged in the non-increasing order,  with  corresponding normalized eigenfunctions $\{\psi_j\}_{1\leq j \leq \# \barLambda_k}$.
For  $1\leq j_1, j_2 \leq \# \barLambda_k = L_0^d$ and an integer $0\leq \ell \leq L_0^{\varepsilon}$, denote by $\mcE_{j_1,j_2,\ell}$ the event that
\[|\lambda_{j_1}-E|\vee|\lambda_{j_2}-E|<s_{\ell} \ {\rm and} \ |\lambda_{j_1-1}-E|\wedge|\lambda_{j_2+1}-E|\geq s_{\ell},\]
where 
\[s_{\ell}=\exp\{-L_1+(2C_1L_2\log L_2 -2\gamma_0 d_{k}+C_2)\ell\}\]
with  $C_2\gg1$  will be specified  later.  By our choice of $\ell$ and \eqref{DS20 scale relationship}, we ensure  that 
\begin{align}\label{bound on s_l}
  s_{\ell}&\leq \exp\{-L_1+(2C_1L_2\log L_2 -2 \gamma_0 d_{k}+C_2)\cdot L_0^{\varepsilon}\} \\
 \notag    & \leq \exp\{-L_1 +2C_1 L_1^{1-2\varepsilon+\frac{\varepsilon}{1-2\varepsilon}}\log L_2 \}\leq \exp\{-\frac{1}{2}L_1\}\\
 \notag    &\leq \exp\{-C_1 L_2 \log L_2\}.
\end{align}

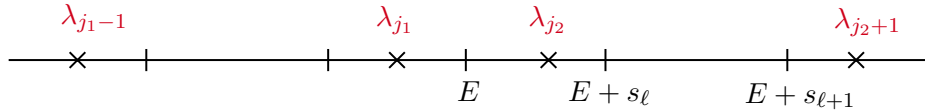
\begin{figure}[htbp]
  \centering

\tikzset{every picture/.style={line width=0.75pt}} 

\begin{tikzpicture}[x=0.75pt,y=0.75pt,yscale=-1,xscale=1]

\draw    (134.71,146) -- (226.43,146) ;
\draw [shift={(226.43,146)}, rotate = 180] [color={rgb, 255:red, 0; green, 0; blue, 0 }  ][line width=0.75]    (0,5.59) -- (0,-5.59)   ;
\draw    (65.29,146) -- (100,146) ;
\draw [shift={(100,146)}, rotate = 45] [color={rgb, 255:red, 0; green, 0; blue, 0 }  ][line width=0.75]    (-5.59,0) -- (5.59,0)(0,5.59) -- (0,-5.59)   ;
\draw    (261.14,146) -- (295.86,146) ;
\draw [shift={(295.86,146)}, rotate = 180] [color={rgb, 255:red, 0; green, 0; blue, 0 }  ][line width=0.75]    (0,5.59) -- (0,-5.59)   ;
\draw    (100,146) -- (134.71,146) ;
\draw [shift={(134.71,146)}, rotate = 180] [color={rgb, 255:red, 0; green, 0; blue, 0 }  ][line width=0.75]    (0,5.59) -- (0,-5.59)   ;
\draw    (226.43,146) -- (261.14,146) ;
\draw [shift={(261.14,146)}, rotate = 45] [color={rgb, 255:red, 0; green, 0; blue, 0 }  ][line width=0.75]    (-5.59,0) -- (5.59,0)(0,5.59) -- (0,-5.59)   ;
\draw    (295.86,146) -- (337.57,146) ;
\draw [shift={(337.57,146)}, rotate = 45] [color={rgb, 255:red, 0; green, 0; blue, 0 }  ][line width=0.75]    (-5.59,0) -- (5.59,0)(0,5.59) -- (0,-5.59)   ;
\draw    (337.57,146) -- (366.29,146) ;
\draw [shift={(366.29,146)}, rotate = 180] [color={rgb, 255:red, 0; green, 0; blue, 0 }  ][line width=0.75]    (0,5.59) -- (0,-5.59)   ;
\draw    (366.29,146) -- (458,146) ;
\draw [shift={(458,146)}, rotate = 180] [color={rgb, 255:red, 0; green, 0; blue, 0 }  ][line width=0.75]    (0,5.59) -- (0,-5.59)   ;
\draw    (458,146) -- (492.71,146) ;
\draw [shift={(492.71,146)}, rotate = 45] [color={rgb, 255:red, 0; green, 0; blue, 0 }  ][line width=0.75]    (-5.59,0) -- (5.59,0)(0,5.59) -- (0,-5.59)   ;
\draw    (492.71,146) -- (527.43,146) ;

\draw (90,116) node [anchor=north west][inner sep=0.75pt]  [color={rgb, 255:red, 208; green, 2; blue, 27 }  ,opacity=1 ] [align=left] {$\lambda_{j_1-1}$};
\draw (249,117) node [anchor=north west][inner sep=0.75pt]  [color={rgb, 255:red, 208; green, 2; blue, 27 }  ,opacity=1 ] [align=left] {$\lambda_{j_1}$};
\draw (326,117) node [anchor=north west][inner sep=0.75pt]  [color={rgb, 255:red, 208; green, 2; blue, 27 }  ,opacity=1 ] [align=left] {$\lambda_{j_2}$};
\draw (480,118) node [anchor=north west][inner sep=0.75pt]  [color={rgb, 255:red, 208; green, 2; blue, 27 }  ,opacity=1 ] [align=left] {$\lambda_{j_2+1}$};
\draw (290,155) node [anchor=north west][inner sep=0.75pt]   [align=left] {$E$};
\draw (346,155) node [anchor=north west][inner sep=0.75pt]   [align=left] {$E+s_{\ell}$};
\draw (435,156) node [anchor=north west][inner sep=0.75pt]   [align=left] {$E+s_{\ell+1}$};

\end{tikzpicture}
\caption{The event $\mcE_{j_1,j_2,\ell}$.}

\end{figure}

From now on, we take  condition  probability on $V_{\text{r},\mcF}$. This leads to a cylindrical decomposition of the probability space: 
\[\{0,1\}^{\barLambda_k}=\bigcup_{v\in \{0,1\}^{\mcF}} \{0,1\}^{\mcS}\times \{V_{\text{r},\mcF}=v  \}:= \bigcup_{v\in \{0,1\}^{\mcF}} \mcC_v. \]
In this representation,  both   $\mcE_{uc}$ and $\mcE_{j_1,j_2,\ell}$ can be viewed as subsets of $\mcC_v$. 

Recall  the definition  of $N_{uc}$ given by \eqref{rho Sperner}.  For $i=0,1$, let $\mcE_{j_1,j_2,\ell,i}$ denote the event that 
\[\mcE_{j_1,j_2,\ell} \ {\rm and} \ \#\left( \left\{|\psi_{j_1}|\geq \exp\{-C_1L_2\log L_2\}\|\psi_{j_1}\|_{\ell^{\infty}(\Lambda_k')}\right\} \cap \{V_{\text{r},\barLambda_k}=i\}\setminus \mcF    \right)\geq \frac{1}{2} N_{uc} .\]
By the pigeonhole principle, \eqref{bound on s_l} and Claim \ref{Claim UC},   we have 
\begin{equation}\label{i=0,1 event inclusion}
  \mcE_{j_1,j_2,\ell}\cap \mcE_{uc}\subset \mcE_{j_1,j_2,\ell,0}\cup  \mcE_{j_1,j_2,\ell,1}.
\end{equation}

Now if $\omega\in \mcE_{j_1,j_2,\ell,i}, x\in \mcS, \omega(x)=i$ and 
\begin{equation}\label{r_3 transversality previous}
  |\psi_{j_1}(x)|\geq \exp\{-C_1L_2\log L_2\}\|\psi_{j_1}\|_{\ell^{\infty}(\Lambda_k')}, 
\end{equation}
then we change the value of $V_{\text{r}}$ at $x$, i.e.,  take 
\begin{equation*}
  \omega_x(y)=\begin{cases}
    \omega(y), \ {\rm if} \ y \neq x,\\
    1-\omega(x), \ {\rm if} \ y=x.
  \end{cases}
\end{equation*}

\begin{claim}\label{Sperner structure}
  $\omega_x\notin \mcE_{j_1,j_2,\ell,i}.$
\end{claim}
\begin{proof}[Proof of Claim \ref{Sperner structure}]
We only consider the case $i=0$, as the proof  of   the case $i=1$ is analogous. Define the shifted operator
\[\widetilde{H}_{\barLambda_k}(\omega)=H_{\barLambda_k}(\omega)-E+s_\ell,\]
of which all eigenvalues are $\widetilde{\lambda}j=\lambda_j-E+s_{\ell}$ (For $i=1$, the corresponding operator is $-(H_{\barLambda_k}(\omega)-E-s_{\ell})$).  Set
\begin{equation}\label{r_1,r_2}
  r_1=2s_\ell,\ r_2=s_{\ell}+s_{\ell+1}.
\end{equation}
Then since $\omega\in \mcE_{j_1,j_2,\ell}$,  the following ordering holds true:
\[0<\widetilde{\lambda}_{j_1}\leq \wtlambda_{j_2}<r_1<r_2<\wtlambda_{j_2+1}.\]

Now recall the set $\widehat{U}_{k,k^{(1)}}\subset\Lambda_k \subset \Lambda_k'$ defined in the proof of Lemma \ref{eigenvalue number}. 
Since $E \in [-\frac{\iota}{2}, 4d+\beta+\frac{\iota}{2}]$, and assuming $k\gg1$ such that 
\[|E-\lambda_{j_1}|\leq s_{\ell} \leq \exp\{-C_1L_2\log L_2\}\ll \frac{\iota}{2}, \]
we have
\begin{equation}\label{lambda still less than h}
  \lambda_{j_1}\in [-\iota,4d+\beta+\iota].
\end{equation}
This implies that the estimate \eqref{k^(1) large 3} is valid for $\psi_{j_1}$, and
\begin{align*}
  \|\psi_{j_1}\|_{\ell^{\infty}(\Lambda_k')} & \geq \|\psi_{j_1}\|_{\ell^{\infty}(\widehat{U}_{k,k^{(1)}})} \geq \sqrt{\frac{1}{2 \# \widehat{U}_{k,k^{(1)}}}} \geq \sqrt{\frac{1}{2 N^{k-k^{(1)}} (16d_{k^{(1)}}+1)^d}} \\
    &\geq L_0^{-\frac{\varepsilon}{4}}\sim d_{0} ^{-\frac{\varepsilon}{4}\cdot \alpha ^{k+1}},
\end{align*}
where in  the last inequality  we used  $k\gg1$. Combining the above estimates  with \eqref{r_3 transversality previous} yields
\begin{equation}\label{r_3 transversality}
  |\psi_{j_1}(x)|\geq  L_0^{-\frac{\varepsilon}{4}} \exp\{-C_1L_2\log L_2\} \geq \exp\{-2C_1L_2\log L_2\}:= r_3.
\end{equation}

Next, recalling \eqref{bound on s_l}, we set 
\begin{equation}\label{r_5}
  r_5= 2\exp\{-C_1 L_2 \log L_2\}\geq \exp\{-C_1 L_2 \log L_2\}+s_\ell > r_2.
\end{equation}
Consider now  a $j$ such that $r_2<\wtlambda_j<r_5$. Then $|\lambda_j-E|\leq r_5-s_{\ell}<r_5$. 
By  similar argument leading  to \eqref{lambda still less than h}, we get   $\lambda_j\in \spec(H_{\barLambda_k})\cap[-\iota,4d+\beta+\iota]$.
Consequently, the estimate \eqref{k^(1) large 1} is applicable to the eigenfunctions corresponding to such $\lambda_j$. Simultaneously, notice that $d(x,\widehat{U}_{k,k^{(1)}})\geq 2d_k$ by $x\notin \Lambda_k'$. We thus obtain 
\begin{equation}\label{r_4}
      \sum_{j:r_2<\wtlambda_j<r_5}|\psi_j(x)|^2 \leq \# \barLambda_k \cdot \exp\{-\gamma_0 d_k\}\leq \exp\{-2\gamma_0 d_k\}:= r_4.
\end{equation}
The proof will be completed after applying the following lemma concerning rank-one perturbations, which is adapted from \cite[Lemma 5.1]{DS20}.
\begin{lem}\label{rank-one perturbation}
Suppose that  the real symmetric matrix \( A \in  \mathbb{R}^{n\times n}  \) has eigenvalues  
\[
\lambda_1 \geq \lambda_2 \geq \cdots \geq \lambda_n \in \mathbb{R}
\]  
with orthonormal eigenbasis \( v_1, v_2, \cdots, v_n \in \mathbb{R}^n \). Then for every $\beta>0$, there is some   $0<c\ll1$ (depending only  on $\beta$) such that, if  
\begin{enumerate}
    \item \(0 < r_1 < r_2 < r_3 < r_4 , r_5 < 1\),  
    \item \(r_1 \leq c \min\{r_3 r_5, r_2 r_3 / r_4\}\),
    \item \(0 < \lambda_j \leq \lambda_i < r_1 < r_2 < \lambda_{i-1}\),  
    \item \(v_{j}^2(x) \geq r_3\),  
    \item \(\displaystyle\sum_{r_2 < \lambda_s < r_5} v_{s}^2(x) \leq r_4\),  
\end{enumerate}
then 
\[
\operatorname{trace} \mathbf{1}_{[r_1, \infty)}(A) < \operatorname{trace} \mathbf{1}_{[r_1, \infty)}(A + \beta e_x \otimes e_x),
\]
where \( e_x \in \mathbb{R}^n \) is the \(x\)-th standard basis element.

\end{lem}
\begin{rmk}
In the continuous model on $\R^d$, \cite{BK05} handled  the rank-one perturbations on a free site by continuing the Bernoulli variable   and   using  eigenvalue variations. This approach seems invalid  in the discrete setting, because the transversal set $T(\omega,\psi_{\lambda})$ is not the whole space, but depends on $\omega$,  which  cannot allow  to take differentiability. 
Furthermore, we note that the assumption $r_4 < r_5$ in \cite[Lemma 5.1]{DS20} is, in fact, unnecessary as seen from the  proof.  We have therefore omitted it here.
\end{rmk}

Having established Lemma \ref{rank-one perturbation}, we recall the parameters $r_1,r_2,r_3,r_4,r_5$ defined in \eqref{r_1,r_2}, \eqref{r_3 transversality}, \eqref{r_4} and \eqref{r_5}. To verify the hypothesis of Lemma \ref{rank-one perturbation}, we compute
\begin{equation*}
  r_2 r_3/r_4 = r_1 \frac{1+\exp\{2C_2L_2\log L_2-\frac{1}{2}\gamma_0 d_k+C_2\}}{2}\cdot \exp\{-2C_1L_2\log L_2+2\gamma_0 d_k\} \geq \frac{e^{C_2}}{2} r_1.
\end{equation*}
Consequently,
\begin{align*}
  \min \{r_3 r_5,r_2 r_3/r_4\} & \geq  \min\{ 2\exp\{-3C_1 L_2 \log L_2\}    , \frac{e^{C_2}}{2} r_1 \} >c r_1.
\end{align*}
Here we choose $C_2\gg1$  and use estimate \eqref{bound on s_l}, which guarantees that for $k\gg1$,
\[r_1\leq \exp\{-\frac{1}{2}L_1\} \ll 2\exp\{-3C_1 L_2 \log L_2\}  .\]
Therefore, Lemma \ref{rank-one perturbation} is applicable to the rank-one perturbation
\[\widetilde{H}_{\barLambda_k}(\omega_x)=\widetilde{H}_{\barLambda_k}(\omega) +\beta e_x\otimes e_x,\]
yielding
\[
\operatorname{trace} \mathbf{1}_{[r_1, \infty)}(\widetilde{H}_{\barLambda_k}(\omega)) < \operatorname{trace} \mathbf{1}_{[r_1, \infty)}(\widetilde{H}_{\barLambda_k}(\omega_x)).
\]
This inequality is precisely equivalent to $\lambda_{j_2}(\omega_x) > s_{\ell+1}$. Hence, $\omega_x \notin \mcE_{j_1,j_2,\ell}$, and we obtain
\[\omega_x \notin \mcE_{j_1,j_2,\ell,0} \subset \mcE_{j_1,j_2,\ell}.\]

\end{proof}

Indeed, Claim \ref{Sperner structure} holds in a more general sense. Recall the event  $\mcE_{j_1,j_2,\ell,i}$. Define
\[T'(\omega,\psi_{j_1},i)= \left\{|\psi_{j_1}|\geq \exp\{-C_1L_2\log L_2\}\|\psi_{j_1}\|_{\ell^{\infty}(\Lambda_k')}\right\} \cap \{V_{\text{r},\barLambda_k}=i\}\setminus \mcF  .\] 
Now suppose $\omega\in \mcE_{j_1,j_2,\ell,i}$ and $\widetilde{\omega}$ is obtained from $\omega$ by modifying it only on the sites in $T'(\omega,\psi_{j_1},i)$,  i.e.,
\[\{y:\ \omega (y) \neq \widetilde{\omega}(y)\} \subset T'(\omega,\psi_{j_1},i) .\]
We still restrict our attention to the case $i=0$. Choose an arbitrary site $x \in {y:\ \omega (y) \neq \widetilde{\omega}(y)}$. The proof of Claim \ref{Sperner structure} shows that $\lambda_{j_2}(\omega_x) > s_{\ell+1}$. Now, by the monotonicity of eigenvalues under positive perturbations, we have
\[\lambda_{j_2}(\widetilde{\omega})>\lambda_{j_2}(\omega_x)>s_{\ell+1},\]
which again implies 
\begin{equation}\label{general Sperner structure}
  \widetilde{\omega}\notin \mcE_{j_1,j_2,\ell,0}.
\end{equation}

If we identify a $\omega$ with the subset of $\mcS$ given by
\[\omega\Leftrightarrow \{x\in \mcS:\ \omega(x)=1 \},\]
then \eqref{general Sperner structure} precisely means that $\mcE_{j_1,j_2,\ell,0}$ is a $\rho$-Sperner family of $\mcS$ (for the definition of $\rho$-Sperner, see \cite[Definition 4.1]{DS20}), where
\begin{equation}\label{Sperner rho final}
  \rho= \frac{\frac{1}{2} N_{uc}}{\# \mcS }.
\end{equation}\label{Sperner estimate for i=0}
Thus, using \cite[Theorem 4.2]{DS20}   enables  us to get  
\begin{align}
  \P(\mcE_{j_1,j_2,\ell,0}| V_{\text{r},\mcF} ) \leq \left(\#\mcS\right)^{-\frac{1}{2}} \cdot \rho^{-1}=\frac{2\sqrt{\# \mcS}}{N_{uc}}.
\end{align}
For $i=1$, if we identify $\omega$ with 
\[\omega\Leftrightarrow \{x\in \mcS:\ \omega(x)=0 \},\]
then a similar  argument will yield  
\begin{align}\label{Sperner estimate for i=1}
  \P(\mcE_{j_1,j_2,\ell,1}| V_{\text{r},\mcF} ) \leq \left(\#\mcS\right)^{-\frac{1}{2}} \cdot \rho^{-1}=\frac{2\sqrt{\# \mcS}}{N_{uc}}.
\end{align}

Finally, the proof of Theorem \ref{Wegner estimate for d=1,2,3} can  be completed once we established  the following claim.

\begin{claim}\label{Wegner set inclusion}
  There exists  a subset $\mcK\subset \{1,2,\cdots,\# \barLambda_k\}$ depending only on the randomness in the frozen sites  set $\mcF$ such that,  $\# \mcK\leq L_0^{\frac{\varepsilon}{2}}$ and 
  \[\{\omega:\  \dist(H_{\barLambda_k},E)< \exp\{-L_1\}\} \cap \{V_{\text{r},\mcF}=v \}\subset \bigcup_{j_1,j_2\in \mcK \atop 0\leq \ell \leq L_0^{\varepsilon} }\mcE_{j_1,j_2,\ell}.\]
\end{claim}

Indeed, supposing Claim \ref{Wegner set inclusion} holds true, then by Claim \ref{Claim UC}, together with \eqref{i=0,1 event inclusion}, \eqref{Sperner estimate for i=0} and \eqref{Sperner estimate for i=1}, we obtain
\begin{align}\label{bare bound with Nuc}
\notag  \P \left(\dist(H_{\barLambda_k},E)< \exp\{-L_1\} | V_{\text{r},\mcF} \right) & \leq \P(\mcE_{uc}^c| V_{\text{r},\mcF} ) +\P\left(\bigcup_{j_1,j_2\in \mcK \atop 0\leq \ell \leq L_0^{\varepsilon} }\mcE_{j_1,j_2,\ell} \bigg| V_{\text{r},\mcF} \right) \\
  \notag      & \leq \exp\{-\varepsilon L_2^{\frac{2}{3}}\} + \sum_{j_1,j_2\in \mcK \atop 0\leq \ell \leq L_0^{\varepsilon}} \P\left(\mcE_{j_1,j_2,\ell} \cap \mcE_{uc} | V_{\text{r},\mcF} \right) \\
    \notag    &\leq  \exp\{-\varepsilon L_2^{\frac{2}{3}}\} + \sum_{j_1,j_2\in \mcK \atop 0\leq \ell \leq L_0^{\varepsilon}} \sum_{i=0,1} \P\left(\mcE_{j_1,j_2,\ell,i} | V_{\text{r},\mcF} \right) \\
      & \leq \exp\{-\varepsilon L_2^{\frac{2}{3}}\} + (\#\mcK)^2 L_0^{\varepsilon} \frac{4\sqrt{\# \mcS}}{N_{uc}}.
\end{align}
Recall that $N_{uc}$ is given in  \eqref{rho Sperner}. Consequently,
\begin{equation*}
    \P \left(\dist(H_{\barLambda_k},E)< \exp\{-L_1\} | V_{\text{r},\mcF} \right)  \leq 
    \begin{cases}
      8\varepsilon^{-1}L_0^{\frac{1}{2}+2\varepsilon}L_2^{-1}, \ d=1, \\
        8\varepsilon^{-3}L_0^{1+2\varepsilon}L_2^{-2}, \ d=2,\\
      8 L_0^{\frac{3}{2}+2\varepsilon}L_2^{-p} , \ d=3.
    \end{cases}
\end{equation*}
This implies
\[\P \left(\dist(H_{\barLambda_k},E)<\exp\{-d_{k+1}^{1-\varepsilon}\} | V_{\text{r},\mcF} \right) \leq L_0^{-q'} \leq d_{k+1}^{-q} \]
for some   $q,q'>0$. Finally, applying the law of total expectation  can  remove the conditioning, and then  yield  \eqref{Wegner poly probability}. The entire proof is conducted under the assumption that $k\gg1$. We choose $k_{\text{in}}$ to be large enough (depending  only on $h,\beta,\alpha,d_0,N$) to ensure the validity of the proof provided  $k\geq k_{\text{in}}$.

\end{proof}

\begin{rmk}
From \eqref{bare bound with Nuc}, it is clear that for the above argument to hold in more general dimensions $d\geq 4$, one  needs to prove in the unique continuation lemma a lower bound of   $N_{uc}\gtrsim L^{\frac{d}{2}+}$ for the number of sites where the  transversality occurs. This remains unsettled. 
\end{rmk}

Finally, we prove Claim \ref{Wegner set inclusion}, thus  finish this section.
\begin{proof}[Proof of Claim \ref{Wegner set inclusion}]
The proof is based on a bootstrap estimate first introduced in \cite{Bou04}, and subsequently employed in \cite{BK05,DS20,LZ22,Li22,LSZ25}.
 
Now, for each $v\in \{0,1\}^{\mcF}$, we take probability condition on the event $\{V_{\text{r},\mcF}=v\}$. The remaining randomness is then denoted by $\zeta \in \{0,1\}^{\mcS}$. We define  
 \begin{equation}\label{mcK}
    \mcK(v)=\left\{ j:\  \exists \zeta \in \{0,1\}^{\mcS} \ {\rm such\ that} \ |\lambda_j(\zeta,v)-E|\leq \exp\{-L_2 \}       \right\}.
  \end{equation}
  
Since $V_{\text{r},\mcF}\equiv v$ has been fixed, we omit the dependence of $\lambda_j$ on $v$ for  simplicity. The function $\lambda_j(\zeta)$ can be naturally extended to a continuous domain by allowing the variable $\zeta$ to vary in $[0,1]^{\mcS}$ instead of $\{0,1\}^{\mcS}$. 

By the first-order eigenvalue variation formula, for every $x\in \mcS$,
\begin{equation}\label{first derivation 1}
  \left|\frac{\partial}{\partial \zeta (x)} \lambda_j(\zeta)\right|=\beta |\psi_j(\zeta)(x)|^2 \leq 1.
\end{equation}
This derivation is valid because the Kato–Rellich theorem guarantees the analyticity of $\lambda_j(\zeta)$ with respective to each single coordinate $\zeta(x)$. Applying the mean value theorem yields 
\begin{equation}\label{mean value theorem}
  \left|   \lambda_j(\zeta_1)-\lambda(\zeta_2)  \right| \leq |\zeta_1-\zeta_2|_1 \leq \#\mcS \cdot \|\zeta_1-\zeta_2\|_{\infty}.
\end{equation}

Now for any eigenvalue $\lambda_j(\zeta),\zeta\in [0,1]^{\mcS}$, if  
\begin{equation}\label{loose bound}
  |\lambda-E|\leq \exp\{-\frac{1}{2}\gamma_0 d_k\},
\end{equation}
then a similar argument  leading to \eqref{lambda still less than h} will ensure $\lambda_j(\zeta)\in [-\iota,4d+\beta+\iota]$. Then the estimate \eqref{k^(1) large 1} will be valid, implying
\begin{equation}\label{faster derivation decay}
  |\psi_j(\zeta)(x)|   \leq \exp\{-\frac{1}{2}\gamma_0 \cdot \dist(x,U_{k,k^{(1)}})\}\leq \exp\{-\gamma_0 d_k\}
\end{equation}
for every $x\in \mcS$. 

For any  $j\in \mcK$, by our definition \eqref{mcK}, there exists  a $\zeta_0\in [0,1]^{\mcS}$ such that
\[|\lambda_j(\zeta_0)-E|\leq \exp\{-L_2 \}.\]
Pave $[0,1]^{\mcS}$ with $e^{-100L_2}$-size cubes (in  the supremum  norm) and assume $\zeta_0\in B$ ($B$  is  of $\exp\{-100L_2\}$-size). Thus for all $\zeta \in B,$  we can apply \eqref{mean value theorem} to control
\begin{align*}
  |\lambda_j(\zeta)-E| & \leq |E-\lambda_j(\zeta_0)| + |\lambda_j(\zeta_0)-\lambda_j(\zeta)| \\
           & \leq \exp\{-L_2\}+ L_0^d \cdot \|\zeta-\zeta_0\|_{\infty} \\
           & \leq \exp\{-L_2\}+ L_0^d \exp\{-100L_2\} \ll \exp\{-\frac{1}{2}\gamma_0 d_k \}.
\end{align*}
This implies \eqref{loose bound} holds for all $\zeta\in B$, and therefore \eqref{faster derivation decay} is valid in the whole $B$, which gives us the refined estimate on the derivative 
\begin{equation}\label{first derivation 2}
  \left|\frac{\partial}{\partial \zeta (x)} \lambda_j(\zeta)\right|=\beta |\psi_j(\zeta)(x)|^2 \leq \exp\{-2\gamma_0 d_k\}.
\end{equation}
Apply \eqref{first derivation 2}, instead of \eqref{first derivation 1}, in the mean value theorem. We obtain 
\begin{align*}
  |\lambda_j(\zeta)-E| \leq \exp\{-L_2\}+ L_0^d \exp\{-2\gamma_0 d_k\}\cdot \|\zeta-\zeta_0\|_{\infty} 
\end{align*}
holds for any $\zeta \in B$. We can propagate the above estimate iteratively over the entire $[0,1]^{\mcS}$ via the $\exp\{-100L_2\}$-size pavements, and finally deduce that 
\begin{align*}
  |\lambda_j(\zeta)-E| & \leq \exp\{-L_2\}+ L_0^d \exp\{-2\gamma_0 d_k\}\cdot \|\zeta-\zeta_0\|_{\infty} \\
      & \leq \exp\{-L_2\}+ L_0^d \exp\{-2\gamma_0 d_k\} \leq \exp\{-\frac{1}{2}\gamma_0 d_k\}
\end{align*}
holds for any $\zeta\in [0,1]^{\mcS}$. Again, by a similar argument  leading to \eqref{lambda still less than h}, we obtain 
\[\lambda_j(\zeta)\in [-\iota,4d+\beta+\iota]\ {\rm for}\ \forall j\in \mcK(v),\ \forall \zeta\in [0,1]^{\mcS}.\]
Therefore, we can apply Lemma \ref{eigenvalue number} to deduce that 
\[\# \mcK(v)\leq \# \left(\spec(H_{\barLambda_k})\cap [-\iota,4d+\beta+\iota] \right) \leq 2 N^{k-k^{(1)}} (16d_{k^{(1)}}+1)^{d} \ll L_0^{\frac{\varepsilon}{2}},\]
supposing that $k\gg1$. 

Finally, suppose $\zeta \in \{\omega: \ \dist(H_{\barLambda_k},E)< \exp\{-L_1\}\} \cap \{V_{\text{r},\mcF}=v \} $. Then there exists an eigenvalue $\lambda_{j_0}(\zeta)$ satisfying 
\[|\lambda_{j_0}(\zeta)-E|<d \exp\{-L_1\}\leq \exp\{-L_2\}.\]
Therefore, by our definition of $\mcK(v)$, we have $j\in \mcK(v)$. Recall that we let $0\leq \ell\leq L_0^{\varepsilon}$ and $s_0=\exp\{-L_1\}$. Thus the interval $(E-s_0,E+s_0)$ contains $\lambda_{j_0}(\zeta)$. Moreover, since
\[\#\mcK(v)\leq L_0^{\frac{\varepsilon}{2}}\ll L_0^{\varepsilon},\]
there must be some $0\leq \ell' \leq L_0^{\varepsilon}$ such that 
\[\big( (E-s_{\ell'+1},E-s_{\ell'}]\cup[E+s_{\ell'},E+s_{\ell'+1}) \big)\cap \{\lambda_j(\zeta):\ j-1 \ {\rm or} \ j \ {\rm or}\ j+1\in\mcK(v)\}=\emptyset .\]
Now from  $\lambda_{j_0}(\zeta)\in (E-s_0,E+s_0)$, it follows  that  
\[\{j\in \mcK(v):\ E-s_{\ell'}<\lambda_j(\zeta)<E+s_{\ell'} \}\neq \emptyset.\]
 Therefore, we can define
\begin{align*}
   &j_1=\min\{j\in \mcK(v):\ E-s_{\ell'}<\lambda_j(\zeta)<E+s_{\ell'} \},\\
   & j_2=\max\{j\in \mcK(v):\ E-s_{\ell'}<\lambda_j(\zeta)<E+s_{\ell'} \},
\end{align*}
and then  $\mcE_{j_1,j_2,\ell'}$ happens. Hence $\zeta\in \mcE_{j_1,j_2,\ell'}$, which completes the proof.

\end{proof}

\section{Wegner estimate on $\Z^d$ for  $d\geq 4$}\label{section dim bigger 4}

In this section, we establish a Wegner estimate in higher dimensions ($d\geq4$). For this higher-dimensional setting, we do not resort to results pertaining to unique continuation arguments; instead, we rely solely on the transversality estimate obtained on one-dimensional subsets, as furnished by the cone property (Lemma \ref{transversality lemma}). Although the transversality in this framework is weakened—manifesting specifically in the relatively small number of points where transversality holds—we can suppress the quantum tunneling effect of particles by tuning the height and width of the hierarchical potential barriers. This adjustment enhances the transversality at individual points, ultimately offsetting the substantial loss of transversality inherent to the higher-dimensional lattice.

The proof of the higher-dimensional Wegner estimate draws substantial inspiration from the eigenvalue shift analysis presented in \cite{Imb21}. The methodology in \cite{Imb21} demonstrates that transversality at a single point can displace at least one eigenvalue far from the target energy $E$, with such a shift occurring over a dyadic scale. However, in the multi-scale analysis employed previously, we have established that the scales typically encountered are Newtonian scales. A key insight underpinning the present proof of the higher-dimensional Wegner estimate is therefore the interpolation of $a$-adic scales between consecutive Newtonian scales, which enables us to derive estimates of higher precision. At the final stage of establishing the probabilistic bounds, we uncover an important martingale structure. In summary, the derivation of the Wegner estimate on $\Z^d$ for $d\ge 4$ constitutes the most critical novel contribution of this work.

The proof of  the higher-dimensional  Wegner estimate   is largely inspired by the analysis of eigenvalue movement   presented in \cite{Imb21}. The technique  of \cite{Imb21} demonstrates  that the transversality at a single point can move at least one eigenvalue far away from the energy $E$, and such movement happens under a dyadic scale.  
However, in the  multi-scale analysis employed earlier, we have shown that the scales that we typically encounter  are {\bf Newtonian scales}. Therefore, a key idea in the present proof for higher dimensional Wegner estimate  is to interpolate $a$-adic scales between two Newtonian scales, thereby obtaining estimates of higher precision. At the final stage of proving the probabilistic estimate, we discover an important martingales structure.  In a word, the proof of Wegner estimate on $\Z^d$ for $d\geq 4$ is the  most important new ingredient of our work. 

We still use the notation $\overline{\Lambda}_k$ and $\overline{\Lambda}_{k}^s$ to stand for the $\frac{1}{10}d_{k+1}$-neighborhood of $\Lambda_k$ and $\Lambda_k(j_s)$ (for $s\geq 2$), respectively. The Wegner estimate in higher dimensions  is  
\begin{thm}[{\bf Wegner estimate for $d\geq 4$}]\label{Wegner estimate for d>3}
 Let $d\geq 4$. For any $0<\beta \leq 1$, there is  some $h_0=h_0(d) \gg 1$ such that,  if $h>h_0,\alpha > N$, then the following holds ture: There exist  some  $q>0$,  $\varepsilon>0$ and  $k_{\text{in}}$ (all of them only  depending  on $h,\beta,\alpha,d_0,N$) such that,  for all $k \geq k_{\text{in}}$ and all $E\in [-\frac{\iota}{2},4d+\beta+\frac{\iota}{2}]$, we have 
\begin{equation}\label{Wegner poly probability 2}
\P\left( \dist\big(\spec(H_{\Lambda}), E\big) < \exp\{-d_{k+1}^{1-\varepsilon}\} \right) \leq d_{k+1}^{-q},
\end{equation}
where $ {\Lambda}$ can be either $\overline{\Lambda}_k$ or $\overline{\Lambda}_k^s$ (with $s\geq 2$), and the probability $\P$ is taken only over the random variables  within $\Lambda$.
\end{thm}

Again, because the sets $\overline{\Lambda}_k^s$ share the same hierarchical structure as $\overline{\Lambda}_k$, we only focus on the case of $\Lambda = \overline{\Lambda}_k$.

\subsection{Refinement of the scales}
We still let $\Lambda'_k$ denote the $2d_k$-neighborhood of $\Lambda_k$, so that $\Lambda_k'\subset \barLambda_k$ and 
\[\diam(\Lambda_k')\approx 16 d_k, \ \diam(\barLambda_k)\sim d_{k+1} .\]
Set $L_0=\diam(\Lambda_k')$ and introduce the following \textbf{$a$-adic scales} for some  integer $a\geq 2$: 
\begin{equation}\label{a-adic scale}
  L_{j+1}=a \cdot L_j, \ j\geq 0.
\end{equation}
Let $P$ be the unique integer satisfying $L_P\leq d_{k+1}^{1/ \sqrt{\alpha}}=d_{k}^{\sqrt{\alpha}} <L_{P+1}$. Consequently, 
\begin{equation}\label{P quantitive}
  P\sim \log_a  \left( \frac{d_k^{\sqrt{\alpha}}}{16 d_k}\right) \sim \frac{\sqrt{\alpha}-1}{\log a} \cdot \log d_k \sim   (\sqrt{\alpha}-1)\frac{\log d_0}{\log a} \cdot \alpha^k.
\end{equation} 
For $0\leq j \leq P$, we define $B_j\equiv \Lambda_k'$ and let $\overline{B}_j$ be the $a L_j$-neighborhood of $B_j$. In this way,  we obtain a nested block structure interpolating between $\Lambda_k'$ and $\barLambda_k$: 
\[\Lambda_k'\subset \barB_0\subset \barB_1\cdots \subset \barB_P\subset \barLambda_k \]
with corresponding scales 
\begin{equation}\label{interpolate a-adic scale}
  \diam(\Lambda_k') = L_0 <L_1<L_2<\cdots<L_P\sim d_k^{\sqrt{\alpha}}\ll d_{k+1}\sim \diam(\barLambda_k).
\end{equation}
The value of $a$ will be chosen at the end of the proof (it  suffices  to take $a \geq 5$) and will then remain fixed throughout.

\subsection{Schur complements and the smallness in different scales}
In this part, we will study the movement of eigenvalues under the $a$-adic scale, using the iterative Schur complement method in \cite{IM16,Imb21}. The following fundamental lemma is adapted from \cite[Lemma 1.4]{IM16}, and we omit its proof.
\begin{lem}\label{fundamental schur lemma}
Let \( K \) be a \((p+q) \times (p+q)\) symmetric matrix in block form
\[
K = \begin{pmatrix}
A & B \\
C & D
\end{pmatrix}
\]
with \( A \) a \( p \times p \) matrix, \( D \) a \( q \times q \) matrix, and \( C = B^{\top} \) ($B^{\top}$ the denotes adjoint matrix of $B$). 
Assume that 
\[
\|(D - E)^{-1}\| \leq \epsilon^{-1}, \  \|B\| \leq \delta, \  \|C\| \leq \delta,\ E\in\R.
\]
Define the Schur complement with respect to \(\lambda\):
\[
S_\lambda = A - B(D - \lambda)^{-1}C.
\]
Let \(\epsilon\) and \(\delta / \epsilon\) be small enough and \(|\lambda - E| \leq \epsilon / 2\). Then

\begin{enumerate}
    \item[(i)] If \(\varphi\) is an eigenvector of  \(S_\lambda\) with eigenvalue \(\lambda\), then 
    \[
    \langle \varphi, -(D - \lambda)^{-1}C\varphi\rangle
    \]
    is an eigenvector of  \(K\) with eigenvalue \(\lambda\), and all eigenvectors of \(K\) with eigenvalue \(\lambda\) are of this form.
    
    \item[(ii)]
    \[
    \|S_\lambda - S_E\| \leq 2\left(\frac{\delta}{\epsilon}\right)^2 |\lambda - E|.
    \]
    
    \item[(iii)] The spectrum of \(K\) in \([E - \epsilon / 2, E + \epsilon / 2]\) is in close agreement with that of \(S_E\) in the following sense. 
    If $\lambda$ is an eigenvalues of \(K\) in \([E - \epsilon / 2, E + \epsilon / 2]\), 
    then there exists a corresponding eigenvalue $\tilde{\lambda}$ of \(S_E\), such that
    \[
    |\lambda - \tilde{\lambda}| \leq 2(\frac{\delta }{ \epsilon} )^2 |\lambda - E|. 
    \]
\end{enumerate}
\end{lem}
\begin{rmk}\label{need additional orthonormal vector argument}
    Let \(\lambda_1 \leq \lambda_2 \leq \cdots \leq \lambda_m\) be the eigenvalues of \(K\) lying in \([E - \epsilon / 2, E + \epsilon / 2]\). 
    Then by the above lemma, for each $\lambda_i$ one can find a corresponding eigenvalue  $\tilde{\lambda}_i$ of $S_E$ to approximate it, with the ordered list 
    \[\tilde{\lambda}_1 \leq \tilde{\lambda}_2 \leq \cdots \leq \tilde{\lambda}_m .\]
    However, such list does not necessarily coincide with the actual ordering (with multiplicity) of the eigenvalues of $S_E$. Hence, Lemma \ref{fundamental schur lemma}  {\textbf{does not}}  imply 
    \begin{equation}\label{schur complement spec inclusion}
      \#\left(\spec(K)\cap [E-\frac{\epsilon}{2},E+\frac{\epsilon}{2}]\right) \leq \# \left( \spec(S_E) \cap [E-\frac{\epsilon}{2}-\frac{\delta^2}{\epsilon},E+\frac{\epsilon}{2}+\frac{\delta^2}{\epsilon}]\right) .
    \end{equation}
    Therefore, to obtain \eqref{schur complement spec inclusion},  some  extra efforts  are required.
\end{rmk}

Now for each $0\leq j\leq P$, we write the matrix $H_{\barB_j}$ in the block form
\[
H_{\barB_j} = \begin{pmatrix}
H_{B_j} & \Gamma_j ^{\top}\\
\Gamma_j & H_{\barB_j\setminus B_j}
\end{pmatrix},
\]
where the connecting matrix $\Gamma_j:\ \ell^2(B_j)\rightarrow \ell^2(\barB_j\setminus B_j)$ comes from  $-\Delta$. The Schur complement in scale $L_j$ with respect to $\mcE\in\R$ is defined by 
\begin{equation}\label{schur complement in j-step}
  \wtSj_\mcE(B_j):= H_{B_j}-\Gamma_j^{\top}(H_{\barB_j\setminus B_j}-\mcE)^{-1}\Gamma_j.
\end{equation}
We will study the movement of eigenvalues of the Schur complements from  scale $L_j$ to scale $L_{j+1}$. Throughout this procedure, we define the following smallness: Recall $\gamma=\frac{1}{4d+h+\beta}$ in Lemma \ref{transversality lemma}, and we set
\begin{itemize}
  \item \textbf{(Spectral distance)} $\varepsilon_j=   \gamma^{(a+1+0.2)L_j} $,
  \item \textbf{(Transversality)} $t_j=\gamma^{(a+1+0.1)L_j}$,
  \item \textbf{(The smallness generated by out-and-in random walks)} $g_j=\gamma^{(2a-0.3)L_j}$.
\end{itemize}
(Indeed, the factor $0.1$ can be replaced by any sufficiently small quantity.) The reason for this definition of smallness will be explained later.
Certainly,
\begin{equation}\label{relationship of smallness}
  \varepsilon_{j+1}\ll t_{j+1}\ll g_j \ll \varepsilon_j \ll t_j. 
\end{equation}
Keeping \eqref{relationship of smallness} in mind will help us better understand the eigenvalue behavior discussed in the following subsection.

\subsection{Movement of eigenvalues}
In this part,  we discuss the movement of eigenvalues. Let $E\in [-\frac{\iota}{2},4d+\beta+\frac{\iota}{2}]$. Denote 
\begin{equation}\label{j-step number of eigenvalues}
  \whn_j(B_j)=\#\left(\spec(\wtSj_E(B_j) )\cap  I_{\varepsilon_{j}}(E)\right),
\end{equation}
where  $I_{\delta}(E):=[E-\delta,E+\delta]$. Obviously, by our definition,  $\whn_j(B_j)$ only depends on the randomness in $\barB_j$.

Inspired by \cite[Proposition 3.2]{Imb21}, we will show that, with probability at least $\frac{1}{2}$, the quantity \eqref{j-step number of eigenvalues} has strict monotonicity property. More precisely, the following lemma holds for $j\geq 1$ if we take $k$ sufficiently large. 
\begin{lem}[{\bf Probabilistic strict monotonicity}]\label{monotonicity lemma}
Assume $a\geq 5$, the height of the barrier $h\gg1$  and the randomness in $\barB_{j-1}$ (i.e.,  $V_{{\rm r},\barB_{j-1}}$) has been fixed.  Then 
  \begin{enumerate}
    \item Deterministically,  we have $\whn_{j}(B_j)\leq \whn_{j-1}(B_{j-1})$.
    \item There exists  a site, $\bar{y}=\bar{y}(V_{{\rm r},\barB_{j-1}})\in \partial^+\barB_{j-1}$, depending only on the randomness in $\barB_{j-1}$ such that,  after further fixing the randomness in $\barB_{j}\setminus (\barB_{j-1}\cup\{\bar{y}\})$, we have
  \begin{equation}\label{monotonicity probabilistic estimate}
    \P\left(\whn_{j}(B_{j}) <\whn_{j-1}(B_{j-1}) \ {\rm or}\ \whn_{j}(B_{j}) =\whn_{j-1}(B_{j-1})=0 \  \bigg| \ V_{{\rm r},\barB_{j}\setminus \{\bar{y}\} } \right)\geq \frac{1}{2}.
  \end{equation}
  The above $\P$ is only about the randomness in $\barB_j$. In other words, the above event fails for at most one value of the random variable $V_{\text{r}}(\bar{y})=\omega(\bar{y})\in \{0,1\}$.
  \end{enumerate}
\end{lem}
\begin{proof}[Proof of Lemma \ref{monotonicity lemma}]
  We denote by $v\in \{0,1\}^{\barB_{j-1}}$ the randomness in $\barB_{j-1}$, and by $\zeta\in \{0,1\}^{\barB_j\setminus\barB_{j-1}}$ the randomness in $\barB_j\setminus \barB_{j-1}$. After fixing the random variables $V_{\text{r},\barB_{j-1}}=v$, we may assume without loss of generality that there is at least one realization 
  \[V_{\text{r},\barB_j\setminus \barB_{j-1}}=\zeta_0\]
  such that $\whn_{j}(B_j)> 0$. Otherwise 
  \[\P\left(\whn_j(B_j)=0 \big| V_{\text{r},\barB_j}=v \right)=1 \]
  and Lemma \ref{monotonicity lemma} holds trivially.

We first introduce the following difference of Schur complements at the same energy $\mcE$ but on two successive scales:  
\begin{equation}\label{difference in scales}
    \diffSj_\mcE := \wtSj_\mcE(B_j)-\wtSjminusone_\mcE(B_{j-1}).
\end{equation}

\begin{claim}\label{estimate on difference in scales}
    Assume $\mcE\in [-\iota,4d+\beta+\iota]$. Then 
  \[ \| \diffSj_\mcE  \| \leq g_{j-1}=\gamma^{(2a-0.3)L_{j-1}}.\]
\end{claim}
\begin{proof}[Proof of Claim \ref{estimate on difference in scales}]
  Simple computation shows that 
  \begin{align*}
    \diffSj_\mcE &= \Gamma_{j-1}^{\top}(H_{\barB_{j-1}\setminus B_{j-1}}-\mcE)^{-1}\Gamma_{j-1}-\Gamma_j^{\top}(H_{\barB_j\setminus B_j}-\mcE)^{-1}\Gamma_j. 
  \end{align*}
  Indeed, noticing that $B_j=B_{j-1}$ and both $\Gamma_j$ and $\Gamma_{j-1}$ essentially connect site $w \in \partial^-B_j$ with one of its neighbor $w'\in \partial^+ B_j$, we can cursorily view $\Gamma_j$ and $\Gamma_{j-1}$ the same. Therefore,
 \begin{align}\label{sandwich 1}
  \notag  \diffSj_\mcE &= \Gamma_{j}^{\top}[ G_{\barB_{j-1}\setminus B_{j-1}}(\mcE)\oplus G_{\barB_{j}\setminus \barB_{j-1}}(\mcE) ]\Gamma_{j}-\Gamma_j^{\top}G_{\barB_j\setminus B_j}(\mcE)\Gamma_j \\
         &= \Gamma_{j}^{\top} \bigg( [ G_{\barB_{j-1}\setminus B_{j-1}}(\mcE)\oplus G_{\barB_{j}\setminus \barB_{j-1}}(\mcE) ]-G_{\barB_j\setminus B_j}(\mcE) \bigg) \Gamma_j .
  \end{align}
  We denote $\Gamma_{\partial \barB_{j-1}}$ the connecting matrix between $\partial^-\barB_{j-1}$ and $\partial^+\barB_{j-1}$ via $-\Delta$.
  By the  resolvent  identity, we get 
  {\small
  \begin{align}\label{vanish term}
    \notag &G_{\barB_{j-1}\setminus B_{j-1}}(\mcE)  \oplus G_{\barB_{j}\setminus \barB_{j-1}}(\mcE) -G_{\barB_j\setminus B_j}(\mcE) \\
   \notag  &  =  [G_{\barB_{j-1}\setminus B_{j-1}}(\mcE) \oplus G_{\barB_{j}\setminus \barB_{j-1}}(\mcE) ] \Gamma_{\partial \barB_{j-1}}G_{\barB_j\setminus B_j}(\mcE)  \\
         & =[G_{\barB_{j-1}\setminus B_{j-1}}(\mcE) \oplus G_{\barB_{j}\setminus \barB_{j-1}}(\mcE) ] \Gamma_{\partial \barB_{j-1}}[G_{\barB_{j-1}\setminus B_{j-1}}(\mcE) \oplus G_{\barB_{j}\setminus \barB_{j-1}}(\mcE) ] \\
     \notag    & \quad  -[G_{\barB_{j-1}\setminus B_{j-1}}(\mcE) \oplus G_{\barB_{j}\setminus \barB_{j-1}}(\mcE) ] \Gamma_{\partial \barB_{j-1}} G_{\barB_j\setminus B_j}(\mcE)  \Gamma_{\partial \barB_{j-1}}[G_{\barB_{j-1}\setminus B_{j-1}}(\mcE) \oplus G_{\barB_{j}\setminus \barB_{j-1}}(\mcE) ].
  \end{align}
  }
  Notice that if both $w,w'\in \partial^+ B_j$, then the element of the first term in \eqref{vanish term}
  \[[G_{\barB_{j-1}\setminus B_{j-1}}(\mcE) \oplus G_{\barB_{j}\setminus \barB_{j-1}}(\mcE) ]  \Gamma_{\partial \barB_{j-1}} [G_{\barB_{j-1}\setminus B_{j-1}}(\mcE) \oplus G_{\barB_{j}\setminus \barB_{j-1}}(\mcE) ] (w,w')\]
  will vanish. Putting  the remaining terms  in \eqref{sandwich 1} yields 
  \begin{align}\label{random walk expansion}
      \diffSj_\mcE (x,y) &= - \sum_{g\in \mcG_{\text{out,in}}(x,y) } G_{\barB_{j-1}\setminus B_{j}}(w,x_1;\mcE)   G_{\barB_j\setminus B_j}(x_1',x_2';\mcE)  G_{\barB_{j-1}\setminus B_{j}}(x_2,w';\mcE).  
  \end{align} 
  Here, $\mcG_{\text{out,in}}(x,y)$ contains all the random walks (which lie in $\barB_j$)
  \[g:\ x\sim w\rightarrow x_1\sim x_1'\rightarrow x_2' \sim x_2\rightarrow w'\sim y \]
  starting at $x$ and terminating at $y$ with $x,y\in \partial^-B_j;  w,w'\in \partial^+ B_j;  x_1,x_2\in \partial^-\barB_{j-1}$ and $x_1',x_2'\in \partial^+\barB_{j-1}$.  (Indeed the summation in \eqref{random walk expansion} is only about the sites $w,w',x_1,x_2,x_1',x_2'$ in $g$, and the summation about the remaining sites in $g$ is actually hidden in each Green's function via the Neumann series expansion.)  Such walks first start in $B_j$, then go out of $\barB_{j-1}$, and finally turn back to  $B_j$ (see Figure  \ref{smallness graph}). Therefore, we must have 
  \begin{equation}\label{length of in and out path}
    \length(g)\geq 2\ \dist(\partial^+ B_j ,\partial^-\barB_{j-1})+4 \geq 2a L_{j-1}+2, 
  \end{equation}
where $\length(\cdot)$ denotes the length of a random walk.

  \begin{figure}[htbp]
    \centering

\tikzset{every picture/.style={line width=0.75pt}} 

\begin{tikzpicture}[x=0.75pt,y=0.75pt,yscale=-0.85,xscale=0.85]

\draw   (76,18) -- (547.71,18) -- (547.71,489.71) -- (76,489.71) -- cycle ;
\draw   (291.5,233.5) -- (332.21,233.5) -- (332.21,274.21) -- (291.5,274.21) -- cycle ;
\draw   (231,173) -- (392.71,173) -- (392.71,334.71) -- (231,334.71) -- cycle ;
\draw [color={rgb, 255:red, 208; green, 2; blue, 27 }  ,draw opacity=1 ]   (284.71,261.29) -- (297.71,262.29) ;
\draw [shift={(297.71,262.29)}, rotate = 4.4] [color={rgb, 255:red, 208; green, 2; blue, 27 }  ,draw opacity=1 ][fill={rgb, 255:red, 208; green, 2; blue, 27 }  ,fill opacity=1 ][line width=0.75]      (0, 0) circle [x radius= 3.35, y radius= 3.35]   ;
\draw [shift={(284.71,261.29)}, rotate = 4.4] [color={rgb, 255:red, 208; green, 2; blue, 27 }  ,draw opacity=1 ][fill={rgb, 255:red, 208; green, 2; blue, 27 }  ,fill opacity=1 ][line width=0.75]      (0, 0) circle [x radius= 3.35, y radius= 3.35]   ;
\draw [color={rgb, 255:red, 208; green, 2; blue, 27 }  ,draw opacity=1 ]   (237.71,221.29) -- (284.71,261.29) ;
\draw [shift={(284.71,261.29)}, rotate = 40.4] [color={rgb, 255:red, 208; green, 2; blue, 27 }  ,draw opacity=1 ][fill={rgb, 255:red, 208; green, 2; blue, 27 }  ,fill opacity=1 ][line width=0.75]      (0, 0) circle [x radius= 3.35, y radius= 3.35]   ;
\draw [shift={(237.71,221.29)}, rotate = 40.4] [color={rgb, 255:red, 208; green, 2; blue, 27 }  ,draw opacity=1 ][fill={rgb, 255:red, 208; green, 2; blue, 27 }  ,fill opacity=1 ][line width=0.75]      (0, 0) circle [x radius= 3.35, y radius= 3.35]   ;
\draw [color={rgb, 255:red, 208; green, 2; blue, 27 }  ,draw opacity=1 ]   (223.71,220.29) -- (237.71,221.29) ;
\draw [shift={(237.71,221.29)}, rotate = 4.09] [color={rgb, 255:red, 208; green, 2; blue, 27 }  ,draw opacity=1 ][fill={rgb, 255:red, 208; green, 2; blue, 27 }  ,fill opacity=1 ][line width=0.75]      (0, 0) circle [x radius= 3.35, y radius= 3.35]   ;
\draw [shift={(223.71,220.29)}, rotate = 4.09] [color={rgb, 255:red, 208; green, 2; blue, 27 }  ,draw opacity=1 ][fill={rgb, 255:red, 208; green, 2; blue, 27 }  ,fill opacity=1 ][line width=0.75]      (0, 0) circle [x radius= 3.35, y radius= 3.35]   ;
\draw [color={rgb, 255:red, 208; green, 2; blue, 27 }  ,draw opacity=1 ]   (353.71,165.29) .. controls (301.71,-1.71) and (84.71,72.29) .. (223.71,220.29) ;
\draw [shift={(223.71,220.29)}, rotate = 46.8] [color={rgb, 255:red, 208; green, 2; blue, 27 }  ,draw opacity=1 ][fill={rgb, 255:red, 208; green, 2; blue, 27 }  ,fill opacity=1 ][line width=0.75]      (0, 0) circle [x radius= 3.35, y radius= 3.35]   ;
\draw [shift={(353.71,165.29)}, rotate = 252.7] [color={rgb, 255:red, 208; green, 2; blue, 27 }  ,draw opacity=1 ][fill={rgb, 255:red, 208; green, 2; blue, 27 }  ,fill opacity=1 ][line width=0.75]      (0, 0) circle [x radius= 3.35, y radius= 3.35]   ;
\draw [color={rgb, 255:red, 208; green, 2; blue, 27 }  ,draw opacity=1 ]   (353.71,165.29) -- (353.71,182.29) ;
\draw [shift={(353.71,182.29)}, rotate = 90] [color={rgb, 255:red, 208; green, 2; blue, 27 }  ,draw opacity=1 ][fill={rgb, 255:red, 208; green, 2; blue, 27 }  ,fill opacity=1 ][line width=0.75]      (0, 0) circle [x radius= 3.35, y radius= 3.35]   ;
\draw [shift={(353.71,165.29)}, rotate = 90] [color={rgb, 255:red, 208; green, 2; blue, 27 }  ,draw opacity=1 ][fill={rgb, 255:red, 208; green, 2; blue, 27 }  ,fill opacity=1 ][line width=0.75]      (0, 0) circle [x radius= 3.35, y radius= 3.35]   ;
\draw [color={rgb, 255:red, 208; green, 2; blue, 27 }  ,draw opacity=1 ]   (353.71,182.29) -- (323.71,227.29) ;
\draw [shift={(323.71,227.29)}, rotate = 123.69] [color={rgb, 255:red, 208; green, 2; blue, 27 }  ,draw opacity=1 ][fill={rgb, 255:red, 208; green, 2; blue, 27 }  ,fill opacity=1 ][line width=0.75]      (0, 0) circle [x radius= 3.35, y radius= 3.35]   ;
\draw [shift={(353.71,182.29)}, rotate = 123.69] [color={rgb, 255:red, 208; green, 2; blue, 27 }  ,draw opacity=1 ][fill={rgb, 255:red, 208; green, 2; blue, 27 }  ,fill opacity=1 ][line width=0.75]      (0, 0) circle [x radius= 3.35, y radius= 3.35]   ;
\draw [color={rgb, 255:red, 208; green, 2; blue, 27 }  ,draw opacity=1 ]   (323.71,227.29) -- (323.71,240.29) ;
\draw [shift={(323.71,240.29)}, rotate = 90] [color={rgb, 255:red, 208; green, 2; blue, 27 }  ,draw opacity=1 ][fill={rgb, 255:red, 208; green, 2; blue, 27 }  ,fill opacity=1 ][line width=0.75]      (0, 0) circle [x radius= 3.35, y radius= 3.35]   ;
\draw [shift={(323.71,227.29)}, rotate = 90] [color={rgb, 255:red, 208; green, 2; blue, 27 }  ,draw opacity=1 ][fill={rgb, 255:red, 208; green, 2; blue, 27 }  ,fill opacity=1 ][line width=0.75]      (0, 0) circle [x radius= 3.35, y radius= 3.35]   ;
\draw [color={rgb, 255:red, 74; green, 144; blue, 226 }  ,draw opacity=1 ][line width=1.5]    (318.71,257.29) -- (319.71,344.29) ;
\draw [shift={(319.71,344.29)}, rotate = 89.34] [color={rgb, 255:red, 74; green, 144; blue, 226 }  ,draw opacity=1 ][fill={rgb, 255:red, 74; green, 144; blue, 226 }  ,fill opacity=1 ][line width=1.5]      (0, 0) circle [x radius= 4.36, y radius= 4.36]   ;
\draw [shift={(318.71,257.29)}, rotate = 89.34] [color={rgb, 255:red, 74; green, 144; blue, 226 }  ,draw opacity=1 ][fill={rgb, 255:red, 74; green, 144; blue, 226 }  ,fill opacity=1 ][line width=1.5]      (0, 0) circle [x radius= 4.36, y radius= 4.36]   ;

\draw (334.21,277.21) node [anchor=north west][inner sep=0.75pt]   [align=left] {$B_j$};
\draw (394.71,337.71) node [anchor=north west][inner sep=0.75pt]   [align=left] {$\barB_{j-1}$};
\draw (549.71,492.71) node [anchor=north west][inner sep=0.75pt]   [align=left] {$\barB_j$};
\draw (184.71,64.71) node [anchor=north west][inner sep=0.75pt]   [align=left] {\textcolor[rgb]{0.82,0.01,0.11}{$g$}};
\draw (202,228) node [anchor=north west][inner sep=0.75pt]  [color={rgb, 255:red, 208; green, 2; blue, 27 }  ,opacity=1 ] [align=left] {$x_1'$};
\draw (232,199) node [anchor=north west][inner sep=0.75pt]  [color={rgb, 255:red, 208; green, 2; blue, 27 }  ,opacity=1 ] [align=left] {$x_1$};
\draw (266,263) node [anchor=north west][inner sep=0.75pt]  [color={rgb, 255:red, 208; green, 2; blue, 27 }  ,opacity=1 ] [align=left] {$w$};
\draw (359,143) node [anchor=north west][inner sep=0.75pt]  [color={rgb, 255:red, 208; green, 2; blue, 27 }  ,opacity=1 ] [align=left] {$x_2'$};
\draw (355.71,185.29) node [anchor=north west][inner sep=0.75pt]  [color={rgb, 255:red, 208; green, 2; blue, 27 }  ,opacity=1 ] [align=left] {$x_2$};
\draw (303,214) node [anchor=north west][inner sep=0.75pt]  [color={rgb, 255:red, 208; green, 2; blue, 27 }  ,opacity=1 ] [align=left] {$w'$};
\draw (209,353) node [anchor=north west][inner sep=0.75pt]   [align=left] {\textcolor[rgb]{0.29,0.56,0.89}{transversality path}};

\end{tikzpicture}

\caption{Smallness generates from random walks.}

\label{smallness graph}

  \end{figure}
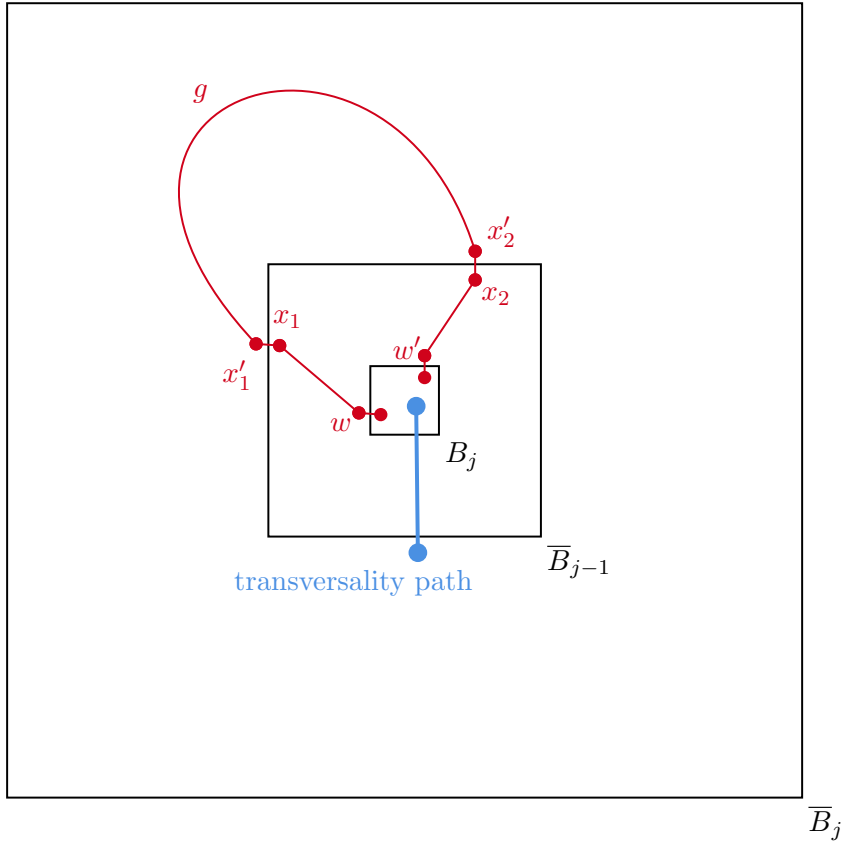

  Now, noticing that the annulus $\barLambda_k\setminus \Lambda'_k$ is non-resonant and $\barB_{j-1}\setminus B_{j},\barB_{j}\setminus B_{j}$ are both subset of $\barLambda_k\setminus \Lambda'_k$ (therefore are both non-resonant), we can apply Lemma \ref{non-resonant Green's function} to the expansion \eqref{random walk expansion}. Recall that 
  \[\frac{2d}{h-2d-\beta-\iota}<1.\]
  We have 
  \begin{align*}
    |\diffSj_\mcE (x,y)| & \leq \sum_{g\in \mcG_{\text{out,in}(x,y)}} \left(\frac{1}{h-4d-\beta-\iota}\right)^3 \cdot \exp\{-\gamma_0 (|w-x_1|_1+|x_1'-x_2'|_1 +|x_2-w'|_1)\} \\
        & \leq \sum_{g\in \mcG_{\text{out,in}(x,y)}} \left(\frac{1}{h-4d-\beta-\iota}\right)^3 \cdot  \left(\frac{2d}{h-2d-\beta-\iota}\right)^{2aL_{j-1}+2-4} \\
        & \leq (\# \partial^+B_j)^2\cdot (\# \partial^-\barB_{j-1})^2 \cdot (\# \partial^+ \barB_{j-1})^2\cdot \left(\frac{1}{h-4d-\beta-\iota}\right)^3 \cdot \left(\frac{2d}{h-2d-\beta-\iota}\right)^{2aL_{j-1}-2}  \\
        & \leq \left(\frac{2d}{h-2d-\beta-\iota}\right)^{(2a-0.1)L_{j-1}} 
  \end{align*}  
  supposing that $k\gg1$  (to ensure $L_{j-1}\geq d_k\gg1$). Therefore,  by the  Schur's  test, 
  \begin{align}\label{barrier high 1}
  \notag  \| \diffSj_\mcE \| & \leq \max_{x\in B_j} \sum_{y\in B_j} |\diffSj_\mcE (x,y)| \leq \# B_j \cdot \left(\frac{2d}{h-2d-\beta-\iota}\right)^{(2a-0.1)L_{j-1}}\\
            &\leq \left(\frac{2d}{h-2d-\beta-\iota}\right)^{(2a-0.2)L_{j-1}} \leq \gamma^{(2a-0.3)L_{j-1}}.
  \end{align}
  The final inequality in \eqref{barrier high 1} requires
  \[\left(\frac{2d}{h-2d-\beta-\iota}\right)^{2a-0.2} \leq \left(\frac{1}{4d+h+\beta}\right)^{2a-0.3} =\gamma^{2a-0.3},\]
  which holds for  sufficiently large  $h$. Recall that $0<\beta\leq 1$ and \eqref{iota}, the above inequality can be ensured by 
  \begin{equation}\label{essential barrier high 1}
      \left(\frac{2d}{h-2d-2}\right)^{2a-0.2} \leq \left(\frac{1}{4d+h+1}\right)^{2a-0.3} .
  \end{equation}
  Since $a$ will be selected later, the required largeness of $h$ depends essentially only on the dimension $d$.
\end{proof}

\begin{rmk}
  The discussion above allows us to interpret the smallness parameter $g_{j-1}$ in the following way. If the expansion \eqref{random walk expansion} is carried out more fully by applying the Neumann series argument  to the Green's function, each site visited by the random walk contributes a factor of order $\gamma$. The length estimate \eqref{length of in and out path} then yields a total smallness of about $\gamma^{(2a-)L_{j-1}}$, which we denote  by $g_{j-1}$.
 Moreover, the proof of Lemma \ref{transversality lemma} shows that the chain obtained via the cone property (recall \eqref{Chain transversality}) originates in $B_j$ and extends to $\partial^+ B_{j-1}$ (see the blue path in Figure  \ref{smallness graph}). Hence the transversality smallness $t_{j-1}$ is significantly larger than $g_{j-1}$, because the latter corresponds to a walk that returns to $B_j$ and then proceeds an additional distance of order $aL_{j-1}$-length. We therefore expect
  \begin{equation}\label{transversality is much larger}
      t_{j-1}/ g_{j-1}\sim \gamma^{-(a-)L_{j-1}}\gg 1.
  \end{equation}
  The comparison in \eqref{transversality is much larger} can be made even stronger by increasing the barrier height $h$, which reduces the value of $\gamma$. This can be interpreted as that the larger $h$ strengthens the transversality at a single site.
\end{rmk}

Recall that we fix $E\in  [-\frac{\iota}{2},4d+\beta+\frac{\iota}{2}]$. Thus using  Lemma \ref{estimate on difference in scales} yields 
\begin{equation}\label{difference at E}
  \| \diffSj_E \|< g_{j-1}. 
\end{equation}

For any eigenvalue $\lambda_0$ of $\wtSj_E(B_j)$ lying in $I_{\varepsilon_j}(E)$, we use \eqref{difference at E} and the spectral stability lemma (Corollary \ref{Spectral stability lemma}) to claim that, there exist a $\widetilde{\lambda}_0 \in \spec(\wtSjminusone_E(B_{j-1}))$ such that 
\begin{align}\label{distance tildelambda0 and E}
 \notag |\widetilde{\lambda}_0-E|& \leq |\widetilde{\lambda}_0-\lambda_0|+|\lambda_0-E|\leq \| \diffSj_E \|+\varepsilon_j \leq \gamma^{(2a-0.3)L_{j-1}} +\gamma^{a(a+1+0.2)L_{j-1}}\\
     & \leq 2 \gamma^{(2a-0.3)L_{j-1}} \ll \varepsilon_{j-1}.
\end{align}
Moreover, in Corollary \ref{Spectral stability lemma},  this correspondence preserves the ordering of the eigenvalues (counting multiplicities). Hence,
  \[\whn_j(B_j) \leq \# \left(\spec(\wtSjminusone_{E}(B_{j-1}))\cap I_{\varepsilon_{j-1} }(E)\right)=\whn_{j-1}(B_{j-1})\]
happens deterministically.

Under the configuration $(v,\zeta_0)$, $\whn_j(B_j)>0$ implies that the $\lambda_0$ exists, and thus if we take $\lambda_1=\lambda_1(v)$ as the eigenvalue of $\wtSjminusone_E(B_{j-1})$ closest to $E$, it will satisfy  \eqref{distance tildelambda0 and E} and 
\begin{equation}\label{distance lambda1 and E}
  |\lambda_1-E|\leq 2 \gamma^{(2a-0.3)L_{j-1}}.
\end{equation}
Obviously, by our choice of $\lambda_1$, it only depends on the randomness in $\barB_{j-1}$.


Next, we introduce the following difference of Schur complements on the same scale but at two distinct energies:
\begin{equation}\label{difference in energy}
  \partialSjminusone_{\mcE,\mcE'}= \wtSjminusone_\mcE(B_{j-1})-\wtSjminusone_{\mcE'}(B_{j-1}).
\end{equation}

\begin{claim}\label{estimate on difference in energy}
  Assume both $\mcE$ and $\mcE'$ are in $ [-\iota,4d+\beta+\iota]$. Then 
  \[ \| \partialSjminusone_{\mcE,\mcE'}  \| \leq \gamma^{1.9}|\mcE-\mcE'|.\]
\end{claim}

\begin{proof}[Proof of Claim \ref{estimate on difference in energy}]
   Simple computation shows that 
  \begin{align*}
    \partialSjminusone_{\mcE,\mcE'} &= \Gamma_{j-1}^{\top}(H_{\barB_{j-1}\setminus B_{j-1}}-\mcE)^{-1}\Gamma_{j-1}-\Gamma_{j-1}^{\top}(H_{\barB_{j-1}\setminus B_{j-1}}-\mcE')^{-1}\Gamma_{j-1} \\
        &=(\mcE-\mcE')\Gamma_{j-1}^{\top}G_{\barB_{j-1}\setminus B_{j-1}}(\mcE)G_{\barB_{j-1}\setminus B_{j-1}}(\mcE')\Gamma_{j-1}.
  \end{align*}
  Clearly, $\|\Gamma_{j-1}\|\leq \|-\Delta\| =2d$. Since $\mcE,\mcE'\in [-\iota,4d+\beta+\iota]$ and $\barB_{j-1}\setminus B_{j-1}$ is non-resonant, we can apply Lemma \ref{non-resonant Green's function} to  obtain 
  \begin{align}\label{barrier high 2}
   \notag \| \partialSjminusone_{\mcE,\mcE'}  \| & \leq (2d)^2 |\mcE-\mcE'|\cdot  \| G_{\barB_{j-1}\setminus B_{j-1}}(\mcE) \|\cdot \| G_{\barB_{j-1}\setminus B_{j-1}}(\mcE') \| \\
      \notag   & \leq \left(\frac{2d}{h-4d-\beta-\iota}\right)^{2} |\mcE-\mcE'| \\
         &\leq \gamma^{1.9} |\mcE-\mcE'|.
  \end{align}
  Inequality \eqref{barrier high 2} requires 
    \[\left(\frac{2d}{h-2d-\beta-\iota}\right)^{2} \leq \left(\frac{1}{4d+h+\beta}\right)^{1.9} =\gamma^{1.9},\]
      which holds for a sufficiently large choice of $h$. Again, the above inequality can be ensured by 
      \begin{equation}\label{essential barrier high 2}
            \left(\frac{2d}{h-2d-2}\right)^{2} \leq \left(\frac{1}{4d+h+1}\right)^{1.9} ,
      \end{equation}
      and the required largeness of $h$ depends essentially only on the dimension $d$.
\end{proof}
\begin{rmk}\label{applicable region}
  In the Wegner estimate, we fix $E\in [-\frac{\iota}{2},4d+\beta+\frac{\iota}{2}] $. Therefore Lemma \ref{non-resonant Green's function}, Claim \ref{estimate on difference in scales} and Claim \ref{estimate on difference in energy} are  applicable for all energy in $I_{\frac{\iota}{2}}(E)\subset [-\iota,4d+\beta+\iota]$. 
\end{rmk}

Now for our fixed $V_{\text{r},\barB_{j-1}}=v\in\{0,1\}^{\barB_{j-1}} $ such that $\zeta_0$ exists, we have proven  that (in \eqref{distance lambda1 and E}) there must be a $\lambda_1=\lambda_1(v)\in \spec(\wtSjminusone_E(B_{j-1}))$ satisfying $|\lambda_1-E|\leq 2g_{j-1}$. Therefore, applying Claim \ref{estimate on difference in energy} yields 
\[\|\partialSjminusone_{E,\lambda_1} \|\leq \gamma^{1.9} \cdot 2g_{j-1}.\]
Then using  Corollary \ref{Spectral stability lemma},   we obtain that, there exists a $\lambda_2=\lambda_2(v)\in \spec(\wtSjminusone_{\lambda_1}(B_{j-1}))$ such that 
\[|\lambda_1-\lambda_2|\leq \|\partialSjminusone_{E,\lambda_1} \| \leq \gamma^{1.9} \cdot 2g_{j-1}. \]
Thus, again applying Claim \ref{estimate on difference in energy} yields 
\[\|\partialSjminusone_{\lambda_1,\lambda_2} \|\leq \gamma^{1.9\times 2} \cdot 2g_{j-1}.\]
Iterating  such a  fixed-point argument, we can obtain that  for $s\geq 2$, there exists  a $\lambda_{s+1}=\lambda_{s+1}(v)\in \spec(\wtSjminusone_{\lambda_s}(B_{j-1}))$ such that 
\[|\lambda_s-\lambda_{s+1}|\leq \|\partialSjminusone_{\lambda_{s-1},\lambda_s} \| \leq \gamma^{1.9s} \cdot 2g_{j-1}\]
and 
\[\|\partialSjminusone_{\lambda_{s},\lambda_{s+1}} \|\leq \gamma^{1.9(s+1)} \cdot 2g_{j-1}.\]
Obviously,  one can ensure all the $\lambda_s\in I_{\frac{\iota}{2}}(E)$ by letting $d_k$ sufficiently large, and by Remark \ref{applicable region}, Claim \ref{estimate on difference in energy} is always applicable during each step of the iterations. The limit  of $\lambda_s,s\rightarrow\infty$ is then denoted by $\mu$, which solves the fixed-point equation 
\begin{equation}\label{fixed point equation 1}
  \mu\in \spec(\wtSjminusone_{\mu}(B_{j-1}))
\end{equation}
in the region 
\[|\mu-E|\leq |E-\lambda_1|+\sum_{s\geq 1}|\lambda_{s+1}-\lambda_s| \leq 2g_{j-1} \sum_{s\geq 0}\gamma^{1.9 s} < 3g_{j-1}. \]
Therefore, we can choose $\lambda$ as the solution of \eqref{fixed point equation 1},  which is closest to $E$, and thus 
\begin{equation}\label{distance of lambda and E}
  \lambda \in \spec(\wtSjminusone_{\lambda}(B_{j-1})),\ |\lambda-E|<3g_{j-1}.
\end{equation}
Clearly, such $\lambda=\lambda(v)$ only depends on the randomness in $\barB_{j-1}$.

We will use the spectral information of $\wtSjminusone_{\lambda}(B_{j-1})$ as an  intermediate term. Denote 
\begin{equation}\label{intermediate term 1}
  \whn=\# \left(\spec(\wtSjminusone_{\lambda}(B_{j-1}))\cap I_{\frac{\varepsilon_{j-1}}{2}}(E)\right).
\end{equation}
\begin{claim}\label{intermediate comparison claim 1}
  $\whn\leq \whn_{j-1}(B_{j-1})$.
\end{claim}
\begin{proof}[Proof of Claim \ref{intermediate comparison claim 1}]
  By Claim \ref{estimate on difference in energy},
  \[\| \partialSjminusone_{\lambda,E}\|\leq \gamma^{1.9}|E-\lambda|<3  \gamma^{1.9} g_{j-1}.\]
  Therefore, by the spectral stability lemma (Corollary \ref{Spectral stability lemma}), each eigenvalue of $\wtSjminusone_{\lambda}$ is approximated  by an eigenvalue of $\wtSjminusone_{E}$ within an error of at most $3\gamma^{1.9}g_{j-1}$. This correspondence preserves the ordering of the eigenvalues (counting multiplicities). Hence,
  \[\whn\leq \# \left(\spec(\wtSjminusone_{E}(B_{j-1}))\cap I_{\frac{\varepsilon_{j-1}}{2}+3  \gamma^{1.9} g_{j-1} }(E)\right). \]
  Finally, notice that 
  \[3  \gamma^{1.9} g_{j-1} = 3  \gamma^{1.9+(2a-0.3)L_{j-1}}\ll \frac{\varepsilon_{j-1}}{2} .\]
  We have 
  \[\whn\leq \# \left(\spec(\wtSjminusone_{E}(B_{j-1}))\cap I_{\varepsilon_{j-1} }(E)\right)=\whn_{j-1}(B_{j-1}).\]
\end{proof}

\par

By our construction of $\lambda=\lambda(v)$, \eqref{distance of lambda and E} holds. Hence, Lemma \ref{fundamental schur lemma} implies that 
\begin{equation}\label{lambda lies in HBj-1 spectrum}
  \lambda\in \spec(H_{\barB_{j-1}}).
\end{equation}
Moreover, if $\psi$ is a corresponding eigenvector of $H_{\barB_{j-1}}$ satisfying 
\[H_{\barB_{j-1}}\psi=\lambda \psi,\]
then $\varphi=R_{B_{j-1}}\psi$ is an eigenvector of $\wtSjminusone_\lambda(B_{j-1})$ for the eigenvalue $\lambda$, and
\begin{equation}\label{eigenvector correspondence}
  \psi =    \langle \varphi, -G_{\barB_{j-1}\setminus B_{j-1}}(\lambda)  \Gamma_{j-1} \varphi\rangle.
\end{equation}
This correspondence is one to one. For simplicity, we denote 
\[\widetilde{G}^{(j-1)}_{\lambda}= -G_{\barB_{j-1}\setminus B_{j-1}}(\lambda)  \Gamma_{j-1}. \]
Recall that we can compute $\diffSj_{\lambda}$ as in Claim \ref{estimate on difference in scales} by \eqref{sandwich 1} and \eqref{vanish term}. That is to say 
\begin{align*}
    \diffSj_\lambda &= \wtSj_{\lambda}(B_j)-\wtSjminusone_{\lambda}(B_{j-1}) \\
      &=   -\Gamma_{j-1}^{\top} G_{\barB_{j-1}\setminus B_{j-1}}(\lambda)  \Gamma_{\partial \barB_{j-1}}  G_{\barB_j\setminus B_j}(\lambda)  \Gamma_{\partial \barB_{j-1}} G_{\barB_{j-1}\setminus B_{j-1}}(\lambda) \Gamma_{j-1} \\
      &= - (\widetilde{G}^{(j-1)}_{\lambda})^{\top} \Gamma_{\partial \barB_{j-1}} G_{\barB_j\setminus B_j}(\lambda)  \Gamma_{\partial \barB_{j-1}} \widetilde{G}^{(j-1)}_{\lambda}
\end{align*}
We choose the orthonormal eigenbasis of $\wtSjminusone_{\lambda}(B_{j-1})$ in $\ell^2(B_{j-1})=\ell^2(B_j)$ as 
\[\{\varphi_1,\varphi_2,\cdots, \varphi_{\whn},\varphi_{\whn+1},\cdots,\varphi_{\# B_j}\},\]
which have their  corresponding eigenvalues  $\mcE_1,\mcE_2\cdots ,\mcE_{\# B_j}$  with 
\[\lambda=\mcE_1 ,\mcE_2,\cdots,\mcE_{\whn}\in I_{\frac{\varepsilon_{j-1}}{2}}(E)\]
and other eigenvalues are not in  $I_{\frac{\varepsilon_{j-1}}{2}}(E)$. Under such basis, the elements of the matrix $\diffSj_{\lambda}$ can be written   as 
\begin{equation}\label{elements under eigenbasis}
  \diffSj_{\lambda,m,m'}=-\left\langle \Gamma_{\partial \barB_{j-1}}  \wtGjminusone_{\lambda} \varphi_m,   G_{\barB_j\setminus B_j}(\lambda)  \Gamma_{\partial \barB_{j-1}} \wtGjminusone_{\lambda} \varphi_{m'} \right\rangle.
\end{equation}
Since $\varphi_1$ is an eigenvector of $\wtSjminusone_{\lambda}(B_{j-1})$ for the  eigenvalue $\lambda$, \eqref{eigenvector correspondence} tells us that $\psi_1=\langle\varphi_1,\wtGjminusone_{\lambda}\varphi_1\rangle$ is a solution of the Dirichlet boundary problem 
\begin{equation}\label{psi1 satisfies the Dirichlet problem}
  H_{\barB_{j-1}} \psi_1=\lambda \psi_1 .
\end{equation}
This allows us to apply Lemma \ref{transversality lemma} to obtain
\[\max_{y\in \partial ^+\barB_{j-1}} \mcI_{\psi_1} (y) \geq \gamma^{aL_{j-1}+L_0} \cdot \max_{n\in B_{j-1}}|\psi_1(n)| .\]
Since  $\varphi_1=R_{B_{j-1}}\psi_1$, we have 
\begin{equation*}
  \max_{y\in \partial ^+\barB_{j-1}} \mcI_{\psi_1} (y)  \geq \gamma^{(a+1)L_{j-1}} \|\varphi_1\|_{\ell^{\infty}(B_{j-1})}.
\end{equation*}
Trivially, one has 
\[\| \varphi_1\| _{\ell^{\infty}(B_{j-1})}\geq \sqrt{\frac{1}{\# B_{j-1}}} \|\varphi_1\| =L_0^{-\frac{d}{2}}.\]
We have
\begin{equation}\label{transversality of psi1}
\max_{y\in \partial ^+\barB_{j-1}} \mcI_{\psi_1} (y)  \geq \gamma^{(a+1)L_{j-1}} L_0^{-\frac{d}{2}} \geq \gamma^{(a+1+0.1)L_{j-1}}=t_{j-1}.
\end{equation}

Building on \eqref{elements under eigenbasis}, we define for each $y\in \partial^+\barB_{j-1}$,
\begin{equation}\label{define a(y)}
  a_m(y):= \Gamma_{\partial \barB_{j-1}}  \wtGjminusone_{\lambda} \varphi_m (y)
\end{equation}
and the vector in $\R^{\whn}$
\begin{equation}\label{define vector bfa(y)}
  \mathbf{a}(y)=(a_1(y),a_2(y),\cdots,a_{\whn}(y)).
\end{equation}
Clearly, \eqref{define a(y)} and \eqref{define vector bfa(y)} also depend only on the randomness in $\barB_{j-1}$. In particular, by our definition, 
 \begin{equation}\label{a1 is influence}
 |a_1(y)|=\mcI_{\psi_1}(y) .
 \end{equation}

We choose $\bar{y}=\bar{y}(v)\in\partial^+\barB_{j-1}$ to be a point so that 
\[|\bfa(y)|=\left(\sum_{1\leq m\leq \whn} |a_m(y)|^2\right)^{\frac{1}{2}}\]
 attains its maximum. This maximum depends only on $\lambda$ and the randomness in $\barB_{j-1}$. Since $\lambda$ itself is taken as the eigenvalue closest to $E$, it is determined solely by the randomness in $\barB_{j-1}$. Hence $\bar{y}$ also depends essentially only on the randomness in $\barB_{j-1}$. If the maximum is attained at more than one point, we select any one $\bar{y}$  among them. 
Hence, by \eqref{transversality of psi1} and \eqref{a1 is influence}, we have 
\begin{equation}\label{transversality of norm of bfa}
  |\bfa(\bar{y})| \geq \max_{y\in \partial ^+\barB_{j-1}} \mcI_{\psi_1} (y)  \geq \gamma^{(a+1+0.1)L_{j-1}}=t_{j-1}.
\end{equation}

Having chosen $\bar{y}$, we further take condition probability  on the randomness in $\barB_{j}\setminus(\barB_{j-1}\cup\{\bar{y}\})$, that is, we fix the configuration
\[V_{\text{r},\barB_{j}\setminus(\barB_{j-1}\cup\{\bar{y}\})}=\zeta'\in \{0,1\}^{\barB_{j}\setminus(\barB_{j-1}\cup\{\bar{y}\})}.\]

Now we use the eigenbasis $\{\varphi_1,\cdots ,\varphi_{\# B_j}\}$ to do further Schur complement argument. At the energy $\mcE$, we write $\wtSjminusone_{\mcE}(B_{j-1})$ and $\wtSj_{\mcE}(B_j)$ in the following block forms (notice that the eigenbasis diagonalizes $\wtSjminusone_{\lambda}(B_{j-1})$):
\begin{equation}\label{eigenbasis block form}
\wtSjminusone_{\mcE}(B_{j-1}) = \begin{pmatrix}
q(\mcE) & r(\mcE) \\
s(\mcE) & t(\mcE)
\end{pmatrix}, \ 
\wtSj_{\mcE} (B_{j}) = \begin{pmatrix}
\tilde{q}(\mcE) & \tilde{r}(\mcE) \\
\tilde{s}(\mcE) & \tilde{t}(\mcE)
\end{pmatrix}.
\end{equation}
Here $q,\tilde{q}$ are restricted to the subspace spanned by $\{\varphi_1,\cdots,\varphi_{\whn}\}$, and $t,\tilde{t}$ to the subspace spanned by $\{\varphi_{\whn+1,}\cdots,\varphi_{\# B_j}\}$. We define the following Schur complements:
\begin{equation}\label{eigenbasis schur complement}
   \smallSjminusone_{\mcE}=q(\mcE)-r(\mcE)(t(\mcE)-\mcE)^{-1} s(\mcE) ,\ \smallSj_{\mcE}=\tilde{q}(\mcE)-\tilde{r}(\mcE)(\tilde{t}(\mcE)-\lambda)^{-1}\tilde{s}(\mcE).
\end{equation}
In particular, since  $\{\varphi_1,\cdots,\varphi_{\# B_j }\}$ diagonalizes $\wtSjminusone_{\lambda} (B_{j-1})$, we have 
\[ r( \lambda)=0,\ s(\lambda) =0 .\]

Since we have already conditioned on the randomness in region $\barB_{j}\setminus\{\bar{y}\}$ (i.e.,  $(v,\zeta')$), we next wish to decompose $\smallSj_E$ in order to extract the principal part depending  only  on the random variable $\omega(\bar{y})$.

\begin{lem}\label{key decomposition lemma}
   Assume that we choose the $a$-adic scale with $a\geq 5$. After conditioning on the randomness in $\barB_{j}\setminus \{\bar{y}\}$, we can decompose 
  {\small
  \begin{align*}
    \smallSj_E & = \frac{-1}{h+2d+\beta \omega(\bar{y})-\lambda}\left(\bfa(\bar{y})\otimes \bfa(\bar{y})+ \sum_{z\sim \bar{y} \atop z\in \partial^+\barB_{j-1}} \frac{1}{h+2d+\beta\omega(z)-\lambda } \cdot(\bfa(\bar{y}) \otimes  \bfa(z)+\bfa(z)\otimes\bfa(\bar{y})  )\right) \\
    \notag & \ \ \ + \mcD + \mcR(\omega(\bar{y})) .
  \end{align*}
  }
  Here $\mcD$ is a deterministic term, and the random term $\mcR(\omega(\bar{y}))$ satisfies
  \begin{equation}\label{estimate on mcR}
    \|\mcR(\omega(\bar{y}))\| \leq \frac{1}{2}\gamma^{2.5} |\bfa(\bar{y})|^2 + \frac{1}{2}\gamma^{0.5L_{j-1}} |\bfa(\bar{y})|^2.
  \end{equation}
\end{lem}
\begin{proof}[Proof of Lemma \ref{key decomposition lemma}]
 Consider 
 \begin{equation}\label{initial decomposition}
   \smallSj_E=\smallSjminusone_{\lambda}+(\smallSj_{\lambda}-\smallSjminusone_\lambda)+(\smallSj_E-\smallSj_{\lambda}).
 \end{equation}
 We analyze the above terms separately.\\

 \noindent {\Large   {\textbf{(Term $\smallSjminusone_{\lambda}$)}}  }\\
 
 \noindent We set 
 \begin{equation}\label{mcD1}
   \mcD_1=\smallSjminusone_{\lambda}.
 \end{equation}
 Obviously, $\mcD_1$ is deterministic, as it  only depends  on the randomness in $\barB_{j-1}$.\\

 \noindent {\Large {\textbf{(Term $(\smallSj_{\lambda}-\smallSjminusone_\lambda)$)}}     }\\

 \noindent By \eqref{eigenbasis block form} and \eqref{eigenbasis schur complement}, direct computation leads to 
\begin{equation}\label{baby decomposition}
  \smallSj_{\lambda}-\smallSjminusone_\lambda =(\tilde{q}(\lambda)-q(\lambda))- \tilde{r}(\lambda)(\tilde{t}(\lambda)-\lambda)^{-1}\tilde{s}(\lambda).
\end{equation}
Since $r(\lambda)=0$ and $s(\lambda)=0$, it is easy to see that  $\tq(\lambda)-q(\lambda),\tr(\lambda),\ts(\lambda)$ and $\tt(\lambda)-t(\lambda)$ contain elements of $\diffSj_{\lambda}$ under the eigenbasis, i.e.,  \eqref{elements under eigenbasis}.
Now under the notation \eqref{define a(y)} and \eqref{define vector bfa(y)}, \eqref{elements under eigenbasis} becomes
\begin{equation}\label{elements in bfa}
    \diffSj_{\lambda,m,m'}=-\sum_{x,y\in \partial^+ \barB_{j-1}}a_{m}(x)G_{\barB_j\setminus B_j }(x,y;\lambda) a_{m'}(y).
\end{equation}
We expand $G_{\barB_{j}\setminus B_j}(\lambda)$ via the Neumann series similar to  \eqref{Neumann series expansion}, leading to the following random walk expansion:
\begin{align}\label{K random walk expansion}
  G_{\barB_{j}\setminus B_j}(x,y;\lambda) =\sum_{g\in \mcG_{\partial^+ \barB_{j-1}}(x,y)}  \frac{1}{V_{x}+2d-\lambda}\cdot \frac{1}{V_{x_1}+2d-\lambda}\cdots \frac{1}{V_{x_{s-1}}+2d-\lambda}\cdot \frac{1}{V_{y}+2d-\lambda}.
\end{align}
Here for $x,y\in \partial^+ \barB_{j-1}$, the set $\mcG_{\partial^+ \barB_{j-1}}(x,y)$ contains all random walks  $g$ contained in $\barB_{j}\setminus B_j$ with 
\[g:\ x\sim x_1\sim x_2\cdots \sim x_{s-1}\sim y,\ s=\length(g) .\]
We decompose $\mcG_{\partial^+ \barB_{j-1}}(x,y)$ into disjoint unions 
\[\mcG_{\partial^+ \barB_{j-1}}(x,y) =\mcG_{\partial^+ \barB_{j-1}}^{(0)} (x,y)\cup \mcG_{\partial^+ \barB_{j-1}}^{(1)} (x,y)\cup \mcG_{\partial^+ \barB_{j-1}}^{(2)} (x,y) ,\]
with
\[\mcG_{\partial^+ \barB_{j-1}}^{(0)} (x,y)= \left\{ g\in \mcG_{\partial^+ \barB_{j-1}} (x,y) :\  \bar{y}\notin g \right\},\]
\[\mcG_{\partial^+ \barB_{j-1}}^{(1)} (x,y)= \left\{ g\in \mcG_{\partial^+ \barB_{j-1}} (x,y) : \ \bar{y}\in g ,\ \length(g)=0 \ {\rm or} \ 1 \right\},\]
\[\mcG_{\partial^+ \barB_{j-1}}^{(2)} (x,y)= \left\{ g\in \mcG_{\partial^+ \barB_{j-1}} (x,y) : \ \bar{y}\in g ,\ \length(g)\geq 2 \right\}.\]
Clearly, $\mcG_{\partial^+ \barB_{j-1}}^{(1)}$ is not empty if and only if $x=\bar{y}$ or $y=\bar{y}$. Therefore, \eqref{K random walk expansion} can be decomposed into 
\begin{equation}\label{decomposition of K}
  K(x,y):= G_{\barB_{j}\setminus B_j}(x,y;\lambda) =K_0(x,y)+K_1(x,y)+K_2(x,y),
\end{equation}
where 
\begin{equation}\label{Ki comes from graph decomposition}
  K_i(x,y)= \sum_{g\in \mcG_{\partial^+ \barB_{j-1}}(x,y)}  \frac{1}{V_{x}+2d-\lambda}\cdot \frac{1}{V_{x_1}+2d-\lambda}\cdots \frac{1}{V_{y}+2d-\lambda}, \ i=1,2,3 .
\end{equation}
Clearly, our decomposition ensures that $K_1(x,y),K_2(x,y)$ depend only  on $V_{\text{r}}(\bar{y})=\omega(\bar{y})$, whereas $K_0(x,y)$ does not. Since $\barB_{j}\setminus B_j$ is non-resonant, we have 
\begin{equation}\label{K1 formula 1}
  K_1(\bar{y},\bar{y})= \frac{1}{h+2d+\beta \omega(\bar{y})-\lambda} 
\end{equation}
and
\begin{equation}\label{K1 formula 2}
  K_1(\bar{y},z)=K_1(z,\bar{y})=\frac{1}{h+2d+\beta \omega(\bar{y})-\lambda} \cdot \frac{1}{h+2d+\beta \omega(z)-\lambda} \ {\rm for}\ \forall z\sim \bar{y},z\in \partial^+\barB_{j-1}.
\end{equation} 
Moreover, the random walks in $\mcG_{\partial^+ \barB_{j-1}}^{(2)}$ have length larger than $2$.
We use $\lambda\in [-\iota,4d+\beta+\iota]$ to estimate
\begin{align}\label{K2 estimate}
\notag  |K_2(x,y)| & \leq \sum_{t \geq (|x-\bar{y}|_1+|\bar{y}-y|_1)\vee 2 \atop r+s=t,r>|x-\bar{y}|_1,s>|\bar{y}-y|_1} \sum_{ g_1 =(x_0=x, x_1,\cdots,x_r=\bar{y}) \atop g_1\in \barB_j\setminus B_j} \sum_{ g_2 =(x_r=\bar{y}, x_{r+1},\cdots,x_r=\bar{y}) \atop g_2\in \barB_j\setminus B_j} \frac{1}{V_{x}+2d-\lambda}\cdots \frac{1}{V_{y}+2d-\lambda}  \\
 \notag  &\leq \sum_{t \geq (|x-\bar{y}|_1+|\bar{y}-y|_1)\vee 2 } (t+1) (2d)^t \left(\frac{1}{h-2d-\beta-\iota}\right)^{t+1} \\
\notag & \lesssim_d \left(\frac{2d}{h-2d-\beta-\iota}\right)^{(|x-\bar{y}|_1+|\bar{y}-y|_1)\vee 2 +1}\\
      &<\gamma^{0.9[(|x-\bar{y}|_1+|\bar{y}-y|_1)\vee 3]}.
\end{align}
Here the validity of \eqref{K2 estimate} also requires 
\begin{equation}\label{essential barrier high 3}
   \frac{2d}{h-2d-\beta-\iota}\leq \frac{2d}{h-2d-2}\lesssim_d   \left(\frac{1}{4d+h+1}\right) ^{1.9} \leq  \gamma^{0.9},
\end{equation}
which will hold when $h$ is sufficiently large (depending only on $d$). 
We first discuss the term $\tq(\lambda)-q(\lambda)$ in \eqref{baby decomposition}. For $m,m'\in \{1,\cdots,\whn \}$, using \eqref{elements in bfa} and the decomposition \eqref{decomposition of K} yields 
\begin{align*}
  (\tq(\lambda)-q(\lambda))_{m,m'} & = \diffSj_{\lambda,m,m'}= -\sum_{x,y\in \partial^+ \barB_{j-1}}a_{m}(x) K(x,y) a_{m'}(y)\\
    &=-\sum_{x,y\in \partial^+ \barB_{j-1}}a_{m}(x)\left(K_0(x,y) +K_1(x,y)+K_2(x,y) \right)a_{m'}(y).
\end{align*}
We set 
\begin{equation}\label{mcD2}
  \mcD_{2,m,m'}=-\sum_{x,y\in \partial^+ \barB_{j-1}}a_{m}(x)K_0(x,y)a_{m'}(y),
\end{equation}
which is deterministic since $\bfa(x),\bfa(y)$ and $K_0(x,y)$ are independent of $\omega(\bar{y})$.
Moreover, recalling t \eqref{K1 formula 1} and \eqref{K1 formula 2}, a direct computation yields 
{\small
\begin{align*}
  &\sum_{x,y\in \partial^+ \barB_{j-1}}a_{m}(x)K_1(x,y)a_{m'}(y) \\
  =& \frac{1}{h+2d+\beta\omega(\bar{y})-\lambda}a_{m}(\bar{y}) a_{m'}(\bar{y}) + \sum_{z\sim \bar{y} \atop z\in \partial^+\barB_{j-1}} \frac{a_m(\bar{y})a_{m'}(z)+a_m(z)a_{m'}(\bar{y})}{(h+2d+\beta\omega(\bar{y})-\lambda)\cdot (h+2d+\beta\omega(z)-\lambda)} \\
  =&\frac{1}{h+2d+\beta \omega(\bar{y})-\lambda}\left(\bfa(\bar{y})\otimes \bfa(\bar{y})+ \sum_{z\sim \bar{y} \atop z\in \partial^+\barB_{j-1}} \frac{1}{h+2d+\beta\omega(z)-\lambda } \cdot(\bfa(\bar{y}) \otimes  \bfa(z)+\bfa(z)\otimes\bfa(\bar{y})  )\right)_{m,m'}.
\end{align*}
}
Next, we set
\begin{equation}\label{mcR1}
  \mcR_{1,m,m'}(\omega(\bar{y}))=-\sum_{x,y\in \partial^+ \barB_{j-1}}a_{m}(x)K_2(x,y)a_{m'}(y).
\end{equation}
\begin{claim}\label{estimate on mcR1}
  We have 
  \[\|\mcR_{1}(\omega(\bar{y}))\| \leq \frac{1}{2}\gamma^{2.5} |\bfa(\bar{y})|^2. \]
\end{claim}
\begin{proof}[Proof of Claim \ref{estimate on mcR1}]
  By our definition,
  \[\mcR_1(\omega(\bar{y}))=-\sum_{x,y\in \partial^+ \barB_{j-1}}K_2(x,y)\cdot \bfa(x)\otimes \bfa(y).\]
  Hence, using the Minkowski inequality and \eqref{K2 estimate} yields
  \begin{align*}
    \| \mcR_{1}(\omega(\bar{y})) \| & \leq \sum_{x,y\in \partial^+ \barB_{j-1}} |K_2(x,y)|\cdot \| \bfa(x)\otimes \bfa(y) \| = \sum_{x,y\in \partial^+ \barB_{j-1}} |K_2(x,y)|\cdot | \bfa(x)|\cdot |\bfa(y) | \\
          &\leq |\bfa(\bar{y})|^2 \sum_{x,y\in \partial^+ \barB_{j-1}} \gamma^{0.9[(|x-\bar{y}|_1+|\bar{y}-y|_1)\vee 3]} \leq |\bfa(\bar{y})|^2 \sum_{x,y\in \Z^d } \gamma^{0.9[(|x-\bar{y}|_1+|\bar{y}-y|_1)\vee 3]} \\
          &\leq |\bfa(\bar{y})|^2 \sum_{r,s\geq 0} (2d)^{r+s}\gamma^{0.9[(r+s)\vee 3]} \lesssim_d |\bfa(\bar{y})|^2 \gamma^{0.9\times 3} \\
          &\leq \frac{1}{2}\gamma^{2.5} |\bfa(\bar{y})|^2 
  \end{align*}
  supposing $h$ is large enough (depending only on $d$) to make sure  $0<\gamma\ll1$.  In the above argument, we also use the fact that $\bar{y}$ is chosen to maximize $|\bfa(y)|$.

\end{proof}
Summarize the above discussions, we have decomposed
{\small
\begin{align}\label{further baby decomposition 1}
  \tq(\lambda)-q(\lambda)&= \mcD_2 +\mcR_1 (\omega(\bar{y}))\\
  \notag &\ \ \ -\frac{1}{h+2d+\beta \omega(\bar{y})-\lambda}\left(\bfa(\bar{y})\otimes \bfa(\bar{y})+ \sum_{z\sim \bar{y} \atop z\in \partial^+\barB_{j-1}} \frac{1}{h+2d+\beta\omega(z)-\lambda } \cdot(\bfa(\bar{y}) \otimes  \bfa(z)+\bfa(z)\otimes\bfa(\bar{y})  )\right).
\end{align}
}

For the remaining term $\tilde{r}(\lambda)(\tilde{t}(\lambda)-\lambda)^{-1}\tilde{s}(\lambda)$ in \eqref{baby decomposition}, we denote it by $\mcR_2(\omega(\bar{y}))$. It can be estimated  as follows:
\begin{claim}\label{estimate on mcR2}
  For 
  \begin{equation}\label{mcR2}
    \mcR_2(\omega(\bar{y}))=-\tilde{r}(\lambda)(\tilde{t}(\lambda)-\lambda)^{-1}\tilde{s}(\lambda),
  \end{equation}
  we have 
  \[ \|\mcR_2(\omega(\bar{y})) \|\leq \frac{1}{6}\gamma^{(a-2)L_{j-1}} |\bfa(\bar{y})|^2 .\] 
\end{claim}
\begin{proof}[Proof of Claim \ref{estimate on mcR2}]
  For $m,m'\in \{1,\cdots,\whn\}$, we expand elements in $\mcR_2(\omega(\bar{y}))$ into 
  \begin{align*}
    \mcR_{2,m,m'}(\omega(\bar{y}))=-\sum_{ \whn+1\leq \tm,\tm'\leq  \# B_j } \diffSj_{\lambda,m,\tm}\cdot [(\tt(\lambda)-\lambda)^{-1}]_{\tm,\tm'}\cdot \diffSj_{\lambda,\tm',m'}.
  \end{align*}
  By \eqref{elements in bfa}, this can be further expanded into 
  \begin{align}\label{mcR2 expansion}
        &\mcR_{2,m,m'}(\omega(\bar{y}))\\
     \notag   &=-\sum_{x,x',y,y'\in \partial^+\barB_{j-1}} \sum_{ \whn+1\leq \tm,\tm'\leq  \# B_j } a_{m}(x)K(x,x') a_{\tm}(x')\cdot [(\tt(\lambda)-\lambda)^{-1}]_{\tm,\tm'}\cdot a_{\tm'}(y')K(y',y) a_{m'}(y).
  \end{align}
  Now, recall that we choose the eigenbasis such that $\mcE_m\notin I_{\frac{\varepsilon_{j-1}}{2}}(E)$ for $\whn\leq m\leq \# B_j$, and we have \eqref{distance of lambda and E}. Thus 
  \[\dist(\spec(t(\lambda)),\lambda)> \frac{\varepsilon_{j-1}}{2}-3g_{j-1}.\]
  Moreover, since $\tt(\lambda)-t(\lambda)$ is the restriction of $\diffSj_{\lambda}$ on the subspace spanned by ${\varphi_{\whn+1},\cdots,\varphi_{\#B_j}}$, applying Claim \ref{estimate on difference in scales} yields 
  \[\| \tt(\lambda)-t(\lambda) \| \leq \|\diffSj_{\lambda}\|\leq g_{j-1}.\]
  Therefore, using  Corollary \ref{Spectral stability lemma},  we have 
  \begin{align*}
      \dist(\spec(\tt(\lambda)),\lambda) & >\dist(\spec(t(\lambda)),\lambda)-g_{j-1}>\frac{\varepsilon_{j-1}}{2}-4g_{j-1} \\
        &=\frac{1}{2}\gamma^{(a+1+0.2)L_{j-1}}-4\gamma^{(2a-0.3)L_{j-1}} >\frac{\varepsilon_{j-1}}{3},
  \end{align*}
  which implies 
  \begin{equation}\label{resolvent of tilde t}
    \|(\tt(\lambda)-\lambda)^{-1}\|< \frac{3}{\varepsilon_{j-1}}.
  \end{equation}

  On the other hand, the definition \eqref{define a(y)} ensures that 
  \begin{align*}
    |a_m(y)| & = |\Gamma_{\partial \barB_{j-1}}  \wtGjminusone_{\lambda} \varphi_m (y)| =| \Gamma_{\partial \barB_{j-1}} G_{\barB_{j-1}\setminus B_{j-1}}(\lambda)  \Gamma_{j-1} \varphi_m(y)| \\
       & \leq \sum_{w\in \partial^{-}\barB_{j-1} \atop  w\sim y} \sum_{w'\in \partial^+ B_{j-1} \atop y'\in \partial^-B_{j-1} , w' \sim y'}  |G_{\barB_{j-1}\setminus B_{j-1}}(w,w';\lambda)|\cdot |\varphi_m(y')| \\ 
       &\leq \#\partial^+ B_{j-1} \cdot \frac{(2d)^2}{h-4d-\beta-\iota} \left(\frac{2d}{h-2d-\beta-\iota}\right)^{\dist(\partial^+B_{j-1},\partial^-\barB_{j-1})}.
  \end{align*}
  Here we used Lemma \ref{non-resonant Green's function} to estimate $G_{\barB_{j-1}\setminus B_{j-1}}(\lambda)$. Now by the fact that  
  \[\dist(\partial^+B_{j-1},\partial^-\barB_{j-1}) =aL_{j-1}-1\]
  and taking $h$ large enough (depending on $d,\beta$), we obtain 
  \begin{equation}\label{decay of a(y)}
    |a_m(y)|<\gamma^{(a-0.1)L_{j-1}}\  {\rm for}\ \forall 1\leq m\leq \#B_j,\  \forall y\in \partial^+\barB_{j-1}.
  \end{equation}

  We can use \eqref{resolvent of tilde t} and \eqref{decay of a(y)} to bound 
  \begin{align*}
    \sum_{ \whn+1\leq \tm,\tm'\leq  \# B_j }  \bigg|a_{\tm}(x')\cdot [(\tt(\lambda)-\lambda)^{-1}]_{\tm,\tm'}\cdot a_{\tm'}(y')\bigg|& \leq  (\# B_j)^2 \gamma^{2(a-0.1)L_{j-1}} \cdot \frac{3}{\gamma^{(a+1+0.2)L_{j-1}}} \\
      &< \gamma^{(a-1-0.5)L_{j-1}}.
  \end{align*}
  Plugging this estimate into \eqref{mcR2 expansion} gives  
  \begin{align*}
    |\mcR_{2,m,m'}(\omega(\bar{y}))| &\leq\gamma^{(a-1-0.5)L_{j-1}} \bigg| \sum_{x,x',y,y'\in \partial^+\barB_{j-1}} a_{m}(x)K(x,x')  K(y',y) a_{m'}(y) \bigg| \\
    &\leq \gamma^{(a-1-0.5)L_{j-1}} |\bfa(\bar{y})|^2 \cdot (\# \partial^+\barB_{j-1})^4 \|K\|^2.
  \end{align*}
  In the above argument,  we again used the fact that $\bar{y}$ maximizes $|\bfa(y)|$. Finally, recall that $K$ is the Green's function $G_{\barB_{j}\setminus B_j}(x,y;\lambda)$. Applying Claim \ref{non-resonant Green's function} and the Schur's test yields
  \begin{align*}
    \|\mcR_{2}(\omega(\bar{y}))\| & \leq \frac{\gamma^{(a-1-0.5)L_{j-1}} }{(h-4d-\beta-\iota)^2}|\bfa(\bar{y})|^2 \cdot (\# \partial^+\barB_{j-1})^4 \cdot (\# B_{j})^2 \\
       &\ll \frac{1}{6}\gamma^{(a-2)L_{j-1}} |\bfa(\bar{y})|^2.
  \end{align*}
The absorption of terms such as $\# B_{j}$ is possible because we can take $k$ sufficiently large, thereby making $L_{j-1}$ large enough.
\end{proof}

It follows from \eqref{baby decomposition}, \eqref{further baby decomposition 1} and \eqref{mcR2} that
{\small
\begin{align}\label{second term decomposition}
  \smallSj_{\lambda}-\smallSjminusone_\lambda &= \mcD_2 +\mcR_1 (\omega(\bar{y})) +\mcR_2(\omega(\bar{y})) \\
  \notag &\ \ \ -\frac{1}{h+2d+\beta \omega(\bar{y})-\lambda}\left(\bfa(\bar{y})\otimes \bfa(\bar{y})+ \sum_{z\sim \bar{y} \atop z\in \partial^+\barB_{j-1}} \frac{1}{h+2d+\beta\omega(z)-\lambda } \cdot(\bfa(\bar{y}) \otimes  \bfa(z)+\bfa(z)\otimes\bfa(\bar{y})  )\right).
\end{align}
}

\noindent {\Large  {\textbf{(Term $(\smallSj_E-\smallSj_{\lambda})$)}}  }\\

\noindent For this term, we compute it as 
\begin{align}\label{baby decomposition 2}
 \notag \smallSj_E-\smallSj_{\lambda} &=(\tq(E)-\tq(\lambda)) +\tr(\lambda) (\tt(\lambda)-\lambda)^{-1} \ts(\lambda)  -\tr(E) (\tt(E)-E)^{-1} \ts(E) \\
                     & =(q(E)-q(\lambda)) \\
         \notag                      & \ \ \ +[(\tq(E)-q(E))-(\tq(\lambda) -q(\lambda)) ]\\
         \notag                     &\ \ \ +\tr(\lambda) (\tt(\lambda)-\lambda)^{-1} \ts(\lambda)  -\tr(E) (\tt(E)-E)^{-1} \ts(E).  
\end{align}
We set 
\begin{equation}\label{mcD3}
  \mcD_3= (q(E)-q(\lambda)),
\end{equation}
which is deterministic, since by our construction,  $q(\mcE)$ and $\lambda$ are independent of $\omega(\bar{y})$. 

Next, for the second line in \eqref{baby decomposition 2}, we set 
\begin{equation}\label{mcR3}
  \mcR_3(\omega(\bar{y}))= (\tq(E)-q(E))-(\tq(\lambda) -q(\lambda)).
\end{equation}
\begin{claim}\label{estimate on mcR3}
  We have 
  \[\|\mcR_3(\omega(\bar{y})) \| \leq \frac{1}{6}\gamma^{(a-2)L_{j-1}} |\bfa(\bar{y})|^2 .\]
\end{claim}
\begin{proof}[Proof of Claim \ref{estimate on mcR3}]
Indeed, $\mcR_3(\omega(\bar{y}))$ can be viewed  as the matrix 
\[\diffSj_{E}-\diffSj_{\lambda} \]
restricted to the subspace spanned by $\{\varphi_1,\cdots,\varphi_{\whn}\}$. Therefore,
\[\|\mcR_3(\omega(\bar{y})) \|\leq \| \diffSj_{E}-\diffSj_{\lambda}  \| .\]
Expanding the Green's function in \eqref{random walk expansion} further by a Neumann series yields 
\begin{equation}\label{scales not derivation}
  \diffSj_{\mcE}(x,y) =-\sum_{g\in \mcG_{\text{out,in}}(x,y) } \frac{1}{V_{x_1}+2d-\mcE} \cdot \frac{1}{V_{x_2}+2d-\mcE} \cdots \frac{1}{V_{x_{s-1}}+2d-\mcE} 
\end{equation}
with the random walk 
\[g=(x_0=x,x_1,x_2,\cdots,x_{s-1},x_s=y),x_{i-1}\sim x_i\]
in $\mcG_{\text{out,in}}(x,y)$. (For the definition of $\mcG_{\text{out,in}}(x,y)$, see the proof of Claim \ref{difference in scales}.) 

Taking the  derivative  about $\mcE$ on \eqref{scales not derivation} and applying \eqref{length of in and out path} yield  that 
\begin{align*}
  |\frac{d}{d \mcE} (\diffSj_{\mcE}(x,y))| & \leq \sum_{g\in \mcG_{\text{out,in}}(x,y) } \left(\frac{1}{h-2d-\beta-\iota}\right) ^{\length(g)} \\
     &\leq \sum_{s\geq 2aL_{j-1}+2}  \left(\frac{2d}{h-2d-\beta-\iota}\right) ^s <\gamma^{(2a-0.1)L_{j-1}}
\end{align*}
holds for all $\mcE\in [-\iota,4d+\beta+\iota]$ and $x,y\in \partial^-B_j$,  provided 
\[ \left(\frac{2d}{h-2d-\beta-\iota}\right)^{2a}< \gamma^{(2a-0.1)}. \]
This can be ensured by 
\begin{equation}\label{essential barrier high 4}
  \left(\frac{2d}{h-2d-2}\right)^{2a}< \left(\frac{1}{4d+h+1}\right)^{(2a-0.1)},
\end{equation}
which requires $h$ to be sufficiently large (depending only on $d$). Thus by the mean value theorem and \eqref{distance of lambda and E},  
\[ |\diffSj_{E}(x,y)-\diffSj_{\lambda}(x,y)| < \gamma^{(2a-0.1)L_{j-1}} |\lambda-E| < 3 g_{j-1} \gamma^{(2a-0.1)L_{j-1}} .\]
Apply the Schur's  test, we get  
\begin{align*}
  \| \diffSj_E-\diffSj_{\lambda}  \|& \leq (\# \partial^-B_j)^2 \cdot 3 g_{j-1} \gamma^{(2a-0.1)L_{j-1}} \leq   \gamma^{(4a-0.5)L_{j-1}} \\
     & = \gamma^{(2a-2-0.7)L_{j-1}} t^2_{j-1} \leq \frac{1}{6} \gamma^{(a-2)L_{j-1}} |\bfa(\bar{y})|^2.
\end{align*}
In the above estimate, we used \eqref{transversality of norm of bfa}. Therefore, we  have proven   
\begin{equation}\label{bound on first part of mcR3}
  \|\mcR_3(\omega(\bar{y})) \| \leq \frac{1}{6}\gamma^{(a-2)L_{j-1}} |\bfa(\bar{y})|^2.
\end{equation}
\end{proof}

Finally, for the third line in \eqref{baby decomposition 2}, we set 
\begin{equation}\label{mcR4}
  \mcR_4(\omega(\bar{y})) =\tr(\lambda) (\tt(\lambda)-\lambda)^{-1} \ts(\lambda)  -\tr(E) (\tt(E)-E)^{-1} \ts(E).
\end{equation}
\begin{claim}\label{estimate on mcR4}
  Assume that we choose the $a$-adic scale with $a\geq 5$. Then we have
  \[\|  \mcR_4(\omega(\bar{y}))  \| \leq \frac{1}{6}\gamma^{0.5L_{j-1}} |\bfa(\bar{y})|^2 .\]
\end{claim}
\begin{proof}[Proof of Claim \ref{estimate on mcR4}]
  Compute $\mcR_4(\omega(\bar{y}))$ as follows:
  \begin{align*}
    \mcR_4(\omega(\bar{y})) &=        (\tr(\lambda) -\tr(E)) (\tt(\lambda)-\lambda)^{-1} \ts(\lambda) \\
             \notag               &\qquad + \tr(E)(\tt(\lambda)-\lambda)^{-1} (\ts(\lambda)-\ts(E)) \\
                   \notag         &\qquad + \tr(E) (\tt(\lambda)-\lambda)^{-1} \cdot [(\tt(E)-\tt(\lambda))-(E-\lambda) ]\cdot (\tt(E)-E)^{-1} \ts(E).
  \end{align*}

We have the following  available estimates:
\begin{itemize} 
 \item By $r(\lambda)=0,s(\lambda)=0$, $|\lambda-E|<3g_{j-1}$ and Claim \ref{estimate on difference in energy}, we have 
             \[\| s(E) \|,\ \| r(E)\| \leq \| \partialSjminusone_{\lambda,E}\|\leq \gamma^{1.9} \cdot 3g_{j-1} \]
             and 
           \[\| \ts(\lambda)-\ts(E) \|,\ \| \tr(\lambda)-\tr(E)\| \leq \| \partialSj_{\lambda,E}\|\leq \gamma^{1.9} \cdot 3g_{j-1} .\]
  \item By $r(\lambda)=0,s(\lambda)=0$ and Claim \ref{difference in scales}, we have 
               \[\|\tr(\lambda)\|, \|\ts(\lambda)\|,\   \| \tr(E)-r(E)\|, \ \| \ts(E)-s(E) \|,\  \|\tt(E)-t(E)\| \leq g_{j-1}.\] 
       Therefore, we also have 
       \[  \| \ts(E) \|,\ \| \tr(E)\| \leq g_{j-1}+ \gamma^{1.9} \cdot 3g_{j-1}\leq 2 g_{j-1}.   \]
  \item By similar argument leading  to \eqref{resolvent of tilde t}, and the fact that 
      \[ \| \tt(\lambda)-\tt(E)\| \leq \| \partialSj_{\lambda,E}\|\leq \gamma^{1.9} g_{j-1},\]
        we have  
        \[\|(\tt(\lambda)-\lambda)^{-1} \|, \ \|(\tt(E)-E)^{-1} \| <\frac{4}{\varepsilon_{j-1}}. \]
\end{itemize}
Use the estimates listed above, we can control  the norm of $\mcR_4(\omega(\bar{y}))$ as 
\begin{align}\label{bound on mcA}
  \notag \|\mcR_4(\omega(\bar{y}))\| &\lesssim \frac{g_{j-1}^2 }{\varepsilon_{j-1}} + \frac{g_{j-1}^3 }{\varepsilon_{j-1}^2} = \gamma^{(3a-1-0.8)L_{j-1}} +\gamma^{(4a-2-1.4)L_{j-1}}\\
    &\lesssim \gamma^{[3a-1-0.8-2(a+1+0.1)]L_{j-1}}t^2_{j-1} \leq \frac{1}{6}\gamma^{0.5L_{j-1}} t^2_{j-1}.
\end{align}
The validity of \eqref{bound on mcA} requires 
\begin{equation}\label{range of a}
  3a-1-0.8-2(a+1+0.1)>0.5,
\end{equation}
which means we need to take $a\geq 5$. Thus, using \eqref{transversality of norm of bfa} and \eqref{bound on mcA} yields 
\[\| \mcR_4(\omega(\bar{y})) \|\leq \frac{1}{6}\gamma^{0.5L_{j-1}} |\bfa(\bar{y}) |^2 .\]

\end{proof}

Combining \eqref{baby decomposition 2}, \eqref{mcD3}, \eqref{mcR3} and \eqref{mcR4} together  yields the decomposition
\begin{equation}\label{third term decomposition}
  \smallSj_E-\smallSj_{\lambda} = \mcD_3+\mcR_3(\omega(\bar{y}))+\mcR_4(\omega(\bar{y})).
\end{equation}

 \noindent {\Large  {\textbf{(Summarize all terms)}}     }\\

\noindent Finally, by \eqref{initial decomposition}, \eqref{mcD1}, \eqref{second term decomposition} and \eqref{third term decomposition}, we take 
\[\mcD=\mcD_1+\mcD_2+\mcD_3\]
and 
\[\mcR(\omega(\bar{y}))= \mcR_1(\omega(\bar{y}))+\mcR_2(\omega(\bar{y}))+\mcR_3(\omega(\bar{y}))+\mcR_4(\omega(\bar{y})).\]
Then the decomposition in Lemma \ref{key decomposition lemma} is obtained. Using Claim \ref{estimate on mcR1}, Claim \ref{estimate on mcR2}, Claim \ref{estimate on mcR3} and Claim \ref{estimate on mcR4} on the random parts  yield \eqref{estimate on mcR}. We have finished  the proof of Lemma \ref{key decomposition lemma}.
\end{proof}

At this stage, we are approaching the final step  of proving Lemma \ref{monotonicity lemma}. Define 
\begin{equation*}
  \whn_s=\#\left(   \spec(\smallSj_E) \cap I_{2\sqrt{\varepsilon_j}}(E)   \right).
\end{equation*}
Now we choose $a\geq 5$. We will show that, after conditioning on the randomness in $\barB_{j}\setminus \{\bar{y}\}$, the event
\begin{equation}\label{immediate monotonicity}
  \whn_s <\whn 
\end{equation}
fails for at most one value of $\omega(\bar{y})$. By construction, the matrix $\smallSj_E$ has order exactly equaling  to $\whn$ (which is also deterministic now). Thus deterministically we have
\[\whn_s\leq \whn.\]
For simplicity, we denote the principal part in Lemma \ref{key decomposition lemma} by 
{\small

\[  \mcM =\left(\bfa(\bar{y})\otimes \bfa(\bar{y})+ \sum_{z\sim \bar{y} \atop z\in \partial^+\barB_{j-1}} \frac{1}{h+2d+\beta\omega(z)-\lambda } \cdot(\bfa(\bar{y}) \otimes  \bfa(z)+\bfa(z)\otimes\bfa(\bar{y})  )\right), \]
}
which is deterministic and independent of $\omega(\bar{y})$.\\

\noindent \textbf{(Case 1: $\whn=1$.)}\\

\noindent When $\whn=1$, then $\smallSj_E$ is a scalar. Since we have the influence estimate 
\begin{equation}\label{value influence}
 \left| \ \frac{1}{h+2d+\beta\omega(\bar{y})-\lambda} \bigg|_{\omega(\bar{y})=0}^{\omega(\bar{y})=1}  \  \right| \geq \frac{\beta}{(h-2d-\beta-\iota)^2}\geq \gamma^2 \beta, 
\end{equation}
using the decomposition in Lemma \ref{key decomposition lemma} yields
\begin{align*}
  \left| \  \smallSj_E \big|_{\omega(\bar{y})=0}^{\omega(\bar{y})=1} \ \right| & \geq \gamma^2 \beta |\mcM| - |\mcR(1)-\mcR(0) | \\
                          &\geq \gamma^2 \beta |\mcM| -2\left( \frac{1}{2}\gamma^{2.5} |\bfa(\bar{y})|^2 + \frac{1}{2}\gamma^{0.5L_{j-1}} |\bfa(\bar{y})|^2 \right) \\
                          &  \geq \gamma^2 \beta |\mcM| -2 \gamma^{2.5} |\bfa(\bar{y})|^2.
\end{align*}
Now,   the second part in $\mcM$ can be bounded by 
\begin{align*}
  &\left|\sum_{z\sim \bar{y} \atop z\in \partial^+\barB_{j-1}} \frac{1}{h+2d+\beta\omega(z)-\lambda } \cdot(\bfa(\bar{y}) \otimes  \bfa(z)+\bfa(z)\otimes\bfa(\bar{y})  ) \right|  \leq  \frac{4d}{h-2d-\beta-\iota} |\bfa(\bar{y})|^2,
\end{align*}
where we used the fact that $\bar{y}$ is chosen to maximize $|\bfa(y)|$. Thus if we choose $h$ large (depending only on $d$) such that
\begin{equation}\label{essential barrier high 5}
  \frac{4d}{h-2d-\beta-\iota} \leq \frac{4d}{h-2d-2}<\frac{1}{2},
\end{equation}
then we  have 
\[|\mcM|\geq \left( 1- \frac{4d}{h-2d-\beta-\iota} \right) |\bfa(\bar{y})|^2 \geq \frac{1}{2}|\bfa(\bar{y})|^2 .\]
Therefore, combining $a\geq 5$ and \eqref{transversality of norm of bfa} yields
\begin{align}\label{beta dependence of height 1}
 \notag \left| \  \smallSj_E \big|_{\omega(\bar{y})=0}^{\omega(\bar{y})=1} \ \right| & \geq (\frac{1}{2}\gamma^2 \beta -2\gamma^{2.5})\cdot |\bfa(\bar{y})|^2 \\
            &\geq \frac{\beta}{4}\gamma^2 t^2_{j-1} \\
         \notag      &=\frac{\beta}{4}\gamma^2 \gamma^{2(a+1+0.1)L_{j-1}} \gg 4 \gamma^{\frac{a}{2}(a+1+0.2)L_{j-1}} = 4\sqrt{\varepsilon_{j}}.
\end{align}
Here, the validity of \eqref{beta dependence of height 1} requires 
\[4\gamma^{0.5}\leq \beta,\]
which requires  $h$ \textbf{{to be sufficiently large (depending on $\beta$). } } However, this $\beta$-dependence can be removed by a technical argument, which we will present in the next subsection. 
The estimate above on the influence of $\smallSj_E$ implies that for at most one value of $\omega(\bar{y})$ do we have $|\smallSj_E-E|\leq 2\sqrt{\varepsilon_j}$. Consequently, inequality \eqref{immediate monotonicity} fails for at most one value of $\omega(\bar{y})\in \{0,1\}$. \\ 

\noindent \textbf{(Case 2: $\whn\geq 2$.)}\\

\noindent When $\whn\geq 2$, then the spread of matrix $\smallSj_E$ is well-defined (see \eqref{definition of spread} for the definition). Since the matrix $\bfa(\bar{y})\otimes \bfa(\bar{y})$ is of rank-one, we have 
\[\spread (\bfa(\bar{y})\otimes \bfa(\bar{y}))=|\bfa(\bar{y})|^2.\]
Moreover, clearly we have 
\begin{align*}
 & \spread\left( \sum_{z\sim \bar{y} \atop z\in \partial^+\barB_{j-1}} \frac{1}{h+2d+\beta\omega(z)-\lambda } \cdot(\bfa(\bar{y}) \otimes  \bfa(z)+\bfa(z)\otimes\bfa(\bar{y})  ) \right) \\
 \leq & 2\cdot \left\| \sum_{z\sim \bar{y} \atop z\in \partial^+\barB_{j-1}} \frac{1}{h+2d+\beta\omega(z)-\lambda } \cdot(\bfa(\bar{y}) \otimes  \bfa(z)+\bfa(z)\otimes\bfa(\bar{y})  ) \right\| \\
 \leq &  \frac{8d}{h-2d-\beta-\iota} |\bfa(\bar{y})|^2,
\end{align*}
where we used that $\bar{y}$ maximizes $|\bfa(y)|$ and $\|\mathbf{v}\otimes \mathbf{w} \|=|\mathbf{v} |\cdot |\mathbf{w}|$ for any vectors  $\mathbf{v},\mathbf{w}$. Again, choose $h$ large enough  (depending only on $d$) and we can ensure 
\begin{equation}\label{essential barrier high 6}
  \frac{8d}{h-2d-\beta-\iota} \leq \frac{8d}{h-2d-2}<\frac{1}{2}.
\end{equation}
Therefore, applying  spread estimate argument, Corollary \ref{spread lemma}, yields
\begin{equation}\label{spread of mcM}
 \spread(\mcM)\geq \frac{1}{2}|\bfa(\omega(\bar{y}))|^2.
\end{equation}
Now set 
\[M_1=\smallSj_E \big|_{\omega(\bar{y})=0},\ M_2 = \smallSj_E \big|_{\omega(\bar{y})=0}^{\omega(\bar{y})=1}. \]
Then 
\[M_2= \frac{1}{h+2d+\beta\omega(\bar{y})-\lambda} \bigg|_{\omega(\bar{y})=0}^{\omega(\bar{y})=1} \cdot \mcM + (\mcR(1)-\mcR(0)). \]
Applying Corollary \ref{spread lemma}, we can obtain 
\begin{align}\label{beta dependence of height 2}
\notag  \spread(M_2)& \geq \left|  \frac{1}{h+2d+\beta\omega(\bar{y})-\lambda} \bigg|_{\omega(\bar{y})=0}^{\omega(\bar{y})=1}  \right|\cdot \spread(\mcM)-\spread(\mcR(1)-\mcR(0)) \\
   \notag & \geq \frac{1}{2}\gamma^2 \beta |\bfa(\bar{y})|^2- 2\| \mcR(1)-\mcR(0)\| \geq (\frac{1}{2}\gamma^2 \beta -2\gamma^{2.5})\cdot |\bfa(\bar{y})|^2  \\
      &\geq   \frac{\beta}{4}\gamma^2 t^2_{j-1} \gg 10\sqrt{\varepsilon_{j}}.
\end{align}
Here, we again require  $h$  to be  \textbf{{large enough (depending on $\beta$)}}.

Now, without loss of generality, we assume that at the value $\omega(\bar{y})=0$, $\whn_s=\whn$. This implies 
\[\spec(M_1)\subset I_{2\sqrt{\varepsilon_j}}(E),\]
and therefore 
\[\spread(M_1)\leq 4 \sqrt{\varepsilon_j}.\]
Thus, for the another value $\omega(\bar{y})=1$, we apply Corollary \ref{spread lemma} to obtain 
\begin{align*}
  \spread(\smallSj_E \big|_{\omega(\bar{y})=1}) & =\spread(M_1+M_2) \geq \spread(M_2)-\spread(M_1) \\
   &\geq 10 \sqrt{\varepsilon_j}-4\sqrt{\varepsilon_j} = 6\sqrt{\varepsilon_j}.
\end{align*} 
This means that at least one eigenvalue of $\smallSj_E$ at $\omega(\bar{y})=1$ is out of $I_{2\sqrt{\varepsilon_j}}(E)$. Thus, \eqref{immediate monotonicity} holds for $\omega(\bar{y})=1$. The above discussion implies that  \eqref{immediate monotonicity} fails for at most one value of $\omega(\bar{y})\in \{0,1\}$.\\

Now, recalling  the Claim \ref{intermediate comparison claim 1}, the following claim will conclude  the proof of Lemma \ref{monotonicity lemma}.

\begin{claim}\label{intermediate comparison claim 2}
  $\whn_{j}(B_j) \leq \whn_s$.
\end{claim}
\begin{proof}[Proof of  Claim \ref{intermediate comparison claim 2}]
  Since $\smallSj_E$ is the Schur complement of
  \begin{equation*}
    \wtSj_{E} (B_{j}) = \begin{pmatrix}
\tilde{q}(E) & \tilde{r}(E) \\
\tilde{s}(E) & \tilde{t}(E)
\end{pmatrix}.
  \end{equation*}
  In the previous content, we already show that $\| \tr(E)\|,\| \ts(E)\|\leq 2g_{j-1} $ and $\|(\tt(E)-E)^{-1} \| \leq \frac{4}{\varepsilon_{j-1}}$. 
  Assume $\lambda_1 \leq \cdots\leq \lambda_{\whn_j(B_j)}$ are the   all eigenvalues  of $\wtSj_E(B_j)$ in $I_{\varepsilon_j}(E)$ (counting the multiplicities).
  Since $\varepsilon_j \ll \frac{\varepsilon_{j-1}}{8}$, Lemma \ref{fundamental schur lemma} ensures that for each $\lambda_i$, there is an  eigenvalue $\wtlambda_i$ of $\smallSj_E$ such that 
  \[|\lambda_i-\wtlambda_i|\leq (\frac{g_{j-1}}{2\varepsilon_{j-1}})^2 |\lambda_i-E|\leq \gamma^{2(a-1-0.5)L_{j-1}} |\lambda_i-E|.\]
  Therefore,
  \[|\wtlambda_i-E|\leq |\lambda_i-E|+|\lambda_i-\wtlambda_i|\leq 2\varepsilon_j .\]
  Thus, $\wtlambda_i \in I_{2\sqrt{\varepsilon_j}}(E) $. 

  However,  the Remark \ref{need additional orthonormal vector argument} reveals that the correspondence of $\lambda_i$ and $\wtlambda_i$ does not necessarily coincide with the actual ordering (with multiplicities) in the spectrum (just like Corollary \ref{Spectral stability lemma}). Thus, we require a further analysis centered on approximate orthogonality. By Lemma \ref{fundamental schur lemma}, we can set 
  \[\bfv_i=\langle \bfw_i,-(\tt(E)-\lambda_i)^{-1}\ts(E) \bfw_i\rangle,\ \| \bfw_i \|=1,\ 1\leq i\leq \whn_j(B_j)\]
  to be the eigenvector of $\wtSj_E(B_j)$ with respect to $\lambda_i$. Since $\bfv_i,\bfv_{i'},i\neq i'$ are multually orthogonal, we have 
  \[|\langle \bfw_i,\bfw_{i'} \rangle |=| \langle  \bfw_i,\tr(E) (\tt(E)-\lambda_i)^{-1} (\tt(E)-\lambda_{i'})^{-1} \ts(E)\bfw_{i'}  \rangle| ,\ i\neq i'.\]
  Since 
  \[|\lambda_i-E|,\ |\lambda_{i'}-E|\leq \varepsilon_j\ll \varepsilon_{j-1},\]
  we have $\| (\tt(E)-\lambda_{i})^{-1} \|,\ \| (\tt(E)-\lambda_{i'})^{-1}\|\leq \frac{5}{\varepsilon_{j-1}}$. Therefore, for $i\neq i'$,  we have 
  \begin{equation}\label{orthogonality 1}
      |\langle \bfw_i,\bfw_{i'} \rangle | \lesssim (\frac{g_{j-1}}{\varepsilon_{j-1}})^2 =\gamma^{2(a-1-0.5)L_{j-1}}.
  \end{equation}
  Moreover, by Lemma \ref{fundamental schur lemma}, $\bfw_i$ is the eigenvector of the Schur complement
  \[\mcS_i=\tq(E)-\tr(E)(\tt(E)-\lambda_i)^{-1}\ts(E) .\]
  Therefore,
  \begin{align*}
    \| (\smallSj_E-E) \bfw_i \| &= \|  (\smallSj_E-E) \bfw_i - (\mcS_i-\lambda_i) \bfw_i\| \\
         & \leq (1+\| \tr(E) (\tt(E)-\lambda_i)^{-1} (\tt(E)-E)^{-1} \ts(E) \|) \cdot |E-\lambda_i| \\
         & \leq 2 \varepsilon_{j}.
  \end{align*}

  Now denote the spectral projection of $\smallSj_E$ on $I_{2\sqrt{\varepsilon_j}}(E)$ by $P_{I_{2\sqrt{\varepsilon_j}}(E)}$, and the spectral projection on $\R\setminus I_{2\sqrt{\varepsilon_j}}(E)$ by $P_{I_{2\sqrt{\varepsilon_j}}(E)^c}$. We have 
  \[\| (\smallSj_E-E) \bfw_i \| \geq 2 \sqrt{\varepsilon_j} \| P_{I_{2\sqrt{\varepsilon_j}}(E)^c} \bfw_i\|. \]
  Hence, 
  \begin{equation}\label{nearly eigenvector}
    \| P_{I_{2\sqrt{\varepsilon_j}}(E)^c} \bfw_i\|\leq \sqrt{\varepsilon_j}.
  \end{equation}
  Denote $P_{I_{2\sqrt{\varepsilon_j}}(E)}\bfw_i$ by $\widetilde{\bfw}_i$. Then 
  \begin{equation}\label{orthogonality 2}
    \|\widetilde{\bfw}_i \|^2 =1-\| P_{I_{2\sqrt{\varepsilon_j}}(E)^c} \bfw_i\|^2=1+\mcO(\varepsilon_j).
  \end{equation}
  By \eqref{orthogonality 1} and \eqref{nearly eigenvector}, for $i\neq i'$,  we also have 
  \begin{align}\label{orthogonality 3}
    |\langle \widetilde{\bfw}_{i},\widetilde{\bfw}_{i'}\rangle| & = |\langle \bfw_i-P_{I_{2\sqrt{\varepsilon_j}}(E)^c}\bfw_i,\bfw_{i'}-P_{I_{2\sqrt{\varepsilon_j}}(E)^c}\bfw_{i'}\rangle| \\
    \notag      &\leq  \gamma^{2(a-1-0.5)L_{j-1}} + 3 \sqrt{\varepsilon_j} \\
    \notag    &=  \gamma^{2(a-1-0.5)L_{j-1}} + 3 \gamma^{\frac{a}{2}(a+1+0.2)L_{j-1}} =\mcO( \gamma^{2(a-1-0.5)L_{j-1}}).
  \end{align}
Then  combining   \eqref{orthogonality 2} and \eqref{orthogonality 3}  implies  
  \begin{equation}
    |\langle \widetilde{\bfw}_{i},\widetilde{\bfw}_{i'}\rangle| =\delta_{i,i'} +\mcO(\gamma^{2(a-1-0.5)L_{j-1}}), \ 1\leq i,i'\leq \whn_j(B_j). 
  \end{equation}
   Moreover,  we have  $ \whn_{j}(B_j)\leq \# B_j $ and 
  \[ \# B_j \cdot \gamma^{2(a-1-0.5)L_{j-1}}\leq L_0^d \gamma^{2(a-1-0.5)L_{j-1}} \ll 1 \]
  supposing $k\gg1$. Thus, we can apply Lemma \ref{orthogonality lemma} to conclude that $\widetilde{\bfw}_i,1\leq i\leq \whn_j(B_j)$ are linearly independent in the subspace ${\rm Range}(P_{I_{2\sqrt{\varepsilon_j}}(E)} )$ (the range of  the operator $P_{I_{2\sqrt{\varepsilon_j}}(E)}$). Therefore,
  \[ \whn_j(B_j)\leq  {\rm dim} \ {\rm Range} (P_{I_{2\sqrt{\varepsilon_j}}(E)})= \whn_s.\]

\end{proof}

Finally, combining  Claim \ref{intermediate comparison claim 1}, \eqref{immediate monotonicity} and Claim \ref{intermediate comparison claim 2} together, we have proven  that, after conditioning on the randomness in $\barB_{j}\setminus \{\bar{y}\}$, the event that 
\[\whn_j(B_j)\leq \whn_s< \whn\leq \whn_{j-1}(B_{j-1})\]
occurs  for at least one value of $V_{\text{r},\bar{y}}=\omega(\bar{y})\in \{0,1\}$. Therefore, \eqref{monotonicity probabilistic estimate} is proved.

\end{proof}

\subsection{Removing the $\beta$-dependence of the barrier height}\label{remove beta dependent}
In the proof above, we observe that the required barrier height $h$ depends on the coupling strength $\beta$, and as $\beta \to 0$, the height $h$ tends to  infinity (see \eqref{beta dependence of height 1} and \eqref{beta dependence of height 2}).  This is definitely not satisfactory if one tries to extend the  present method to handle the ABM on $\Z^d$ for $d\geq 4$. However, using the trick described below, we can remove this $\beta$-dependence of $h$, thereby obtain  a non‑perturbative lower bound on the barrier height with respect to the  coupling strength.

From \eqref{beta dependence of height 1} and \eqref{beta dependence of height 2}, where the $\beta$-dependence of $h$ originates from, it is clear that the dependence is due primarily to the quantity
\[\gamma^{2.5}|\bfa(\bar{y})|^2 ,\]
which arises from estimating the remainder term $\mathcal{R}_1(\omega(\bar{y}))$ (See Claim \ref{estimate on mcR1}).  Since the remainder term $\mcR_1(\omega(\bar{y}))$ comes from the $K_2$ in decomposition \eqref{decomposition of K}, consequently, to improve the estimate of this remainder term,  it needs to  expand the random leading term $K_1(x,y)$ into   higher orders terms.  So we redefine 
\[\mcG_{\partial^+ \barB_{j-1}}(x,y) =\mcG_{\partial^+ \barB_{j-1}}^{(0)} (x,y)\cup \mcG_{\partial^+ \barB_{j-1}}^{(1)} (x,y)\cup \mcG_{\partial^+ \barB_{j-1}}^{(2)} (x,y) \]
with
\[\mcG_{\partial^+ \barB_{j-1}}^{(0)} (x,y)= \left\{ g\in \mcG_{\partial^+ \barB_{j-1}} (x,y) :\  \bar{y}\notin g \right\},\]
\[\mcG_{\partial^+ \barB_{j-1}}^{(1)} (x,y)= \left\{ g\in \mcG_{\partial^+ \barB_{j-1}} (x,y) : \ \bar{y}\in g ,\ \length(g)\leq \mcL \right\},\]
\[\mcG_{\partial^+ \barB_{j-1}}^{(2)} (x,y)= \left\{ g\in \mcG_{\partial^+ \barB_{j-1}} (x,y) : \ \bar{y}\in g ,\ \length(g)\geq \mcL+1 \right\}.\]
Here, $\mcL$ is a large length parameter that will be selected later in the argument. The decomposition $K_0,K_1,K_2$ will still be given by \eqref{Ki comes from graph decomposition}.  In  this setup, applying the argument used in the proof of Claim \ref{estimate on mcR1} directly yields that the upper bound on $\mcR_1(\omega(\bar{y}))$  will become 
\begin{equation}\label{adjusting estimate on mcR1}
  \| \mcR_1(\omega(\bar{y}))\| \leq \frac{1}{2} \gamma^{0.85(\mcL+2)} |\bfa(\bar{y})|^2.
\end{equation}
However, the leading term (which comes from graphs in $\mcG_{\partial^+ \barB_{j-1}}^{(1)} (x,y)$ and $K_1$) in Lemma \ref{key decomposition lemma} will contain more terms, and will  become 
\begin{equation}\label{adjust leading term 1}
  - \sum_{x,y\in \partial^+ \barB_{j-1}} K_1(x,y) {\bf a}(x) \otimes {\bf a}(y).
\end{equation}
If we denote 
\begin{equation*}
  \mcG_{\partial^+ \barB_{j-1}}^{(1)} = \bigcup_{x,y\in \partial^+ \barB_{j-1}} \mcG_{\partial^+ \barB_{j-1}}^{(1)} (x,y),
\end{equation*}
then \eqref{adjust leading term 1} becomes
\begin{equation}\label{adjust leading term 2}
  -\sum_{g\in \mcG_{\partial^+ \barB_{j-1}}^{(1)}} \frac{1}{h+2d+\beta \omega(x_0)-\lambda}\cdot \frac{1}{h+2d+\beta\omega(x_1)-\lambda} \cdots   \frac{1}{h+2d+\beta\omega(x_s)-\lambda} \bfa(x_1)\otimes \bfa(x_s),
\end{equation}
where  $\bar{y}\in g=(x_0,\cdots,x_s),x_{i}\sim x_{i+1}$ and $x_0,x_s \in \partial^+ \barB_{j-1} $. We denote  by 
\[\mathfrak{t}(g)=\#\{0\leq i \leq s: \ x_s = \bar{y}\} \]
the time of the random walk $g$ arriving at  $\bar{y}$.
We do a further decomposition on $\mcG_{\partial^+ \barB_{j-1}}^{(1)}$ based on the length of random walks
\begin{equation*}
  \mcG_{\partial^+ \barB_{j-1}}^{(1)}=\bigcup_{0\leq s\leq \mcL} \mcG_{\partial^+ \barB_{j-1}}^{(1),s},
\end{equation*}
where 
\begin{equation*}
  \mcG_{\partial^+ \barB_{j-1}}^{(1),s}=\left\{  g\in \mcG_{\partial^+ \barB_{j-1}}^{(1)}: \ \length(g)=s \right\}.
\end{equation*}
Using the notation introduced above, equation \eqref{adjust leading term 2} becomes 
\begin{align}
 \mathfrak{L}(\omega(\bar{y})):=  &-  \frac{1}{h+2d+\beta \omega(\bar{y})-\lambda} \bfa(\bar{y})\otimes \bfa(\bar{y}) \\
   \notag  & -\sum_{1\leq s\leq \mcL} \sum_{g\in \mcG_{\partial^+ \barB_{j-1}}^{(1),s}} \left(\frac{1}{h+2d+\beta \omega(\bar{y})-\lambda}\right)^{\mathfrak{t}(g)} \prod_{x_i\in g \atop x_i\neq \bar{y}} \frac{1}{h+2d+\beta\omega(x_i)-\lambda} \cdot \bfa(x_0)\otimes \bfa(x_s) \\
  \notag :=& -  \frac{1}{h+2d+\beta \omega(\bar{y})-\lambda} \bfa(\bar{y})\otimes \bfa(\bar{y}) + \mathfrak{L}'(\omega(\bar{y})).
\end{align}
Clearly, the estimates in  Claim \ref{estimate on mcR2}, Claim \ref{estimate on mcR3} and Claim \ref{estimate on mcR2} remain valid, and thus the Lemma \ref{key decomposition lemma} will become 
  \begin{align*}
    \smallSj_E & = \frac{-1}{h+2d+\beta \omega(\bar{y})-\lambda} =\mathfrak{L}(\omega(\bar{y})) + \mcD + \mcR(\omega(\bar{y})) .
  \end{align*}
  Here $\mcD$ is a deterministic term, and the random term $\mcR(\omega(\bar{y}))$ satisfies
  \begin{equation*}
    \|\mcR(\omega(\bar{y}))\| \leq \frac{1}{2}\gamma^{0.85(\mcL+1)} |\bfa(\bar{y})|^2 + \frac{1}{2}\gamma^{0.5L_{j-1}} |\bfa(\bar{y})|^2.
  \end{equation*}
  Choose  $k$ sufficiently large to ensure $L_{j-1}\geq L_0\sim d_{k} \geq 10 \mcL$. Then this  will yield  
  \begin{equation}\label{adjust estimate on mcR}
     \|\mcR(\omega(\bar{y}))\| \leq \gamma^{0.85(\mcL+1)} |\bfa(\bar{y})|^2 .
  \end{equation}
  Moreover, for the higher order leading term $\mathfrak{L}'(\omega(\bar{y}))$, we have 
  \begin{align*}
      \sup_{\omega(\bar{y})\in (0,1)}\left\| \frac{d }{d  \omega(\bar{y})} \mathfrak{L}'(\omega(\bar{y})) \right\| & \leq  \beta \cdot \sum_{1\leq s\leq \mcL}   \sum_{g\in \mcG_{\partial^+ \barB_{j-1}}^{(1),s}} \mathfrak{t}(g) \left(\frac{1}{h-2d-\beta-\iota}\right)^{s+2} |\bfa(x_0)|\cdot |\bfa(x_s)| \\
          & \leq  \beta |\bfa(\bar{y})|^2 \sum_{1\leq s\leq \mcL} (s+1) \cdot \left(\# \mcG_{\partial^+ \barB_{j-1}}^{(1),s} \right) \cdot \left(\frac{1}{h-2d-\beta-\iota}\right)^{s+2}.
  \end{align*}
  In the above estimate, we used that $\bar{y}$ maximizes $|\bfa(y)|$, and that $\mathfrak{t}(g)\leq \length(g)+1$. Since the random walks in $\mcG_{\partial^+ \barB_{j-1}}^{(1),s}$ must pass through the fixed site $\bar{y}$, we have $\# \mcG_{\partial^+ \barB_{j-1}}^{(1),s}\leq (s+1) (2d)^s$. Therefore,
  \begin{align}\label{essential barrier high 7}
     \notag   \sup_{\omega(\bar{y})\in (0,1)}  \left\| \frac{d}{d \omega(\bar{y})} \mathfrak{L}'(\omega(\bar{y})) \right\| & \leq  \beta  |\bfa(\bar{y})|^2 \sum_{1\leq s\leq \mcL} (s+1)^2 \cdot (2d)^s \cdot \left(\frac{1}{h-2d-\beta-\iota}\right)^{s+2} \\
       \notag      & \lesssim_d \beta (\frac{2d}{h-2d-\beta-\iota})^3 |\bfa(\bar{y})|^2 \\
             &\leq \frac{1}{10} \beta \gamma^{2} |\bfa(\bar{y})|^2, 
  \end{align}
 which is  uniform  about $\mcL\geq 1$.
  The validity of \eqref{essential barrier high 7} also needs the largeness of $h$, which  only depends   on $d$. Hence, applying the mean value theorem, we have 
  \begin{equation}\label{hihger order leading term small}
    \| \mathfrak{L}'(0) -\mathfrak{L}'(1)\| \leq \frac{1}{5} \beta \gamma^2 |\bfa(\bar{y})|^2.
  \end{equation}
  We have the following cases: \\

  \noindent \textbf{(Case 1: $\whn=1$.)}\\

  By \eqref{value influence} and \eqref{hihger order leading term small}, we have 
  \[\|\mathfrak{L}(1)-\mathfrak{L}(0) \|\geq \gamma^2 \beta |\bfa(\bar{y})|^2  - \| \mathfrak{L}'(0) -\mathfrak{L}'(1)\| \geq \frac{1}{2}\gamma^2 \beta |\bfa(\bar{y}) |^2, \]
  and thus by \eqref{adjust estimate on mcR}, the estimate \eqref{beta dependence of height 1} becomes 
  \begin{align}\label{no beta dependence of height 1}
 \notag \left| \  \smallSj_E \big|_{\omega(\bar{y})=0}^{\omega(\bar{y})=1} \ \right| & \geq (\frac{1}{2}\gamma^2 \beta -2\gamma^{0.85(\mcL+2)})\cdot |\bfa(\bar{y})|^2 \\
            &\geq \frac{\beta}{4}\gamma^2 t^2_{j-1}\gg  4\sqrt{\varepsilon_{j}}.
  \end{align}
  Now, \eqref{no beta dependence of height 1} only requires us to  choose $\mcL$ large enough  so that 
  \[ \gamma^{0.85(\mcL+2)-2}\ll \beta \Leftrightarrow \mcL\gg \log \beta / \log \gamma.\]
  So we remove the $\beta$-dependence of $h$ in this case.\\

  \noindent \textbf{(Case 2: $\whn\geq 2$.)}\\

  The case $\whn\geq 2$ can be handled in a similar manner. By \eqref{value influence}, \eqref{hihger order leading term small} and Corollary \ref{spread lemma}, we have 
  \begin{align*}
    \spread(\mathfrak{L}(1)-\mathfrak{L}(0)) &\geq  \gamma^2 \beta |\bfa(\bar{y})|^2 -\spread(\mathfrak{L}'(1)-\mathfrak{L}'(0)) \\
                               &\geq  \gamma^2 \beta |\bfa(\bar{y})|^2 -2\| \mathfrak{L}'(1)-\mathfrak{L}'(0)\| \\
                               &\geq  \frac{1}{2} \gamma^2 \beta |\bfa(\bar{y})|^2.
  \end{align*}
  Therefore, by \eqref{adjust estimate on mcR}, the estimate \eqref{beta dependence of height 2} becomes 
  \begin{align}\label{no beta dependence of height 2}
\notag  \spread(M_2)& \geq \spread(\mathfrak{L}(1)-\mathfrak{L}(0))-\spread(\mcR(1)-\mcR(0)) \\
   \notag & \geq \frac{1}{2}\gamma^2 \beta |\bfa(\bar{y})|^2- 2\| \mcR(1)-\mcR(0)\| \geq (\frac{1}{2}\gamma^2 \beta -2\gamma^{0.85(\mcL+2)})\cdot |\bfa(\bar{y})|^2  \\
      &\geq   \frac{\beta}{4}\gamma^2 t^2_{j-1} \gg 10\sqrt{\varepsilon_{j}}
\end{align}
provided  $\mcL$ large enough. We remove the $\beta$-dependence of $h$ again.

Finally, since the largeness requirements for the barrier height $h$ arising elsewhere depend only on the dimension $d$ (see \eqref{essential barrier high 1}, \eqref{essential barrier high 2}, \eqref{essential barrier high 3}, \eqref{essential barrier high 4}, \eqref{essential barrier high 5}, \eqref{essential barrier high 6} and \eqref{essential barrier high 7}), the final choice of $h_0$ depends solely on $d$ as well. 
Moreover, observing the requirement \eqref{essential barrier high 1}, if we assume $h_0(d)\sim d^n$ grows polynomially in $d$, we must have 
\[d^{-(2a-0.2)(n-1)} \lesssim d^{-(2a-0.3)n},\]
which gives $n>20a-2$. Consequently, with the choice of  $a\geq 5$, we obtain at least that $h_0(d)\gg d^{50}\gg 4d+1$. This is the condition we imposed on the barrier height in Theorem \ref{high dimension localization}.

\subsection{A martingale type  argument}\label{martingale section}

Based on the discussions above, we have shown that there exists a large constant $h_0(d)$ such that, if we choose $h>h_0(d)$, then for all sufficiently large scales $k$,  Lemma \ref{monotonicity lemma} holds true. We denote by $k^{(2)}$ (depending only  on $h,\beta,d_0$) the smallest $k$  for  which the above argument remains valid. 

We now proceed to prove the Wegner estimate, i.e., Theorem \ref{Wegner estimate for d>3}.  Lemma \ref{monotonicity lemma} indicates  that the transversality at some point on the boundary of $\barB_{j-1}$ can, with high probability, move at least one eigenvalue away from the energy $E$. Therefore, we want to select a sufficiently large collection of points so that all eigenvalues closed  to $E$ are successively shifted away. This procedure will  generate  a martingale with a particular structure, which then allows us to complete the proof of Theorem \ref{Wegner estimate for d>3} via a large deviation estimate.

\begin{proof}[Proof of Theorem \ref{Wegner estimate for d>3}]

Throughout the proof, we let $k\geq k_{\text{in}}>\max \{k^{(1)}, k^{(2)} \}$, where $k^{(1)}$ is the lower bound of scales such that Lemma \ref{eigenvalue number} holds true. 

Recall that we take the $a$-adic scales  as in  \eqref{a-adic scale} and \eqref{interpolate a-adic scale}. Moreover, recall that we set $B_j\equiv \Lambda'_k$ and $\barB_j$ the $aL_{j-1}$-neighborhood of $B_j$ for $0\leq j\leq P$. In this proof, we just let $a=5$. 

We construct a filtration of the probability space $\{0,1\}^{\barLambda_k}$ as follows.

First, we set 
\begin{equation}\label{mcF0}
  \mcF_0= \sigma(\barB_0)
\end{equation}
to be the $\sigma$-algebra generated by the randomness in $\barB_0$. Letting $\baromega_0\in \{0,1\}^{\barB_0}$ denote the configuration of this randomness, $\mcF_0$ induces the cylindrical decomposition
\[\{0,1\}^{\barLambda_k}=\bigcup_{\baromega_0} \{\baromega_0\}\times \{0,1\}^{\barLambda_k\setminus \barB_0 }:= \bigcup_{\baromega_0} \mcC_{\baromega_0}.\]
Now we can apply Lemma \ref{monotonicity lemma}. If $\baromega_0$ is fixed, then we can find a point $\bar{y}_1 =\bar{y}_1(\baromega_0)\in \partial^+\barB_0$ (that is to say $\bar{y}_1$ is $\mcF_0$-measurable) satisfying the conclusions of Lemma \ref{monotonicity lemma}. 
Next,  for each fixed $\baromega_0$ (and then $\bar{y}_1$ is fixed), we set the configuration of the randomness in $\barB_1\setminus \{\bar{y}_1\}$ to be $\hatomega_0 \in \{0,1\}^{\barB_1\setminus (\barB_0\cup \{\bar{y}_1\})}$. Let $\mcQ_0$ denote the $\sigma$-algebra matching with the following cylindrical decomposition 
\[ \{0,1\}^{\barLambda_k}=\bigcup_{\baromega_0,\hatomega_0} \{(\baromega_0,\hatomega_0)\}\times \{0,1\}^{(\barLambda_k\setminus \barB_1)\cup \{\bar{y}_1\} }:= \bigcup_{\baromega_0,\hatomega_0}\mcC_{\baromega_0,\hatomega_0}. \]
Then, clearly $\mcF_0\subset \mcQ_0$.

If we condition on $\mcQ_0$, i.e.,  we just consider the randomness in each cylinder $\mcC_{\baromega_0,\hatomega_0}$, the randomness of $V_{\text{r},\bar{y}_1}=\omega(\bar{y}_1)$ is still free and an application of the Lemma \ref{monotonicity lemma} yields 
\begin{equation}\label{prob about bar y 1}
   \P_{\omega(\bar{y}_1)}    \left(\whn_{1}(B_{1}) <\whn_{0}(B_{0}) \ {\rm or}\ \whn_{1}(B_{1}) =\whn_{0}(B_{0})=0 \  \bigg| \ V_{\text{r},\barB_{j}\setminus \{\bar{y}\} }=(\baromega_0,\hatomega_0) \right)\geq \frac{1}{2}.
\end{equation}
Here $\P_{\omega(\bar{y}_1)}$ means that indeed the probability is only about the randomness of $\omega(\bar{y}_1)$. We denote the above event by 
\begin{equation*}
   \event_1: \ \whn_{1}(B_{1}) <\whn_{0}(B_{0}) \ {\rm or}\ \whn_{1}(B_{1}) =\whn_{0}(B_{0})=0 \,
\end{equation*}
and let $\mcZ_1=\mathbf{1}_{\event_1}(\omega(\bar{y}_1))$. Then $\mcZ_1$ is $\mcF_1$-measurable, where
\[\mcF_1=\sigma(\barB_1)\] 
is the the $\sigma$-algebra generated by the randomness in $\barB_1$. Clearly, $\mcQ_0\subset \mcF_1$, and if we set the configuration of $\omega(\bar{y}_1)$ to be $\baromega_1$, then $\mcF_1$ will match with the following cylindrical decomposition 
\[ \{0,1\}^{\barLambda_k}=\bigcup_{\baromega_0,\hatomega_0,\baromega_1} \{(\baromega_0,\hatomega_0,\baromega_1)\}\times \{0,1\}^{(\barLambda_k\setminus \barB_1)} :=\bigcup_{\baromega_0,\hatomega_0,\baromega_1}\mcC_{\baromega_0,\hatomega_0,\baromega_1}. \]
Moreover, by \eqref{prob about bar y 1}, we have 
\begin{equation}\label{expectation 1}
  \E(\mcZ_1)=\E(\E(\mcZ_1 \ | \ \mcC_{\baromega_0,\hatomega_0})) \geq \frac{1}{2}.
\end{equation}

Now, since $\mcF_1$ realizes  the randomness in $\barB_1$, applying Lemma \ref{monotonicity lemma} again will yield a $\mcF_1$-measurable $\bar{y}_2=\bar{y}_2(\baromega_0,\hatomega_0,\baromega_1)\in \partial^+ \barB_1$. If we condition on $\mcF_1$, then the value of $\mcZ_1$ and the choice of $\bar{y}_2$ will be fixed. We then define $\mcQ_1$ as the $\sigma$-algebra matching with the following cylindrical decomposition 
\[ \{0,1\}^{\barLambda_k}=\bigcup_{\baromega_0,\hatomega_0,\baromega_1,\hatomega_1} \{(\baromega_0,\hatomega_0,\baromega_1,\hatomega_1)\}\times \{0,1\}^{(\barLambda_k\setminus \barB_2)\cup \{\bar{y}_2\} } := \bigcup_{\baromega_0,\hatomega_0,\baromega_1,\hatomega_1}\mcC_{\baromega_0,\hatomega_0,\baromega_1,\hatomega_1}, \]
where $\hatomega_1\in \{0,1\}^{\barB_2\setminus (\barB_1\cup \{\bar{y}_2\})}$ is the configuration of the randomness in ${\barB_2\setminus (\barB_1\cup \{\bar{y}_2\})}$. Then, clearly $\mcF_1\subset \mcQ_1$. If we denote the event 
\begin{equation*}
   \event_2:\  \ \whn_{2}(B_{2}) <\whn_{1}(B_{1}) \ {\rm or}\ \whn_{2}(B_{2}) =\whn_{1}(B_{1})=0 ,
\end{equation*}
then using  Lemma \ref{monotonicity lemma} yields that in each cylinder $\mcC_{\baromega_0,\hatomega_0,\baromega_1,\hatomega_1}$,
\begin{equation}\label{prob about bar y 2}
   \P_{\omega(\bar{y}_2)}    \left(  \event_2 \ \bigg| \ \mcC_{\baromega_0,\hatomega_0,\baromega_1,\hatomega_1} \right)\geq \frac{1}{2}.
\end{equation}
Therefore, we can define $\mcZ_2=\mathbf{1}_{\event_2}(\omega(\bar{y}_2))$, which is measurable in the $\sigma$-algebra 
\[\mcF_2=\sigma(\barB_2).\]
Moreover, since $\mcZ_1$ is $\mcF_1$-measurable, $\mcF_1\subset \mcQ_1\subset \mcF_2$ and \eqref{prob about bar y 2}, we have 
\begin{equation}\label{expectation 1}
  \E(\mcZ_2\ | \ \mcZ_1)=\E(\E(\mcZ_2 \ | \ \mcC_{\baromega_0,\hatomega_0,\baromega_1,\hatomega_1} ) | \mcZ_1) \geq \frac{1}{2}.
\end{equation}

By iterating the above construction, we obtain the following filtration
\[\mcF_0\subset \mcQ_0\subset \mcF_1\subset\mcQ_1\subset \cdots \subset \mcF_{P-1}\subset\mcQ_{P-1}\subset \mcF_{P}:=\sigma(\barLambda_k) \]
together with a (site-mixed) martingale $(\bar{y}_{j+1},\mcZ_j),0 \leq j\leq P$ satisfying:
\begin{itemize}
  \item We set $\mcZ_0=0$. We have  $(\bar{y}_{j+1},\mcZ_{j})$ is $\mcF_{j}$-measurable. The realization of $\mcZ_j$ happens between the two filtration elements $\mcQ_{j-1}\subset \mcF_j$. Moreover, 
        \begin{equation}\label{martingale expectation}
            \E(\mcZ_j\ | \ \mcZ_{j-1})\geq \frac{1}{2},\ 1\leq j\leq P.
        \end{equation}
  \item The $\sigma$-algebra $\mcF_j$ match with the cylindrical decomposition 
     \[ \{0,1\}^{\barLambda_k}=\bigcup_{\baromega_0,\hatomega_0,\cdots,\baromega_j} \{(\baromega_0,\hatomega_0,\cdots,\baromega_j)\}\times \{0,1\}^{(\barLambda_k\setminus \barB_j)}:= \bigcup_{\baromega_0,\hatomega_0,\cdots,\baromega_j}\mcC_{\baromega_0,\hatomega_0,\cdots,\baromega_j} \]
       and the $\sigma$-algebra $\mcQ_j$ match with the cylindrical decomposition
  \[ \{0,1\}^{\barLambda_k}=\bigcup_{\baromega_0,\hatomega_0,\cdots,\baromega_j,\hatomega_j} \{(\baromega_0,\hatomega_0,\cdots,\baromega_j,\hatomega_j)\}\times \{0,1\}^{(\barLambda_k\setminus \barB_{j+1})\cup \{\bar{y}_{j+1}\} } := \bigcup_{\baromega_0,\hatomega_0,\cdots,\baromega_j,\hatomega_j}\mcC_{\baromega_0,\hatomega_0,\cdots,\baromega_j,\hatomega_j}. \]     
\end{itemize}
Here we say the martingale is \textbf{``site-mixed"}, because during the process the site $\bar{y}_j$ itself is a random variable, and the filtration $\mcQ_{j-1}$ is constructed based on the realization of $\bar{y}_j$. (For a visual illustration of this construction, see Figure  \ref{martingale Wegner fig}.)

\begin{figure}[htbp]
  \centering

\tikzset{every picture/.style={line width=0.75pt}} 

\begin{tikzpicture}[x=0.75pt,y=0.75pt,yscale=-0.9,xscale=0.9]

\draw   (17.82,22.82) -- (626.54,22.82) -- (626.54,631.54) -- (17.82,631.54) -- cycle ;
\draw   (316,321) -- (328.36,321) -- (328.36,333.36) -- (316,333.36) -- cycle ;
\draw   (195.57,200.57) -- (448.79,200.57) -- (448.79,453.79) -- (195.57,453.79) -- cycle ;
\draw   (295.36,300.36) -- (349,300.36) -- (349,354) -- (295.36,354) -- cycle ;
\draw   (275.02,280.02) -- (369.34,280.02) -- (369.34,374.34) -- (275.02,374.34) -- cycle ;
\draw   (253.31,258.31) -- (391.04,258.31) -- (391.04,396.04) -- (253.31,396.04) -- cycle ;
\draw  [color={rgb, 255:red, 208; green, 2; blue, 27 }  ,draw opacity=1 ][fill={rgb, 255:red, 208; green, 2; blue, 27 }  ,fill opacity=1 ] (298.91,312.15) .. controls (298.84,310.39) and (297.35,308.97) .. (295.59,308.97) .. controls (293.83,308.97) and (292.46,310.39) .. (292.53,312.15) .. controls (292.6,313.91) and (294.09,315.34) .. (295.85,315.34) .. controls (297.61,315.34) and (298.98,313.91) .. (298.91,312.15) -- cycle ;
\draw  [color={rgb, 255:red, 208; green, 2; blue, 27 }  ,draw opacity=1 ][fill={rgb, 255:red, 208; green, 2; blue, 27 }  ,fill opacity=1 ] (335.91,374.15) .. controls (335.84,372.39) and (334.35,370.97) .. (332.59,370.97) .. controls (330.83,370.97) and (329.46,372.39) .. (329.53,374.15) .. controls (329.6,375.91) and (331.09,377.34) .. (332.85,377.34) .. controls (334.61,377.34) and (335.98,375.91) .. (335.91,374.15) -- cycle ;
\draw  [color={rgb, 255:red, 208; green, 2; blue, 27 }  ,draw opacity=1 ][fill={rgb, 255:red, 208; green, 2; blue, 27 }  ,fill opacity=1 ] (394.91,279.15) .. controls (394.84,277.39) and (393.35,275.97) .. (391.59,275.97) .. controls (389.83,275.97) and (388.46,277.39) .. (388.53,279.15) .. controls (388.6,280.91) and (390.09,282.34) .. (391.85,282.34) .. controls (393.61,282.34) and (394.98,280.91) .. (394.91,279.15) -- cycle ;
\draw  [dash pattern={on 4.5pt off 4.5pt}] (122.18,127.18) -- (522.18,127.18) -- (522.18,527.18) -- (122.18,527.18) -- cycle ;
\draw  [color={rgb, 255:red, 208; green, 2; blue, 27 }  ,draw opacity=1 ][fill={rgb, 255:red, 208; green, 2; blue, 27 }  ,fill opacity=1 ] (198.91,424.15) .. controls (198.84,422.39) and (197.35,420.97) .. (195.59,420.97) .. controls (193.83,420.97) and (192.46,422.39) .. (192.53,424.15) .. controls (192.6,425.91) and (194.09,427.34) .. (195.85,427.34) .. controls (197.61,427.34) and (198.98,425.91) .. (198.91,424.15) -- cycle ;
\draw    (122.18,541.18) -- (522.18,541.18) ;
\draw [shift={(522.18,541.18)}, rotate = 180] [color={rgb, 255:red, 0; green, 0; blue, 0 }  ][line width=0.75]    (0,5.59) -- (0,-5.59)(10.93,-3.29) .. controls (6.95,-1.4) and (3.31,-0.3) .. (0,0) .. controls (3.31,0.3) and (6.95,1.4) .. (10.93,3.29)   ;
\draw [shift={(122.18,541.18)}, rotate = 0] [color={rgb, 255:red, 0; green, 0; blue, 0 }  ][line width=0.75]    (0,5.59) -- (0,-5.59)(10.93,-3.29) .. controls (6.95,-1.4) and (3.31,-0.3) .. (0,0) .. controls (3.31,0.3) and (6.95,1.4) .. (10.93,3.29)   ;
\draw    (17.82,647.54) -- (626.54,647.54) ;
\draw [shift={(626.54,647.54)}, rotate = 180] [color={rgb, 255:red, 0; green, 0; blue, 0 }  ][line width=0.75]    (0,5.59) -- (0,-5.59)(10.93,-3.29) .. controls (6.95,-1.4) and (3.31,-0.3) .. (0,0) .. controls (3.31,0.3) and (6.95,1.4) .. (10.93,3.29)   ;
\draw [shift={(17.82,647.54)}, rotate = 0] [color={rgb, 255:red, 0; green, 0; blue, 0 }  ][line width=0.75]    (0,5.59) -- (0,-5.59)(10.93,-3.29) .. controls (6.95,-1.4) and (3.31,-0.3) .. (0,0) .. controls (3.31,0.3) and (6.95,1.4) .. (10.93,3.29)   ;

\draw (213,313.71) node [anchor=north west][inner sep=0.75pt]   [align=left] {\textcolor[rgb]{0.82,0.01,0.11}{\textbf{$\cdots$}}};
\draw (294.53,315.15) node [anchor=north west][inner sep=0.75pt]   [align=left] {\textcolor[rgb]{0.82,0.01,0.11}{$\bar{y}_1$}};
\draw (334.72,377.15) node [anchor=north west][inner sep=0.75pt]   [align=left] {\textcolor[rgb]{0.82,0.01,0.11}{$\bar{y}_2$}};
\draw (393.72,282.15) node [anchor=north west][inner sep=0.75pt]   [align=left] {\textcolor[rgb]{0.82,0.01,0.11}{$\bar{y}_3$}};
\draw (197.72,427.15) node [anchor=north west][inner sep=0.75pt]   [align=left] {\textcolor[rgb]{0.82,0.01,0.11}{$\bar{y}_P$}};
\draw (328,336.36) node [anchor=north west][inner sep=0.75pt]   [align=left] {$\barB_0$};
\draw (348,357) node [anchor=north west][inner sep=0.75pt]   [align=left] {$\barB_1$};
\draw (368.34,377.34) node [anchor=north west][inner sep=0.75pt]   [align=left] {$\barB_2$};
\draw (410,435.34) node [anchor=north west][inner sep=0.75pt]   [align=left] {$\barB_{P-1}$};
\draw (498,508.34) node [anchor=north west][inner sep=0.75pt]   [align=left] {$\barB_P$};
\draw (605,613.71) node [anchor=north west][inner sep=0.75pt]   [align=left] {$\barLambda_k$};
\draw (300,551.71) node [anchor=north west][inner sep=0.75pt]   [align=left] {$L_P\sim d_k^{\sqrt{\alpha}}$};
\draw (307,656.71) node [anchor=north west][inner sep=0.75pt]   [align=left] {$\sim d_{k+1}$};

\end{tikzpicture}

\caption{The construction of the chain of sites in the martingale.}
\label{martingale Wegner fig}
  
\end{figure}
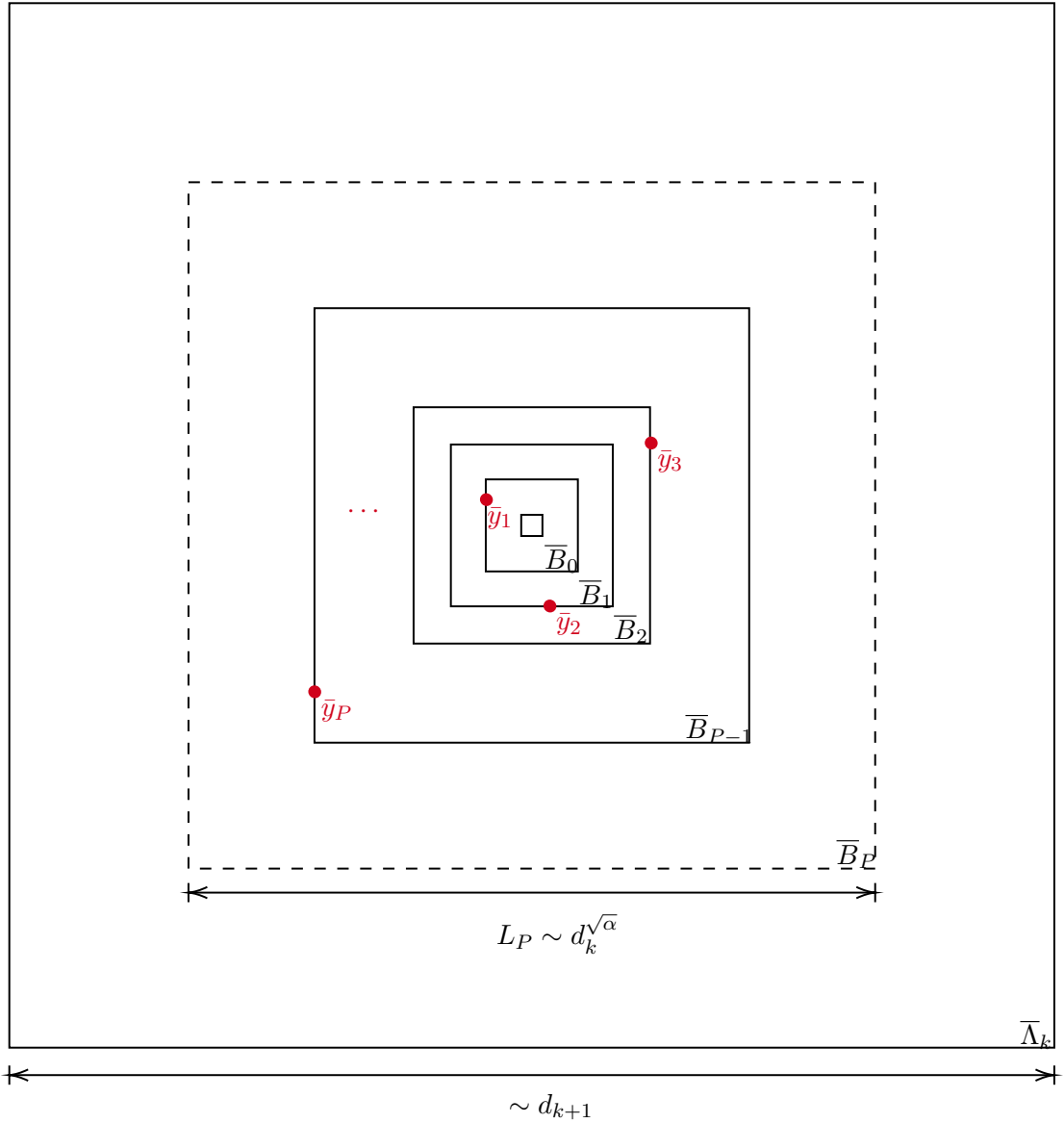

Now,  we let 
\begin{equation}\label{mcX}
  \mcX_j=\mcZ_1+\mcZ_2+\cdots + \mcZ_j,\ 1\leq j\leq P.
\end{equation}
Moreover, we let $\widetilde{\mcZ}_j=\mcZ_j-\frac{1}{2}$ and 
\begin{equation}\label{mcY}
  \mcY_j=\widetilde{\mcZ}_1+\widetilde{\mcZ}_2+\cdots+\widetilde{\mcZ}_j=\mcX_j-\frac{j}{2},  \ 1\leq j\leq P.
\end{equation}
Then applying \eqref{martingale expectation} yields 
\[\E(\widetilde{\mcZ}_{j+1} \ | \ \widetilde{\mcZ}_j)\geq 0\]
and therefore 
\[\E(\mcY_{j+1}  \ |\ \mcY_j)\geq \mcY_j.\]
This means $(\mcY_j)_{1\leq j\leq P}$ is a submartingale. Moreover, by our construction, 
\[|\mcY_{j+1}-\mcY_j|=|\widetilde{\mcZ}_{j+1}|\leq \frac{1}{2}.\]
Hence we can apply the Azuma's inequality and \eqref{P quantitive} to obtain 
\begin{align}\label{Azuma 1}
  \P\left(\mcX_P\leq \frac{1}{10}P\right) & = \P\left(\mcY_P\leq -(\frac{1}{2}-\frac{1}{10})P\right) \\
     \notag          &\leq \exp\left\{-2 \frac{\big((\frac{1}{2}-\frac{1}{10})P\big)^2}{\sum_{j=1}^P (\frac{1}{2})^2}\right\} \leq \exp\{-P\}\\
     \notag   &\leq \exp\left\{     -\frac{1}{2} \frac{\sqrt{\alpha}-1}{\alpha\cdot \log 5} \cdot \log d_{k+1}    \right\}=d_{k+1}^{-q}
\end{align}
for some constant $q>0$.

Clearly, the event $\mcX_P > \frac{1}{10}P$ represents  that the strict monotonicity in \eqref{monotonicity probabilistic estimate} happens more than $\frac{1}{10}P$ times, and therefore (Lemma \ref{monotonicity lemma} tells us that deterministically $\whn_{j}(B_j)\leq \whn_{j-1}(B_{j-1})$)
\begin{equation}\label{decreasing final}
  \whn_P(B_P)\leq \max\left\{ \ \whn_0(B_0)-\frac{1}{10}P, \ 0 \ \right\}.
\end{equation}
Now, recall that the $\whn_0(B_0)$ is the number of eigenvalues (counting multiplicities) of the Schur complement 
 \begin{equation}\label{schur complement in 0-step}
  \widetilde{S}^{(0)}_E (B_0)= H_{B_0}-\Gamma_0^{\top}(H_{\barB_0\setminus B_0}-E)^{-1}\Gamma_0
\end{equation}
in the interval $I_{\varepsilon_0}(E)$ for  $\varepsilon_0=\gamma^{6.2L_0}$. Since $\barB_0\setminus B_0$ is non-resonant, by Lemma \ref{non-resonant Green's function},  we get  
\[\| \widetilde{S}^{(0)}_E (B_0)- H_{B_0}  \|\leq \| \Gamma_0\|^2\cdot \|G_{\barB_0\setminus B_0}(E) \|\leq \frac{4d^2}{h-4d-\beta-\iota}\leq \frac{\iota}{4}. \]
The last inequality follows from our choice $h\geq h_0(d)\gg d^{50}$ with $d\geq 4$ (so that $\iota=\frac{1}{100}$ by \eqref{iota}). Therefore, applying Corollary \ref{Spectral stability lemma} will yield 
\begin{equation}\label{bound whn0 1}
  \whn_0(B_0)\leq \#\left( \spec(H_{B_0})\cap I_{\varepsilon_0+\frac{\iota}{4}}(E) \right).
\end{equation}
Recall that we fix  $E\in [-\frac{\iota}{2},4d+\beta+\frac{\iota}{2}]$ and $\varepsilon_0\ll \frac{\iota}{4} $ if we choose $k$ large. So
\[I_{\varepsilon_0+\frac{\iota}{4}}(E) \subset [-\iota,4d+\beta+\iota].\]
By \eqref{bound whn0 1}, $B_0=\Lambda'_k$ and Remark \ref{Lambda' has less eigenvalue number}, we obtain 
\begin{equation}\label{bound whn0 2}
  \whn_0(B_0)\leq \# \left(  \spec(H_{\Lambda'_k})\cap [-\iota,4d+\beta+\iota] \right) \leq 2 N^{k-k^{(1)}} (16d_{k^{(1)}}+1)^d. 
\end{equation}
Therefore, by \eqref{decreasing final} and \eqref{P quantitive}, the event $\mcX_P > \frac{1}{10}P$ implies 
\begin{align}\label{eigenvalue vanish}
  \whn_P(B_P) & \leq \max \left\{ 2 N^{k-k^{(1)}} (16d_{k^{(1)}}+1)^d - C\cdot(\sqrt{\alpha}-1)\frac{\log d_0}{\log a} \cdot \alpha^k ,\ 0\right\} =0.
\end{align}
The validity of \eqref{eigenvalue vanish} can be ensured by taking sufficiently large $k$ (depending only  on $h,\beta,d_0,\alpha,N$), since {\textbf{we assume $\alpha>N$}} and $k^{(1)}$ depends only on $h,\beta,d_0,\alpha,N$. Thus, by \eqref{Azuma 1},  we have  
\begin{equation}
  \P\left(\whn_P(B_P) =0 \right) \geq \P(\mcX_P > \frac{1}{10}P)\geq 1-d_{k+1}^{-q}.
\end{equation}

Finally, we return the control of $\whn_P(B_P)$ back to counting  the number of eigenvalues of $H_{\barLambda_k}$. Assume there exists  a $\lambda$ lying in $\spec(H_{\barB_P})\cap I_{\varepsilon_P/2}(E)$.  Then by Lemma \ref{fundamental schur lemma},  $\lambda$ is also an eigenvalue of the Schur complement $\widetilde{S}^{(P)}_\lambda(B_P)$. Moreover, since we choose $h>h_0(d)$, we can use Claim \ref{estimate on difference in energy} to deduce that 
\[ \|\widetilde{S}^{(P)}_\lambda(B_P)-\widetilde{S}^{(P)}_E(B_P)  \|\leq \gamma^{1.9}|\lambda-E|\leq \frac12 \gamma^{1.9} \varepsilon_P .\]
Therefore, by Corollary \ref{Spectral stability lemma}, there must be an eigenvalue $\wtlambda$ of $\widetilde{S}^{(P)}_E(B_P)$ such that 
\[|\wtlambda-E|\leq |\wtlambda-\lambda|+|\lambda-E|\leq \left(\frac{1}{2}\gamma^{1.9}+\frac{1}{2}\right)\varepsilon_P<\varepsilon_P,\]
which means $\whn_P(B_P)>0$. Thus, \eqref{eigenvalue vanish} tells us 
\begin{align}\label{before bootstrap}
  \P\left(\dist(\spec(H_{\barB_P}) , E) <\varepsilon_P=\gamma^{6.2L_P} \right) \leq \P(\whn_P(B_P)>0)\leq d_{k+1}^{-q}.
\end{align}

\bigskip

However, \eqref{before bootstrap} only ensures an exponential decay upper bound with respect the block length $\diam(\barB_P)=L_P$. But we need a sub-exponential decay upper bound in \eqref{Wegner poly probability 2}. So we need the following additional bootstrap argument. 

Denote the Schur complement of $H_{\barLambda_k}-\mcE$ about the block matrix restricted in $B_P=\Lambda'_k$ by 
\[\widetilde{S}^{\barLambda_k}_\mcE (B_P)= H_{B_P}-\Gamma_P^{\top} (H_{\barLambda_k \setminus B_P}-E)^{-1}\Gamma_P . \]
For $\mcE,\mcE'\in [-\iota,4d+\beta+\iota]$, the argument used in the proof of Claim \ref{difference in scales} can be applied analogously to  yield 
\begin{equation}\label{final difference in scales}
  \| \widetilde{S}^{\barLambda_k}_\mcE (B_P) - \widetilde{S}^{(P)}_\mcE (B_P) \| \leq \gamma^{(2a-0.3)L_P}=\gamma^{9.7L_P}.
\end{equation}
Now assume there is a $\lambda_0$ lying in $\spec(\widetilde{S}^{\barLambda_k}_E (B_P) )\cap I_{\varepsilon_P/10 }(E)$. By \eqref{final difference in scales},  we have 
\[   \| \widetilde{S}^{\barLambda_k}_E (B_P) - \widetilde{S}^{(P)}_E (B_P) \| \leq \gamma^{(2a-0.3)L_P}=\gamma^{9.7L_P}=g_P, \]
and thus Corollary \ref{Spectral stability lemma} implies there is an eigenvalue $\wtlambda_0$ of $\widetilde{S}^{(P)}_E (B_P)$ such that 
\[|E-\wtlambda_0|\leq |\lambda_0-E|+|\lambda_0-\wtlambda_0|< g_P +\frac{\varepsilon_P}{10}\leq \frac{\varepsilon_P}{5}. \]
Using Claim \ref{estimate on difference in energy}, we also have 
\[  \|  \widetilde{S}^{(P)}_E (B_P) -\widetilde{S}^{(P)}_{\wtlambda_0} (B_P)\| \leq \gamma^{1.9} |E-\wtlambda_0|,\]
and thus Corollary \ref{Spectral stability lemma} implies there is an eigenvalue $\wtlambda_1$ of $\widetilde{S}^{(P)}_{\wtlambda_0} (B_P)$ such that 
\[|\wtlambda_1-\wtlambda_0|\leq\gamma^{1.9 } |E-\wtlambda_0|. \]
Therefore, the fixed-point argument   leading to \eqref{distance of lambda and E} will help us to find a $\wtlambda$ such that 
\begin{equation}\label{bootstrap fixed point}
  \wtlambda\in \spec(\widetilde{S}^{(P)}_{\wtlambda} (B_P) ),\ |\wtlambda-E|<2|E-\wtlambda_0|<\frac{2}{5}\varepsilon_P.
\end{equation}
By Lemma \ref{fundamental schur lemma} and \eqref{bootstrap fixed point}, we  yield $\wtlambda\in \spec(H_{\barB_P})$, and therefore 
\[ \dist(\spec(H_{\barB_P}) , E) <\frac{2}{5}\varepsilon_P<\varepsilon_P .\]
That is to say, by \eqref{before bootstrap}, 
\begin{align}\label{after bootstrap}
 \P\left(\dist(\spec(H_{\barLambda_k}) , E) <\frac{1}{10}\varepsilon_P\right) \leq  \P\left(\dist(\spec(H_{\barB_P}) , E) <\varepsilon_P \right) \leq d_{k+1}^{-q}.
\end{align}
Finally, noticing that 
 \[ \frac{1}{10}\varepsilon_P \sim \exp\{-C\cdot d_{k+1}^{\frac{1}{\sqrt{\alpha}}} \} \ll \exp\{-d_{k+1}^{1-\varepsilon}\} \]
if we choose $\varepsilon$ to be smaller than $1-\alpha^{-1/2}$, we finish the proof of the Wegner estimate \eqref{Wegner poly probability 2}.

The entire proof is conducted under the assumption that $k$ is sufficiently large. We choose $k_{\text{in}}$ to be large enough (depending only on $h,\beta,\alpha,d_0,N$) to ensure the validity of the proof when $k\geq k_{\text{in}}$.

\end{proof}

\begin{rmk}
The final bootstrap argument relies on the existence of an annular non-resonant region surrounding $\barB_P$, in which  a similar structure also presents in  the important Bourgain's geometric lemma  \cite{Bou07} in the quasi-periodic Schr\"odinger operators case.
\end{rmk}

\section{Elimination of the energy and the proof of the localization}\label{section eliminate the energy}
In this section, we will use the Wegner estimates Theorem \ref{Wegner estimate for d=1,2,3} and Theorem \ref{Wegner estimate for d>3} to prove the our main theorems on the localization. 

From now on, for every $(k,N)$-hierarchical structure we simply write $2\leq s\leq N$ instead of $2\leq s\leq N_k\leq N$. This simplified  notation does not affect any of the estimates, since we always have the bound $ N_k \leq N$.

\subsection{Eliminating  the energy}
Recall that we use the notation $\overline{\Lambda}_k$ and $\overline{\Lambda}_{k}^s$ to stand for the $\frac{1}{10}d_{k+1}$-neighborhood of $\Lambda_k$ and $\Lambda_k(j_s)$ (for $s\geq 2$). 

To prove the Anderson localization, it's sufficient to establish exponential decay of all generalized eigenfunctions in the corresponding spectral interval. Since we have \eqref{spectral inclusion}, the well-known Shnol's lemma tells us that the spectral interval  $[0,4d+\beta]$ contains  the following generalized eigenvalues (of  full spectral measure): 
\begin{equation}\label{generalized eigenvalue}
  \mcE=\mcE(\omega)\in [0,4d+\beta] \ {\rm with\ generalized \ eigenfunction \ } \psi.
\end{equation}
Here the generalized eigenfunction satisfies 
\begin{equation}\label{generalized eigenfunction}
  |\psi(m_{\mcE})| \geq 1, \ |\psi(x)|\leq C_1(d) (|m_{\mcE}|+1)^{\frac{d}{2}+1}\cdot (|x|+1)^{\frac{d}{2}+1} \ {\rm for \ all\ } x\in \Z^d. 
\end{equation}
We emphasize that both  the generalized eigenvalue $\mcE$ and the site $m_{\mcE}=m_{\mcE}(\omega)$ depend on the random potential $V_{\text{r}}=\omega$, while the constant $C_1(d)$ only depends on  $d$ (For the statement and proof of Shnol's lemma, see e.g. \cite[Proposition 7.4]{Kir08}). 

Now we choose an arbitrary generalized eigenvalue $\mcE\in [0,4d+\beta]$. Then  we have 
\begin{lem}
  For any $\omega\in\Omega$, there is a scale $K_1=K_1(\mcE,\omega,h,\beta,d)>0$  such that, for  $\forall k\geq K_1$, we have 
  \begin{equation}\label{approximate the generalized eigenvalue}
      \dist\left(\mcE,\spec(H_{\barLambda_k})\cap [-\frac{\iota}{2},4d+\beta+\frac{\iota}{2}]\right)< \exp\left\{  -\frac{1}{50}\gamma_0 d_{k+1}      \right\}.
  \end{equation}

\end{lem}
\begin{proof}
We decompose 
\[H-\mcE = (H_{\mcB}-\mcE)\oplus (H_{\mcW}-\mcE)+\Gamma,\]
where $\Gamma$ is the connecting matrix between the potential barrier $\mcB$ and the potential well $\mcW$. Since  $\mcB$ is non-resonant, Lemma \ref{non-resonant Green's function}, \eqref{generalized eigenfunction} and
\[(H-\mcE)\psi=0,\]
we use the Poisson's formula to get for all $x\in \mcB$,
\begin{align}\label{Poisson Pieres}
  |\psi(x)|&=\left|   -\sum_{w\in \partial^-\mcB \atop w'\in \partial^+\mcB } G_{\mcB}(\mcE)(x,w) \cdot \psi(w')     \right| \\
\notag  &\leq \frac{2d C_1(d) (|m_{\mcE}|+1)^{\frac{d}{2}+1}}{h-4d-\beta-\iota}  \sum_{w\in \partial^-\mcB} \exp\{-\gamma_0 |x-w|_1\} (2+|w|)^{\frac{d}{2}+1} \\
 \notag & \lesssim_{d,h,\beta,\omega,\mcE } \sum_{w\in \partial^-\mcB} \exp\{-\gamma_0 |x-w|_1\} (2+|x|+|w-x|)^{\frac{d}{2}+1} \\
 \notag &   \lesssim_{d,h,\beta,\omega,\mcE } (2+|x|)^{\frac{d}{2}+1} \sum_{k \geq \dist(x,\partial^-\mcB)} k^{d+\frac{d}{2}+1} \exp\{-\gamma_0 \cdot k\} \\
\notag & \lesssim_{d,h,\beta,\omega,\mcE } (2+|x|)^{\frac{d}{2}+1} \exp \{-\gamma_0 \dist(x,\partial^-\mcB) \}.
\end{align}
Therefore, if $x\in \partial^{\pm}\barLambda_k$, then we have 
\[\dist(x,\partial^-\mcB) \geq \frac{1}{10}d_{k+1},\]
which combined with \eqref{Poisson Pieres} implies 
\begin{equation}\label{Poisson Pieres 2}
  |\psi(x)| \lesssim_{d,h,\beta,\omega,\mcE } (2+\frac{1}{10}d_{k+1})^{\frac{d}{2}+1} \exp \{-\frac{1}{10}\gamma_0 d_{k+1} \}.
\end{equation} 
Take $\psi_{\barLambda_k}=R_{\barLambda_k}\psi$. It's easy to check that $(H-\mcE)\psi_{\barLambda_k}$ supports only in $\partial^{\pm}\barLambda_k$  with 
\begin{align}\label{Poisson Pieres 3}
 \notag \| (H-\mcE) \psi_{\barLambda_k} \| & \leq 2d\cdot \sqrt{\# \partial^{\pm}\barLambda_k} \cdot \| \psi \|_{\ell^{\infty}( \partial^{\pm}\barLambda_k )} \\
     \notag  & \lesssim_{d,h,\beta,\omega,\mcE } (d_{k+1})^{{d}+1} \exp \{-\frac{1}{10}\gamma_0 d_{k+1} \} \\
         &\leq \exp\{-\frac{1}{11}\gamma_0 d_{k+1}\}. 
\end{align}
The validity of \eqref{Poisson Pieres 3} requires $k$ to be  large enough (depending only on $d,h,\beta,\omega,\mcE$). 

Now let $\mcE_1,\mcE_2,\cdots,\mcE_{\# \barLambda_k}$ denote all  the eigenvalues of $H_{\barLambda_k}$ with corresponding normalized eigenvectors $\varphi_1,\cdots ,\varphi_{\# \barLambda_k}$. Since $\{\varphi_s\}_{1\leq s\leq \#\barLambda_k}$ forms an eigenbasis of $\ell^{2}(\barLambda_k)$, we have 
\[\psi_{\barLambda_k}=\sum_{s} c_s\varphi_s.\] 
Under this representation, using \eqref{Poisson Pieres 3} yields 
\begin{align*}
\left(\sum_{s:\ |\mcE_s-\mcE|\geq \exp\{-\frac{1}{50}\gamma_0 d_{k+1}\}}  c_s^2\right)^{\frac{1}{2}} \exp\{-\frac{1}{50}\gamma_0 d_{k+1}\} & \leq \left(\sum_{s} |\mcE_s-\mcE|^2 c_s^2\right)^{\frac{1}{2}}  = \|(H_{\barLambda_k}-\mcE)\psi_{\barLambda_k} \| \\
       & \leq  \| (H-\mcE) \psi_{\barLambda_k} \|\leq \exp\{-\frac{1}{11}\gamma_0 d_{k+1}\}.
\end{align*}
Therefore,
\begin{equation}\label{coefficient 1}
  \sum_{s:\ |\mcE_s-\mcE|\geq \exp\{-\frac{1}{50}\gamma_0 d_{k+1}\}}  c_s^2 \leq \exp\{-\frac{1}{11}\gamma_0 d_{k+1}\}.
\end{equation}
Moreover, if we let $k$ be sufficiently large so that  $m_{\mcE}\in \barLambda_k$ (thus the largeness here depends  on $\omega,\mcE$), we will obtain 
\begin{equation}\label{coefficient 2}
  \| \psi_{\barLambda_k} \|^2  =\sum_{s}c_s^2 \geq |\psi(m_{\mcE})|^2 \geq 1.
\end{equation}
Combining \eqref{coefficient 1} and \eqref{coefficient 2} yields 
\[ \sum_{s:\ |\mcE_s-\mcE|< \exp\{-\frac{1}{50}\gamma_0 d_{k+1}\} } c_s^2\geq 1- \exp\{-\frac{1}{11}\gamma_0 d_{k+1}\}>0,\]
which means that there is some  $\mcE_{s'}\in \spec(H_{\barLambda_k})$ such that 
\[|\mcE_{s'}-\mcE|< \exp\{-\frac{1}{50}\gamma_0 d_{k+1}\}\ll \frac{\iota}{2}.\]
Since we have chosen  $\mcE\in [0,4d+\beta]$, we obtain $\mcE_{s'}\in [-\frac{\iota}{2},4d+\beta+\frac{\iota}{2}]$. This proves \eqref{approximate the generalized eigenvalue}.

\end{proof}

The next  lemma  concerns the separation of the spectrum, which is useful for  eliminating  the energy dependence in the Wegner estimate.
\begin{lem}[{\bf Separation of the spectrum}]\label{separation of spectrum lemma}
For almost every $\omega$,  there exists  a scale $K_2=K_2(\omega,h,d_0,\beta,N,\alpha)>0$ such that, for $\forall k\geq K_2$, we have 
  \begin{equation}\label{separation event}
    \dist\left( \spec(H_{\barLambda_k}),\spec(H_{\barLambda^s_k})\right)\geq \exp \{-d_{k+1}^{1-\varepsilon} \}.
  \end{equation}
\end{lem}
\begin{proof}[Proof of Lemma \ref{separation of spectrum lemma}]
  We take condition probability  on $V_{\text{r},\barLambda_k}$, i.e.,  the randomness in $\barLambda_k$. Then the spectrum 
  \[\spec(H_{\barLambda_k})\cap [-\frac{\iota}{2},4d+\beta+\frac{\iota}{2}]:= \{E_{i}\}_{1\leq i\leq M}\]
  is now fixed, where we denote
  \[M=\# \left(    \spec(H_{\barLambda_k})\cap [-\frac{\iota}{2},4d+\beta+\frac{\iota}{2}]  \right).\]
  By Lemma \ref{eigenvalue number}, we have $M\leq 2 N^{k-k^{(1)}} (16d_{k^{(1)}}+1)^{d}$.

  Now we can apply the Wegner estimates, Theorem \ref{Wegner estimate for d=1,2,3} and Theorem \ref{Wegner estimate for d>3}. We let $k\geq k_{\text{in}}$. For each pair $(\barLambda_k^s,E_i)$ with $2\leq s\leq N$ and $1\leq i\leq M$,  we have 
  \[ \P\left( \dist\big(\spec(H_{\barLambda_k^s}), E_i \big) < \exp\{-d_{k+1}^{1-\varepsilon}\} \ \bigg| \ V_{\text{r},\barLambda_k}   \right) \leq d_{k+1}^{-q}  \]
  and therefore, 
  \begin{align}\label{not separation event}
    &\P\left(\dist\big( \spec(H_{\barLambda_k}),\spec(H_{\barLambda^s_k})\big)< \exp \{-d_{k+1}^{1-\varepsilon} \}  \ \bigg| \ V_{\text{r},\barLambda_k} \right) \\
    \notag =&\P\left( \exists2\leq s\leq N, \ 1\leq i\leq M \ {\rm s.t.,} \ \dist\big(\spec(H_{\barLambda_k^s}), E_i \big) < \exp\{-d_{k+1}^{1-\varepsilon}\} \ \bigg| \ V_{\text{r},\barLambda_k} \right) \\
    \notag \leq & N\cdot M\cdot d_{k+1}^{-q} \leq 2 N^{1-k^{(1)}} (16d_{k^{(1)}}+1)^{d} \cdot N^k d_0^{-q\cdot \alpha^k}.
  \end{align}
  Clearly, $\sum\limits_{k\geq 0} N^k d_0^{-q\cdot \alpha^k}<\infty$. Hence, by the Borel-Cantelli lemma, the event in \eqref{not separation event} almost surely does not occur for all sufficiently large $k$. This means almost surely there will be a $K_2>0$ (depending  on $\omega$, $h,d_0,\beta,N,\alpha$) such that,  for any $k\geq K_2$, \eqref{separation event} happens.

\end{proof}

Since the parameters $h,d_0,\beta,N,\alpha$ have been fixed, in the remainder of the proof we will suppress their dependences  of quantities  in the notation. 

\subsection{Proof of the Anderson localization}
For almost every $\omega\in\Omega$, we obtain the following scales:
\begin{itemize}
  \item $k\geq k_{\text{in}}\ \Rightarrow$ The Wegner estimates hold true,
  \item $k\geq K_1(\omega,\mcE)\ \Rightarrow$ Lemma \ref{approximate the generalized eigenvalue} holds true, 
  \item $k\geq K_2(\omega)\ \Rightarrow$ Lemma \ref{separation of spectrum lemma} holds true.
\end{itemize}

\begin{proof}[Proof of Theorem \ref{low dimension localization} and Theorem \ref{high dimension localization}]
For almost every $\omega$ (fixed), we arbitrarily choose a generalized eigenvalue $\mcE(\omega)$ in $[0,4d+\beta]$ with the generalized eigenfunction $\psi$. 

We take 
\[ k\geq \max\{ k_{\text{in}}, K_1(\omega,\mcE),K_2(\omega)\}:= k_{\text{AL}}(\omega).\]
Then, we can apply Lemma \ref{approximate the generalized eigenvalue} and Lemma \ref{separation of spectrum lemma} to obtain
\begin{align*}
  \dist(\mcE,\spec(H_{\barLambda_k^s}) ) \geq  \exp \{-d_{k+1}^{1-\varepsilon} \}- \exp\left\{  -\frac{1}{50}\gamma_0 d_{k+1}      \right\} \geq \frac{1}{2}\exp \{-d_{k+1}^{1-\varepsilon} \} 
\end{align*}
for any $2\leq s\leq N$. Therefore, 
\begin{equation}\label{resonant block L2 norm}
  \| G_{\barLambda_k^s}(\mcE)\|\leq 2 \exp\{d_{k+1}^{1-\varepsilon}\}\ {\rm for} \ \forall   2\leq s\leq N.
\end{equation}
Let $D_k$ denote the $2d_{k}$-neighborhood of $\Lambda_k$, and $D_k^s$ the $2d_k$-neighborhood of $\Lambda_k(j_s)$. Denote the annular region $D_{k+1}\setminus \Lambda_k$ by $\mcA_k$. Then we have the following geometric structure: 
\begin{itemize}
  \item $\mcA_k$ and $\mcA_{k+1}$ have an overlap of size  $2d_{k+1}$, and $\mcA_{k-1}\cap\mcA_{k+1}=\emptyset$, 
  \item $\barLambda_k^s\subset \mcA_k$, and  $\mcA_k\setminus (\cup_{2\leq s\leq N} \barLambda_k^s)$ is non-resonant. Moreover, $\dist(\barLambda_k^s,D_k)\geq \frac{9}{5}d_{k+1}$ and $\dist(\barLambda_k^s,\barLambda_k^{s'})\geq \frac{17}{10}d_{k+1},s\neq s'$ .
\end{itemize}
Now take $L=d_{k},L_1=\diam(\mcA_k)\approx 16 d_{k+1}\sim d_k^{\alpha}$ and $L'=\diam(\barLambda_k^s)\approx \frac{1}{5}d_{k+1}$. Then $1\ll L \ll L'\leq \frac{1}{2}L_1$. Moreover, we have 
\[\diam (D_k^s)\sim L, \ {\rm the \ number \ of \ } D_k^s \ {\rm in} \ \mcA_k \ {\rm is \  less \ than \ } N \ \Rightarrow {\rm denote} \ \mathfrak{S}=\cup_s D_k^s, \]
\[\diam(\barLambda_k^s)\sim L',\  \barLambda_k^s \ {\rm contains\ a}  \  \frac{L'}{10}{\rm-neighborhood\ of \ } D_k^s\ \Rightarrow {\rm denote} \ \mathfrak{S}' =\cup_s \barLambda_k^s.\]
Thus, we can construct a class of $L$-size blocks (denoted by $\mathfrak{F}$) to cover  $\mcA_k\setminus \mathfrak{S}'$ such that,  every element of $\mathfrak{F}$ is contained in $\mcA_{k}\setminus (\cup_s \Lambda_k(j_s))$. Moreover, we can ensure that for each $n\in \mcA_k\setminus\mathfrak{S}$, there is a $\Lambda'\in \mathfrak{F}$ such that 
\[Q_{\frac{L}{5}}(n)\subset \Lambda'.\]
We denote $\widetilde{\mcA}_k=\cup_{\mathfrak{F}}\Lambda'$.
Therefore, in $\mcA_k$,  the geometric structure condition of   Lemma \ref{coupling lemma} is satisfied. So applying  Lemma \ref{coupling lemma} combined with  \eqref{resonant block L2 norm} implies 
\begin{equation}\label{good annuls off diagonal decay}
        |G_{\mcA_k}(x,y;\mcE) | \leq \exp\{-\frac{1}{2}\gamma_0 |x-y|\} \ {\rm if} \ |x-y|\geq \frac{1}{10} d_{k+1} \geq \frac{1}{200} L_1.
\end{equation}
Moreover, we also have 
\begin{equation}\label{good covering off diagonal decay}
        |G_{\widetilde{\mcA}_k}(x,y;\mcE) | \leq \exp\{-\frac{1}{2}\gamma_0 |x-y| \} \ {\rm if} \ |x-y|\geq  L = d_k.
\end{equation}

We additionally denote by $\widetilde{D}_k$ the $d_k$-neighborhood of $\Lambda_k$, which is an intermediate block satisfying $\Lambda_{k}\subset \widetilde{D}_k\subset D_k$. Now for any $x\notin \widetilde{D}_{k_{\text{AL}}}$, by our construction,  there will be a $k\geq k_{\text{AL}}$ such that $x\in \widetilde{D}_{k+1}\setminus \widetilde{D}_k$. We distinguish the following cases: \\
\underline{If $\dist(x,\partial^+ \Lambda_k)\geq \frac{1}{10}d_{k+1}$}, then we see  $x\in \mcA_k$ and $\dist(x,\mcA_k)\geq \frac{1}{10}d_{k+1}\geq \frac{1}{200}\diam(\mcA_k)\geq \frac{|x|}{200}$. 
Applying the Poisson's formula  together with  \eqref{generalized eigenfunction},  \eqref{good annuls off diagonal decay} will lead to 
\begin{align}\label{exp decay 1}
  |\psi(x)| & \leq \sum_{w\in \partial^-\mcA_k \atop w'\sim w,w'\in \partial^+\mcA_k } |G_{\mcA_k}(x,w;\mcE)|\cdot |\psi(w')| \\
        \notag   & \lesssim_d (|m_{\mcE}|+1)^{\frac{d}{2}+1} \cdot (\# \partial^-\mcA_k) \exp\{-\frac{1}{2}\gamma_0 \cdot \frac{1}{20}d_{k+2}\} \\
  \notag         & \lesssim_{d,\omega,\mcE} \exp\{-\frac{1}{500}\gamma_0 |x|\}.
\end{align}\\
\underline{If $\dist(x,\partial^+\Lambda_k)\leq \frac{1}{10}d_{k+1}$}, we see  $x\in \widetilde{\mcA}_k$. In this case, since 
\[\dist(x,\barLambda_{k}^s)\geq \dist (\barLambda_k^s,\Lambda_k)-\dist(x,\Lambda_k)\geq \frac{9}{5}d_{k+1} -\frac{1}{10}d_{k+1} \geq \dist(x,\Lambda_k),\]
we must have $\dist(x,\widetilde{A}_k)=\dist(x,\partial^+\Lambda_k)$. Moreover, recalling  $x\notin \widetilde{D}_k$, we have $|x|\geq 7d_k$ and $\dist(x,\partial^+\Lambda_k)\geq d_k$, which also leads to 
\[\dist(x,\widetilde{A}_k)=\dist(x,\partial^+\Lambda_k)\geq |x|-6d_k\geq \frac{1}{7}|x|. \] 
Applying the Poisson's formula together with \eqref{generalized eigenfunction} and \eqref{good covering off diagonal decay}  will lead to 
\begin{align}\label{exp decay 2}
  |\psi(x)| & \leq \sum_{w\in \partial^-\widetilde{\mcA}_k \atop w'\sim w,w'\in \partial^+\widetilde{\mcA}_k } |G_{\widetilde{\mcA}_k}(x,w;\mcE)|\cdot |\psi(w')| \\
        \notag   & \lesssim_d (|m_{\mcE}|+1)^{\frac{d}{2}+1} \cdot (\# \partial^+\widetilde{\mcA}_k) \exp\{-\frac{1}{2}\gamma_0 \cdot \dist(x,\widetilde{\mcA}_k)\} \\
     \notag      & \lesssim_{d,\omega,\mcE} \exp\{-\frac{1}{14}\gamma_0 |x|\}.
\end{align}

Combining \eqref{exp decay 1} and \eqref{exp decay 2}, we have 
\begin{equation}\label{exp decay 3}
|\psi(x)|\lesssim_{d,\omega,\mcE} \exp\{-\frac{1}{500}\gamma_0 |x|\}\ {\rm for}\ \forall x\notin \widetilde{D}_{k_{\text{AL}}}.
\end{equation}

Hence,  $\psi$ exhibits exponential decay, and consequently $\mcE$ is an eigenvalue. This completes the proof of Theorem \ref{low dimension localization} and Theorem \ref{high dimension localization}.

\end{proof}

\subsection{Proof of the dynamical localization}
Having established Anderson localization, we now proceed to prove the dynamical localization.
\begin{proof}[Proof of Theorem \ref{dynamical localization}]
Since the Anderson localization (Theorem \ref{low dimension localization}, Theorem \ref{high dimension localization}) holds almost surely, we may assume (i.e.,  choose $\omega$ such that Anderson localization happens)
\[\spec(H)\cap [0,h) = \{E_r\}_{r \in \mathfrak{r}} \]
consists of eigenvalues of $H$ in the spectral interval $[0,4d+\beta]$, indexed by a countable set $\mathfrak{r}$ (counting multiplicities). Let $\varphi_r$ denote the corresponding normalized eigenfunction  of $E_r$. Then we have 
\[P_{[0,h)}(H)\delta_0 =\sum_{r} \varphi_r(0) \varphi_r,\]
and therefore,
\begin{align}\label{MSD}
  r^2(t,\omega) &= \sum_{z\in \Z^d} |n|^2 \cdot \left|  \sum_{r} e^{-itE_r}\varphi_r(0)\varphi_r(n) \right|^2 \\
  \notag & \leq \sum_{n\in \Z^d}|n|^2 \left(\sum_r |\varphi_r(0)|\cdot |\varphi_r(n)|\right)^2 \\
  \notag  & \leq \sum_{n\in \Z^d} \sum_{r} \sum_{t} |\varphi_r(0)| \cdot |\varphi_t(0)| \cdot |n|^2 \cdot |\varphi_r(n)|\cdot |\varphi_t(n)|.
\end{align}

Now for $j\geq 1$, let 
\[I_j=\left\{r:\ |\varphi_r(0)| \geq \exp\{-\frac{1}{2}\gamma_0 d_{j+1} \}  \right\}.\]
\begin{claim}\label{resonant in center claim}
  If $r\in I_j$, then for each $k\geq j$, we must have 
\begin{equation}\label{resonant in center}
  \dist(E_r,\spec(H_{\barLambda_k}))\leq \frac{1}{2}\exp \{-d_{k+1}^{1-\varepsilon}\}.
\end{equation}
\end{claim}
\begin{proof}
  Assume \eqref{resonant in center} fails. Then we have 
  \begin{equation}\label{center good}
      \| G_{\barLambda_k}(E_r) \| \leq 2\exp\{d_{k+1}^{1-\varepsilon}\}.
  \end{equation}
  Let $\widetilde{\Lambda}_k$ be the $d_{k+1}$-neighborhood of $\barLambda_k$, and recall that $\Lambda'_k$ is the $2d_k$-neighborhood of $\Lambda_k$. We have 
  \[\widetilde{\Lambda}_k = (\widetilde{\Lambda}_k \setminus \Lambda'_k)\cup \barLambda_k.\]
  Now set $L=d_k,L'=\diam(\barLambda_k)\approx \frac{1}{5}d_{k+1}, L_1=\diam(\widetilde{\Lambda}_k)\approx \frac{11}{5}d_{k+1}\sim L^{\alpha}$. Then $1\ll L\ll L'\leq \frac{1}{2}L_1$, and 
  \[ \diam(\Lambda'_k)\sim L .\]
  Moreover, $\barLambda_k$ contains the $\frac{L'}{10}$-neighborhood of $\Lambda'_k$. Since $\widetilde{\Lambda}_k\setminus \Lambda_k$ is non-resonant, we can construct a class of $L$-size blocks (denoted by $\mathfrak{F}$) to cover $\widetilde{\Lambda}_k\setminus \barLambda_k$ such that,  every element of $\mathfrak{F}$ is contained in $\widetilde{\Lambda}_k\setminus \Lambda_k$. Moreover, we can ensure that for each $n\in \widetilde{\Lambda}_k \setminus\Lambda'_k $, there is a $\Lambda'\in \mathfrak{F}$ such that 
 \[Q_{\frac{L}{5}}(n)\subset \Lambda'.\]
  
  By above geometric structure and \eqref{center good}, we can apply Lemma \ref{coupling lemma} (with $\mathfrak{S}=\Lambda'_k,\mathfrak{S}'=\barLambda_k$) to conclude 
  \begin{equation}\label{whole block off diagonal decay}
        |G_{\widetilde{\Lambda}_k}(x,y;E_r) | \leq \exp\{-\frac{1}{2}\gamma_0 |x-y|\} \ {\rm if} \ |x-y|\geq \frac{1}{100} d_{k+1} \geq \frac{1}{200} L_1.
  \end{equation}
  Thus, using \eqref{whole block off diagonal decay} together with the Poisson's formula, we obtain 
  \begin{align*}
    |\varphi_r(0)| & \leq \sum_{w\in \partial^-\widetilde{\Lambda}_k \atop w'\sim w,w'\in \partial^+\widetilde{\Lambda}_k } |G_{\widetilde{\Lambda}_k}(0,w;E_r)|\cdot |\varphi_r (w')| \\
                    & \leq 2d (\# \partial^-\widetilde{\Lambda}_k) \exp\{  -\frac{1}{2}\gamma_0 \cdot \dist(0,\partial^-\widetilde{\Lambda}_k)\}\\
                    & < \exp\{  -\frac{1}{2}\gamma_0 d_{k+1}\} \leq   \exp\{  -\frac{1}{2}\gamma_0 d_{j+1}\}, 
  \end{align*}
  which contradicts the condition $r\in I_j$.

\end{proof}

Based on Claim \ref{resonant in center claim}, we can decompose \eqref{MSD} into  (we set $I_0=\emptyset$)
\begin{equation}\label{decompose MSD 1}
  \sum_{i=1}^{\infty} \sum_{j=1}^{\infty} \sum_{r\in I_i\setminus I_{i-1}} \sum_{t\in I_{j}\setminus I_{j-1}} \left( \sum_{n\in \Z^d}  |n|^2 |\varphi_r(n)|\cdot |\varphi_t(n)|\right) |\varphi_r(0)| \cdot | \varphi_t(0)|.
\end{equation}
Recall that $K_2=K_2(\omega)$ denotes the smallest scale  so that Lemma \ref{separation of spectrum lemma} holds. Therefore, if $r\in I_j$, by Claim \ref{resonant in center claim},  
\begin{align*}
  \dist(E_r,\spec(H_{\barLambda_k^s})) & \geq \dist(\spec(H_{\barLambda_k}),\spec(H_{\barLambda_k^s}))-\dist(E_r,\spec(H_{\barLambda_k})) \\
     &\geq \exp\{-d_{k+1}^{1-\varepsilon}\}-\frac{1}{2}\exp\{-d_{k+1}^{1-\varepsilon}\} \geq \frac{1}{2}\exp\{-d_{k+1}^{1-\varepsilon}\}.
\end{align*}
Thus, we have 
\begin{equation}\label{L2 norm bound, dynamical}
  \|G_{\barLambda_k^s}(E_r) \| \leq 2\exp\{d_{k+1}^{1-\varepsilon}\}\ {\rm for} \ \forall k\geq \max\{j, K_2\},\ \forall 2\leq s\leq N.
\end{equation}
Using  \eqref{L2 norm bound, dynamical} together with the previous argument (that leads to \eqref{exp decay 3}), we can prove that 
\begin{equation}\label{exp decay 4}
  |\varphi_r(x)|\lesssim_{d} \exp\{-\frac{1}{500}\gamma_0 |x|\}\ {\rm for}\ \forall |x| > 2 \max\{d_{j+1}, d_{K_2+1} \}.
\end{equation}
(Here we note that, because $\varphi_r$ has already been normalized, we may replace the bound \eqref{generalized eigenfunction} in Poisson’s formulas \eqref{exp decay 1} and \eqref{exp decay 2} by the uniform estimate $|\varphi_r(w')|\leq 1$. Consequently, the constant appearing in \eqref{exp decay 4} is independent of the index $r\in I_j$.) Thus, applying Hilbert-Schmidit argument yields
\begin{align*}
  \| \varphi_r \| ^2_{\ell^2(\Z^d\setminus Q_{2 \max\{d_{j+1}, d_{K_2+1} \}}(0))} & \leq \sum_{|x|\geq 2 \max\{d_{j+1}, d_{K_2+1} \}} \exp \left\{-\frac{1}{250}\gamma_0 |x|\right\} \leq \frac{1}{2},
\end{align*}
and 
\begin{align*}
  \# Q_{2 \max\{d_{j+1}, d_{K_2+1} \}}(0)&\geq \sum_{|x|\leq Q_{2 \max\{d_{j+1}, d_{K_2+1} \}}(0)}  \sum_{r} |\varphi_r(n)|^2 \\
     & \geq \sum_{r\in I_j} (1-\frac{1}{2})\geq \frac{1}{2} \cdot \# I_j.
\end{align*}
This means 
\begin{equation}\label{number of Ij}
  \# I_j \lesssim \left( d_{\max\{j,K_2\}+1} \right)^d\ {\rm for}\  \forall j\geq 1.
\end{equation}

\bigskip

Based on \eqref{number of Ij}, we further decompose \eqref{decompose MSD 1} into 
\begin{equation*}
  \left(\sum_{1\leq i \leq K_2} \sum_{1\leq j\leq K_2} + \sum_{i > K_2} \sum_{1\leq j\leq K_2} +\sum_{1\leq i \leq K_2} \sum_{ j> K_2} + \sum_{ i >K_2} \sum_{ j> K_2}\right)\sum_{r\in I_i\setminus I_{i-1}} \sum_{t\in I_{j}\setminus I_{j-1}}  \cdots .
\end{equation*}
We discuss:
\begin{itemize}
  \item For the summation over ${1\leq i \leq K_2}, {1\leq j\leq K_2}$, we can estimate 
          \begin{align*}
            &\sum_{1\leq i \leq K_2} \sum_{1\leq j\leq K_2}  \sum_{r\in I_i\setminus I_{i-1}} \sum_{t\in I_{j}\setminus I_{j-1}} \left( \sum_{n\in \Z^d}  |n|^2 |\varphi_r(n)|\cdot |\varphi_t(n)|\right) |\varphi_r(0)| \cdot | \varphi_t(0)| \\
                    =&  \sum_{r\in I_{K_2}} \sum_{t\in I_{K_2}} \left( \sum_{n\in \Z^d}  |n|^2 |\varphi_r(n)|\cdot |\varphi_t(n)|\right) |\varphi_r(0)| \cdot | \varphi_t(0)| \\
                    \leq & (\# I_{K_2})^2 \cdot \sup_{r,t\in I_{K_2}}\left(   \sum_{n\in \Z^d}  |n|^2 |\varphi_r(n)|\cdot |\varphi_t(n)| \right) \\
                    \overset{\text{\eqref{exp decay 4}}}{\lesssim_d } & (d_{K_2+1})^{2d} \left( \sum_{|n|\leq 2d_{K_2+1} }|n|^2 +\sum_{|n|>2d_{K_2+1}} |n|^2 \exp\{-\frac{1}{250}\gamma_0 |n|\}\right) <\infty. 
          \end{align*}

    \item For the summation over ${ i > K_2}, {1\leq j\leq K_2}$ and ${1\leq i \leq K_2}, {j> K_2}$, we can estimate, for example,
          \begin{align*}
            &\sum_{ i > K_2} \sum_{1\leq j\leq K_2}  \sum_{r\in I_i\setminus I_{i-1}} \sum_{t\in I_{j}\setminus I_{j-1}} \left( \sum_{n\in \Z^d}  |n|^2 |\varphi_r(n)|\cdot |\varphi_t(n)|\right) |\varphi_r(0)| \cdot | \varphi_t(0)| \\
                    \leq &   \sum_{t\in I_{K_2}} \sum_{ i > K_2}   \sum_{r\in I_i\setminus I_{i-1}} \left( \sum_{n\in \Z^d}  |n|^2 |\varphi_r(n)| \right) |\varphi_r(0)| \\
                    \leq & (\# I_{K_2})  \cdot\sum_{ i > K_2} \# I_i \cdot  \left( \sup_{r\in I_{i}}   \sum_{n\in \Z^d}  |n|^2 |\varphi_r(n)| \right)  \exp\{-\frac{1}{2}\gamma_0 d_i\} \\
                    \overset{\text{\eqref{exp decay 4}}}{\lesssim_d } & (d_{K_2+1})^{d} \sum_{ i > K_2} (d_{i+1})^d   \exp\{-\frac{1}{2}\gamma_0 d_i\} \cdot  \left( \sum_{|n|\leq 2d_{i+1} }|n|^2 +\sum_{|n|>2d_{i+1}} |n|^2 \exp\{-\frac{1}{500}\gamma_0 |n|\} \right)\\
                    \lesssim_d & (d_{K_2+1})^{d} \sum_{ i > K_2} (d_{i+1})^{2d+2}  \exp\{-\frac{1}{2}\gamma_0 d_i\} <\infty.
          \end{align*}
     \item For the summation over ${i > K_2}, { j> K_2}$, we can estimate 
          \begin{align}\label{third case term, dynamical}
          \notag  &\ \ \ \sum_{i > K_2} \sum_{ j> K_2}  \sum_{r\in I_i\setminus I_{i-1}} \sum_{t\in I_{j}\setminus I_{j-1}} \left( \sum_{n\in \Z^d}  |n|^2 |\varphi_r(n)|\cdot |\varphi_t(n)|\right) |\varphi_r(0)| \cdot | \varphi_t(0)| \\
                    =&  \sum_{i > K_2} \sum_{ j> K_2} \left( \sum_{r\in I_i\setminus I_{i-1}}  \sum_{n\in \Z^d}  |n|^2 |\varphi_r(n)|\cdot |\varphi_r(0)|\right)  \cdot  \left( \sum_{t\in I_j\setminus I_{j-1}}  \sum_{n\in \Z^d}  |n|^2 |\varphi_t(n)|\cdot |\varphi_t(0)|\right).
          \end{align}
          By the same argument used in the second case, we have 
          \begin{equation*}
             \left( \sum_{r\in I_i\setminus I_{i-1}}  \sum_{n\in \Z^d}  |n|^2 |\varphi_r(n)|\cdot |\varphi_r(0)|\right)\lesssim (d_{i+1})^{2d+2} \exp\{-\frac{1}{2}\gamma_0 d_i\}.
          \end{equation*}
          Hence, 
          \begin{equation*}
            \text{\eqref{third case term, dynamical}} \lesssim_d \sum_{i > K_2} \sum_{ j> K_2} (d_{i+1})^{2d+2} \exp\{-\frac{1}{2}\gamma_0 d_i\}\cdot (d_{j+1})^{2d+2} \exp\{-\frac{1}{2}\gamma_0 d_j\}<\infty.
          \end{equation*}

\end{itemize}
Summarizing all the estimates above, we obtain
\[r^2(t,\omega)<\infty\]
for almost every $\omega$, which completes the proof of Theorem \ref{dynamical localization}.

\end{proof}

\appendix
\section{Proof of Theorem \ref{UC for d arbitrary}}\label{Appendix UC}
In this section, we present the proof of Theorem \ref{UC for d arbitrary}. The core idea of the proof bears strong similarity to that of the higher-dimensional Wegner estimate in Section \ref{section dim bigger 4}: we leverage the cone property to construct a one-dimensional subset that exhibits transversality, and then use the arguments developed in Section \ref{martingale section} to build a martingale on this subset. We emphasize that the martingale structure employed here is strikingly analogous to the one introduced in Section \ref{martingale section}. In particular, during the construction of this martingale, the sites we select form a martingale themselves, one that depends on the preceding randomness—a feature we refer to as the {\bf "site-mixed" property}.

\begin{proof}[Proof of Theorem \ref{UC for d arbitrary}]
To get  better insight, we first consider the two-dimensional case.
  
\begin{Def}
For a fixed length $L\gg 1$, \textbf{the elementary propagating region} of length $L$ in $\Z^2$ is defined by 
\[\mcI_L:= \left\{ (n_1,n_2)\in \Z^2: \ n_2\geq 0 ,\ n_1+ n_2 \leq L+1, \  n_2-n_1 \leq L+1    \right\}. \]   

\end{Def}

\begin{rmk}
The region $\mcI_L$ can be interpreted as follows. Given initial data $u_0$ on $(\mcP_0\cup \mcP_1)\cap \mcI_L$, the solution $u$ of the Dirichlet problem
\begin{equation*}
\begin{cases}
\Delta u = W u, \\
u \equiv u_0 \ \text{on} \ (\mcP_0\cup\mcP_1)\cap \mcI_L,
\end{cases}
\end{equation*}
is uniquely determined precisely on $\mcI_L$. Moreover, $\mcI_L$ is the maximal region on which the values  of $u$ can be determined.
\end{rmk}

Now we consider the tilted line 
\[\mcT_k:=  \{(n_1,n_2)\in \Z^2 : \ n_2-n_1  = k \}.\]
Since the  equation $H(\omega) u-\lambda u=0 $ locally reads  as 
\begin{align}
\label{mcT 0}  u(n_1,n_2) &= -u(n_1-1,n_2-1)\\
 \label{mcT 1}   &\ \ \ + (2d+V(n_1,n_2-1)-\lambda)\cdot u(n_1,n_2-1) \\
  \label{mcT 2}  &\ \ \ -u(n_1+1,n_2-1)-u(n_1,n_2-2).
\end{align}
Clearly, \eqref{mcT 0} is about the information of $u$ on $\mcT_{n_2-n_1}$,  \eqref{mcT 1} on $\mcT_{n_2-n_1-1}$ and \eqref{mcT 2}  on $\mcT_{n_2-n_1-2}$. When $n_2\gg 1$, we can iterate the relationships along $\mcT_{n_2-n_1}$ to obtain 
{\Small
\begin{align*}
  &u(n_1,n_2) \\
  &= -\big(-u(n_1-2,n_2-2) + (2d+V(n_1-1,n_2-2)-\lambda)\cdot u(n_1-1,n_2-2) -u(n_1,n_2-2)-u(n_1-1,n_2-3) \big)\\
   &\quad + (2d+V(n_1,n_2-1)-\lambda)\cdot u(n_1,n_2-1) \\
   &\quad -u(n_1+1,n_2-1)-u(n_1,n_2-2)\\
   &=u(n_1-2,n_2-2 ) \\
   &\quad +(2d+V(n_1,n_2-1)-\lambda)\cdot u(n_1,n_2-1) - (2d+V(n_1-1,n_2-2)-\lambda)\cdot u(n_1-1,n_2-2) \\
  &\quad -u(n_1+1,n_2-1)+u(n_1-1,n_2-3) \\
  &=\cdots \\
  &=\sum_{x\in \mcT_{n_2-n_1-1}} a_x \cdot (2d+V(x)-\lambda) \cdot u(x)  +\sum_{y\in \mcT_{n_2-n_1-2}} b_y \cdot u(y) + {\rm Remains}(u_0).
\end{align*}
}
Here, we have $|a_x|=1$, $a_x,b_y$ are  coefficients independent of the randomness, and ${\rm Remains}(u_0)$ is a term only involving the initial data $u_0$  (see Figure  \ref{iteration figure}).

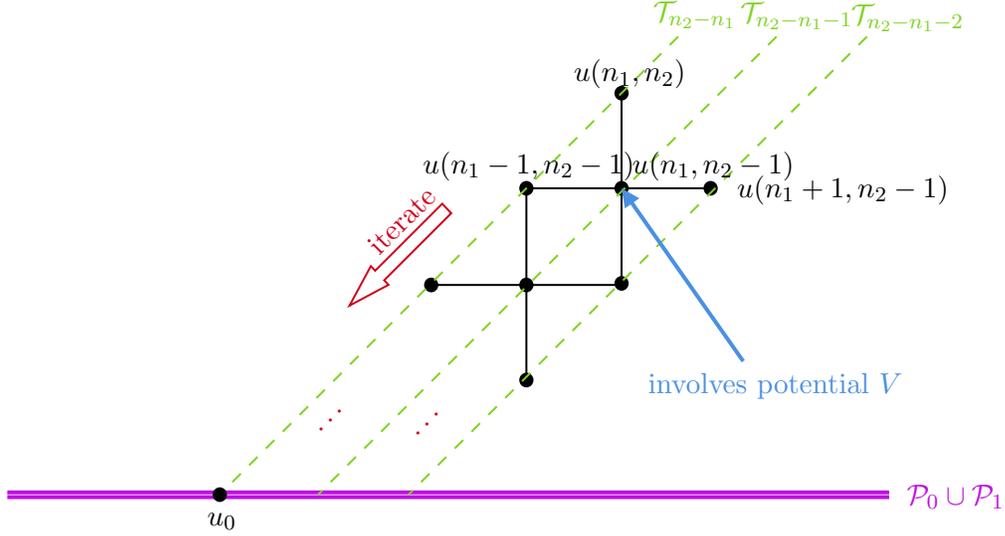
\begin{figure}[htbp]
  
\centering

\tikzset{every picture/.style={line width=0.75pt}} 

\begin{tikzpicture}[x=0.75pt,y=0.75pt,yscale=-0.8,xscale=0.8]

\draw  [fill={rgb, 255:red, 0; green, 0; blue, 0 }  ,fill opacity=1 ] (403.14,116) .. controls (403.14,113.79) and (404.87,112) .. (407,112) .. controls (409.13,112) and (410.86,113.79) .. (410.86,116) .. controls (410.86,118.21) and (409.13,120) .. (407,120) .. controls (404.87,120) and (403.14,118.21) .. (403.14,116) -- cycle ;
\draw  [fill={rgb, 255:red, 0; green, 0; blue, 0 }  ,fill opacity=1 ] (403.14,236) .. controls (403.14,233.79) and (404.87,232) .. (407,232) .. controls (409.13,232) and (410.86,233.79) .. (410.86,236) .. controls (410.86,238.21) and (409.13,240) .. (407,240) .. controls (404.87,240) and (403.14,238.21) .. (403.14,236) -- cycle ;
\draw  [fill={rgb, 255:red, 0; green, 0; blue, 0 }  ,fill opacity=1 ] (343.14,176) .. controls (343.14,173.79) and (344.87,172) .. (347,172) .. controls (349.13,172) and (350.86,173.79) .. (350.86,176) .. controls (350.86,178.21) and (349.13,180) .. (347,180) .. controls (344.87,180) and (343.14,178.21) .. (343.14,176) -- cycle ;
\draw [color={rgb, 255:red, 189; green, 16; blue, 224 }  ,draw opacity=1 ][line width=1.5]    (19.71,367.86) -- (575.71,367.86)(19.71,370.86) -- (575.71,370.86) ;
\draw    (347,176) -- (467,176) ;
\draw    (407,116) -- (407,236) ;
\draw  [fill={rgb, 255:red, 0; green, 0; blue, 0 }  ,fill opacity=1 ] (459.29,176) .. controls (459.29,173.79) and (461.01,172) .. (463.14,172) .. controls (465.27,172) and (467,173.79) .. (467,176) .. controls (467,178.21) and (465.27,180) .. (463.14,180) .. controls (461.01,180) and (459.29,178.21) .. (459.29,176) -- cycle ;
\draw  [fill={rgb, 255:red, 0; green, 0; blue, 0 }  ,fill opacity=1 ] (343.14,297) .. controls (343.14,294.79) and (344.87,293) .. (347,293) .. controls (349.13,293) and (350.86,294.79) .. (350.86,297) .. controls (350.86,299.21) and (349.13,301) .. (347,301) .. controls (344.87,301) and (343.14,299.21) .. (343.14,297) -- cycle ;
\draw  [fill={rgb, 255:red, 0; green, 0; blue, 0 }  ,fill opacity=1 ] (283.14,237) .. controls (283.14,234.79) and (284.87,233) .. (287,233) .. controls (289.13,233) and (290.86,234.79) .. (290.86,237) .. controls (290.86,239.21) and (289.13,241) .. (287,241) .. controls (284.87,241) and (283.14,239.21) .. (283.14,237) -- cycle ;
\draw    (287,237) -- (407,237) ;
\draw    (347,177) -- (347,297) ;
\draw [color={rgb, 255:red, 126; green, 211; blue, 33 }  ,draw opacity=1 ][fill={rgb, 255:red, 0; green, 0; blue, 0 }  ,fill opacity=1 ] [dash pattern={on 4.5pt off 4.5pt}]  (442.71,80.29) -- (153.71,369.29) ;
\draw  [fill={rgb, 255:red, 0; green, 0; blue, 0 }  ,fill opacity=1 ] (149.86,369.29) .. controls (149.86,367.08) and (151.58,365.29) .. (153.71,365.29) .. controls (155.84,365.29) and (157.57,367.08) .. (157.57,369.29) .. controls (157.57,371.49) and (155.84,373.29) .. (153.71,373.29) .. controls (151.58,373.29) and (149.86,371.49) .. (149.86,369.29) -- cycle ;
\draw  [fill={rgb, 255:red, 0; green, 0; blue, 0 }  ,fill opacity=1 ] (403.14,176) .. controls (403.14,173.79) and (404.87,172) .. (407,172) .. controls (409.13,172) and (410.86,173.79) .. (410.86,176) .. controls (410.86,178.21) and (409.13,180) .. (407,180) .. controls (404.87,180) and (403.14,178.21) .. (403.14,176) -- cycle ;
\draw  [fill={rgb, 255:red, 0; green, 0; blue, 0 }  ,fill opacity=1 ] (343.14,237) .. controls (343.14,234.79) and (344.87,233) .. (347,233) .. controls (349.13,233) and (350.86,234.79) .. (350.86,237) .. controls (350.86,239.21) and (349.13,241) .. (347,241) .. controls (344.87,241) and (343.14,239.21) .. (343.14,237) -- cycle ;
\draw [color={rgb, 255:red, 126; green, 211; blue, 33 }  ,draw opacity=1 ][fill={rgb, 255:red, 0; green, 0; blue, 0 }  ,fill opacity=1 ] [dash pattern={on 4.5pt off 4.5pt}]  (561.71,80.29) -- (272.71,369.29) ;
\draw [color={rgb, 255:red, 126; green, 211; blue, 33 }  ,draw opacity=1 ][fill={rgb, 255:red, 0; green, 0; blue, 0 }  ,fill opacity=1 ] [dash pattern={on 4.5pt off 4.5pt}]  (504.71,80.29) -- (215.71,369.29) ;
\draw  [color={rgb, 255:red, 208; green, 2; blue, 27 }  ,draw opacity=1 ] (235.5,250.03) -- (248.93,225.13) -- (251.66,228.03) -- (293.96,185.73) -- (299.63,191.74) -- (257.34,234.04) -- (260.07,236.93) -- cycle ;
\draw [color={rgb, 255:red, 74; green, 144; blue, 226 }  ,draw opacity=1 ][line width=1.5]    (409.3,179.27) -- (483.71,285.14) ;
\draw [shift={(407,176)}, rotate = 54.9] [fill={rgb, 255:red, 74; green, 144; blue, 226 }  ,fill opacity=1 ][line width=0.08]  [draw opacity=0] (11.61,-5.58) -- (0,0) -- (11.61,5.58) -- cycle    ;

\draw (210.76,327.42) node [anchor=north west][inner sep=0.75pt]  [color={rgb, 255:red, 208; green, 2; blue, 27 }  ,opacity=1 ,rotate=-315] [align=left] {$\cdots$};
\draw (585,361) node [anchor=north west][inner sep=0.75pt]   [align=left] {\textcolor[rgb]{0.74,0.06,0.88}{$\mcP_0\cup\mcP_1$}};
\draw (243,210.18) node [anchor=north west][inner sep=0.75pt]  [color={rgb, 255:red, 208; green, 2; blue, 27 }  ,opacity=1 ,rotate=-315] [align=left] {iterate};
\draw (144,379) node [anchor=north west][inner sep=0.75pt]   [align=left] {$u_0$};
\draw (425,56) node [anchor=north west][inner sep=0.75pt]  [color={rgb, 255:red, 126; green, 211; blue, 33 }  ,opacity=1 ] [align=left] {$\mcT_{n_2-n_1}$};
\draw (480,57) node [anchor=north west][inner sep=0.75pt]  [color={rgb, 255:red, 126; green, 211; blue, 33 }  ,opacity=1 ] [align=left] {$\mcT_{n_2-n_1-1}$};
\draw (550,58) node [anchor=north west][inner sep=0.75pt]  [color={rgb, 255:red, 126; green, 211; blue, 33 }  ,opacity=1 ] [align=left] {$\mcT_{n_2-n_1-2}$};
\draw (375,93) node [anchor=north west][inner sep=0.75pt]   [align=left] {$u(n_1,n_2)$};
\draw (280,151) node [anchor=north west][inner sep=0.75pt]   [align=left] {$u(n_1-1,n_2-1)$};
\draw (478,166) node [anchor=north west][inner sep=0.75pt]   [align=left] {$u(n_1+1,n_2-1)$};
\draw (412,151) node [anchor=north west][inner sep=0.75pt]   [align=left] {$u(n_1,n_2-1)$};
\draw (422,291) node [anchor=north west][inner sep=0.75pt]   [align=left] {\textcolor[rgb]{0.29,0.56,0.89}{involves potential $V$}};
\draw (271.76,329.42) node [anchor=north west][inner sep=0.75pt]  [color={rgb, 255:red, 208; green, 2; blue, 27 }  ,opacity=1 ,rotate=-315] [align=left] {$\cdots$};

\end{tikzpicture}

\caption{The iteration of the solution till the initial data.}
\label{iteration figure}

\end{figure}
This fact also has an analogue for the tilted line
\[\mcT'_k := \{(n_1,n_2)\in \Z^2: \ n_1+n_2=k\}. \]

Based on the above discussions, we can construct the martingale in $\mcI_L$ as follows. 

Choose the site $X_0=0\in \Z^2$  and take the $\sigma$-algebra with 
\[\mcB_0=\{\Omega,\emptyset\}\subset \mcF_0=\sigma (\{0\}) . \]
 Now, we already know the initial data 
\[u\equiv u_0 \ {\rm on} \ \mcP_0\cup \mcP_1.\]

\noindent \textbf{{(Step 1)}}\\

\noindent By the cone property (i.e., Proposition \ref{cone property})  and \eqref{cone property parameter}, there must be a site 
\[X_1\in  \mathfrak{C}(0,\bfe_2)=\{(0,1),(1,1),(-1,1),(0,2)\}\]
such that $|u(X_1)|\geq \gamma |u(0)|$. Since we already know the value of $u=u_0$ in $\mcP_1$, we can judge that, if $X_1$ belongs to   $\mcP_1$, we choose $X_1$ to be the site with the  smallest first coordinate; else if 
\[|u_0(x)|<\gamma |u(0)|\ {\rm for}\ \forall x\in \mcP_1,\]
then $X_1$ is  chosen to be $(0,2)$. Thus, by our construction,  {$X_1$ is $\mcF_0$-measurable} (deterministic).\\

\noindent \textbf{{(Step 2)}}\\

\noindent Define  the $\sigma$-algebra 
\begin{equation}\label{mcB 1}
  \mcB_1=\sigma\left( \mathfrak{B}_1 \right), \ \mathfrak{B}_1 =\mcI_L\cap \{0\leq n_2\leq (X_1)_2\} \setminus \{X_1\} .
\end{equation}
Now if the randomness in $\mathfrak{B}_1$ is fixed, then the value of $u$ in $(\mcI_L\cap \{0\leq n_2\leq (X_1)_2+1\}) \setminus \{X_1+\bfe_2\}$ will be determined.  See Figure \ref{construction of mcL1}.  In particular, the value $|u(X_1)|$ will be determined, which means  {$|u(X_1)|$ is $\mcB_1$-measurable}.

Using again the cone property, there must be a site 
\[X_2 \in  \mathfrak{C}(X_1,\bfe_2)=\{A_1,B_1,C_1,D_1\}\]
with 
\[A_1=X_1+\bfe_2-\bfe_1,B_1 =X_1+\bfe_2+\bfe_1,C_1=X_1+\bfe_2,D_1=X_1+2\bfe_2,\]
such that $|u(X_2)|\geq \gamma |u(X_1)|$. We can judge that,   if $X_2$ can be $A_1$ or $B_1$, we choose the one with smallest first coordinate; else if  
\[|u(A_1)|<\gamma |u(X_1)|,|u(B_1)|<\gamma |u(X_1)|,\]
then we must have $X_2\in \{C_1,D_1\}$. (This means, although conditioning on $\mcB_1$ cannot determine $X_2$, but can determine the relevant place where $X_2$ may occur.)

By the  preceding discussion, we can define the tilted line as follows.  If $X_2=A_1$ or $X_2\in \{C_1,D_1\}$, we take 
\[\mathfrak{L}_1= \{n_2-n_1=(X_1)_2-(X_1)_1\}=\mcT_{(X_1)_2-(X_1)_1},\]
which contains $X_1$, but does not pass through $X_2$; else if $X_2=B_1$, we take 
\[\mathfrak{L}_1= \{n_2+n_1=(X_1)_1+(X_1)_2\}=\mcT'_{(X_1)_1+(X_1)_2},\]
which also contains $X_1$ but does not pass through $X_2$. See Figure \ref{construction of mcL1}.

\begin{figure}[htbp]
  \centering

\tikzset{every picture/.style={line width=0.75pt}} 

\begin{tikzpicture}[x=0.75pt,y=0.75pt,yscale=-1,xscale=1]

\draw  [color={rgb, 255:red, 74; green, 144; blue, 226 }  ,draw opacity=1 ][fill={rgb, 255:red, 74; green, 144; blue, 226 }  ,fill opacity=1 ] (322,236) .. controls (322,233.79) and (323.73,232) .. (325.86,232) .. controls (327.99,232) and (329.71,233.79) .. (329.71,236) .. controls (329.71,238.21) and (327.99,240) .. (325.86,240) .. controls (323.73,240) and (322,238.21) .. (322,236) -- cycle ;
\draw  [fill={rgb, 255:red, 0; green, 0; blue, 0 }  ,fill opacity=1 ] (342,236) .. controls (342,233.79) and (343.73,232) .. (345.86,232) .. controls (347.99,232) and (349.71,233.79) .. (349.71,236) .. controls (349.71,238.21) and (347.99,240) .. (345.86,240) .. controls (343.73,240) and (342,238.21) .. (342,236) -- cycle ;
\draw  [fill={rgb, 255:red, 0; green, 0; blue, 0 }  ,fill opacity=1 ] (362,236) .. controls (362,233.79) and (363.73,232) .. (365.86,232) .. controls (367.99,232) and (369.71,233.79) .. (369.71,236) .. controls (369.71,238.21) and (367.99,240) .. (365.86,240) .. controls (363.73,240) and (362,238.21) .. (362,236) -- cycle ;
\draw  [fill={rgb, 255:red, 0; green, 0; blue, 0 }  ,fill opacity=1 ] (382,236) .. controls (382,233.79) and (383.73,232) .. (385.86,232) .. controls (387.99,232) and (389.71,233.79) .. (389.71,236) .. controls (389.71,238.21) and (387.99,240) .. (385.86,240) .. controls (383.73,240) and (382,238.21) .. (382,236) -- cycle ;
\draw  [fill={rgb, 255:red, 0; green, 0; blue, 0 }  ,fill opacity=1 ] (402,236) .. controls (402,233.79) and (403.73,232) .. (405.86,232) .. controls (407.99,232) and (409.71,233.79) .. (409.71,236) .. controls (409.71,238.21) and (407.99,240) .. (405.86,240) .. controls (403.73,240) and (402,238.21) .. (402,236) -- cycle ;
\draw  [fill={rgb, 255:red, 0; green, 0; blue, 0 }  ,fill opacity=1 ] (302,236) .. controls (302,233.79) and (303.73,232) .. (305.86,232) .. controls (307.99,232) and (309.71,233.79) .. (309.71,236) .. controls (309.71,238.21) and (307.99,240) .. (305.86,240) .. controls (303.73,240) and (302,238.21) .. (302,236) -- cycle ;
\draw  [fill={rgb, 255:red, 0; green, 0; blue, 0 }  ,fill opacity=1 ] (282,236) .. controls (282,233.79) and (283.73,232) .. (285.86,232) .. controls (287.99,232) and (289.71,233.79) .. (289.71,236) .. controls (289.71,238.21) and (287.99,240) .. (285.86,240) .. controls (283.73,240) and (282,238.21) .. (282,236) -- cycle ;
\draw  [fill={rgb, 255:red, 0; green, 0; blue, 0 }  ,fill opacity=1 ] (262,236) .. controls (262,233.79) and (263.73,232) .. (265.86,232) .. controls (267.99,232) and (269.71,233.79) .. (269.71,236) .. controls (269.71,238.21) and (267.99,240) .. (265.86,240) .. controls (263.73,240) and (262,238.21) .. (262,236) -- cycle ;
\draw  [fill={rgb, 255:red, 0; green, 0; blue, 0 }  ,fill opacity=1 ] (242,236) .. controls (242,233.79) and (243.73,232) .. (245.86,232) .. controls (247.99,232) and (249.71,233.79) .. (249.71,236) .. controls (249.71,238.21) and (247.99,240) .. (245.86,240) .. controls (243.73,240) and (242,238.21) .. (242,236) -- cycle ;
\draw  [fill={rgb, 255:red, 0; green, 0; blue, 0 }  ,fill opacity=1 ] (462,256) .. controls (462,253.79) and (463.73,252) .. (465.86,252) .. controls (467.99,252) and (469.71,253.79) .. (469.71,256) .. controls (469.71,258.21) and (467.99,260) .. (465.86,260) .. controls (463.73,260) and (462,258.21) .. (462,256) -- cycle ;
\draw  [fill={rgb, 255:red, 0; green, 0; blue, 0 }  ,fill opacity=1 ] (342,215) .. controls (342,212.79) and (343.73,211) .. (345.86,211) .. controls (347.99,211) and (349.71,212.79) .. (349.71,215) .. controls (349.71,217.21) and (347.99,219) .. (345.86,219) .. controls (343.73,219) and (342,217.21) .. (342,215) -- cycle ;
\draw  [fill={rgb, 255:red, 0; green, 0; blue, 0 }  ,fill opacity=1 ] (362,215) .. controls (362,212.79) and (363.73,211) .. (365.86,211) .. controls (367.99,211) and (369.71,212.79) .. (369.71,215) .. controls (369.71,217.21) and (367.99,219) .. (365.86,219) .. controls (363.73,219) and (362,217.21) .. (362,215) -- cycle ;
\draw  [fill={rgb, 255:red, 0; green, 0; blue, 0 }  ,fill opacity=1 ] (382,215) .. controls (382,212.79) and (383.73,211) .. (385.86,211) .. controls (387.99,211) and (389.71,212.79) .. (389.71,215) .. controls (389.71,217.21) and (387.99,219) .. (385.86,219) .. controls (383.73,219) and (382,217.21) .. (382,215) -- cycle ;
\draw  [fill={rgb, 255:red, 0; green, 0; blue, 0 }  ,fill opacity=1 ] (402,215) .. controls (402,212.79) and (403.73,211) .. (405.86,211) .. controls (407.99,211) and (409.71,212.79) .. (409.71,215) .. controls (409.71,217.21) and (407.99,219) .. (405.86,219) .. controls (403.73,219) and (402,217.21) .. (402,215) -- cycle ;
\draw  [fill={rgb, 255:red, 0; green, 0; blue, 0 }  ,fill opacity=1 ] (302,215) .. controls (302,212.79) and (303.73,211) .. (305.86,211) .. controls (307.99,211) and (309.71,212.79) .. (309.71,215) .. controls (309.71,217.21) and (307.99,219) .. (305.86,219) .. controls (303.73,219) and (302,217.21) .. (302,215) -- cycle ;
\draw  [fill={rgb, 255:red, 0; green, 0; blue, 0 }  ,fill opacity=1 ] (282,215) .. controls (282,212.79) and (283.73,211) .. (285.86,211) .. controls (287.99,211) and (289.71,212.79) .. (289.71,215) .. controls (289.71,217.21) and (287.99,219) .. (285.86,219) .. controls (283.73,219) and (282,217.21) .. (282,215) -- cycle ;
\draw  [fill={rgb, 255:red, 0; green, 0; blue, 0 }  ,fill opacity=1 ] (262,215) .. controls (262,212.79) and (263.73,211) .. (265.86,211) .. controls (267.99,211) and (269.71,212.79) .. (269.71,215) .. controls (269.71,217.21) and (267.99,219) .. (265.86,219) .. controls (263.73,219) and (262,217.21) .. (262,215) -- cycle ;
\draw  [fill={rgb, 255:red, 0; green, 0; blue, 0 }  ,fill opacity=1 ] (242,215) .. controls (242,212.79) and (243.73,211) .. (245.86,211) .. controls (247.99,211) and (249.71,212.79) .. (249.71,215) .. controls (249.71,217.21) and (247.99,219) .. (245.86,219) .. controls (243.73,219) and (242,217.21) .. (242,215) -- cycle ;
\draw  [fill={rgb, 255:red, 0; green, 0; blue, 0 }  ,fill opacity=1 ] (322,256) .. controls (322,253.79) and (323.73,252) .. (325.86,252) .. controls (327.99,252) and (329.71,253.79) .. (329.71,256) .. controls (329.71,258.21) and (327.99,260) .. (325.86,260) .. controls (323.73,260) and (322,258.21) .. (322,256) -- cycle ;
\draw  [fill={rgb, 255:red, 0; green, 0; blue, 0 }  ,fill opacity=1 ] (342,256) .. controls (342,253.79) and (343.73,252) .. (345.86,252) .. controls (347.99,252) and (349.71,253.79) .. (349.71,256) .. controls (349.71,258.21) and (347.99,260) .. (345.86,260) .. controls (343.73,260) and (342,258.21) .. (342,256) -- cycle ;
\draw  [fill={rgb, 255:red, 0; green, 0; blue, 0 }  ,fill opacity=1 ] (362,256) .. controls (362,253.79) and (363.73,252) .. (365.86,252) .. controls (367.99,252) and (369.71,253.79) .. (369.71,256) .. controls (369.71,258.21) and (367.99,260) .. (365.86,260) .. controls (363.73,260) and (362,258.21) .. (362,256) -- cycle ;
\draw  [fill={rgb, 255:red, 0; green, 0; blue, 0 }  ,fill opacity=1 ] (382,256) .. controls (382,253.79) and (383.73,252) .. (385.86,252) .. controls (387.99,252) and (389.71,253.79) .. (389.71,256) .. controls (389.71,258.21) and (387.99,260) .. (385.86,260) .. controls (383.73,260) and (382,258.21) .. (382,256) -- cycle ;
\draw  [fill={rgb, 255:red, 0; green, 0; blue, 0 }  ,fill opacity=1 ] (402,256) .. controls (402,253.79) and (403.73,252) .. (405.86,252) .. controls (407.99,252) and (409.71,253.79) .. (409.71,256) .. controls (409.71,258.21) and (407.99,260) .. (405.86,260) .. controls (403.73,260) and (402,258.21) .. (402,256) -- cycle ;
\draw  [fill={rgb, 255:red, 0; green, 0; blue, 0 }  ,fill opacity=1 ] (302,256) .. controls (302,253.79) and (303.73,252) .. (305.86,252) .. controls (307.99,252) and (309.71,253.79) .. (309.71,256) .. controls (309.71,258.21) and (307.99,260) .. (305.86,260) .. controls (303.73,260) and (302,258.21) .. (302,256) -- cycle ;
\draw  [fill={rgb, 255:red, 0; green, 0; blue, 0 }  ,fill opacity=1 ] (282,256) .. controls (282,253.79) and (283.73,252) .. (285.86,252) .. controls (287.99,252) and (289.71,253.79) .. (289.71,256) .. controls (289.71,258.21) and (287.99,260) .. (285.86,260) .. controls (283.73,260) and (282,258.21) .. (282,256) -- cycle ;
\draw  [fill={rgb, 255:red, 0; green, 0; blue, 0 }  ,fill opacity=1 ] (262,256) .. controls (262,253.79) and (263.73,252) .. (265.86,252) .. controls (267.99,252) and (269.71,253.79) .. (269.71,256) .. controls (269.71,258.21) and (267.99,260) .. (265.86,260) .. controls (263.73,260) and (262,258.21) .. (262,256) -- cycle ;
\draw  [fill={rgb, 255:red, 0; green, 0; blue, 0 }  ,fill opacity=1 ] (242,256) .. controls (242,253.79) and (243.73,252) .. (245.86,252) .. controls (247.99,252) and (249.71,253.79) .. (249.71,256) .. controls (249.71,258.21) and (247.99,260) .. (245.86,260) .. controls (243.73,260) and (242,258.21) .. (242,256) -- cycle ;
\draw  [fill={rgb, 255:red, 0; green, 0; blue, 0 }  ,fill opacity=1 ] (182,256) .. controls (182,253.79) and (183.73,252) .. (185.86,252) .. controls (187.99,252) and (189.71,253.79) .. (189.71,256) .. controls (189.71,258.21) and (187.99,260) .. (185.86,260) .. controls (183.73,260) and (182,258.21) .. (182,256) -- cycle ;
\draw  [fill={rgb, 255:red, 0; green, 0; blue, 0 }  ,fill opacity=1 ] (162,256) .. controls (162,253.79) and (163.73,252) .. (165.86,252) .. controls (167.99,252) and (169.71,253.79) .. (169.71,256) .. controls (169.71,258.21) and (167.99,260) .. (165.86,260) .. controls (163.73,260) and (162,258.21) .. (162,256) -- cycle ;
\draw  [fill={rgb, 255:red, 0; green, 0; blue, 0 }  ,fill opacity=1 ] (142,256) .. controls (142,253.79) and (143.73,252) .. (145.86,252) .. controls (147.99,252) and (149.71,253.79) .. (149.71,256) .. controls (149.71,258.21) and (147.99,260) .. (145.86,260) .. controls (143.73,260) and (142,258.21) .. (142,256) -- cycle ;
\draw  [fill={rgb, 255:red, 0; green, 0; blue, 0 }  ,fill opacity=1 ] (162,236) .. controls (162,233.79) and (163.73,232) .. (165.86,232) .. controls (167.99,232) and (169.71,233.79) .. (169.71,236) .. controls (169.71,238.21) and (167.99,240) .. (165.86,240) .. controls (163.73,240) and (162,238.21) .. (162,236) -- cycle ;
\draw  [fill={rgb, 255:red, 0; green, 0; blue, 0 }  ,fill opacity=1 ] (182,236) .. controls (182,233.79) and (183.73,232) .. (185.86,232) .. controls (187.99,232) and (189.71,233.79) .. (189.71,236) .. controls (189.71,238.21) and (187.99,240) .. (185.86,240) .. controls (183.73,240) and (182,238.21) .. (182,236) -- cycle ;
\draw  [fill={rgb, 255:red, 0; green, 0; blue, 0 }  ,fill opacity=1 ] (182,216) .. controls (182,213.79) and (183.73,212) .. (185.86,212) .. controls (187.99,212) and (189.71,213.79) .. (189.71,216) .. controls (189.71,218.21) and (187.99,220) .. (185.86,220) .. controls (183.73,220) and (182,218.21) .. (182,216) -- cycle ;
\draw  [fill={rgb, 255:red, 0; green, 0; blue, 0 }  ,fill opacity=1 ] (482,256) .. controls (482,253.79) and (483.73,252) .. (485.86,252) .. controls (487.99,252) and (489.71,253.79) .. (489.71,256) .. controls (489.71,258.21) and (487.99,260) .. (485.86,260) .. controls (483.73,260) and (482,258.21) .. (482,256) -- cycle ;
\draw  [fill={rgb, 255:red, 0; green, 0; blue, 0 }  ,fill opacity=1 ] (502,256) .. controls (502,253.79) and (503.73,252) .. (505.86,252) .. controls (507.99,252) and (509.71,253.79) .. (509.71,256) .. controls (509.71,258.21) and (507.99,260) .. (505.86,260) .. controls (503.73,260) and (502,258.21) .. (502,256) -- cycle ;
\draw  [fill={rgb, 255:red, 0; green, 0; blue, 0 }  ,fill opacity=1 ] (462,236) .. controls (462,233.79) and (463.73,232) .. (465.86,232) .. controls (467.99,232) and (469.71,233.79) .. (469.71,236) .. controls (469.71,238.21) and (467.99,240) .. (465.86,240) .. controls (463.73,240) and (462,238.21) .. (462,236) -- cycle ;
\draw  [fill={rgb, 255:red, 0; green, 0; blue, 0 }  ,fill opacity=1 ] (462,216) .. controls (462,213.79) and (463.73,212) .. (465.86,212) .. controls (467.99,212) and (469.71,213.79) .. (469.71,216) .. controls (469.71,218.21) and (467.99,220) .. (465.86,220) .. controls (463.73,220) and (462,218.21) .. (462,216) -- cycle ;
\draw  [fill={rgb, 255:red, 0; green, 0; blue, 0 }  ,fill opacity=1 ] (482,236) .. controls (482,233.79) and (483.73,232) .. (485.86,232) .. controls (487.99,232) and (489.71,233.79) .. (489.71,236) .. controls (489.71,238.21) and (487.99,240) .. (485.86,240) .. controls (483.73,240) and (482,238.21) .. (482,236) -- cycle ;
\draw  [color={rgb, 255:red, 208; green, 2; blue, 27 }  ,draw opacity=1 ][fill={rgb, 255:red, 208; green, 2; blue, 27 }  ,fill opacity=1 ] (322,196) .. controls (322,193.79) and (323.73,192) .. (325.86,192) .. controls (327.99,192) and (329.71,193.79) .. (329.71,196) .. controls (329.71,198.21) and (327.99,200) .. (325.86,200) .. controls (323.73,200) and (322,198.21) .. (322,196) -- cycle ;
\draw  [color={rgb, 255:red, 208; green, 2; blue, 27 }  ,draw opacity=1 ][fill={rgb, 255:red, 208; green, 2; blue, 27 }  ,fill opacity=1 ] (322,216) .. controls (322,213.79) and (323.73,212) .. (325.86,212) .. controls (327.99,212) and (329.71,213.79) .. (329.71,216) .. controls (329.71,218.21) and (327.99,220) .. (325.86,220) .. controls (323.73,220) and (322,218.21) .. (322,216) -- cycle ;
\draw  [dash pattern={on 0.84pt off 2.51pt}] (325.71,80.15) -- (506.81,255.72) -- (145.96,256.28) -- cycle ;
\draw [color={rgb, 255:red, 126; green, 211; blue, 33 }  ,draw opacity=1 ]   (131.71,227.14) -- (314.71,227.14) ;
\draw [color={rgb, 255:red, 184; green, 233; blue, 134 }  ,draw opacity=1 ]   (314.71,244.14) -- (314.71,227.14) ;
\draw [color={rgb, 255:red, 126; green, 211; blue, 33 }  ,draw opacity=1 ]   (314.71,244.14) -- (335.71,244.14) ;
\draw [color={rgb, 255:red, 126; green, 211; blue, 33 }  ,draw opacity=1 ]   (335.71,228.14) -- (335.71,244.14) ;
\draw [color={rgb, 255:red, 126; green, 211; blue, 33 }  ,draw opacity=1 ]   (526.71,226.14) -- (335.71,228.14) ;
\draw [color={rgb, 255:red, 126; green, 211; blue, 33 }  ,draw opacity=1 ]   (131.71,227.14) -- (131.71,278.14) ;
\draw [color={rgb, 255:red, 126; green, 211; blue, 33 }  ,draw opacity=1 ]   (131.71,278.14) -- (528.71,275.14) ;
\draw [color={rgb, 255:red, 126; green, 211; blue, 33 }  ,draw opacity=1 ]   (526.71,226.14) -- (528.71,275.14) ;
\draw [color={rgb, 255:red, 74; green, 144; blue, 226 }  ,draw opacity=1 ]   (334.71,301.14) .. controls (348.5,260.76) and (347.74,275.68) .. (326.83,241.59) ;
\draw [shift={(325.86,240)}, rotate = 58.84] [color={rgb, 255:red, 74; green, 144; blue, 226 }  ,draw opacity=1 ][line width=0.75]    (10.93,-3.29) .. controls (6.95,-1.4) and (3.31,-0.3) .. (0,0) .. controls (3.31,0.3) and (6.95,1.4) .. (10.93,3.29)   ;
\draw [color={rgb, 255:red, 245; green, 166; blue, 35 }  ,draw opacity=1 ][line width=1.5]    (429.71,130.14) -- (244.71,319.14) ;
\draw [color={rgb, 255:red, 144; green, 19; blue, 254 }  ,draw opacity=1 ][line width=1.5]    (227.71,132.14) -- (402.71,318.14) ;

\draw (313,50) node [anchor=north west][inner sep=0.75pt]   [align=left] {$\mcI_L$};
\draw (536,258) node [anchor=north west][inner sep=0.75pt]   [align=left] {\textcolor[rgb]{0.49,0.83,0.13}{$\mathfrak{B}_1$}};
\draw (315,300) node [anchor=north west][inner sep=0.75pt]   [align=left] {\textcolor[rgb]{0.29,0.56,0.89}{$X_1$}};
\draw (294,191) node [anchor=north west][inner sep=0.75pt]   [align=left] {\textcolor[rgb]{0.82,0.01,0.11}{$A_1$}};
\draw (316,168) node [anchor=north west][inner sep=0.75pt]   [align=left] {\textcolor[rgb]{0.82,0.01,0.11}{$D_1$}};
\draw (338,190) node [anchor=north west][inner sep=0.75pt]   [align=left] {\textcolor[rgb]{0.82,0.01,0.11}{$B_1$}};
\draw (434,110) node [anchor=north west][inner sep=0.75pt]  [color={rgb, 255:red, 245; green, 166; blue, 35 }  ,opacity=1 ] [align=left] {\textcolor[rgb]{0.96,0.65,0.14}{\textbf{$\mcT_{(X_1)_2-(X_1)_1}$}}};
\draw (200,112) node [anchor=north west][inner sep=0.75pt]   [align=left] {\textbf{\textcolor[rgb]{0.74,0.06,0.88}{$\mcT'_{(X_1)_1+(X_1)_2}$}}};
\draw (201,234) node [anchor=north west][inner sep=0.75pt]   [align=left] {$\cdots$};
\draw (422,234) node [anchor=north west][inner sep=0.75pt]   [align=left] {$\cdots$};

\end{tikzpicture}

\caption{Construction of $\mathfrak{L}_1$.}
\label{construction of mcL1}

\end{figure}
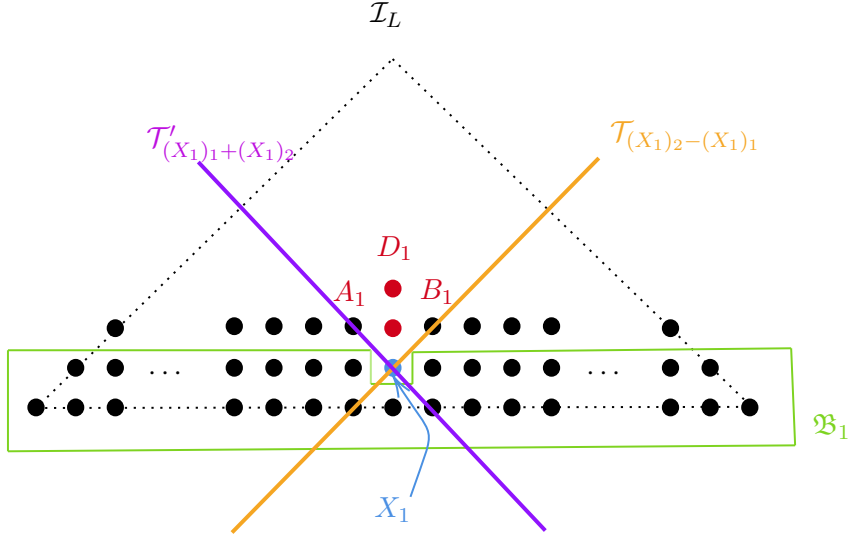

Therefore, the tilted line $\mathfrak{L}_1$ is determined, and it uniquely determines the site 
\[\mathfrak{e}_1\in (\mathfrak{L}_1+\bfe_2) \cap \partialside \mcI_L,\]
where 
 \[ \partialside \mcI_L := \mcI_L \cap (\mcT_{L+1}\cup \mcT_{L}\cup \mcT'_{L+1}\cup \mcT'_{L})\]
is the side boundary of $\mcI_L$. 
{Thus both $\mathfrak{L}_1$ and $\mathfrak{e}_1$ are $\mcB_1$-measurable.}\\

Next, we denote the region under the tilted line $\mcT_k$ (resp. $\mcT'_k$) by 
\[\mcT_{\leq k}= \{n_2-n_1\leq k \}, \ \mcT'_{\leq k}= \{n_1+n_2\leq k\}.\]
With this notation, we denote the region under $\mathfrak{L}_1$ by $\mathfrak{R}_1$, i.e.,  
\begin{equation*}
  \mathfrak{R}_1 = 
  \begin{cases}
    \mcT_{\leq (X_1)_2-(X_1)_1},\ {\rm if }\ \mathfrak{L}_1=\mcT_{(X_1)_2-(X_1)_1}, \\
    \mcT'_{\leq (X_1)_1+(X_1)_2},\ {\rm if }\ \mathfrak{L}_1=\mcT'_{(X_1)_1+(X_1)_2}.
  \end{cases}
\end{equation*}
We construct the next filtration element by 
\begin{equation}\label{mcF 1}
    \mcF_1=\sigma(\mathfrak{F}_1), \ \mathfrak{F}_1:= (\mathfrak{R}_1\cap \mcI_L )\cup \mathfrak{B}_1.
\end{equation}
Our construction ensures that 
\begin{equation}\label{1 step relation}
 \mathfrak{B}_1 \subset \mathfrak{F}_1, \  X_1 \notin \mathfrak{B}_1, \   X_1\in \mathfrak{F}_1, \ X_2 \notin \mathfrak{F}_1. 
\end{equation}
{Therefore, $\mcB_1\subset \mcF_1$.} 

If we further condition on $\mcF_1$, clearly the value of $u$ in the region 
\[\{ x\in \mcI_L: \ x\in \mathfrak{F}_1  \ {\rm or} \ x-\bfe_2\in \mathfrak{F}_1 \}\]
will be determined, which contains the sites $C_1$ and $\mathfrak{e}_1$. Hence, we can do further discussions under the case $X_2\in \{C_1,D_1\}$: If $|u(C_1)|\geq \gamma |u(X_1)|$, then we let $X_2=C_1$; else, we let $X_2=D_1$. Thus the site $X_2$ is determined. 

Moreover, the value $u(\mathfrak{e}_1)$ is now determined. By our previous discussion, we can express it in the following form 
\begin{equation}\label{u e1 value1}
  u(\mathfrak{e}_1) = \sum_{x\in \mathfrak{L}_1} a_x \cdot (2d+V(x)-\lambda) \cdot u(x)  +\sum_{y\in \mathfrak{L}_1-e_2} b_y \cdot u(y) + {\rm Remains}(u_0).
\end{equation}
Observing that the value of $u$ in the right hand side of \eqref{u e1 value1} will be determined if we condition on the randomness in $\mathfrak{F}_1\setminus \{X_1\}$, which means
\[u(\mathfrak{e}_1) = a_{X_1} \cdot V(X_1)u(X_1)+ {\rm a\ deterministic \ value},\ |a_{X_1}|=1. \]
Since $V(X_1)=V_{\text{hi}}(X_1)+ \beta \omega(X_1)$ and we have transversality  estimate at $X_1$ 
\[|u(X_1)|\geq \gamma |u(X_0)|=\gamma |u(0)|,\]
we can conclude that 
\begin{align*}
   \P\left( |u(\mathfrak{e}_1)| \geq \frac{1}{2}\beta |u(X_1)| \ \bigg|   \ \mcB_1 \right)& = \E \left(    \P_{\omega(X_1)}\left( |u(\mathfrak{e}_1)| \geq \frac{1}{2}\beta |u(X_1)| \ \bigg|   \ \sigma(\mathfrak{F}_1 \setminus \{X_1\})  \right)  \bigg| \mcB_1 \right) \\
               &\geq \frac{1}{2}.
\end{align*}
Therefore, we define 
\begin{equation}\label{martingale UC 1}
  \mcZ_1= \mathbf{1}_{\{|u(\mathfrak{e}_1)| \geq \frac{1}{2}\beta |u(X_1)| \}}(\omega(X_1)).
\end{equation}
Then  {$X_2,\mcZ_1$ are $\mcF_1$-measurable,} and 
\begin{equation*}
  \E(\mcZ_1 | \mcF_0)\geq \frac{1}{2}.
\end{equation*}

\noindent \textbf{{(Step 3)}}\\

\noindent Analogously, we take 
\begin{equation}\label{mcB 2}
  \mcB_2=\sigma\left( \mathfrak{B}_2 \right), \ \mathfrak{B}_2 =(\mcI_L\cap \{0\leq n_2\leq (X_2)_2\} \setminus \{X_2\})\cup \mathfrak{F}_1  .
\end{equation}

If the randomness in $\mathfrak{B}_2$ is fixed, then the value of $u$ in $(\mcI_L\cap \{0\leq n_2\leq (X_2)_2+1\}) \setminus \{X_2+e_2\}$ will be determined. In particular, the value $|u(X_2)|$ will be determined, which means {$|u(X_2)|$ is $\mcB_2$-measurable}.

By similar argument in {\bf Step 2}, we can determine the relevant place of $X_3$, and then construct the tilted line $\mathfrak{L}_2$ and $\mathfrak{e}_2$. Such construction satisfies 
\[X_2\in \mathfrak{L}_2,\ X_3\notin {\mathfrak{L}_2},\ \mathfrak{e}_k\in (\mathfrak{L}_2+\bfe_2)\cap \partialside \mcI_L. \]
We also have  {$\mathfrak{L}_2,\mathfrak{e}_2$ are $\mcB_2$-measurable.} 

Next, with $\mathfrak{L}_2$ determined, the $\sigma$-algebra 
\begin{equation}\label{mcF 2}
    \mcF_2=\sigma(\mathfrak{F}_2), \ \mathfrak{F}_2:= (\mathfrak{R}_2\cap \mcI_L )\cup \mathfrak{B}_2 
\end{equation}
will be obtained, and we have 
\begin{equation}\label{2 step relation}
 \mathfrak{B}_2 \subset \mathfrak{F}_2,\   X_2 \notin \mathfrak{B}_2, \   X_2\in \mathfrak{F}_2, \ X_3 \notin \mathfrak{F}_2. 
\end{equation}
{Therefore, $\mcB_2\subset \mcF_2$.}

If we further condition on $\mcF_1$, using similar discussions  as in {\bf Step 2}  concludes  that $X_3$ will be determined, and that 
\begin{align*}
   \P\left( |u(\mathfrak{e}_2)| \geq \frac{1}{2}\beta |u(X_2)| \ \bigg|   \ \mcB_2 \right)& = \E \left(    \P_{\omega(X_2)}\left( |u(\mathfrak{e}_2)| \geq \frac{1}{2}\beta |u(X_2)| \ \bigg|   \ \sigma(\mathfrak{F}_2 \setminus \{X_2\})  \right)  \bigg| \mcB_2 \right) \\
               &\geq \frac{1}{2}.
\end{align*}
Therefore, we define 
\begin{equation}\label{martingale UC 2}
  \mcZ_2= \mathbf{1}_{\{|u(\mathfrak{e}_2)| \geq \frac{1}{2}\beta |u(X_2)|\} }(\omega(X_2)).
\end{equation}
Hence {$X_3,\mcZ_2$ are $\mcF_2$-measurable,} and 
\begin{equation*}
  \E(\mcZ_2| \mcZ_1)= \E(\mcZ_2 | \mcF_1)\geq \frac{1}{2}.
\end{equation*}

\noindent \textbf{{(Step 4)}}$\cdots$\\

By iterating the above constructions, we obtain the following filtration
\[\mcB_0\subset \mcF_0\subset \mcB_1\subset\mcF_1\subset  \mcB_2\subset\mcF_2\cdots \subset  \mcB_T \subset \mcF_T,\]
together with a (site-mixed) martingale $(|u(X_j)|,\mathfrak{L}_j,\mathfrak{e}_j, X_{j+1},\mcZ_j),1 \leq j\leq T$ satisfying:
\begin{itemize}
  \item Denote the region under $\mathfrak{L}_j$ by $\mathfrak{R}_j$. Then 
   \begin{align*}
    \mcB_j=\sigma\left( \mathfrak{B}_j \right),& \ \mathfrak{B}_j =(\mcI_L\cap \{0\leq n_2\leq (X_j)_2\} \setminus \{X_j\})\cup \mathfrak{F}_{j-1}  ,\\
      \mcF_j=\sigma(\mathfrak{F}_j),& \ \mathfrak{F}_j = (\mathfrak{R}_j\cap \mcI_L )\cup \mathfrak{B}_j.
  \end{align*}
  \item We have $X_j\notin \mathfrak{B}_j,X_{j}\in \mathfrak{L}_j\subset \mathfrak{F}_j,X_{j+1}\notin \mathfrak{F}_j$. Moreover, $X_{j+1}\in \mathfrak{C}(X_j,\bfe_2)$ has  transversality:  
  \[|u(X_{j+1})|\geq \gamma|u(X_j)|\geq \gamma^{j+1}|u(0)|.\]
  We also have $\mathfrak{e}_j\in \partialside \mcI_L \cap (\mathfrak{L}_j+\bfe_2)$,  and $\mathfrak{e}_j,1\leq j\leq T$ are distinct.
  \item $(|u(X_j)|,\mathfrak{L}_j,\mathfrak{e}_j)$ are $\mcB_j$-measurable,  and $ (X_{j+1},\mcZ_j)$ are $\mcF_j$-measurable. Moreover, 
        \[ \mcZ_j= \mathbf{1}_{\{|u(\mathfrak{e}_j)| \geq \frac{1}{2}\beta |u(X_j)|\} }(\omega(X_j)),\ \E(\mcZ_{j+1}| \mcZ_j)\geq \frac{1}{2}. \]
\end{itemize}
For a visual illustration, see Figure \ref{martingale UC graph in dim 2}.

\begin{figure}[htbp]
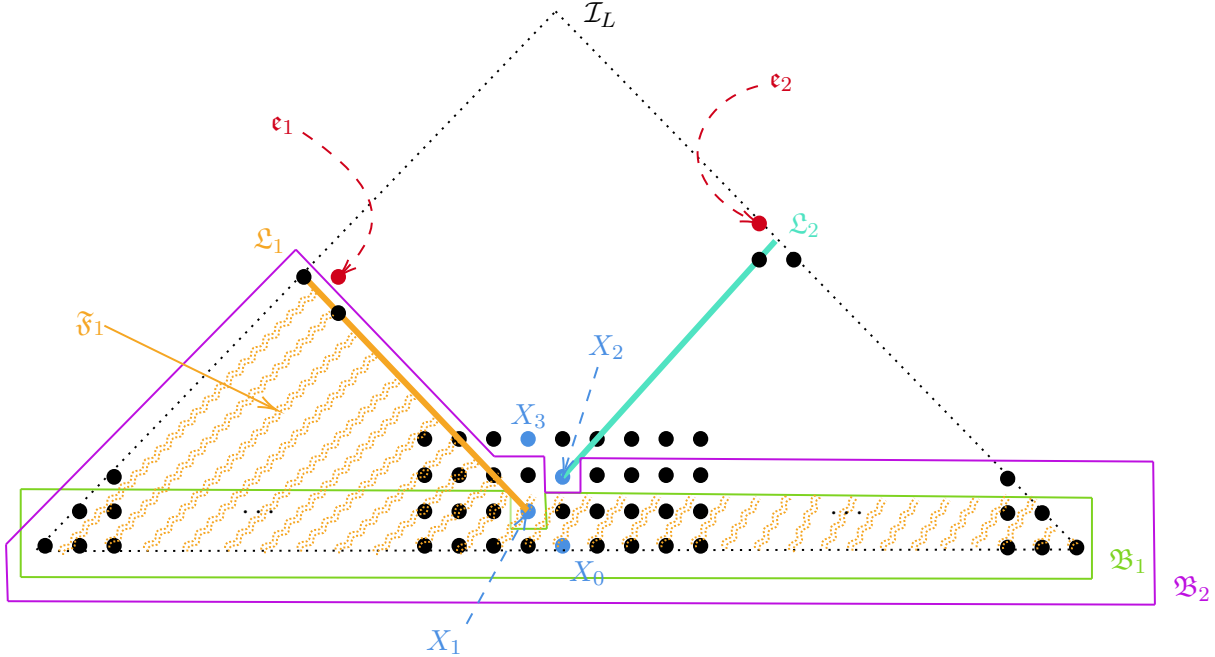


  \centering

\tikzset{every picture/.style={line width=0.75pt}} 


 
\caption{Construction of the martingale $(|u(X_j)|,\mathfrak{L}_j,\mathfrak{e}_j, X_{j+1},\mcZ_j)$.} 
\label{martingale UC graph in dim 2}
\end{figure}

We set the length of the martingale to be  $T=\frac{1}{10}L$. As in Section \ref{martingale section}, we define
\begin{equation}\label{mcY, UC}
  \mcY_j=\mcX_j-\frac{1}{2}j=\mcZ_1+\mcZ_2+\cdots + \mcZ_j-\frac{1}{2}j,\ 1\leq j\leq T.
\end{equation}
Then $(\mcY_j)_{1\leq j\leq T}$ becomes a submartingale. Applying Azuma's inequality yields
\begin{align}\label{Azuma UC}
  \P\left(\mcX_T \leq \frac{1}{10}T \right) & = \P\left(\mcY_T \leq -(\frac{1}{2}-\frac{1}{10})T \right) \leq \exp\{-T\}=\exp\{-\frac{1}{10}L\}.
\end{align}
Clearly, both the probabilistic estimate \eqref{Azuma UC} and the previous discussion   depend on the value of $\lambda$ and the initial data $u_0$. Furthermore, if the event $\mcX_T>\frac{1}{10}T$ happens, then there are more than $\frac{1}{10}T=\frac{1}{100}L$ many   $j$'s  such that 
\[|u(\mathfrak{e}_j)|\geq \frac{1}{2}\beta |u(X_j)|\geq \frac{1}{2}\beta\gamma^{\frac{1}{10}L} |u(0)|.\] 
That is to say, if we  define  
\[\Omega(\lambda,u_0,\mcI_L)= \ \#\left(  \partialside \mcI_L\cap \{ x:\  |u(x)| \geq  \frac{1}{2}\beta\gamma^{\frac{1}{10}L} |u(0)| \} \right) \ {\rm is \ larger \ than\ } \frac{L}{100},\]
then 
\begin{equation}\label{abundance in boundary of mcI L}
  \P\left(\Omega(\lambda,u_0,\mcI_L) \right) \geq \P\left(\mcX_T > \frac{1}{10}T \right)\geq 1-\exp\{-\frac{L}{10}\}.
\end{equation}
The  estimate \eqref{abundance in boundary of mcI L} is only about the randomness in $\mcI_L$.

Note that \eqref{abundance in boundary of mcI L} only provides a lower bound on the number of sites along the side boundary on which e $u$ is bounded from below. Define
 \[L_s=  \frac{L}{10}+10s, \ 0\leq s\leq \frac{L}{100}. \] 
Then the side boundaries $\partialside \mcI_{L_s},0\leq s\leq \frac{L}{100}$ are mutually disjoint, and are all contained in $Q_L(0)$. By \eqref{abundance in boundary of mcI L}, 
\begin{align*}
    \P\left(\bigcap_{1\leq s\leq \frac{L}{100}}\Omega(\lambda,u_0,\mcI_{L_s}) \right) \geq 1-\sum_{0\leq s\leq \frac{L}{100}}\exp\{-\frac{L_s}{10}\} \geq 1-\exp\{-\frac{L}{101}\},
\end{align*}
which means that for each $0\leq s\leq \frac{L}{100}$ and  in $\partialside \mcI_{L_s},$  there will be more than $\frac{L_s}{100}$  many sites such that 
\begin{equation}\label{lower dim 2 each boundary}
  |u(x)|\geq \frac{1}{2}\beta \gamma^{\frac{1}{10}L_s} |u(0)|\geq \frac{1}{2}\beta \gamma^{\frac{1}{50}L}|u(0)| \geq \exp \{{-\frac{1}{25}|\log \gamma|\cdot L}\}  |u(0)| 
\end{equation}
supposing $L\gg1$ l (depending only  on $h,\beta$). Therefore in $Q_L(0)$,  the total number of the sites satisfying \eqref{lower dim 2 each boundary} is bigger than 
\[  \sum_{ 0\leq s\leq \frac{L}{100}} \frac{L_s}{100}\geq  10^{-5}L^2 .\]
This implies  
\begin{equation}\label{increase dimension trick}
    \P \left(\mathcal{E}_{\mathrm{uc}}(L , u_0, \lambda)\right)\geq \P\left(\bigcap_{1\leq s\leq \frac{L}{100}}\Omega(\lambda,u_0,\mcI_{L_s}) \right) \geq 1-\exp\{-\frac{L}{1001}\}. 
\end{equation}
We have proven  Theorem \ref{UC for d arbitrary}  on $\Z^2$.

The argument presented above can be  extended to all dimensions $d\geq 3$.

Recall that we denote the standard basis of $\Z^d$ by $\{\bfe_1,\cdots,\bfe_d\}$, and a site $n\in\Z^d$ can be represented as 
\[n=(n_1,\cdots,n_d)=n_1\bfe_1+n_2\bfe_2+\cdots + n_d\bfe_d.\]
We give an order for element in $\Z^d$ via:  
\begin{equation}\label{order in Zd}
  n <n' \Leftrightarrow n_{i_0}<n_{i_0}'\ {\rm with} \ i_0:= \min \{i:\  n_i\neq n_i' \}. 
\end{equation}

\begin{Def}
For a fixed length $L\gg 1$, \textbf{the elementary propagating region} of length $L$ in $\Z^d$ is defined by 
\[\mcI_L := \left\{ (n_1,\cdots,n_d)\in \Z^d: \ n_d\geq 0 \ {\rm and }\ \sum_{1\leq i \leq d-1} |n_i|+ n_d \leq L+1    \right\}. \]   
\end{Def}
We define the following tilted hyperplanes in $\Z^d$:
 \[\mcT_k=\{n\in \Z^d:\  \langle n,\bfv_-\rangle = k\}, \ \bfv_-=\bfe_d-(\bfe_1+\bfe_2+\cdots+\bfe_{d-1}),\]
\[\mcT'_k=\{n\in \Z^d: \ \langle n,\bfv_+\rangle = k\}, \ \bfv_+=\bfe_d+(\bfe_1+\bfe_2+\cdots+\bfe_{d-1}).\]
Now for a site $n\in \mcT_k$ with $n_d\gg 1$, we may iterate the local form of  $H(\omega)u-\lambda u=0$ along the hyperplane $\mcT_k$ as follows. 
{\Small\begin{align}
\notag &u(n) \\
\notag  &= -\left(\sum_{1\leq i\leq d-1}u(n-\bfe_d -\bfe_i) \right) + (2d+V(n-\bfe_d)-\lambda)\cdot u(n-\bfe_d) -\left(\sum_{1\leq i\leq d-1} u(n-\bfe_d+\bfe_i)+u(n-2\bfe_d)\right)\\
\notag   &= (-1)^2 \cdot  \left(\sum_{1\leq i\leq d-1} \sum_{1\leq j\leq d-1} u(n-2\bfe_d-\bfe_i-\bfe_j)\right)\\
\notag   &\quad +(2d+V(n-\bfe_d)-\lambda)\cdot u(n-\bfe_d) +(-1)\cdot \left(\sum_{1\leq i\leq d-1}  (2d+V(n-2\bfe_d-\bfe_i)-\lambda)\cdot u(n-2\bfe_d -\bfe_i) \right) \\
 \notag  &\quad + \sum_{y\in \mcT_{k-2}} b_y^{(2)} \cdot u(y)\\
 \notag  & =\cdots \\
 \notag  &= (-1)^k \cdot  \left(\sum_{(i_1,i_2,\cdots,i_k) \atop 1\leq i_s \leq d-1 ,\forall 1\leq s\leq k}  u(n-k\bfe_d-\bfe_{i_1}-\bfe_{i_2}-\cdots-\bfe_{i_k})\right)\\
 \label{d dim expansion}  &\quad +\sum_{0\leq s\leq k-1}(-1)^{s}\cdot \left(\sum_{(i_1,i_2,\cdots,i_{s})  \atop 1\leq i_r \leq d-1 ,\forall 1\leq r\leq s }  (2d+V(n-(s+1)\bfe_d -\bfe_{i_1}-\cdots -\bfe_{i_s})-\lambda)\cdot u(n-(s+1)\bfe_d -\bfe_{i_1}-\cdots -\bfe_{i_s}) \right) \\
\notag   &\quad + \sum_{y\in \mcT_{k-2}} b_y^{(k)} \cdot u(y)\\
 \notag  &=\cdots\\
 \notag  &=\sum_{x\in \mcT_{k-1}} a_x \cdot (2d+V(x)-\lambda) \cdot u(x)  +\sum_{y\in \mcT_{k-2}} b_y \cdot u(y) + {\rm Remains}(u_0).
\end{align}}
From the above formula, it's easy to see that the coefficients  involving information on $\mcT_{k-1}$ will satisfy 
\[|a_x| {\rm \ is \ a \ binomial \ coefficient }\ \Rightarrow \ |a_x|\geq 1.\]
Similar discussion also works for $\mcT'_k$.

Now, choose the site $X_0=0\in \Z^d$, and take the $\sigma$-algebra
\[\mcB_0=\{\Omega,\emptyset\}\subset \mcF_0=\sigma (\{0\}), \mathfrak{F}_0=\{0\} . \]

\noindent \textbf{{(Step 1)}}\\

\noindent By the cone property Proposition \ref{cone property} and \eqref{cone property parameter}, there must be a site 
\[X_1\in  \mathfrak{C}(0,\bfe_2)=\{ \bfe_d\pm \bfe_i: \ 1\leq i\leq d-1 \}\cup\{ 2\bfe_d\}\]
such that $|u(X_1)|\geq \gamma |u(0)|$. Since we already know the value of $u=u_0$ in $\mcP_1$, we can judge that,  if $X_1$ can be in $\mcP_1$, we choose $X_1$ to be the smallest site under the order \eqref{order in Zd}; else, $X_1$ can only be chosen to be $2\bfe_2$. Thus, by our construction, {$X_1$ is $\mcF_0$-measurable} (deterministic).\\

\noindent \textbf{{(Step 2)}}\\

\noindent Define the $\sigma$-algebra 
\begin{equation}\label{mcB 1, d arbitrary}
  \mcB_1=\sigma\left( \mathfrak{B}_1 \right), \ \mathfrak{B}_1 =\mcI_L\cap \{0\leq n_2\leq (X_1)_d\} \setminus \{X_1\} .
\end{equation}
Now if the randomness in $\mathfrak{B}_1$ is fixed, then the value of $u$ in $(\mcI_L\cap \{0\leq n_2\leq (X_1)_d+1\}) \setminus \{X_1+\bfe_d\}$ will be determined. In particular, the value $|u(X_1)|$ will be determined, which means {$|u(X_1)|$ is $\mcB_1$-measurable}. Moreover, if we decompose 
\[\mathfrak{C}(X_1,\bfe_d)=A_1\cup B_1\cup\{C_1,D_1\}\]
with 
\[A_1=\{X_2+\bfe_d-\bfe_i :\ 1\leq i\leq d-1\}, \ B_1 = \{X_2+\bfe_d+\bfe_i :\ 1\leq i\leq d-1\},\ C_1=X_1+\bfe_d,D_1=X_1+2\bfe_d,\]
then the value of $u$ in $A_1$ and $B_1$ will also be determined.

By the cone property, there must be a site $X_2 \in  \mathfrak{C}(X_1,\bfe_2)$ such that $|u(X_2)|\geq \gamma |u(X_1)|$. We can judge that,  if $X_2$ can lie in $A_1\cup B_1$, we choose the smallest one under the order \eqref{order in Zd}; else,
we must have $X_2\in \{C_1,D_1\}$.

Next, we define the tilted hyperplane as follows. If $X_2\in A_1$ or $X_2\in \{C_1,D_1\}$, we take $\mathfrak{L}_1 =\mcT_{\langle X_1,\bfv_-\rangle}$. In this case, we have 
\[X_2-X_1\in \{\bfe_d-\bfe_i: \ 1\leq i\leq d-1 \} \cup \{\bfe_d,2\bfe_d\}\Rightarrow \langle X_2,\bfv_-\rangle  > \langle X_1,\bfv_-\rangle.\]
Moreover, if we consider the point $y\in \mathfrak{C}(X_2,\bfe_d)$, we have 
\[y-X_2\in \{\bfe_d \pm\bfe_i: \ 1\leq i\leq d-1 \} \cup \{\bfe_d,2\bfe_d\}\Rightarrow \langle y,\bfv_-\rangle  \geq  \langle X_2,\bfv_-\rangle.\]
Therefore, $X_1$ lies in $\mathfrak{L}_1$, but $X_2$ and the cone $\mathfrak{C}(X_2,\bfe_d)$ lie in the region strictly above $\mathfrak{L}_1$ (this ensures that the subsequent $X_3$ will also lie strictly above $\mathfrak{L}_1$). Else, if $X_2\in B_1$, we take $\mathfrak{L}_1 =\mcT'_{\langle X_1,\bfv_+\rangle}$.
Then, similarly, $X_1$ lies in $\mathfrak{L}_1$, but $X_2$ and the cone $\mathfrak{C}(X_2,\bfe_d)$ lie in the region strictly above $\mathfrak{L}_1$.

After determining $\mathfrak{L}_1$, we consider the set 
\[\mathfrak{e}'_1 = (\mathfrak{L}_1+\bfe_d) \cap \partialside \mcI_L,\]
where 
 \[ \partialside \mcI_L := \mcI_L \cap (\mcT_{L+1}\cup \mcT_{L}\cup \mcT'_{L+1}\cup \mcT'_{L} )\]
is a part of the side boundary of $\mcI_L$. We construct $\mathfrak{e}_1\subset \mathfrak{e}'_1$ as follows: For a $n\in \mathfrak{e}_1$, we can expand $u(n)$ (by the previous discussion) into   
\[u(n)= \sum_{x\in \mathfrak{L}_1} a_x \cdot (2d+V(x)-\lambda) \cdot u(x)  +\sum_{y\in \mathfrak{L}_1-\bfe_d} b_y \cdot u(y) + {\rm Remains}(u_0).\]
If the above formula involves the potential at $X_1$, i.e.,  $V(X_1)$, we let $n\in \mathfrak{e}_1$.
{Thus both $\mathfrak{L}_1$ and $\mathfrak{e}_1$ are $\mcB_1$-measurable.}


Since both $(\mathfrak{L}_1+\bfe_d)$ and $\partialside \mcI_L$ are (union of) $(d-1)$-dimensional hyperplanes, intuitively we can expect $\mathfrak{e}_1$ is a $(d-2)$-dimensional set. This can be strictly proven by the following combinatorial  geometry argument.
\begin{claim}\label{number of e1}
  We have 
  \[\# \mathfrak{e}_1 \gtrsim_d L^{d-2}. \]
\end{claim}
\begin{proof}[Proof of Claim \ref{number of e1}]
  Without loss of generality, we assume $\mathfrak{L}_1+\bfe_d=\mcT_{\langle X_1,\bfv_-\rangle+1}$. Clearly, if $n\in \mathfrak{L}_1$, then the closest sites near $n$ in $\mathfrak{L}_1$ are
  \[n \pm (\bfe_d+\bfe_i),\ 1\leq i\leq d-1.\]
  Moreover, we have 
  \[ \langle n \pm (\bfe_d+\bfe_i) , \bfv_+\rangle = \langle n , \bfv_+\rangle \pm 2. \]
  This will ensure that $ (\mathfrak{L}_1+\bfe_d) \cap \mcT'_L \neq \emptyset$ or  $ (\mathfrak{L}_1+\bfe_d) \cap \mcT'_{L+1} \neq \emptyset$, and thus $\mathfrak{e}'_1\neq \emptyset$. Without loss of generality,  we assume $(\mathfrak{L}_1+\bfe_d) \cap \mcT'_{L+1} \neq \emptyset$. We solve the equation 
  \begin{equation*}
    \begin{cases}
       \langle n , \bfv_+\rangle =L +1,\\
       \langle n , \bfv_-\rangle =\langle X_1 , \bfv_-\rangle+1,
    \end{cases}
  \end{equation*}
  which gives 
  \begin{equation}\label{top set}
    \mathfrak{e}'_1= \left\{n\in \mcI_L:\  n_1+\cdots+n_{d-1}=\frac{1}{2}(L- \langle X_1 , \bfv_-\rangle) \ {\rm and} \ n_d=\frac{1}{2}(L+ \langle X_1 , \bfv_-\rangle+2) \right\}.
  \end{equation}
  Now for any $y\in \mathfrak{e}'_1$, the term \eqref{d dim expansion} in our previous discussion about the expansion of $u(y)$ tells us that, the sites, which lie in $\mcP_{(X_1)_d}$ and are involved in the expansion, are of form
  \[ x_{(i_1,\cdots,i_s)}(y)\triangleq y-(k+1)\bfe_k -\bfe_{i_1}-\bfe_{i_2}\cdots -\bfe_{i_k}  \in \mcP_{(X_1)_d},\  1\leq i_1,\cdots,i_k\leq d-1.\]
  This forces 
   \[y_d-(k+1)=\frac{1}{2}(L+ \langle X_1 , \bfv_-\rangle+2)-k-1=(X_1)_d \Rightarrow k= \frac{1}{2}(L+ \langle X_1 , \bfv_-\rangle)-(X_1 )_d. \]
  Therefore, 
  \begin{equation}\label{bottom set}
      \langle x_{(i_1,\cdots,i_s)} , \bfe_1+\cdots+\bfe_{d-1} \rangle = \frac{1}{2}(L- \langle X_1 , \bfv_-\rangle) -k = \langle X_1,\bfe_1+\cdots+\bfe_{d-1} \rangle.
  \end{equation}
  Now we define the set
  \[\mathfrak{X}(y)=\{ x_{(i_1,\cdots,i_s)}(y):  \ 1\leq i_1,\cdots,i_k\leq d-1\}, \]
  which is the same as  
  \[\mathfrak{W}_k:= \{n\in\Z^d : \ n_1,n_2,\cdots,n_{d-1}\geq 0,  n_d=0,n_1+\cdots n_{d-1}=k\}\]
   up to a rigid transformation (first do a reflection, and then shift the origin $0$ to $y-(k+1)\bfe_k= {\rm Proj}_{\mcP_{(X_1)_d}}(y)$). We choose a special $y_0\in \mathfrak{e}'_1$ such that 
  \[\left|(y_0)_i -\frac{1}{2(d-1)}(L- \langle X_1 , \bfv_-\rangle  ) \right|\leq 1\ {\rm for}\ \forall 1\leq i\leq d-1,\]
  which can be seen as the center of $\mathfrak{e}'_1$. By a geometric intuition,  it is natural  to believe the center of 
  \[\{  n\in \mcI_L: \ n_d=(X_1)_d, n_1 +\cdots n_{d-1}= \langle X_1,\bfe_1+\cdots + \bfe_{d-1}\rangle  \}\]
  will be the center of $\mathfrak{X}(y_0)$ (See Figure  \ref{choose the center} for an  illustration in $\Z^3$), which is denoted  by  $x'$ and
  \[ (x')_d=(X_1)_d, \ \left|(x')_i -\frac{1}{(d-1)}\langle X_1,\ \bfe_1+\cdots+\bfe_{d-1} \rangle  \right|\leq 1,\ 1\leq i \leq d-1.  \]

  \begin{figure}[htbp]
    \centering

\tikzset{every picture/.style={line width=0.75pt}} 

\begin{tikzpicture}[x=0.75pt,y=0.75pt,yscale=-0.8,xscale=0.8]

\draw    (388.21,90.29) -- (469.45,329.29) ;
\draw  [dash pattern={on 0.84pt off 2.51pt}]  (388.21,90.29) -- (297.76,202) ;
\draw    (388.21,90.29) -- (169,329.29) ;
\draw    (388.21,90.29) -- (598.21,202) ;
\draw  [dash pattern={on 0.84pt off 2.51pt}]  (297.76,202) -- (598.21,202) ;
\draw  [dash pattern={on 0.84pt off 2.51pt}]  (297.76,202) -- (169,329.29) ;
\draw    (169,329.29) -- (469.45,329.29) ;
\draw    (598.21,202) -- (469.45,329.29) ;
\draw [color={rgb, 255:red, 189; green, 16; blue, 224 }  ,draw opacity=1 ][line width=1.5]    (510.21,155.29) -- (434.21,225.29) ;
\draw [color={rgb, 255:red, 208; green, 2; blue, 27 }  ,draw opacity=1 ]   (472.21,190.29) -- (562.21,237.29) ;
\draw [color={rgb, 255:red, 208; green, 2; blue, 27 }  ,draw opacity=1 ]   (472.21,190.29) -- (505.21,293.29) ;
\draw [color={rgb, 255:red, 208; green, 2; blue, 27 }  ,draw opacity=1 ]   (472.21,190.29) -- (393.21,293.29) ;
\draw [color={rgb, 255:red, 208; green, 2; blue, 27 }  ,draw opacity=1 ]   (393.21,293.29) -- (505.21,293.29) ;
\draw [color={rgb, 255:red, 80; green, 227; blue, 194 }  ,draw opacity=1 ][line width=1.5]    (393.21,293.29) -- (451.21,235.29) ;
\draw [color={rgb, 255:red, 208; green, 2; blue, 27 }  ,draw opacity=1 ] [dash pattern={on 4.5pt off 4.5pt}]  (451.21,235.29) -- (562.21,237.29) ;
\draw [color={rgb, 255:red, 208; green, 2; blue, 27 }  ,draw opacity=1 ] [dash pattern={on 4.5pt off 4.5pt}]  (472.21,190.29) -- (451.21,235.29) ;
\draw    (389.11,116.64) -- (387.42,66.93) ;
\draw [shift={(387.32,63.93)}, rotate = 88.06] [fill={rgb, 255:red, 0; green, 0; blue, 0 }  ][line width=0.08]  [draw opacity=0] (8.93,-4.29) -- (0,0) -- (8.93,4.29) -- cycle    ;
\draw    (480,267) -- (567.21,265.34) ;
\draw [shift={(570.21,265.29)}, rotate = 178.91] [fill={rgb, 255:red, 0; green, 0; blue, 0 }  ][line width=0.08]  [draw opacity=0] (8.93,-4.29) -- (0,0) -- (8.93,4.29) -- cycle    ;
\draw  [color={rgb, 255:red, 74; green, 144; blue, 226 }  ,draw opacity=1 ][fill={rgb, 255:red, 74; green, 144; blue, 226 }  ,fill opacity=1 ] (469.11,190.29) .. controls (469.11,188.57) and (470.5,187.18) .. (472.21,187.18) .. controls (473.93,187.18) and (475.32,188.57) .. (475.32,190.29) .. controls (475.32,192) and (473.93,193.39) .. (472.21,193.39) .. controls (470.5,193.39) and (469.11,192) .. (469.11,190.29) -- cycle ;
\draw  [color={rgb, 255:red, 74; green, 144; blue, 226 }  ,draw opacity=1 ][fill={rgb, 255:red, 74; green, 144; blue, 226 }  ,fill opacity=1 ] (426.11,257.29) .. controls (426.11,255.57) and (427.5,254.18) .. (429.21,254.18) .. controls (430.93,254.18) and (432.32,255.57) .. (432.32,257.29) .. controls (432.32,259) and (430.93,260.39) .. (429.21,260.39) .. controls (427.5,260.39) and (426.11,259) .. (426.11,257.29) -- cycle ;
\draw  [color={rgb, 255:red, 65; green, 117; blue, 5 }  ,draw opacity=1 ][fill={rgb, 255:red, 65; green, 117; blue, 5 }  ,fill opacity=1 ] (412.11,272.29) .. controls (412.11,270.57) and (413.5,269.18) .. (415.21,269.18) .. controls (416.93,269.18) and (418.32,270.57) .. (418.32,272.29) .. controls (418.32,274) and (416.93,275.39) .. (415.21,275.39) .. controls (413.5,275.39) and (412.11,274) .. (412.11,272.29) -- cycle ;
\draw [color={rgb, 255:red, 80; green, 227; blue, 194 }  ,draw opacity=1 ] [dash pattern={on 4.5pt off 4.5pt}]  (208.21,171.29) .. controls (247.81,141.59) and (390.32,202.05) .. (435.86,248.87) ;
\draw [shift={(437.21,250.29)}, rotate = 226.89] [color={rgb, 255:red, 80; green, 227; blue, 194 }  ,draw opacity=1 ][line width=0.75]    (10.93,-3.29) .. controls (6.95,-1.4) and (3.31,-0.3) .. (0,0) .. controls (3.31,0.3) and (6.95,1.4) .. (10.93,3.29)   ;

\draw (518,131) node [anchor=north west][inner sep=0.75pt]   [align=left] {\textcolor[rgb]{0.74,0.06,0.88}{$\mathfrak{e}'_1$}};
\draw (358,54) node [anchor=north west][inner sep=0.75pt]   [align=left] {$\bfe_3$};
\draw (545,276) node [anchor=north west][inner sep=0.75pt]   [align=left] {$\bfe_1+\bfe_2$};
\draw (449,171) node [anchor=north west][inner sep=0.75pt]   [align=left] {\textcolor[rgb]{0.29,0.56,0.89}{$y_0$}};
\draw (414,235) node [anchor=north west][inner sep=0.75pt]   [align=left] {\textcolor[rgb]{0.29,0.56,0.89}{$x'$}};
\draw (385,262) node [anchor=north west][inner sep=0.75pt]   [align=left] {\textcolor[rgb]{0.25,0.46,0.02}{$X_1$}};
\draw (181,175) node [anchor=north west][inner sep=0.75pt]   [align=left] {\textcolor[rgb]{0.31,0.89,0.76}{$\mathfrak{X}(y_0)$}};

\end{tikzpicture}

\caption{Illustration of the proof of Claim \ref{number of e1} for $d=3$.}
\label{choose the center}

  \end{figure}
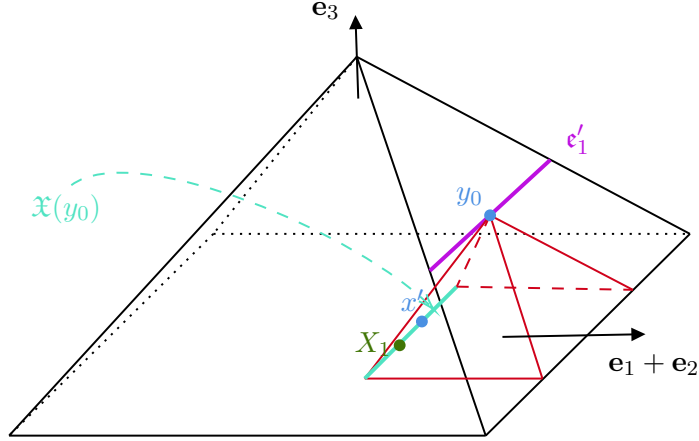

  So we compute  for $1\leq i\leq d-1$,
  \begin{align*}
    \left|(x'-y_0)_i +\frac{1}{(d-1)}k  \right|= \left|(x'-y_0)_i + \frac{1}{2(d-1)}(L-\langle X_1 , \bfv_-\rangle-2  \langle X_1,\bfe_1+\cdots+\bfe_{d-1} \rangle  ) \right| \leq 2 .
  \end{align*} 
  That proves $x'$ is the center of $\mathfrak{X}(y_0)$, since $\frac{k}{d-1}(\bfe_1+\cdots +\bfe_{d-1})$ is the center of $\mathfrak{W}_k$.

  Finally, assume $X_1\in Q_{\frac{L}{100d}}(0)$. Then we have 
  \begin{align}\label{X1 near center}
    |X_1-x'|_1\leq |X_1|_1+|x'|_1 \leq 2|X_1|_1\leq \frac{L}{50}\leq \frac{k}{10}. 
  \end{align} 
  That means $X_1$ is also near the center of $\mathfrak{X}(y_0)$. Moreover, using  \eqref{top set} and \eqref{bottom set}  ensures that if we shift $y_0$ by a vector $m\in \Z^d$ with 
  \[m_1+\cdots + m_{d-1}=0, \ y_0+m\in \mcI_L,\]
  then we still have $y_0+m\in \mathfrak{e}'_1$ and exactly 
  \[\mathfrak{X}(y_0+m)= \mathfrak{X}(y_0)+m.\]
  Since \eqref{X1 near center} ensures that for all $m\in \mathfrak{W}_\frac{k}{50}-{\rm center \ of \ } \mathfrak{W}_\frac{k}{50}$, $\mathfrak{X}(y_0)+m$ still contains $X_1$. Thus, 
  \[y_0 +m \in \mathfrak{e}_1\ {\rm for}\ \forall m\in \mathfrak{W}_{\frac{k}{50}}-{\rm center \ of \ } \mathfrak{W}_\frac{k}{50}. \]
  Therefore we conclude that (since $k\geq \frac{L}{5}$)
  \[\# \mathfrak{e}_1 \geq \# \mathfrak{W}_{L/250} = \binom{L/250+d-2}{d-2}\gtrsim_d L^{d-2}. \]

\end{proof}
\begin{rmk}\label{whole ej is d-2 dim}
 From the proof one readily sees that if the construction is iterated at most  for $\frac{L}{200d}$ many  times (i.e.,  $T\leq \frac{L}{200d}$), then $X_j \in Q_{\frac{L}{100d}}(0)\ {\rm for} \ \forall 1\leq j\leq T$. Thus the proof  also applies for obtaining  $\# \mathfrak{e}_j\gtrsim_d L^{d-2}\ {\rm for}\ \forall 1\leq j\leq T$.
\end{rmk}

Next, We construct the filtration element by 
\begin{equation}\label{mcF 1, d arbitrarily}
    \mcF_1=\sigma(\mathfrak{F}_1), \ \mathfrak{F}_1:= (\mathfrak{R}_1\cap \mcI_L )\cup \mathfrak{B}_1,
\end{equation}
where $\mathfrak{R}_1$ is the half space under $\mathfrak{L}_1$.
Our construction ensures that 
\begin{equation}\label{1 step relation}
 \mathfrak{B}_1 \subset \mathfrak{F}_1,  X_1 \notin \mathfrak{B}_1,   X_1\in \mathfrak{F}_1, X_2 \notin \mathfrak{F}_1. 
\end{equation}
{Therefore, $\mcB_1\subset \mcF_1$.} 

If we further condition on $\mcF_1$, clearly the value of $u$ in the region 
\[\{ x\in \mcI_L: \ x\in \mathfrak{F}_1  \ {\rm or} \ x-\bfe_d\in \mathfrak{F}_1 \}\]
will be determined, which contains the sites $C_1$ and the set $\mathfrak{e}_1$. Hence, we can do further discussion under the case $X_2\in \{C_1,D_1\}$: If $|u(C_1)|\geq \gamma |u(X_1)|$, then we let $X_2=C_1$; else, we let $X_2=D_1$. Thus the site $X_2$ is determined. 

Moreover, for any $n\in \mathfrak{e}_1$, the value $u(n)$ is now determined. By our definition of $\mathfrak{e}_1$, if we further condition on the randomness in $\mathfrak{F}_1\setminus \{X_1\}$, we can express $u(n),\forall n\in \mathfrak{e}_1$ in the following form:
\begin{equation}\label{u e1 value}
u(\mathfrak{e}_1) = a_{X_1} \cdot V(X_1)u(X_1)+ {\rm a\ deterministic \ value},\ |a_{X_1}|\geq 1. 
\end{equation}
Since $V(X_1)=V_{\text{hi}}(X_1)+ \beta \omega(X_1)$ and we have  the transversality at $X_1$:  
\[|u(X_1)|\geq \gamma |u(X_0)|=\gamma |u(0)|,\]
we can conclude that $|u(n)|<\frac{1}{2}\beta |u(X_1)|$ only holds for at most one value of $\omega(X_1)\in \{0,1\}$. 
Hence, by the pigeonhole principle,  
\begin{align*}
   &\P\left( \# \{ n\in \mathfrak{e}_1: \ |u(n)| \geq \frac{1}{2}\beta |u(X_1)| \}  \geq \frac{1}{2}\# \mathfrak{e}_1 \ \bigg|   \ \mcB_1 \right) \\
   &= \E \left(    \P_{\omega(X_1)}\left( \# \{ n\in \mathfrak{e}_1: \ |u(n)| \geq \frac{1}{2}\beta |u(X_1)| \}  \geq \frac{1}{2}\# \mathfrak{e}_1 \ \bigg|   \ \sigma(\mathfrak{F}_1 \setminus \{X_1\})  \right)  \bigg| \mcB_1 \right) \\
               &\geq \frac{1}{2}.
\end{align*}
Therefore, we define 
\begin{align*}
  \event_1:   & \ {\rm  the \ event \ }\# \{ n\in \mathfrak{e}_1:\  |u(n)| \geq \frac{1}{2}\beta |u(X_1)| \}  \geq \frac{1}{2}\# \mathfrak{e}_1, \\
  &\mcZ_1= \mathbf{1}_{\event_1}(\omega(X_1)).
\end{align*}
Then {$X_2,\mcZ_1$ are $\mcF_1$-measurable,} and 
\begin{equation*}
  \E(\mcZ_1 | \mcF_0)\geq \frac{1}{2}. 
\end{equation*}

\noindent \textbf{{(Step 3)}} $\cdots$\\

Iterate the above constructions  for  $T=\frac{L}{200d}$ amy  times. We obtain the following filtration
\[\mcB_0\subset \mcF_0\subset \mcB_1\subset\mcF_1\subset  \mcB_2\subset\mcF_2\cdots \subset  \mcB_T \subset \mcF_T,\]
together with a (site-mixed) martingale $(|u(X_j)|,\mathfrak{L}_j,\mathfrak{e}_j, X_{j+1},\mcZ_j),1 \leq j\leq T$ with the same properties of the martingale we constructed in the  $d=2$ case. The only tiny differences are 
\begin{itemize}
  \item  $\mathfrak{e}_j \subset \partialside \mcI_L \cap (\mathfrak{L}_j+\bfe_d)$ is now a set with cardinality $\# \mathfrak{e}_j\gtrsim_d L^{d-2}$ (by Remark \ref{whole ej is d-2 dim}). Moreover, $\mathfrak{e}_j \ (1\leq j\leq T)$ are mutually disjoint.
  \item The definition of $\mcZ_j$ is different. In higher dimensions, it is 
   \begin{align*}
  \mcZ_j= \mathbf{1}_{\event_j}(\omega(X_j)), 
\end{align*} 
where $\event_j$ is the event that $\# \{ n\in \mathfrak{e}_j:\  |u(n)| \geq \frac{1}{2}\beta |u(X_j)| \} \geq \frac{1}{2}\# \mathfrak{e}_j$. We still have 
\[E(\mcZ_{j+1}|\mcZ_j)\geq \frac{1}{2}.\]

\end{itemize}

Therefore, by  the  Azuma's inequality for  the submartingale, we can conclude that,  with probability larger than $1-\exp\{-T\}$, there are more than $\frac{1}{10}T=\frac{L}{2000d}$ many  $j$'s  such that 
\[\# \{ n\in \mathfrak{e}_j: |u(n)| \geq \frac{1}{2}\beta |u(X_j)| \geq \frac{1}{2}\beta\gamma^{\frac{1}{10}L} |u(0)| \} \geq \frac{1}{2}\# \mathfrak{e}_j \gtrsim_d L^{d-2} .\] 
That is to say, if we define  
\[\Omega(\lambda,u_0,\mcI_L):= \ \#\left(  \partialside \mcI_L\cap \{ x: \ |u(x)| \geq  \frac{1}{2}\beta\gamma^{\frac{1}{10}L} |u(0)| \} \right) \gtrsim_d \frac{L}{2000d}\cdot L^{d-2},\]
then 
\begin{equation}\label{abundance in boundary of mcI L, d arbitrary}
  \P\left(\Omega(\lambda,u_0,\mcI_L) \right) \geq \P\left(\mcX_T > \frac{1}{10}T \right)\geq 1-\exp\{-\frac{L}{200d}\}.
\end{equation}
The   estimate \eqref{abundance in boundary of mcI L, d arbitrary} is only about the randomness in $\mcI_L$.\\

Finally, by the  similar trick leading  to \eqref{increase dimension trick}, we take 
 \[L_s=  \frac{L}{10}+10s, \ 0\leq s\leq \frac{L}{100}. \] 
Then the side boundaries $\partialside \mcI_{L_s}\ (0\leq s\leq \frac{L}{100})$ are mutually disjoint, and are all contained in $Q_L(0)$. By \eqref{abundance in boundary of mcI L, d arbitrary}, the event 
$\cap_{1\leq s\leq \frac{L}{100}}\Omega(\lambda,u_0,\mcI_{L_s}) $ will imply that in $Q_L(0)$,  the total number of  sites with  
\begin{equation*}
  |u(x)|\geq \frac{1}{2}\beta \gamma^{\frac{1}{10}L_s} |u(0)| \geq \exp \{{-\frac{1}{25}|\log \gamma|\cdot L}\}  |u(0)|,
\end{equation*}
satisfies 
\[  \gtrsim_d \sum_{ 0\leq s\leq \frac{L}{100}} L_s^{d-1}\gtrsim_d L^d .\]
This yields  
\begin{equation*}
    \P \left(\mathcal{E}_{\mathrm{uc}}(L , u_0, \lambda)\right)\geq \P\left(\bigcap_{1\leq s\leq \frac{L}{100}}\Omega(\lambda,u_0,\mcI_{L_s}) \right) \geq 1-\exp\{-\frac{L}{2001d}\}.
\end{equation*}

We  have completed the proof  Theorem \ref{UC for d arbitrary} for arbitrary $d\geq 3$.

\end{proof}

\section{Some useful  lemmas}
In this section, we present several useful linear algebra lemmas  applied in our proofs. We begin with  the Weyl's inequality and its corollaries.
\begin{lem}[{\bf Weyl's inequality}]\label{Weyl inequality lemma}
  Let \( A, B \) be self-adjoint on some $n$-dimensional  inner product space \( V \)  
with all  eigenvalues ordered in the descending order  $\lambda_1(\cdot) \geq \cdots \geq \lambda_n(\cdot).$ Then for any   \(1 \leq i, j \leq n\), we have
\begin{equation}\label{Weyl inequality}
  \lambda_{i+j-1}(A+B) \leq \lambda_i(A) + \lambda_j(B) \leq \lambda_{i+j-n}(A+B),
\end{equation}
where we set  \(\lambda_k = -\infty\) if \(k > n\),  and \(\lambda_k = +\infty\) if \(k < 1\).
\end{lem}

The proof of  the Weyl's inequality follows directly from the min-max principle  and is omitted here. Two important corollaries of  the Weyl's inequality are:
\begin{cor}[{\bf Spectral stability}]\label{Spectral stability lemma}
For each \(k = 1, \cdots, n\), we have
\[
\lambda_k(A + B) - \lambda_k(A) \in [\lambda_n(B), \lambda_1(B)]
\]
and thus 
\[
|\lambda_k(A + B) - \lambda_k(A)| \leq \|B\|.
\]
\end{cor}

We define the spread of a self-adjoint matrix as the difffference between its largest and smallest eigenvalues. That is  
\begin{equation}\label{definition of spread}
  \spread(\cdot)=\lambda_1(\cdot)-\lambda_n(\cdot).
\end{equation}
\begin{cor}[{\bf Spread's lower bound}]\label{spread lemma}
 We have 
  \[\spread(A+B)\geq |\spread(A)-\spread(B)|.\]
\end{cor}
\begin{proof}[Proof of Corollary \ref{spread lemma}]
  Setting $i=n,j=1$ and then $i=1,j=n$ in \eqref{Weyl inequality} yields 
  \[\lambda_n(A+B)\leq \lambda_n(A)+\lambda_1(B)\leq \lambda_1(A+B),\ \lambda_n(A+B)\geq \lambda_1(A)+\lambda_n(B)\leq \lambda_1(A+B), \]
  which  gives
  \[-\spread(A+B)\leq \spread(A)-\spread(B)\leq \spread(A+B).\]
  This concludes the desired proof. 
\end{proof}

Finally,  we have 
\begin{lem}[{\bf Approximate orthogonality}]\label{orthogonality lemma}
Let $\bfv_1,\cdots ,\bfv_n$ be $n$ vectors satisfying 
\[\langle \bfv_i,\bfv_j \rangle =\delta_{i,j}+\mcO(\varepsilon).\]
If $0<n\varepsilon\ll 1$, then $\bfv_1,\cdots ,\bfv_n$ are linearly independent.
\end{lem}
\begin{proof}[Proof of Lemma \ref{orthogonality lemma}]
  Consider the Gram matrix $G$ of $\bfv_1,\cdots ,\bfv_n$. We have 
  \[G=(\langle \bfv_i,\bfv_j \rangle)_{n\times n} =I+\mcO(\varepsilon),\]
  where $I$ is the identity matrix of order $n$. 
  Then by Schur's  test, we have 
  \[ \| \mcO(\varepsilon) \| \lesssim \varepsilon\cdot n\ll 1.\]
  Thus, $G$ is invertible, which implies the  linear  independence of $\bfv_1,\cdots ,\bfv_n$.

\end{proof}

\section{A coupling lemma}

The following coupling lemma is standard in the multi-scale analysis, and one can find its proof in \cite[Lemma 3.1, Remark 3.6]{LSZ25}.

\begin{lem}\label{coupling lemma}
Fix $L_1\sim L^{\alpha },1\ll L\ll L' \leq \frac{1}{2} L_1$ ($\alpha>1$).  Let $\Lambda$ be a $L_1$-size block. Assume $\mathfrak{S} \subset \mathfrak{S}' \subset \Lambda\subset \Z^d$ satisfying 
\begin{itemize}
  \item $\mathfrak{S}$ is a union of no more than $N$ many $L$-size blocks $\Lambda_0'$,
  \item $\mathfrak{S}'$ is union of $L'$-size blocks $\Lambda_1'$, which is the $\frac{L'}{10}$-neighborhood of each $\Lambda_0'\in \mathfrak{S}$. Moreover, distinct elements in $\mathfrak{S}'$ are at distance $\gtrsim L'$.
\end{itemize}
 Assume there is  a class of $L $-size non-resonant blocks  
\[\mathfrak F=\{\Lambda':\ \Lambda'\subset \Lambda\}\]
covering $\Lambda\setminus \mathfrak{S}'$, and satisfying that for each $n\in \Lambda\setminus \mathfrak{S}$, there is a $\Lambda'\in \mathfrak F$ s.t.,  
      \begin{equation}\label{good block nhd}
        Q_{\frac{L}{5}}(n)\cap \Lambda \subset \Lambda'.
      \end{equation}
Assume further for $E\in\R$, 
  \begin{equation}\label{center block L2 estimate}
           \| G_{\Lambda_1'} (E)\| \leq  2\exp\{(L')^{1-\varepsilon}\} \ {\rm for} \ \forall \Lambda_{1}' \in \mathfrak{S}'. 
  \end{equation}
Then 
\begin{align}
\label{annuls L2 norm}  \| G_{\Lambda}(E)\| &<\exp \{  L_1^{1-\varepsilon}\}, \\
\label{annuls off-diagonal decay}  |G_{\Lambda}(x,y;E)| &<\exp \{-\frac12 \gamma_0 |x-y|\} \ {\rm for}\ \forall |x-y|\geq \frac{L_1}{200}. 
\end{align}
Moreover, If we denote $S=\cup_{\mathfrak{F}}\Lambda'$, then 
\begin{align}
\label{covering annuls L2 norm}  \| G_{S}(E)\| &<\exp \{  2 L^{1-\varepsilon}\}, \\
\label{covering annuls off-diagonal decay}  |G_{S}(x,y;E)| &<\exp \{-\frac12 \gamma_0 |x-y|\} \ {\rm for}\ \forall |x-y|\geq L. 
\end{align}
\end{lem}


\section*{Data Availability}
		The manuscript has no associated data.
		\section*{Declarations}
		{\bf Conflicts of interest} \ The authors  state  that there is no conflict of interest.

   \bibliographystyle{alpha}

\begin{thebibliography}{WSPZ89}

\bibitem[AM93]{AM93}
M.~Aizenman and S.~Molchanov.
\newblock Localization at large disorder and at extreme energies: an elementary
  derivation.
\newblock {\em Comm. Math. Phys.}, 157(2):245--278, 1993.

\bibitem[And58]{And58}
P.~W. Anderson.
\newblock Absence of diffusion in certain random lattices.
\newblock {\em Phys. Rev.}, 109(5):1492--1505, 1958.

\bibitem[AW15]{AW15}
M.~Aizenman and S.~Warzel.
\newblock {\em Random operators}, volume 168 of {\em Graduate Studies in
  Mathematics}.
\newblock American Mathematical Society, Providence, RI, 2015.
\newblock Disorder effects on quantum spectra and dynamics.

\bibitem[BGS02]{BGS02}
J.~Bourgain, M.~Goldstein, and W.~Schlag.
\newblock Anderson localization for {S}chr\"{o}dinger operators on {$\bold
  Z^2$} with quasi-periodic potential.
\newblock {\em Acta Math.}, 188(1):41--86, 2002.

\bibitem[BK05]{BK05}
J.~Bourgain and C.~E. Kenig.
\newblock On localization in the continuous {A}nderson-{B}ernoulli model in
  higher dimension.
\newblock {\em Invent. Math.}, 161(2):389--426, 2005.

\bibitem[BLMS22]{BLMS}
L.~Buhovsky, A.~Logunov, E.~Malinnikova, and M.~Sodin.
\newblock A discrete harmonic function bounded on a large portion of {$\Bbb
  Z^2$} is constant.
\newblock {\em Duke Math. J.}, 171(6):1349--1378, 2022.

\bibitem[Bou03]{Bou03}
J.~Bourgain.
\newblock Random lattice {S}chr\"{o}dinger operators with decaying potential:
  some higher dimensional phenomena.
\newblock In {\em Geometric aspects of functional analysis}, volume 1807 of
  {\em Lecture Notes in Math.}, pages 70--98. Springer, Berlin, 2003.

\bibitem[Bou04]{Bou04}
J.~Bourgain.
\newblock On localization for lattice {S}chr\"{o}dinger operators involving
  {B}ernoulli variables.
\newblock In {\em Geometric aspects of functional analysis}, volume 1850 of
  {\em Lecture Notes in Math.}, pages 77--99. Springer, Berlin, 2004.

\bibitem[Bou07]{Bou07}
J.~Bourgain.
\newblock Anderson localization for quasi-periodic lattice {S}chr\"{o}dinger
  operators on {$\Bbb Z^d$}, {$d$} arbitrary.
\newblock {\em Geom. Funct. Anal.}, 17(3):682--706, 2007.

\bibitem[CKM87]{CKM87}
R.~Carmona, A.~Klein, and F.~Martinelli.
\newblock Anderson localization for {B}ernoulli and other singular potentials.
\newblock {\em Comm. Math. Phys.}, 108(1):41--66, 1987.

\bibitem[CSZ23]{CSZ23}
H.~Cao, Y.~Shi, and Z.~Zhang.
\newblock Localization and regularity of the integrated density of states for
  {S}chr\"{o}dinger operators on {$\Bbb Z^d$} with {$C^2$}-cosine like
  quasi-periodic potential.
\newblock {\em Comm. Math. Phys.}, 404(1):495--561, 2023.

\bibitem[DLS85]{DLS85}
F.~Delyon, Y.~L\'{e}vy, and B.~Souillard.
\newblock Anderson localization for multidimensional systems at large disorder
  or large energy.
\newblock {\em Comm. Math. Phys.}, 100(4):463--470, 1985.

\bibitem[DS20]{DS20}
J.~Ding and C.~K. Smart.
\newblock Localization near the edge for the {A}nderson {B}ernoulli model on
  the two dimensional lattice.
\newblock {\em Invent. Math.}, 219(2):467--506, 2020.

\bibitem[DSS02]{DSS02}
D.~Damanik, R.~Sims, and G.~Stolz.
\newblock Localization for one-dimensional, continuum, {B}ernoulli-{A}nderson
  models.
\newblock {\em Duke Math. J.}, 114(1):59--100, 2002.

\bibitem[FMSS85]{FMSS85}
J.~Fr\"{o}hlich, F.~Martinelli, E.~Scoppola, and T.~Spencer.
\newblock Constructive proof of localization in the {A}nderson tight binding
  model.
\newblock {\em Comm. Math. Phys.}, 101(1):21--46, 1985.

\bibitem[FS83]{FS83}
J.~Fr\"{o}hlich and T.~Spencer.
\newblock Absence of diffusion in the {A}nderson tight binding model for large
  disorder or low energy.
\newblock {\em Comm. Math. Phys.}, 88(2):151--184, 1983.

\bibitem[FSW90]{FSW90}
J.~Fr\"{o}hlich, T.~Spencer, and P.~Wittwer.
\newblock Localization for a class of one-dimensional quasi-periodic
  {S}chr\"{o}dinger operators.
\newblock {\em Comm. Math. Phys.}, 132(1):5--25, 1990.

\bibitem[GK13]{GK13}
F.~Germinet and A.~Klein.
\newblock A comprehensive proof of localization for continuous {A}nderson
  models with singular random potentials.
\newblock {\em J. Eur. Math. Soc. (JEMS)}, 15(1):53--143, 2013.

\bibitem[IM16]{IM16}
J.~Z. Imbrie and R.~Mavi.
\newblock Level spacing for non-monotone {A}nderson models.
\newblock {\em J. Stat. Phys.}, 162(6):1451--1484, 2016.

\bibitem[Imb21]{Imb21}
J.~Z. Imbrie.
\newblock Localization and eigenvalue statistics for the lattice {A}nderson
  model with discrete disorder.
\newblock {\em Rev. Math. Phys.}, 33(8):Paper No. 2150024, 50, 2021.

\bibitem[Jit07]{Jit07}
S.~Jitomirskaya.
\newblock Ergodic {S}chr\"{o}dinger operators (on one foot).
\newblock In {\em Spectral theory and mathematical physics: a {F}estschrift in
  honor of {B}arry {S}imon's 60th birthday}, volume~76 of {\em Proc. Sympos.
  Pure Math.}, pages 613--647. Amer. Math. Soc., Providence, RI, 2007.

\bibitem[JLMS81]{JLMS81}
G.~Jona-Lasinio, F.~Martinelli, and E.~Scoppola.
\newblock New approach to the semiclassical limit of quantum mechanics. {I}.
  {M}ultiple tunnelings in one dimension.
\newblock {\em Comm. Math. Phys.}, 80(2):223--254, 1981.

\bibitem[JLMS84]{JLMS84}
G.~Jona-Lasinio, F.~Martinelli, and E.~Scoppola.
\newblock Quantum particle in a hierarchical potential with tunnelling over
  arbitrarily large scales.
\newblock {\em J. Phys. A}, 17(12):L635--L638, 1984.

\bibitem[JLMS85]{JLMS85}
G.~Jona-Lasinio, F.~Martinelli, and E.~Scoppola.
\newblock Multiple tunnelings in {$d$} dimensions: a quantum particle in a
  hierarchical potential.
\newblock {\em Ann. Inst. H. Poincar\'{e} Phys. Th\'{e}or.}, 42(1):73--108,
  1985.

\bibitem[Kir08]{Kir08}
W.~Kirsch.
\newblock An invitation to random {S}chr\"{o}dinger operators.
\newblock In {\em Random {S}chr\"{o}dinger operators}, volume~25 of {\em Panor.
  Synth\`eses}, pages 1--119. Soc. Math. France, Paris, 2008.
\newblock With an appendix by Fr\'{e}d\'{e}ric Klopp.

\bibitem[Kri08]{Kri08}
E.~Kritchevski.
\newblock Hierarchical {A}nderson model.
\newblock 2008.

\bibitem[Li22]{Li22}
L.~Li.
\newblock Anderson-{B}ernoulli localization at large disorder on the 2{D}
  lattice.
\newblock {\em Comm. Math. Phys.}, 393(1):151--214, 2022.

\bibitem[LSZ25a]{LSZ25b}
S.~Liu, Y.~Shi, and Z.~Zhang.
\newblock Extended states for the random {S}chr{\"o}dinger operator on
  $\mathbb{Z}^d\ (d\geq5)$ with decaying {B}ernoulli potential.
\newblock {\em arXiv:2505.04077}, 2025.

\bibitem[LSZ25b]{LSZ25}
S.~Liu, Y.~Shi, and Z.~Zhang.
\newblock On localization for the alloy-type {A}nderson-{B}ernoulli model with
  long-range hopping.
\newblock {\em arXiv:2508.12714}, 2025.

\bibitem[LZ22]{LZ22}
L.~Li and L.~Zhang.
\newblock Anderson-{B}ernoulli localization on the three-dimensional lattice
  and discrete unique continuation principle.
\newblock {\em Duke Math. J.}, 171(2):327--415, 2022.

\bibitem[MS85]{MS85}
F.~Martinelli and E.~Scoppola.
\newblock Remark on the absence of absolutely continuous spectrum for
  {$d$}-dimensional {S}chr\"{o}dinger operators with random potential for large
  disorder or low energy.
\newblock {\em Comm. Math. Phys.}, 97(3):465--471, 1985.

\bibitem[MS87]{MS87}
F.~Martinelli and E.~Scoppola.
\newblock Introduction to the mathematical theory of {A}nderson localization.
\newblock {\em Riv. Nuovo Cimento (3)}, 10(10):1--90, 1987.

\bibitem[SVW98]{SVW98}
C.~Shubin, R.~Vakilian, and T.~Wolff.
\newblock Some harmonic analysis questions suggested by {A}nderson-{B}ernoulli
  models.
\newblock {\em Geom. Funct. Anal.}, 8(5):932--964, 1998.

\bibitem[SW86]{SW86}
B.~Simon and T.~Wolff.
\newblock Singular continuous spectrum under rank one perturbations and
  localization for random {H}amiltonians.
\newblock {\em Comm. Pure Appl. Math.}, 39(1):75--90, 1986.

\bibitem[Weg81]{Weg81}
F.~Wegner.
\newblock Bounds on the density of states in disordered systems.
\newblock {\em Z. Phys. B}, 44(1-2):9--15, 1981.

\bibitem[WSPZ89]{WSPZ89}
D~W{\"u}rtz, T~Schneider, A~Politi, and M~Zannetti.
\newblock Renormalization group and dynamical maps for the hierarchical
  tight-binding problem.
\newblock {\em Physical Review B}, 39(11):7829, 1989.

\end{thebibliography}

\end{document}